\documentclass[sn-mathphys,pdflatex]{sn-jnl}

\usepackage[mathscr]{euscript}
\usepackage{dsfont,bm}
\usepackage{enumitem}
\usepackage{esint} 
\usepackage{mathtools}

\jyear{2021}%

\makeatletter
\newtheoremstyle{thmstyleone}
{18pt plus2pt minus1pt}
{18pt plus2pt minus1pt}
{\itshape}
{0pt}
{\bfseries}
{}
{.5em}
{\thmname{#1}\thmnumber{\@ifnotempty{#1}{ }\@upn{#2}}%
  \thmnote{ {\the\thm@notefont(#3)}}}
\newtheoremstyle{thmstylethree}
{18pt plus2pt minus1pt}
{18pt plus2pt minus1pt}
{\normalfont}
{0pt}
{\bfseries}
{}
{.5em}
{\thmname{#1}\thmnumber{\@ifnotempty{#1}{ }\@upn{#2}}%
  \thmnote{ {\the\thm@notefont(#3)}}}
\makeatother

\theoremstyle{thmstyleone}%
\newtheorem{theorem}{Theorem}[section]
\newtheorem{corollary}[theorem]{Corollary}
\newtheorem{lemma}[theorem]{Lemma}
\newtheorem{proposition}[theorem]{Proposition}

\theoremstyle{thmstylethree}%
\newtheorem{definition}[theorem]{Definition}
\newtheorem{assumption}[theorem]{Assumption}
\newtheorem{remark}[theorem]{Remark}
\newtheorem{example}[theorem]{Example}
\newtheorem{problem}[theorem]{Problem}
\newtheorem{framework}[theorem]{Framework}
\newtheorem{notation}[theorem]{Notation}
\newtheorem*{acknowledgements}{Acknowledgements}

\numberwithin{equation}{section}
\renewcommand{\figurename}{Figure}
\renewcommand{\bibname}{References}

\newcommand{\Capa}{\operatorname{Cap}}
\newcommand{\mr}[1]{{\tt \href{http://www.ams.org/mathscinet-getitem?mr=#1}{MR#1}}}
\newcommand{\arxiv}[1]{{\tt \href{http://arxiv.org/abs/#1}{arXiv:#1}}}
\newcommand{\giveset}[1]{\left\{ #1 \right\}}
\newcommand{\abs}[1]{{\left\lvert #1 \right\rvert}}
\newcommand\norm[1]{\left\lVert#1\right\rVert} 
\newcommand{\one}{\mathds{1}} 
\DeclareMathOperator*{\esssup}{ess\,sup}
\DeclareMathOperator*{\essinf}{ess\,inf}
\newcommand{\diam}{\mathop{{\rm diam}}\nolimits}
\newcommand{\supp}{\mathop{{\rm supp}}\nolimits}
\newcommand{\dcw}{d_{\operatorname{cw}}}
\newcommand{\on}[1]{\operatorname{#1}}

\newcommand{\homeo}{\operatorname{\mathrm{Homeo}}}
\newcommand{\id}[1]{\mathrm{id}_{#1}}
\newcommand{\GSC}{\mathrm{GSC}}
\newcommand{\interior}{\operatorname{int}}
\newcommand{\measure}{m}
\newcommand{\contfunc}{C}
\newcommand{\attainsss}{\mathcal{G}}

\def\sA {{\mathcal A}} \def\sB {{\mathcal B}} \def\sC {{\mathcal C}}
\def\sD {{\mathcal D}} \def\sE {{\mathcal E}} \def\sF {{\mathcal F}}
\def\sG {{\mathcal G}} \def\sH {{\mathcal H}} 
\def\sJ {{\mathcal J}}  \def\sL {{\mathcal L}}
 \def\sN {{\mathcal N}} 
  
\def\sS {{\mathcal S}}

 \def\bN {{\mathbb N}} 
  \def\bR {{\mathbb R}}
\def\bS {{\mathbb S}}  
  
 \def\bZ {{\mathbb Z}}

\newcommand{\ver}{Version of July 18, 2022} 

\begin{document}

\makeatletter
\renewenvironment{proof}[1][\proofname]{\par\removelastskip
  \pushQED{\qed}%
  \normalfont \topsep7.5\p@\@plus7.5\p@\relax%
  \trivlist%
  \item[\hskip\labelsep%
        \itshape%
    #1\@addpunct{}]\ignorespaces%
}{%
  \popQED\endtrivlist\@endpefalse%
}%
\makeatother

\title[On the conformal walk dimension]{\thispagestyle{plain}\begin{picture}(0,0)\put(-50,89.5){\normalsize\ver}\end{picture}On the conformal walk dimension: Quasisymmetric uniformization for symmetric diffusions\protect\footnotemark[0]}


\author*[1,2]{\fnm{Naotaka} \sur{Kajino}\protect\footnotemark[0]}\email{nkajino@kurims.kyoto-u.ac.jp}

\author[3]{\fnm{Mathav} \sur{Murugan}\protect\footnotemark[0]}\email{mathav@math.ubc.ca}

\affil[1]{\orgdiv{Department of mathematics, Graduate School of Science}, \orgname{Kobe University}, \orgaddress{\city{Kobe} \postcode{657-8501}, \country{Japan}}}

\affil*[2]{\orgdiv{Research Institute for Mathematical Sciences}, \orgname{Kyoto University}, \orgaddress{\city{Kyoto} \postcode{606-8502}, \country{Japan}} (current address)}

\affil[3]{\orgdiv{Department of Mathematics}, \orgname{University of British Columbia}, \orgaddress{\city{Vancouver}, \state{BC} \postcode{V6T 1Z2}, \country{Canada}}}


\abstract{\mathversion{normal}%
		We introduce the notion of conformal walk dimension, which
		serves as a bridge between elliptic and parabolic Harnack inequalities.
		The importance of this notion is due to the fact that,
		for a given strongly local, regular symmetric Dirichlet space
		in which every metric ball has compact closure (\emph{MMD space}),
		the finiteness of the conformal walk dimension characterizes
		the conjunction of the metric doubling property and the elliptic Harnack inequality.
		Roughly speaking, the conformal walk dimension of an MMD space is defined as
		the infimum over all possible values of the walk dimension with which the parabolic Harnack inequality can be made to hold
		by a time change of the associated diffusion and by a quasisymmetric change of the metric. 
		We show that the conformal walk dimension of any MMD space satisfying
		the metric doubling property and the elliptic Harnack inequality is two,
		and provide a necessary condition for a pair of such changes to attain the
		infimum defining the conformal walk dimension when it is attained by the original pair.
		We also prove a necessary condition for the existence of such a pair attaining the infimum in the setting of a self-similar Dirichlet form
		on a self-similar set, and apply it to show that the infimum
		fails to be attained for the Vicsek set and the $N$-dimensional Sierpi\'{n}ski gasket with $N\geq 3$,
		in contrast to the attainment for the two-dimensional Sierpi\'{n}ski gasket due to Kigami
		[\emph{Math.\ Ann.}\ \textbf{340} (2008), no.\ 4, 781--804].}




\renewcommand{\thefootnote}{}
\footnotetext{This work was supported by the Research Institute for Mathematical Sciences, an International Joint Usage/Research Center located in Kyoto University.}
\footnotetext{Naotaka Kajino was partially supported by JSPS KAKENHI Grant Numbers JP17H02849, JP18H01123.}
\footnotetext{Mathav Murugan was partially supported by NSERC and the Canada research chairs program.}
\renewcommand{\thefootnote}{\arabic{footnote}}
\setcounter{footnote}{0}

\maketitle

\makeatletter
\def\@oddhead{\hfill On the conformal walk dimension\qquad\thepage}
\def\@evenhead{\thepage\qquad N.\ Kajino and M.\ Murugan\hfill}
\makeatother

\tableofcontents

	\section{Introduction}
	\emph{What is the ``best'' way to parametrize a space?} 
	This vaguely stated question is the motivation for our work and several earlier works.
	By a parametrization, we mean a bijection $f:X \to M$ between the given space $X$ and another ``model space'' $M$ with more desirable properties.
	For example, the Riemann mapping theorem (or more generally, the uniformization theorem for Riemann surfaces) and geometric flows like the Ricci flow can be viewed as an attempt to answer the above question. In the Riemann mapping theorem example, $X$ is a proper simply connected domain in $\mathbb{C}$, $M$ is the unit disk, and $f$ is a conformal map.
	In the Ricci flow example, $X$ is a Riemannian manifold, $M$ is a Riemannian manifold with constant Ricci curvature, and  $f$ is a diffeomorphism. 
	This work aims to formulate and answer this question for spaces satisfying Harnack inequalities. In this work, $X$ is a space equipped with a symmetric diffusion
	that satisfies the elliptic Harnack inequality, $M$ satisfies the stronger parabolic Harnack inequality and $f$ is a quasisymmetry along with a time change of the diffusion on $X$ (quasisymmetry is an analogue of conformal maps for metric spaces).
	
	This paper uses quasiconformal geometry and time change of diffusion processes to understand the relationship between elliptic and parabolic Harnack inequalities. 
	The analysis using quasiconformal geometry also leads to a natural uniformization problem for spaces satisfying the elliptic Harnack inequality.
	Our results can be viewed as a bridge between analysis in smooth and fractal spaces and also as a bridge between elliptic and parabolic Harnack inequalities.

	We informally describe the setup and results.
	A more precise treatment is given in Section \ref{s:fr}.
	The setup of this work is a metric space equipped with a Radon measure $m$
	with full support and an $m$-symmetric diffusion process.
	Equivalently, we consider a metric space $(X,d)$ equipped with such $m$ and
	a strongly local, regular symmetric Dirichlet form $(\sE,\sF)$ on $L^2(X,m)$.
	We always assume that $B(x,r):=B_{d}(x,r):=\{y\in X \mid d(x,y)<r\}$ has compact closure in $X$
	for any $(x,r)\in X\times(0,\infty)$ and that $X$ contains at least two elements, and call $(X,d,m,\sE,\sF)$
	a \emph{metric measure Dirichlet space} or an \emph{MMD space} for short.
	Associated to an MMD space $(X,d,m,\sE,\sF)$ is a non-negative self-adjoint operator $\sL$ on $L^2(X,m)$ such that the corresponding Markov semigroup
	$(P_t)_{t \ge 0}$ is given by $P_t= e^{-t \sL}$. The operator $\sL$ is called the \emph{generator} of $(X,d,m,\sE,\sF)$, which
	is an analog of the Laplace operator in the abstract setting of MMD spaces. We refer to \cite{FOT, CF} for the theory of Dirichlet forms.

	We recall that this setup includes Brownian motion on a Riemannian manifold, where $d$ is the  Riemannian distance function, $m$ is the Riemannian measure,  $\sF$ is the Sobolev space $W^{1,2}$, and $\sE(f,f)= \int \lvert \nabla f \rvert^2 \, dm$, where $\nabla$ denotes the Riemannian gradient.
	In this case, the corresponding generator $\sL$ is the Laplace--Beltrami operator with a minus sign (so that $\sL$ is non-negative definite). This setup also covers non-smooth settings like diffusions on fractals including the Sierpi\'{n}ski gasket and the Sierpi\'{n}ski carpet.
	We refer the reader to \cite{Bar98} for an introduction to diffusions on fractals. Random walks on graphs can also be studied in this framework because the corresponding cable processes share many properties with the original random walks (see \cite{BB04} for this approach).

	An MMD space has an associated sheaf of \emph{harmonic} and \emph{caloric} functions. Roughly speaking, harmonic and caloric functions are generalization of solutions to the ``Laplace equation'' $\Delta h \equiv 0$ and the ``heat equation'' $\partial_t u - \Delta u \equiv 0$, respectively.
	Let $\sL$ denote the generator of an MMD space $(X,d,m,\sE,\sF)$. Let $h : U \to \bR$ be a measurable function in an open set $U$. We say that $h$ is harmonic in $U$, if it satisfies $\sL h \equiv 0$ in $U$ interpreted in a weak sense.
	Similarly, we say that a space-time function $u : (a,b) \times U \to \bR$ is caloric in $(a,b) \times U$ if it satisfies the ``heat equation'' $\partial_t u + \sL u \equiv 0$
	interpreted in a weak sense. 
	
	Harnack inequalities are fundamental regularity estimates that have numerous applications in partial differential equations and probability theory. We refer to \cite{Kas} for a nice survey on Harnack inequality and its variants.
	We recall the (scale-invariant) elliptic and parabolic Harnack inequalities.
	We say that an MMD space $(X,d,m,\sE,\sF)$ satisfies the \emph{elliptic Harnack inequality} (abbreviated as EHI),
	if there exist $C >1$ and $\delta \in (0,1)$ such that for all $x \in X$, $r>0$ and for any non-negative harmonic function $h$ on the ball $B(x,r)$, we have
	\begin{equation} \label{EHI-int} \tag*{EHI}
	\esssup_{B(x,\delta r)} h \le C \essinf_{B(x,\delta r)} h.
	\end{equation}
	We say that  an MMD space $(X,d,m,\sE,\sF)$ satisfies the \emph{parabolic Harnack inequality} with walk dimension $\beta>0$ (abbreviated as $\on{PHI(\beta)}$), if there exist $0<C_1< C_2 < C_3 < C_4 <\infty$, $C_5>1$ and $\delta \in (0,1)$ such that for all $x \in X$, $r>0$ and for any non-negative bounded caloric function $u$ on the space-time cylinder $Q=(a,a+C_4 r^\beta) \times B(x,r)$, we have
	\begin{equation} \label{PHI-int} \tag*{$\on{PHI(\beta)}$}
	\esssup_{Q_-} u \le C_5 \essinf_{Q_+} u,
	\end{equation}
	where $Q_-=(a+C_1 r^\beta,a+C_2 r^\beta) \times B(x,\delta r)$ and  $Q_+=(a+C_3 r^\beta,a+C_4 r^\beta) \times B(x,\delta r)$.
	
	We briefly review some earlier works on Harnack inequalities, referring the reader to \cite{Kas} for a more detailed survey of the literature. In a series of celebrated works, Moser showed EHI and PHI(2) for uniformly elliptic divergence form operators on $\mathbb{R}^n$ \cite{Mo1,Mo2}.
	Cheng, Li and Yau obtained gradient estimates that imply \ref{EHI-int} and PHI(2) for Riemannian manifolds with non-negative Ricci curvature \cite{CY,LY,Yau}. Grigor'yan and Saloff-Coste independently characterized PHI(2) using the volume doubling property and the Poincar\'e inequality \cite{Gri,Sal}.
	This characterization was extended to $\on{PHI(\beta)}$ by Barlow, Bass and Kumagai \cite{BB04, BBK}. A similar characterization of the simpler elliptic Harnack inequality remained open until recently \cite{Bas13,BM,BCM}.

	Note that every harmonic function lifts to a caloric function.
	More precisely, if $h$ is harmonic on $B(x,r)$, then $u(t,x)= h(x)$ is caloric on $(a,b) \times B(x,r)$ for all $b >a$. This lift immediately shows that%
	\footnote{To be precise, while for various purposes it is convenient to formulate
	\ref{EHI-int} without assuming the boundedness of non-negative harmonic functions $h$,
	we follow for simplicity the formulation of \ref{PHI-int} in \cite[Subsection 3.1]{BGK}
	which a priori requires the boundedness of non-negative caloric functions $u$, and then
	\eqref{e:p-ehi} is obvious only under the extra assumption of the boundedness of $h$.
	It turns out that this extra assumption can be dropped, but the proof of
	this fact is non-trivial; see Theorem \ref{t:phichar} and its proof for details.}
	\begin{equation} \label{e:p-ehi}
	\textrm{\ref{PHI-int}} \implies \textrm{\ref{EHI-int}}, \quad \mbox{for all $\beta>0$.}
	\end{equation}
	However, the converse of the above implication fails. Indeed, Delmotte has constructed an example of a space that satisfies EHI but fails to satisfy $\on{PHI(\beta)}$ for any $\beta>0$ \cite{Del} (see also \cite{BCM}).
	Nevertheless, one can characterize the elliptic Harnack inequality in terms of the parabolic Harnack inequality \cite{BM,BCM}.

	The main idea behind the characterization of \ref{EHI-int} is to reparametrize the space and time of the associated diffusion process so that it satisfies \ref{PHI-int} for some $\beta>0$.
	In the theory of regular symmetric Dirichlet forms the \emph{Revuz correspondence} provides a bijection between the time changes of the process and the family of \emph{smooth measures}.
	Roughly speaking, smooth measures are Radon measures that do not charge any set of capacity zero. If $\mu$ is a smooth measure for an MMD space $(X,d,m,\sE,\sF)$, then it defines a ``time-changed'' Dirichlet space $(X,d,\mu,\sE^\mu,\sF^\mu)$ as well as a ``time-changed'' Markov process.
	We say that a measure $\mu$ is \emph{admissible}, if $\mu$ is a smooth measure and has full quasi-support for the Dirichlet form $(\sE,\sF)$, which amounts to saying that $\mu$ represents a ``time change'' keeping the form $\sE$ essentially unchanged (see Definitions \ref{d:admissible} and \ref{d:TCDF}).
	We denote the collection of admissible measures by $\sA(X,d,m,\sE,\sF)$.
	Next, we recall the definition of conformal gauge.
	\begin{definition}[Conformal gauge] \label{d:cgauge}
			Let $(X,d)$ be a metric space and $\theta$ be another metric on $X$. We say that $d$ is \emph{quasisymmetric} to $\theta$, if there exists a homeomorphism $\eta:[0,\infty) \to [0,\infty)$ such that
			\[
			\frac{\theta(x,a)}{\theta(x,b)} \le \eta\left(\frac{d(x,a)}{d(x,b)}\right)   \quad \mbox{for all triples of points $x,a,b \in X$, $x \neq b$.}
			\]
			The \emph{conformal gauge} of a metric space $(X,d)$ is defined as
			\begin{equation} \label{e:cgauge-dfn}
			\sJ(X,d):= \{ \theta \colon X \times X \to [0,\infty) \mid \mbox{$\theta$ is a metric on $X$, $d$ is quasisymmetric to $\theta$} \}.
			\end{equation}
			By \cite[Proposition 10.6]{Hei}, being quasisymmetric is an equivalence relation among metrics. That is,
			\begin{equation} \label{e:cgauge}
			 d\in\sJ(X,d) \qquad \mbox{and} \qquad
			\sJ(X,\theta)= \sJ(X,d) \quad \mbox{for all $\theta \in \sJ(X,d)$}.
			\end{equation}
	\end{definition}
	The notion of quasisymmetry is an extension of conformal map to the context of metric spaces. 
	Quasisymmetric maps on the real line  were introduced by Beurling and Ahlfors, and were studied as boundary values of quasiconformal self-maps of the upper half-plane \cite{BA}.
	The above definition on general metric spaces is due to Tukia and V\"ais\"al\"a \cite{TV}.
	This is the reason behind the terminology ``conformal gauge''. We refer to \cite{Hei, HK} for expositions of the theory of quasisymmetric maps and quasiconformal geometry on metric spaces.

	To characterize the elliptic Harnack inequality, we reparametrize the space by choosing a new metric in the {conformal gauge} of $(X,d)$ and we reparametrize time by choosing a new symmetric measure that is {admissible}. More precisely, 
	given an MMD space $(X,d,m,\sE,\sF)$ satisfying \ref{EHI-int}, we seek to find a metric $\theta \in \sJ(X,d)$ and a measure $\mu \in \sA(X,d,m,\sE,\sF)$ such that
	the corresponding time-changed MMD space $(X,\theta,\mu,\sE^\mu,\sF^\mu)$ equipped with the new metric $\theta$ satisfies \ref{PHI-int} for some $\beta>0$.
	In other words, we seek to \emph{upgrade} \ref{EHI-int} to \ref{PHI-int} by reparametrizing space and time.
	This motivates the notion of conformal walk dimension.
	\begin{definition}[Conformal walk dimension] \label{dfn:dcw}
			The \emph{conformal walk dimension} $\dcw$ of an MMD space $(X,d,m,\sE,\sF)$ is defined as
			\begin{equation} \label{e:dcw}
			d_{\operatorname{cw}}:=\inf \biggl\{ \beta >0 \biggm\vert
			\begin{minipage}{210pt}
			there exist $\mu \in \sA(X,d,m,\sE,\sF)$ and  $\theta \in \sJ(X,d)$ 
			such that $(X,\theta,\mu,\sE^\mu,\sF^\mu)$ satisfies \ref{PHI-int}
			\end{minipage}
			\biggr\},
			\end{equation}
			where $\inf \emptyset := \infty$ and $(\sE^\mu,\sF^\mu)$ denotes the time-changed Dirichlet form on $L^2(X,\mu)$.
	\end{definition}
	We remark that, if $(X,d,m,\sE,\sF)$ satisfies \ref{PHI-int}, then it is easy to see that for any $\alpha \in (0,1]$ the MMD space $(X,d^{\alpha},m ,\sE,\sF)$ satisfies $\on{PHI}(\beta/\alpha)$ and $d^{\alpha} \in \sJ(X,d)$.
	This shows that it is easy to increase the walk dimension by changing the metric to a different one in the conformal gauge, but it is non-trivial to decrease the walk dimension. This explains the ``infimum'' in \eqref{e:dcw}.
	Another remark, which is based on an observation in \cite[Section 1]{Hin02}, is that the lower bound
	\begin{equation} \label{e:dcw-lower}
	d_{\operatorname{cw}} \geq 2
	\end{equation}
	is essentially known to experts; indeed, \eqref{e:dcw-lower} can be obtained from the so-called
	Varadhan-type Gaussian off-diagonal asymptotics of the associated Markov semigroup due to
	\cite[Theorem 2.7]{AH} combined with the characterization of \ref{PHI-int} for $\beta > 1$
	by the \emph{volume doubling property} \ref{VD} and
	the \emph{heat kernel estimates} \hyperlink{hke}{$\on{HKE(\beta)}$} with walk dimension $\beta$
	(see Definitions \ref{d:vd-rvd}, \ref{d:HKE} and Theorem \ref{t:phichar}).%
	\footnote{Since
	$(\mathcal{E},\mathcal{F})$ satisfies the locality assumption in \cite{AH}
	by its strong locality and \cite[Theorem 2.4.3]{CF}, we can apply \cite[Theorem 2.7]{AH},
	which the conjunction of \ref{VD} and \hyperlink{hke}{$\on{HKE(\beta)}$} with $\beta \in (1,2)$ would contradict
	in view of the finiteness of the limit in \cite[Theorem 2.7]{AH} implied by \cite[Proposition 5.1]{AH}
	and \hyperlink{hke}{$\on{HKE(\beta)}$}. Therefore the conjunction of
	\ref{VD} and \hyperlink{hke}{$\on{HKE(\beta)}$} for $\beta\in(1,2)$ cannot hold,
	hence neither can $\on{PHI(\beta/\alpha)}$ for any $\beta\in(0,2)$ and any
	$\alpha\in(\beta/2,\beta)\cap(0,1]$ by Theorem \ref{t:phichar} and thus
	neither can \ref{PHI-int} for any $\beta\in(0,2)$ by the
	previous remark, proving \eqref{e:dcw-lower}.
	\endgraf In Lemma \ref{l:dcw-lowerbound}, we will give an alternative proof of \eqref{e:dcw-lower} based on a relatively simple result from \cite{Mur1}.}

	Two natural questions arise. What is the value of $\dcw$? When is the infimum in \eqref{e:dcw} attained?  The answer to the first question is given below.
	We assume that our metric space $(X,d)$ satisfies the \emph{metric doubling property}, i.e.,
	admits $N \in \bN$ such that for all $x \in X$ and $r>0$, the ball $B(x,r)$ can be covered by $N$ balls of radii $r/2$.
	
	Our first main result (Theorem \ref{t:ehichar}) is that the value of the conformal walk dimension is always two,
	i.e., we always have the equality in \eqref{e:dcw-lower}, for any MMD space
	satisfying the metric doubling property and \ref{EHI-int}.
	In other words, we have the equivalence among the following three conditions,
	sharpening the existing characterization of \ref{EHI-int}
	(more precisely, its conjunction with the metric doubling property):
	\begin{enumerate}[label=\textup{(\alph*)},align=left,leftmargin=*,topsep=5pt,parsep=0pt,itemsep=2pt]
		\item \label{it:ehichar-intro-a}$(X,d,m,\sE,\sF)$ satisfies the metric doubling property and \ref{EHI-int}.
		\item \label{it:ehichar-intro-b}$\dcw <\infty$. 
		\item \label{it:ehichar-intro-c}$\dcw=2$.
	\end{enumerate}
	
	The equivalence between \ref{it:ehichar-intro-a} and \ref{it:ehichar-intro-b} is contained in \cite{BM,BCM}.
	That \ref{it:ehichar-intro-c} implies \ref{it:ehichar-intro-b} is obvious.
	Our contribution to the above equivalence is the proof that \ref{it:ehichar-intro-a} implies \ref{it:ehichar-intro-c}.
	Therefore our result sharpens the characterization of \ref{EHI-int} in \cite{BM,BCM}.
	The result that \ref{it:ehichar-intro-a} implies \ref{it:ehichar-intro-c} is particularly interesting on fractals as we explain below. Diffusions on many regular fractals are known to satisfy \ref{PHI-int} with $\beta>2$.
	These are often called \emph{anomalous diffusions} to distinguish them from the classical smooth settings like the Euclidean space where one often has Gaussian space-time scaling and $\on{PHI(2)}$.
	However, by the above equivalence one can ``improve'' from \ref{PHI-int} to $\on{PHI(2 + \varepsilon)}$ for any $\varepsilon>0$ even on fractals.
	So this result serves as a bridge between anomalous space-time scaling in fractals and Gaussian space-time scaling seen in smooth settings.
	
	It is worth mentioning that the proof that \ref{it:ehichar-intro-a} implies \ref{it:ehichar-intro-b} in \cite{BM,BCM} does not give a universal upper bound for $\dcw$.
	The bound on $\dcw$ obtained there depends on the constants in \ref{EHI-int} and could be arbitrarily large.
	To improve the previous \ref{it:ehichar-intro-a}-implies-\ref{it:ehichar-intro-b} result to the \ref{it:ehichar-intro-a}-implies-\ref{it:ehichar-intro-c} result, we need a new construction of metrics and measures.
	
	We briefly discuss this new construction in the proof that \ref{it:ehichar-intro-a} implies \ref{it:ehichar-intro-c}, which we will achieve in Section \ref{sec:dcw-proof}.
	The inspiration behind our argument is the uniformization theorem for Riemann surfaces.
	In the proof of the uniformization theorem, the Green's function of a Riemann surface (or a subset of the surface) plays an essential role in constructing the uniformizing map \cite[Chapter 15]{Mar19}.
	We use certain cutoff functions across annuli with small Dirichlet energy at different scales and locations as a substitute for the Green's function. 
	It is helpful to think of these cutoff functions as equilibrium potentials across annuli.
	Roughly speaking, the  diameter of a ball under the new metric $\theta \in \sJ(X,d)$ for our construction is
	proportional to the average gradient of the equilibrium potential chosen at a suitable location and scale.

	On a technical level, our proof relies heavily on the theory of Gromov hyperbolic spaces. We view $X$ as the boundary of a Gromov hyperbolic space called the hyperbolic filling.
	The conformal gauge of $X$ is essentially in a bijective correspondence to the bi-Lipschitz changes of the metric on the hyperbolic filling.
	A desired bi-Lipshitz change of the metric on the hyperbolic filling is constructed using equilibrium potentials as described above.
	A major ingredient in the proof is a combinatorial description of the conformal gauge due to Carrasco Piaggio \cite{Car13}, which we will adapt for our purpose in Section \ref{sec:HypFill}.
	
	Our first main result described above (Theorem \ref{t:ehichar}) is a partial converse to the trivial implication \ref{PHI-int}$\implies$\ref{EHI-int} in \eqref{e:p-ehi}.
	The equivalence between \ref{it:ehichar-intro-a} and \ref{it:ehichar-intro-c} clarifies the extent to which the converse of this trivial implication holds.
	Although the value of $\dcw$ has a 
	simple description, the following questions remain open in general. 
	\begin{problem} \label{prb:attainment-Gaussian-unif}
		Given an MMD space $(X,d,m,\sE,\sF)$ that satisfies the metric doubling property and \ref{EHI-int}:
		\begin{enumerate}[label=\textup{(\arabic*)},align=left,leftmargin=*,topsep=5pt,parsep=0pt,itemsep=2pt]
			\item (Attainment problem) Determine whether the infimum in \eqref{e:dcw} is attained.
			\item (Gaussian uniformization problem) Describe all the pairs $(\theta,\mu)$ of metrics $\theta \in \sJ(X,d)$
			and measures $\mu \in \sA(X,d,m,\sE,\sF)$ such that the corresponding time-changed MMD space
			$(X,\theta,\mu,\sE^\mu,\sF^\mu)$ satisfies $\on{PHI(2)}$.
		\end{enumerate}
	\end{problem}
	
	We describe two examples of self-similar fractals for which  a positive answer to the attainment problem is known.
	Kigami has shown in \cite{Kig08} that the MMD space corresponding to the Brownian motion on the two-dimensional Sierpi\'{n}ski gasket attains the infimum, where $\mu$ is the Kusuoka measure and $\theta$ is the associated intrinsic metric.
	Further examples of admissible measures that attain the infimum for the two-dimensional Sierpi\'{n}ski gasket is described in \cite{Kaj12};
	see Theorem \ref{thm:SG2-attained} below and the references in its proof for details.
	In retrospect, Kigami's measurable Riemannian structure on the Sierpi\'{n}ski gasket is the first evidence towards the implication \ref{it:ehichar-intro-a}$\implies$\ref{it:ehichar-intro-c} in Theorem \ref{t:ehichar}.
	Another example of a fractal that attains the infimum in \eqref{e:dcw} is the two-dimensional snowball described in \cite{Mur}. The ``snowball'' fractal can be viewed as a limit of Riemann surfaces and is a two-dimensional analog of the von Koch snowflake.
	In this example, the answer to the attainment problem is  obtained by considering a limit of uniformizing maps to $\bS^2$ and using the conformal invariance of Brownian motion.
	Our terminology ``Gaussian uniformization problem'' is inspired by this example and the classical fact
	from \cite{Stu96} (see also Proposition \ref{p:metmeas} and Theorem \ref{t:phichar} below) that $\on{PHI(2)}$ is equivalent to Gaussian heat kernel estimates.

	Nevertheless, the infimum in \eqref{e:dcw} need not be attained in general. We show in Subsection \ref{ssec:examples} that the Vicsek set and the $N$-dimensional Sierpi\'{n}ski gasket with $N \geq 3$ fail to attain the infimum in \eqref{e:dcw}.
	The examples with non-attainment of $\dcw$ rely on the following result (Theorems \ref{t:attain-harmonic-func} and \ref{t:attain-harmonic-func-GSC}).
	For a ``regular'' self-similar fractal, if the infimum in \eqref{e:dcw} is attained
	by a quasisymmetric metric $\theta$ and an admissible measure $\mu$,
	then it is possible to choose $\mu$ as the energy measure of a function that is harmonic outside a canonical boundary.
	This result immediately implies the non-attainment of $\dcw$ for the Vicsek set, since the energy measure of any such harmonic function fails to have full support.
	The non-attainment of $\dcw$ for the $N$-dimensional Sierpi\'{n}ski gasket with $N \geq 3$ requires a more delicate analysis of the intrinsic metric (see Definition \ref{d:dint})
	associated to the energy measure.

	Next, we mention some progress towards the Gaussian uniformization problem. 
	If $(X,\theta,\mu,\sE^\mu,\sF^\mu)$ satisfies $\on{PHI(2)}$, then
	it easily follows by combining the results in \cite{KM} and \cite{Mur1} (see Proposition \ref{p:metmeas}-\ref{it:conseq-PHI2-met}) that
	the metric $\theta$ is bi-Lipschitz equivalent to the intrinsic metric of $(X,\theta,\mu,\sE^\mu,\sF^\mu)$ and in particular
	that the metric $\theta$ is determined by the measure $\mu$ up to a bi-Lipschitz change.
	Therefore, in order to find a metric $\theta \in \sJ(X,d)$ and a measure $\mu \in \sA(X,d,m,\sE,\sF)$ in the Gaussian uniformization problem, it is enough to find an appropriate measure $\mu$.
	Furthermore, by \cite[Propositions 4.5 and 4.7]{KM}, we know that any such $\mu$ is a minimal energy-dominant measure, i.e., mutually absolutely continuous with respect to the whole family of energy measures (see Definition \ref{d:minimal-energy-dominant}), so that
	any two admissible measures that arise in the Gaussian uniformization problem are mutually absolutely continuous.
	We strengthen this result by showing in Section \ref{sec:GUP} that any two admissible measures that arise in the
	Gaussian uniformization problem are $A_\infty$-related in $(X,d)$ in the sense of Muckenhoupt (Theorem \ref{t:ainf}).
	For the MMD space corresponding to the Brownian motion on $\bR^n$, we also prove that
	the measures $A_{\infty}$-related to the Lebesgue measure are the only ones arising
	in the Gaussian uniformization problem for $n=1$ (Theorem \ref{t:1dgaussian})
	but are not for $n \geq 2$ (Example \ref{x:ainf}).
	
	This last result on $A_\infty$-relation between admissible measures in the Gaussian uniformization problem and its proof are inspired by a similar result for Ahlfors regular conformal dimension on Loewner spaces \cite[Theorem 7.11]{HK}.
	The combinatorial description of conformal gauge used in the proof of Theorem \ref{t:ehichar} was developed in \cite{Car13} for studying Ahlfors regular conformal dimension. Therefore, we find it appropriate to recall the definition of Ahlfors regular conformal dimension and discuss some related questions.
	
	Given a metric space $(X,d)$ and a Borel measure $\mu$ on $X$, we say that $\mu$ is \emph{$p$-Ahlfors regular} if there exists $C>0$ such that
	\[
	C^{-1} r^p \le \mu(B(x,r)) \le C r^p \quad \mbox{for all $x \in X$ and $r>0$ such that $B(x,r) \neq X$.}
	\]
	It is easy to verify that if a $p$-Ahlfors regular Borel measure $\mu$ exists on $(X,d)$, then the $p$-dimensional Hausdorff measure $\sH^p$ is also $p$-Ahlfors regular and the Hausdorff dimension of $(X,d)$ is $p$.
	Therefore, the existence of a $p$-Ahlfors regular measure is a property of the metric $d$.
	The \emph{Ahlfors regular conformal dimension} of a metric space $(X,d)$ is defined as
	\begin{align} \label{e:dARC}
	d_{\on{ARC}}(X,d)&= 
	\inf \Biggl\{ p >0 \Biggm\vert
	\begin{minipage}{170pt}
	there exist $\theta \in \sJ(X,d)$ and a $p$-Ahlfors regular Borel measure $\mu$ on $(X,\theta)$ 
	\end{minipage}
	\Biggr\}.
	\end{align}
	
	The attainment problem for the Ahlfors regular conformal dimension of the standard (two-dimensional) Sierpi\'{n}ski carpet is a well-known open question \cite[Problem 6.2]{BK05}.
	An important motivation for studying this attainment problem is Cannon's conjecture in geometric group theory.
	Cannon's conjecture states that every finitely generated, Gromov-hyperbolic group  $G$
	whose boundary (in the sense of Gromov) is homeomorphic to the $2$-sphere is isomorphic to
	a Kleinian group, i.e., a discrete group of M\"{o}bius transformations on the Riemann sphere.
	Bonk and Kleiner have shown that Cannon's conjecture is equivalent to the attainment of Ahlfors regular conformal dimension of the boundary of such a group \cite[Theorem 1.1]{BK05}.
	Our results and proof techniques will make it clear that there are similarities between the attainment problems for Ahlfors regular conformal dimension and conformal walk dimension.
	We hope that some of the methods we develop towards the attainment problem for conformal walk dimension will have applications to the analogous attainment problem for Ahlfors regular conformal dimension.
	
	Our work suggests that it would be useful to develop a theory of non-linear Dirichlet forms to study Ahlfors regular conformal dimension of fractals.
	In particular, Theorems \ref{t:attain-harmonic-func} and \ref{t:attain-harmonic-func-GSC} show that if the infimum in \eqref{e:dcw} is attained on
	a self-similar fractal, then an optimal admissible measure can be chosen to be the energy measure of a harmonic function. This result
	and its proof suggest that one might be able to construct an optimal Ahlfors regular measure attaining the Ahlfors regular conformal dimension as the ``energy measure'' of a $p$-harmonic function.
	However, the notions of energy measure and $p$-harmonic functions for non-linear Dirichlet energy  remain to be developed on fractals
	(non-linear Dirichlet energy can be formally viewed as $\int \lvert\nabla f\rvert^p $ with $p \neq 2$ and the corresponding $p$-harmonic functions
	can be viewed as minimizers of the non-linear Dirichlet energy). There is a well-developed non-linear potential theory in smooth settings
	(see \cite{HKM} and references therein), but a similar theory is yet to be developed on fractals.
	
\begin{notation} \label{ntn:intro}
Throughout this paper, we use the following notation and conventions.
\begin{enumerate}[label=\textup{(\alph*)},align=left,leftmargin=*,topsep=5pt,parsep=0pt,itemsep=2pt]
\item \label{it:ntn-intro-set-inclusion}The symbols $\subset$ and $\supset$ for set inclusion
	\emph{allow} the case of the equality.
\item \label{it:ntn-intro-ineq-mod-const}For $[0,\infty]$-valued quantities $A$ and $B$, we write $A \lesssim B$
	to mean that there exists an implicit constant $C \in [1,\infty)$ depending on some unimportant
	parameters such that $A \leq CB$. We write $A \asymp B$ if $A \lesssim B$ and $B \lesssim A$.
\item \label{it:ntn-intro-natural-numbers}$\mathbb{N}:=\{n\in\mathbb{Z}\mid n>0\}$, i.e., $0\not\in\mathbb{N}$.
\item \label{it:ntn-intro-cardinality}The cardinality (the number of elements) of a set $A$ is denoted by $\#A\in\mathbb{N}\cup\{0,\infty\}$.
\item \label{it:ntn-intro-max-min}
	We set $a\vee b:=\max\{a,b\}$, $a\wedge b:=\min\{a,b\}$, $a^{+}:=a\vee 0$ and
	$a^{-}:=-(a\wedge 0)$ for $a,b\in[-\infty,\infty]$, and
	we use the same notation also for $[-\infty,\infty]$-valued functions and
	equivalence classes of them. All numerical functions in this paper are assumed
	to be $[-\infty,\infty]$-valued.
\item \label{it:ntn-intro-indicator}Let $X$ be a non-empty set. We define $\one_{A}=\one_{A}^{X}\in\mathbb{R}^{X}$ for $A\subset X$ by
	$\one_{A}(x):=\one_{A}^{X}(x):=\bigl\{\begin{smallmatrix}1 & \textrm{if $x\in A$,}\\ 0 & \textrm{if $x\not\in A$.}\end{smallmatrix}$
\item \label{it:ntn-intro-contfunc}For a topological space $X$, we set
	$\contfunc(X):=\{f\colon X\to\mathbb{R}\mid\textrm{$f$ is continuous}\}$ and
	$\contfunc_{\mathrm{c}}(X):=\{f\in \contfunc(X)\mid\textrm{$X\setminus f^{-1}(0)$ has compact closure in $X$}\}$.
\item \label{it:ntn-intro-measure} Let $(X,\mathcal{B})$ be a measurable space and let $\mu,\nu$ be $\sigma$-finite
	measures on $(X,\mathcal{B})$. We write $\nu \ll \mu$ to mean that
	$\nu$ is absolutely continuous with respect to $\mu$, and $\nu\leq\mu$
	to mean that $\nu(A)\leq\mu(A)$ for any $A\in\mathcal{B}$, or equivalently,
	$\nu \ll \mu$ and $d\nu/d\mu\leq 1$ $\mu$-a.e.
\end{enumerate}
\end{notation}
	
	\section{Framework and results} \label{s:fr}
	
	In this section, we recall the background definitions and state our main results
	for general MMD spaces. Our other main results on the attainment problem
	for self-similar sets are treated separately in Section \ref{sec:pfcsss}.
	
	\subsection{Metric measure Dirichlet space and energy measure}\label{ssec:MMD-EnergyMeas}
	
	Throughout this paper, we consider a metric space $(X,d)$ in which
	$B(x,r):=B_{d}(x,r):=\{y\in X \mid d(x,y)<r\}$ is relatively compact
	(i.e., has compact closure) in $X$ for any $(x,r)\in X\times(0,\infty)$,
	and a Radon measure $m$ on $X$ with full support, i.e., a Borel measure
	$m$ on $X$ which is finite on any compact subset of $X$ and strictly positive on any
	non-empty open subset of $X$. We always assume that $X$ contains at least two elements,
	and such a triple $(X,d,m)$ is referred to as a \emph{metric measure space}.
	We set $\overline{B}(x,r):=\overline{B}_{d}(x,r):=\{y\in X \mid d(x,y)\leq r\}$ for
	$(x,r)\in X\times(0,\infty)$ and $\diam_{d}(A):=\diam(A,d):=\sup_{x,y\in A}d(x,y)$ for $A\subset X$ ($\sup\emptyset:=0$).
	
	Furthermore let $(\mathcal{E},\mathcal{F})$ be a \emph{symmetric Dirichlet form} on $L^{2}(X,m)$;
	by definition, $\mathcal{F}$ is a dense linear subspace of $L^{2}(X,m)$, and
	$\mathcal{E}:\mathcal{F}\times\mathcal{F}\to\mathbb{R}$
	is a non-negative definite symmetric bilinear form which is \emph{closed}
	($\mathcal{F}$ is a Hilbert space under the inner product $\mathcal{E}_{1}:= \mathcal{E}+ \langle \cdot,\cdot \rangle_{L^{2}(X,m)}$)
	and \emph{Markovian} ($f^{+}\wedge 1\in\mathcal{F}$ and $\mathcal{E}(f^{+}\wedge 1,f^{+}\wedge 1)\leq \mathcal{E}(f,f)$ for any $f\in\mathcal{F}$).
	Recall that $(\mathcal{E},\mathcal{F})$ is called \emph{regular} if
	$\mathcal{F}\cap \contfunc_{\mathrm{c}}(X)$ is dense both in $(\mathcal{F},\mathcal{E}_{1})$
	and in $(\contfunc_{\mathrm{c}}(X),\|\cdot\|_{\mathrm{sup}})$, and that
	$(\mathcal{E},\mathcal{F})$ is called \emph{strongly local} if $\mathcal{E}(f,g)=0$
	for any $f,g\in\mathcal{F}$ with $\supp_{m}[f]$, $\supp_{m}[g]$ compact and
	$\supp_{m}[f-a\one_{X}]\cap\supp_{m}[g]=\emptyset$ for some $a\in\mathbb{R}$. Here,
	for a Borel measurable function $f:X\to[-\infty,\infty]$ or an
	$m$-equivalence class $f$ of such functions, $\supp_{m}[f]$ denotes the support of the measure $\lvert f\rvert\,dm$,
	i.e., the smallest closed subset $F$ of $X$ with $\int_{X\setminus F}\lvert f\rvert\,dm=0$,
	which exists since $X$ has a countable open base for its topology; note that
	$\supp_{m}[f]$ coincides with the closure of $X\setminus f^{-1}(0)$ in $X$ if $f$ is continuous.
	The pair $(X,d,m,\mathcal{E},\mathcal{F})$ of a metric measure space $(X,d,m)$ and a strongly local,
	regular symmetric Dirichlet form $(\mathcal{E},\mathcal{F})$ on $L^{2}(X,m)$ is termed
	a \emph{metric measure Dirichlet space}, or an \emph{MMD space} in abbreviation.
	We refer to \cite{FOT,CF} for details of the theory of symmetric Dirichlet forms.
	
	We next recall the definition of energy measure and some relevant notions.
	Note that $fg\in\mathcal{F}$ for any $f,g\in\mathcal{F}\cap L^{\infty}(X,m)$ by \cite[Theorem 1.4.2-(ii)]{FOT}
	and that $\{(-n)\vee(f\wedge n)\}_{n=1}^{\infty}\subset\mathcal{F}$ and
	$\lim_{n\to\infty}(-n)\vee(f\wedge n)=f$ in norm in $(\mathcal{F},\mathcal{E}_{1})$
	by \cite[Theorem 1.4.2-(iii)]{FOT}.
	
	\begin{definition}[{\cite[(3.2.13), (3.2.14) and (3.2.15)]{FOT}}]\label{d:EnergyMeas}
		Let $(X,d,m,\mathcal{E},\mathcal{F})$ be an MMD space.
		The \emph{energy measure} $\Gamma(f,f)$ of $f\in\mathcal{F}$
		associated with $(X,d,m,\mathcal{E},\mathcal{F})$ is defined,
		first for $f\in\mathcal{F}\cap L^{\infty}(X,m)$ as the unique ($[0,\infty]$-valued)
		Borel measure on $X$ such that
		\begin{equation}\label{e:EnergyMeas}
		\int_{X} g \, d\Gamma(f,f)= \mathcal{E}(f,fg)-\frac{1}{2}\mathcal{E}(f^{2},g) \qquad \textrm{for all $g \in \mathcal{F}\cap \contfunc_{\mathrm{c}}(X)$,}
		\end{equation}
		and then by
		$\Gamma(f,f)(A):=\lim_{n\to\infty}\Gamma\bigl((-n)\vee(f\wedge n),(-n)\vee(f\wedge n)\bigr)(A)$
		for each Borel subset $A$ of $X$ for general $f\in\mathcal{F}$. 
	\end{definition}
	
	\begin{definition}[{\cite[Definition 2.1]{Hin10}}]\label{d:minimal-energy-dominant}
		Let $(X,d,m,\mathcal{E},\mathcal{F})$ be an MMD space. A $\sigma$-finite Borel measure
		$\nu$ on $X$ is called a \emph{minimal energy-dominant measure}
		of $(\mathcal{E},\mathcal{F})$ if the following two conditions are satisfied:
		\begin{enumerate}[label=\textup{(\roman*)},align=left,leftmargin=*,topsep=5pt,parsep=0pt,itemsep=2pt]
			\item\label{it:domination}(Domination) For every $f \in \mathcal{F}$, $\Gamma(f,f) \ll \nu$.
			\item\label{it:minimality}(Minimality) If another $\sigma$-finite Borel measure $\nu'$
			on $X$ satisfies condition \ref{it:domination} with $\nu$ replaced
			by $\nu'$, then $\nu \ll \nu'$.
		\end{enumerate}
		Note that by \cite[Lemmas 2.2, 2.3 and 2.4]{Hin10}, a minimal energy-dominant measure of
		$(\mathcal{E},\mathcal{F})$ always exists and is precisely a $\sigma$-finite
		Borel measure $\nu$ on $X$ such that for each Borel subset $A$ of $X$,
		$\nu(A)=0$ if and only if $\Gamma(f,f)(A)=0$ for all $f\in\mathcal{F}$.
		In particular, any two minimal energy-dominant measures are mutually absolutely continuous.
	\end{definition}
	
	\begin{definition}\label{d:dint}
		Let $(X,d,m,\mathcal{E},\mathcal{F})$ be an MMD space. We define its
		\emph{intrinsic metric} $d_{\on{int}}\colon X\times X\to[0,\infty]$ by
		\begin{equation}\label{e:dint}
		d_{\on{int}}(x,y) := \sup \{f(x) -f(y) \mid
		\textrm{$f \in \mathcal{F}_{\on{loc}} \cap \contfunc(X)$, $\Gamma(f,f) \leq m$} \},
		\end{equation}
		where
		\begin{equation}\label{e:Floc}
		\mathcal{F}_{\on{loc}} := \Biggl\{ f \Biggm\vert
		\begin{minipage}{245pt}
		$f$ is an $m$-equivalence class of $\mathbb{R}$-valued Borel measurable functions
		on $X$ such that $f \one_{V} = f^{\#} \one_{V}$ $m$-a.e.\ for some $f^{\#}\in\mathcal{F}$
		for each relatively compact open subset $V$ of $X$
		\end{minipage}
		\Biggr\}
		\end{equation}
		and the energy measure $\Gamma(f,f)$ of $f\in\mathcal{F}_{\on{loc}}$ associated with
		$(X,d,m,\mathcal{E},\mathcal{F})$ is defined as the unique Borel measure on $X$
		such that $\Gamma(f,f)(A)=\Gamma(f^{\#},f^{\#})(A)$ for any relatively compact
		Borel subset $A$ of $X$ and any $V,f^{\#}$ as in \eqref{e:Floc} with $A\subset V$;
		note that $\Gamma(f^{\#},f^{\#})(A)$ is independent of a particular choice of such $V,f^{\#}$
		by \cite[Corollary 3.2.1]{FOT}.
	\end{definition}

We remark that the intrinsic metric need not always be a metric.
In general, it is only a pseudo-metric; that is, $d_{\on{int}}(x,x)=0$ for all $x \in X$, $d_{\on{int}}(x,y) \in [0,\infty], d_{\on{int}}(x,y)= d_{\on{int}}(y,x)$ for all $x,y \in X$ and $ d_{\on{int}}(x,z) \le  d_{\on{int}}(x,y)+ d_{\on{int}}(y,z)$ for all $x,y,z \in X$.
A sufficient condition for when the intrinsic metric is bi-Lipschitz equivalent to the original metric is given in Proposition \ref{p:metmeas}.
On the other hand, \cite[Theorem 2.13-(a)]{KM} gives a family of examples for which the intrinsic metric is identically zero.
In the setting of \cite[Theorem 2.13-(a)]{KM}, $\Gamma(f,f) \le m$ implies that $f$ is identically constant. We refer the reader to \cite[Propostion 5.7]{BeSa} for an example which shows that the intrinsic metric could be infinite.

	\subsection{Harnack inequalities}
	We recall the definition of harmonic and caloric functions.
	\begin{definition}\label{d:harmonic}
		Let $(X,d,m,\mathcal{E},\mathcal{F})$ be an MMD space and let $\sF_e$ denote its extended Dirichlet space.
		Recall that the \emph{extended Dirichlet space} $\sF_e$ of $(X,d,m,\sE,\sF)$ is defined as
		the space of $m$-equivalence classes of functions $f\colon X \to \bR$
		such that $\lim_{n\to \infty} f_n= f$ $m$-a.e.\ on $X$ for some
		$\sE$-Cauchy sequence $(f_n)_{n \in \bN}$ in $\sF$, that the limit
		$\mathcal{E}(f,f):=\lim_{n\to\infty}\mathcal{E}(f_{n},f_{n})\in\mathbb{R}$
		exists, is independent of a choice of such $(f_n)_{n \in \bN}$ for each
		$f \in \sF_e$ and defines an extension of $\mathcal{E}$ to $\sF_e \times \sF_e$,
		and that $\mathcal{F}=\mathcal{F}_{e}\cap L^{2}(X,m)$;
		see \cite[Definition 1.1.4 and Theorem 1.1.5]{CF}. We remark that the definition of
		the energy measure $\Gamma(f,f)$ associated with $(X,d,m,\mathcal{E},\mathcal{F})$
		also extends canonically to $f\in\sF_e$; see \cite[p.\ 123 and Theorem 5.2.3]{FOT}.
		
		A function $h \in \mathcal{F}_e$ is said to be \emph{$\mathcal{E}$-harmonic}
		on an open subset $U$ of $X$, if
		\begin{equation}\label{e:harmonic}
		\mathcal{E}(h,f)=0 \qquad \textrm{for all $f\in\mathcal{F}\cap \contfunc_{\mathrm{c}}(X)$ with $\supp\nolimits_{m}[f]\subset U$.}
		\end{equation}
		
		Let $I$ be an open interval in $\bR$. We say that a function $u: I \to L^2(X,m)$ is \emph{weakly differentiable} at $t_0 \in I$ if for any $f \in L^2(X,m)$ the function $t \mapsto \langle u(t), f \rangle$ is differentiable at $t_0$, where $\langle \cdot, \cdot \rangle$ denotes the inner product in $L^2(X,m)$. If $u$ is weakly differentiable at $t_0$, then by the   uniform boundedness principle, there exists a (unique) function $w \in L^2(X,m)$ such that
		\[
		\lim_{t \to t_0} \left\langle \frac{u(t)-u(t_0)}{t- t_0}, f \right\rangle = \langle w , f \rangle, \quad \mbox{ for all $f \in L^2(X,m)$.}
		\]
		We say that the function $w$ is the \emph{weak derivative} of the function $u$ at $t_0$ and write $w=u'(t_0)$. 
		
		Let $I$ be an open interval in $\bR$ and let $\Omega$ be an open subset of $X$.
		A function $u: I \to \sF$ is said to be \emph{caloric} in $I \times \Omega$ if  $u$ is weakly differentiable in the space $L^2(\Omega)$ at any $t \in I$, and for any $f \in \sF \cap \contfunc_{\mathrm{c}}(\Omega)$, and for any $t \in I$,
		\begin{equation}
		\langle u', f \rangle + \sE(u,f)=0. \label{e:caloric}
		\end{equation}
	\end{definition}

	\begin{definition}[Harnack inequalities] \label{d:harnack}
		We say that  an MMD space $(X,d,m,\sE,\sF)$ satisfies the \emph{elliptic Harnack inequality} (abbreviated as \hypertarget{ehi}{$\on{EHI}$}),
		if there exist $C >1$ and $\delta \in (0,1)$ such that for all $x \in X$, $r>0$ and for any $h \in \sF_e$ that is non-negative on $B(x,r)$ and $\mathcal{E}$-harmonic on $B(x,r)$, we have
		\begin{equation} \label{EHI} \tag*{$\on{EHI}$}
		\esssup_{B(x,\delta r)} h \le C \essinf_{B(x,\delta r)} h.
		\end{equation}
		We say that  an MMD space $(X,d,m,\sE,\sF)$ satisfies the \emph{parabolic Harnack inequality} with walk dimension $\beta$ (abbreviated as \hypertarget{phi}{$\on{PHI(\beta)}$}),
		if there exist $0<C_1< C_2 < C_3 < C_4 <\infty$, $C_5>1$ and $\delta \in (0,1)$ such that for all $x \in X$, $r>0$ and for any non-negative bounded caloric function $u$ on the space-time cylinder $Q=(a,a+C_4 r^\beta) \times B(x,r)$, we have
		\begin{equation} \label{PHI} \tag*{$\on{PHI(\beta)}$}
		\esssup_{Q_-} u \le C_5 \essinf_{Q_+} u,
		\end{equation}
		where $Q_-=(a+C_1 r^\beta,a+C_2 r^\beta) \times B(x,\delta r)$ and $Q_+=(a+C_3 r^\beta,a+C_4 r^\beta) \times B(x,\delta r)$.
	\end{definition}
	
	\begin{remark} \label{rmk:harnack}
	The formulation of \ref{EHI} in Definition \ref{d:harnack} above is slightly different
	from that in \cite[Definition 4.2-(i)]{BCM}, but it is easy to see from the relative
	compactness in $X$ of all balls in $(X,d)$ and \cite[Proposition 2.9-(ii)]{BCM}
	that the former implies the latter, so that we can freely use the implications of
	\ref{EHI} established in \cite{BCM} under the present formulation of \ref{EHI}.
	\end{remark}
	
	\subsection{Admissible measures and time-changed Dirichlet space}
 Given an MMD space $(X,d,m,\sE,\sF)$ and $A\subset X$, we define its $1$-capacity as
	\begin{equation} \label{e:defCap1}
	\Capa_1(A) := \inf \bigl\{ \sE_{1}(f,f) \bigm\vert \textrm{$f \in \sF$, $f \geq 1$ $m$-a.e.\ on a neighborhood of $A$} \bigr\},
	\end{equation}
	where $\sE_{1}:=\mathcal{E}+\langle\cdot,\cdot\rangle_{L^{2}(X,m)}$ as defined before. 
	For disjoint Borel sets $A,B$ such that $B$ is closed and $\overline{A} \Subset B^c$ (by $A \Subset B^c$, we mean that $\overline{A}$ is compact and $\overline{A} \subset B^c$), we
	define $\sF(A,B)$ as the set of function $\phi \in \sF$ such that $\phi \equiv 1$ in an open neighborhood of $A$, and $\supp \phi \subset B^c$.
	For such sets $A$ and $B$, we define the capacity between them as
	\begin{equation} \label{e:defCap}
	\Capa(A,B) := \inf \giveset{\sE(f,f) \mid f \in \sF(A,B)}.
	\end{equation}
	
	\begin{definition}[Smooth measures] \label{d:smooth}
		Let $(X, d,m,\sE,\sF)$ be an MMD space. 
		A Radon measure $\mu$ on $X$, i.e., a Borel measure $\mu$ on $X$ which is finite on
		any compact subset of $X$, is said to be \emph{smooth} if $\mu$ charges no set
		of zero capacity (that is, $\mu(A)=0$ for any Borel subset $A$ of $X$ with $\Capa_{1}(A)=0$).
	\end{definition}
	
	For example, the energy measure $\Gamma(f,f)$ of $f\in\mathcal{F}_{e}$ is smooth by
	\cite[Lemma 3.2.4]{FOT}. An essential feature of a smooth Radon measure $\mu$ on $X$
	is that the $\mu$-equivalence class of each $f\in\mathcal{F}_{e}$ is canonically determined
	by considering a \emph{quasi-continuous} $m$-version of $f$, which exists by \cite[Theorem 2.1.7]{FOT}
	and is unique \emph{q.e.}\ (i.e., up to sets of capacity zero) by \cite[Lemma 2.1.4]{FOT};
	see \cite[Section 2.1]{FOT} and \cite[Sections 1.2, 1.3 and 2.3]{CF} for the definition and
	basic properties of quasi-continuous functions with respect to a regular symmetric Dirichlet form.
	In what follows, we always consider a quasi-continuous $m$-version of $f\in\mathcal{F}_{e}$.
	
	An increasing sequence $\{F_k; k\geq 1\}$ of closed subsets of an MMD space $(X,d,m,\sE,\sF)$ is said to be a \emph{nest} if
	$\bigcup_{k\geq 1} \sF_{F_k} $ is $\sqrt{\sE_1}$-dense in $\sF$, where
	$\sF_{F_k}:=\{f\in \sF \mid \textrm{$f=0$ $m$-a.e.\ on $X\setminus F_k$} \}$. 
	Recall that $D \subset X$ is \emph{quasi-open} if there exists a nest $\giveset{F_n}$ such that $D \cap F_n$ is an open subset of $F_n$ in the relative topology for each $n \in \bN$.
	The complement in $X$ of a quasi-open set is called \emph{quasi-closed}.
	We are interested in the class of admissible measures, namely smooth Radon measures
	having full quasi-support in the sense defined as follows.
	
	\begin{definition}[Full quasi-support, admissible measures; {\cite[(4.6.3) and (4.6.4)]{FOT}}] \label{d:admissible}
			Let $(X, d,m,\sE,\sF)$ be an MMD space. A smooth Radon measure $\mu$ on $X$
			is said to have \emph{full quasi-support} if for any quasi-closed set $F$
			with $\mu (X \setminus F)=0$ we have $\Capa_1(X \setminus F)=0$.
			A Borel measure $\mu$ on $X$ is said to be \emph{admissible} if $\mu$ is a smooth
			Radon measure on $X$ with full quasi-support, and the set of admissible measures
			with respect to $(X,d,m,\sE,\sF)$ is denoted by $\sA(X,d,m,\sE,\sF)$.
	\end{definition}
	
	\begin{definition} [Time-changed Dirichlet form] \label{d:TCDF}
			Let $(X, d,m,\sE,\sF)$ be an MMD space.
			If $\mu$ is a smooth Radon measure, it defines a time change of the process whose associated Dirichlet form is
			called the trace Dirichlet form and denoted by $(\sE^\mu,\sF^\mu)$ (see \cite[Section 6.2]{FOT} and \cite[Section 5.2]{CF}).
			Assume $\mu \in \sA(X,d,m,\sE,\sF)$. Then the trace Dirichlet form $(\sE^\mu,\sF^\mu)$ is given by
			\begin{equation} \label{e:defTC}
			\sF^\mu = \sF_e \cap L^2(X,\mu) \qquad \mbox{and} \qquad
			\sE^\mu(u,v) = \sE(u,v) \quad \mbox{for $u,v \in \sF^\mu$,}
			\end{equation}  
			and $(\sE^\mu,\sF^\mu)$ is a strongly local, regular symmetric Dirichlet form on $L^2(X,\mu)$
			by \cite[Theorems 5.1.5, 6.2.1 and Exercise 3.1.1]{FOT}.
			We also note that by \cite[Theorem 5.2.11]{CF},
			\begin{equation} \label{e:admiss}
			\sA(X,d,m,\sE,\sF)=\sA(X,d,\mu,\sE^\mu,\sF^\mu).
			\end{equation}
			In probabilistic terms, $(\sE^\mu,\sF^\mu)$ is the Dirichlet form of the time-changed process
			$(\omega,t) \mapsto Y_{\tau_t(\omega)}(\omega)$, where $(Y_t)_{t \ge 0}$ is an $m$-symmetric
			diffusion on $X$ whose Dirichlet form is $(\sE,\sF)$ and $\tau_t$ is the right-continuous
			inverse of the positive continuous additive functional $(A_t)_{t \geq 0}$ of $(Y_t)_{t \geq 0}$
			with Revuz measure $\mu$; see \cite[Section 6.2]{FOT} and \cite[Section 5.2]{CF} for details.
	\end{definition}

	\subsection{Main results}
	Our first main result is that the value of the conformal walk dimension is an invariant
	for MMD spaces satisfying the metric doubling property and the elliptic Harnack inequality \ref{EHI}.
	We recall from Definition \ref{dfn:dcw} that the conformal walk dimension $\dcw$ of an MMD space $(X,d,m,\sE,\sF)$ is the infimum over all $\beta>0$
	such that there exist an admissible measure $\mu \in \sA(X,d,m,\sE,\sF)$ and a metric $\theta \in \sJ(X,d)$ such that
	the time-changed MMD space $(X,\theta,\mu,\sE^\mu,\sF^\mu)$ satisfies the parabolic Harnack inequality \ref{PHI};
	note here that $(X,\theta,\mu,\sE^\mu,\sF^\mu)$ is indeed an MMD space since,
	for any $(x,r)\in X\times(0,\infty)$, $\diam(B_{\theta}(x,r),d)<\infty$ by $\theta \in \sJ(X,d)$
	and \cite[Proposition 10.8]{Hei} and hence $B_{\theta}(x,r)$ is relatively compact in $X$.
	\begin{theorem} [Universality of conformal walk dimension] \label{t:ehichar}
		Let $(X,d,m,\sE,\sF)$ be an MMD space
		and let $\dcw$ denote its conformal walk dimension.
		Then the following are equivalent:
		\begin{enumerate}[label=\textup{(\alph*)},align=left,leftmargin=*,topsep=5pt,parsep=0pt,itemsep=2pt]
			\item \label{it:ehichar-a}$(X,d,m,\sE,\sF)$ satisfies the metric doubling property and \ref{EHI}.
			\item \label{it:ehichar-b}$\dcw <\infty$. 
			\item \label{it:ehichar-c}$\dcw=2$.
		\end{enumerate}
	\end{theorem}
	The proof of Theorem \ref{t:ehichar} is concluded in Subsection \ref{ssec:prf-dcw-2}
	after long preparations in Section \ref{sec:HypFill} and Subsection \ref{ssec:conseq-Harnack}.
	We give a brief description of the proof in Subsection \ref{ssec:outline-prf-dcw-2}.
	
	The next question is whether or not the infimum in the definition \eqref{e:dcw} of $\dcw$
	is attained. To this end we first describe the metric and measure.
	The following result is essentially contained in \cite{KM};
	see the beginning of Subsection \ref{ssec:conseq-PHI2} for the proof.
	\begin{proposition} \label{p:metmeas}
		Let $(X,d,m,\sE,\sF)$ be an MMD space that satisfies \hyperlink{phi}{$\on{PHI}(2)$}. Then the following hold:
		\begin{enumerate}[label=\textup{(\alph*)},align=left,leftmargin=*,topsep=5pt,parsep=0pt,itemsep=2pt]
			\item\label{it:conseq-PHI2-met}The metric $d$ is bi-Lipschitz equivalent to the intrinsic metric $d_{\on{int}}$.
			\item\label{it:conseq-PHI2-meas}The symmetric measure $m$ is a minimal energy-dominant measure.
		\end{enumerate}
	\end{proposition}
	By Proposition \ref{p:metmeas}-\ref{it:conseq-PHI2-met}, in order to find a metric $\theta \in \sJ(X,d)$ and a measure
	$\mu \in \sA(X,d,m,\sE,\sF)$ in the Gaussian uniformization problem, it is enough to find an
	appropriate measure $\mu$ as the metric $\theta$ is determined by the measure up to bi-Lipschitz equivalence.
	This observation is useful in studying the Gaussian uniformization problem,
	since constructing measures is typically easier than constructing metrics.
	By Proposition \ref{p:metmeas}-\ref{it:conseq-PHI2-meas}, any such measures are mutually absolutely continuous.
	In fact, we have the following improvement, asserting that any two measures $\mu_1,\mu_2$
	attaining the infimum in the definition of $\dcw$ are $A_\infty$-related in the sense of Muckenhoupt (see Definition \ref{d:ainf}).
	\begin{theorem} \label{t:ainf}
		Let $(X,d,m,\sE,\sF)$ be an MMD space. For $i=1,2$, let $d_i \in \sJ(X,d)$,
		$m_i \in \sA(X,d,m,\sE,\sF)$ and assume that the time-changed MMD space
		$(X,d_i,m_i,\sE^{m_i},\sF^{m_i})$ satisfies \hyperlink{phi}{$\on{PHI}(2)$}.
		Then the measures $m_1$ and $m_2$ are $A_\infty$-related in $(X,d)$.
	\end{theorem}
	The proof of Theorem \ref{t:ainf} is given in the first half of Subsection \ref{ssec:Ainfty}.
	Using Proposition \ref{p:metmeas}, Theorem \ref{t:ainf} and sharp constants of Poincar\'e inequalities in \cite{CW},
	we also answer in Theorem \ref{t:1dgaussian} the Gaussian uniformization problem for the
	MMD space $(\bR,d,m,\sE,\sF)$ corresponding to the Brownian motion on $\bR$,
	by proving that any Borel measure $\mu$ on $\bR$ that is $A_{\infty}$-related to $m$ admits
	$\theta \in \sJ(\bR,d)$ such that $(\bR,\theta,\mu,\sE^{\mu},\sF^{\mu})$ satisfies \hyperlink{phi}{$\on{PHI}(2)$}.
	This result is not true for the MMD space $(\bR^n,d,m,\sE,\sF)$ corresponding
	to the Brownian motion on $\bR^n$ with $n \geq 2$ as shown in Example \ref{x:ainf}, and
	we do not have an exact characterization of the Borel measures on $\bR^n$ which are $A_\infty$-related to $m$
	and admit $\theta \in \sJ(\bR^n,d)$ such that $(\bR^n,\theta,\mu,\sE^{\mu},\sF^{\mu})$ satisfies \hyperlink{phi}{$\on{PHI}(2)$},
	except that we give an implicit one for $n=2$ in Proposition \ref{p:strongainf}.
	
	\subsection{Outline of the proof of Theorem \ref{t:ehichar}} \label{ssec:outline-prf-dcw-2}
	
	As pointed out in the introduction, the equivalence between \ref{it:ehichar-a} and \ref{it:ehichar-b} is already known. Since the implication from \ref{it:ehichar-c} to \ref{it:ehichar-b} is trivial, it suffices to show that \ref{it:ehichar-b} implies \ref{it:ehichar-c}.
	Hence, we may assume that $(X,d,m,\sE,\sF)$ satisfies \hyperlink{phi}{$\on{PHI(\gamma)}$} for some $\gamma>2$. Let $\beta >2$ be arbitrary. We wish to construct a metric $\theta \in \sJ(X,d)$ and $\mu \in \sA(X,d,m,\sE,\sF)$ such that the time-changed MMD space $(X,\theta,\mu,\sE^\mu,\sF^\mu)$ satisfies \hyperlink{phi}{$\on{PHI(\beta)}$}. 
	
	To sketch the main ideas, we further assume that $(X,d)$ is compact and the diameter of $(X,d)$ is normalized to $\frac 1 3$. The non-compact case follows by the same argument as the compact case, by considering $X$ as a limit of compact sets.
	Thanks to the known characterizations of \hyperlink{phi}{$\on{PHI(\beta)}$} as stated in Theorem \ref{t:phichar}
	and the preservation of \ref{EHI} between $(X,d,m,\sE,\sF)$ and $(X,\theta,\mu,\sE^\mu,\sF^\mu)$ as stated in Lemma \ref{l:EHI-TC-QS},
	it is enough to construct a metric $\theta \in \sJ(X,d)$ and a measure $\mu \in \sA(X,d,m,\sE,\sF)$ such that
	\begin{equation} \label{e:capest}
	\mu (B_\theta (x,r)) \asymp r^\beta \Capa(B_\theta(x,r), B_\theta(x,2r)^c),\quad \mbox{for all $x \in X, r \lesssim \diam(X,\theta)$.}
	\end{equation}
	The above estimate  relating the measure $\mu$ and capacity implies that $\mu$ is a smooth measure with full quasi-support and satisfies the following volume doubling and reverse volume doubling properties (see Proposition \ref{p:smooth}): there exists $C_D >0$ and $c_D \in (0,1)$ such that
	\begin{align} \label{e:vd}
	\frac{\mu(B_d(x,2r))}{\mu(B_d(x,r))} &\le C_D, \quad \mbox{for all $x \in X, r>0$, and}\\
	\label{e:rvd} \frac{\mu(B_d(x,r))}{\mu(B_d(x,2r))} &\le c_D, \quad \mbox{for all $x \in X, 0<r\lesssim \diam(X,d)$.}
	\end{align} 
	The estimate \eqref{e:capest} along with Theorem \ref{t:phichar}, implies \ref{PHI}
	because volume doubling, reverse volume doubling, and \hyperlink{ehi}{EHI} are preserved by the quasisymmetric change
	of the metric from $d$ to $\theta$.

	The construction of a metric $\theta$ and a measure $\mu$ is a modification of \cite{Car13}, but instead of the Ahlfors
	regularity required in \cite{Car13} we need to establish \eqref{e:capest}.
	Following \cite{Car13}, we construct the metric $\theta$ and the measure $\mu$ that satisfy \eqref{e:capest} using a multi-scale argument.
	This part of the argument relies on theory of Gromov hyperbolic spaces.
	The basic idea behind the approach is to construct a Gromov hyperbolic graph (called the \emph{hyperbolic filling}) whose boundary (in the sense of Gromov) corresponds to the given metric space $(X,d)$.
	A well-known result in Gromov hyperbolic spaces asserts that any metric in the conformal gauge $\sJ(X,d)$ up to bi-Lipschitz equivalence can be obtained by a bounded perturbation of edge weights on the hyperbolic filling.
	We recall the basic results about Gromov hyperbolic spaces and their boundaries in Subsection \ref{s:gromov}.

	We first sketch the construction of this hyperbolic graph, postponing a more precise definition to Subsection \ref{s:hf}.
	We choose a parameter $a>10^2$ and
	cover the space $X$ using a covering $\sS_n$ with balls of radii $2 a^{-n}$ such that for any two distinct balls $B_{d}(x_1, 2a^{-n}),B_{d}(x_2, 2a^{-n}) \in \sS_n$,
	we have $B_{d}(x_1, a^{-n}/2) \cap B_{d}(x_2, a^{-n}/2)=\emptyset$ (we think of these balls as `approximately pairwise disjoint covering' at scale $a^{-n}$). Therefore the covering $\sS_n$ corresponds to scale $a^{-n}$ for all $n \in \bN_{\ge 0}$.
	In what follows, for a ball $B$, we denote by $x_B$ and $r_B$ the center and radius of $B$. For a ball $B$ and $\lambda >0$, we denote by $\lambda\cdot B$, the ball $B_{d}(x_B,\lambda r_B)$.
	
	We define a \emph{tree of vertical edges} with vertex set $\coprod_{n \ge 0} \sS_n$ by choosing for each ball $B \in \sS_n, n \ge 1$ a `parent ball' $B' \in \sS_{n-1}$ such that $x_{B'}$ is a closest point to $x_B$ in the set $\giveset{x_C: C \in \sS_{n-1}}$. By the assumption on the diameter, $\sS_0$ is a singleton. The edges in this tree are called \emph{vertical edges}. We choose another parameter $\lambda \ge 10$ to define another set of edges on $\coprod_{n \ge 0} \sS_n$ called \emph{horizontal edges}. Two distinct balls $B, \widetilde{B} \in \sS_n, n \ge 0$ share a horizontal edge  if and only if $\lambda \cdot B \cap \lambda \cdot \widetilde{B} \neq \emptyset$. The set of edges of the hyperbolic filling is defined as the union of the sets of horizontal and vertical edges. 
	
	In our construction, the vertical edge weights play a more central role and the values of horizontal edge weights are less important. The weight of the vertical edge between $B \in \sS_n, B' \in \sS_{n-1}$ can be interpreted as the relative diameter under the $\theta$-metric.
	To describe the data required in our construction of the metric,
	suppose for the moment that $\theta$ is a metric on $X$ with the desired properties, and
	let us define the \emph{relative diameter} of $B \in \sS_n, n \ge 1$ as 
	\begin{equation} \label{e:reldiam}
	\rho(B):=\frac{\diam(B,\theta)}{\diam(B',\theta)},
	\end{equation}
	where    $B' \in \sS_{n-1}$ is such that there is a vertical edge between $B'$ and $B$ ($B'$ is the parent of $B$ in the tree of vertical edges). It turns out that the `relative diameter' in \eqref{e:reldiam} contains enough information about $\theta$ to reconstruct the metric $\theta$ (up to bi-Lipschitz equivalence).  Hence we could reduce the problem of construction of $\theta \in \sJ(X,d)$ to constructing the function $\rho(\cdot)$ on $\coprod_{n \ge 1} \sS_n$; see Theorem \ref{t:qscomb}.
	
	It is therefore enough to construct $\rho(\cdot)$ in \eqref{e:reldiam}. Next, we describe two key conditions that the relative diameter $\rho(\cdot)$ defined in \eqref{e:reldiam} must satisfy.
	For a ball $B \in \sS_{n-1}, n \ge 1$, let us denote by $\Gamma_{n}(B)$ the set of horizontal paths in $\sS_n$ defined by
	\[
	\Gamma_n(B)= \Biggl\{ (B_i)_{i=0}^N \Biggm\vert
	\begin{minipage}{210pt}
	$N \in \bN,  B_i \in \sS_n$  for all $i=0,\ldots,N$; $B_i$ and $B_{i+1}$ share a horizontal edge for all $i=0,\ldots,N-1$, $x_{B_0} \in B$, $x_{B_1},\ldots,x_{B_{N-1}} \in 2 \cdot B$, $x_{B_{N}} \not\in 2\cdot B$
	\end{minipage}  \Biggr\}.
	\]
	The first condition on $\rho(\cdot)$ is 
	\begin{equation} \label{e:s1}
	\sum_{i=1}^N \rho(B_i) \gtrsim 1, \quad \mbox{for all $(B_i)_{i=0}^N \in \Gamma_{n}(B)$ and  $B \in \sS_{n-1}, n \ge 1$.}
	\end{equation}
	The condition \eqref{e:s1} is a consequence of the fact that $\theta \in \sJ(X,d)$ and that $(X,d)$ is a uniformly perfect metric space. 
	We like to think of \eqref{e:s1} as a `\emph{no shortcuts condition}' as it disallows the possibility of short cuts in the $\theta$-metric from $B_d(x,r)$ to $B_d(x,2r)^c$.

	The second condition arises from the estimate \eqref{e:capest}. 
	For any ball $B \in \sS_k, k \ge 0$, by $\sD_{k+1}(B)$, we denote its descendants in $\sS_{k+1}$; that is $\sD_{k+1}(B)$ is the set of elements $B' \in \sS_{k+1}$ such that $B'$ and $B$ share a vertical edge.
	The second condition that $\rho$ must satisfy is the following estimate:
	\begin{equation} \label{e:s2}
	\sum_{B' \in \sD_{k+1}(B)} \rho(B')^\beta \Capa(B', (2\cdot B')^c)  \lesssim  \Capa(B, (2\cdot B)^c), 
	\end{equation}
	for all $B \in \sS_k$ and $k$ large enough.
	To explain that \eqref{e:s2} should hold for a metric $\theta$ and a measure $\mu$ with the desired properties,
	we first observe that the volume doubling property \eqref{e:vd} implies that  
	\[
	\sum_{B' \in \sD_{k+1}(B)}  \mu(B') \lesssim \mu(B), \quad \mbox{for all $B \in \sS_k$}.
	\]
	By \eqref{e:capest}, $\theta \in \sJ(X,d)$ and comparing capacity of annuli in $\theta$ and $d$ metrics on the basis of EHI \cite[Lemmas 5.22 and 5.23]{BCM},
	one obtains $\mu(B) \asymp \diam(B,\theta)^\beta \Capa(B, (2\cdot B)^c)$ for all $B \in \sS_k$ and for all large enough $k$. Combining these estimates  with \eqref{e:vd}, we obtain \eqref{e:s2}. 
	To summarize, the conditions \eqref{e:s1} and \eqref{e:s2} arise from the metric and the measure, respectively. It turns out that the necessary conditions \eqref{e:s1} and \eqref{e:s2} on $\rho(\cdot)$ are `almost sufficient' to construct $\rho$; see Theorem \ref{t:suff}. 
	
	We note that there is a tension between the estimates  \eqref{e:s1} and \eqref{e:s2}. On the one hand, in order to satisfy \eqref{e:s1}, the function $\rho(\cdot)$ must be large enough, whereas \eqref{e:s2} imposes that $\rho(\cdot)$ cannot be too large. Next, we sketch how to construct  $\rho$ that satisfies these seemingly conflicting requirements in \eqref{e:s1} and \eqref{e:s2}.  Let $B \in \sS_{k}$, $u \in \contfunc_{\mathrm{c}}(X)$ be such that
	\begin{equation} \label{e:chooseu}
	u \equiv 1 \mbox{ on $B_{d}(x_B, 1.1 r_B)$,} \quad  u \equiv 0 \mbox{ on $B_{d}(x_B,1.9 r_B)^c$,} \quad \sE(u,u) \asymp \frac{m(B)}{r_B^\gamma}.
	\end{equation}
	Such a function exists because of \eqref{e:capest} for $(X,d,m,\sE,\sF)$ with $\beta$ replaced by $\gamma$ and a covering argument (the constants $1.1$ and $1.9$ are essentially arbitrary and could be replaced by $1+\varepsilon$ and $2-\varepsilon$ for small enough $\varepsilon$).
	It is helpful to think of $u$ as the equilibrium potential corresponding to $\Capa(B_{d}(x_B,1.1 r_B), B_{d}(x_B,1.9 r_B)^c)$. Let us define the functions $u_B,\rho_B: \sS_{k+1} \to [0,\infty)$ as 
	\begin{align*}
	u_B(B') &:= \fint_{B'} u\,dm = \frac{1}{m(B')} \int_{B'}u\,dm, \\
	\rho_{B}(B'') &= \sum_{B'' \in \sS_{k+1}, B'' \sim B'} \abs{u_B(B'')-u_B(B')},
	\end{align*}
	where $ B'' \sim B'$ means that $ B''$ and $B'$ share a horizontal edge (or equivalently, $ \lambda\cdot B'' \cap \lambda\cdot B' \neq \emptyset$).
	From the definitions it is clear that
	$u_B$ is a discrete version of $u$ and $\rho_B$ is a discrete version of the gradient of $u$. Using a Poincar\'e inequality on $(X,d,m,\sE,\sF)$ and the bound $\Capa(B', (2\cdot B')^c) \asymp \frac{m(B')}{r_{B'}^\gamma}$, we obtain the following estimate (see \eqref{e:suff3}):
	\begin{equation} \label{e:s2loc}
	\sum_{B' \in \sD_{k+1}(B)} \rho_B(B')^2 \Capa(B', (2\cdot B')^c)  \lesssim \sE(u,u) \lesssim \Capa(B, (2\cdot B)^c).
	\end{equation}
	Since $u_{B}(B_0)=1$ for any $B_0 \in \sS_{k+1}$ with $x_{B_0} \in B$ and $u_{B}(B_N)=0$ for any $B_N \in \sS_{k+1}$ with $x_{B_N} \notin (2 \cdot B)^c$, by the triangle inequality
	\begin{equation} \label{e:s1loc}
	\sum_{i=1}^N \rho_B(B_i) \ge  1, \quad \mbox{for all $(B_i)_{i=0}^N \in \Gamma_{k+1}(B)$.}
	\end{equation}
	Clearly, \eqref{e:s1loc} and \eqref{e:s2loc} are local versions of \eqref{e:s1} and \eqref{e:s2} with $\beta =2$, respectively.
	Here, by ``local version'' we mean that the estimates \eqref{e:s1} and \eqref{e:s2} are satisfied for a function $\rho_{B}$ dependent on $B$ for each fixed $B$.
	To ensure \eqref{e:s1} and \eqref{e:s2} for all scales and locations, we define
	\begin{equation} \label{e:firstdef}
	\rho(B')= \sup_{B \in \sS_k}\rho_B(B'), \quad \mbox{for all $B' \in \sS_{k+1}$,}
	\end{equation}
	where $\rho_B$ is defined as above at all locations and scales. This ensures \eqref{e:s1} and \eqref{e:s2} at all scales with $\beta=2$.
	
	However, $\rho$ should satisfy further additional conditions that the above construction need not obey. Since $\theta \in \sJ(X,d)$, one obtains that
	\begin{equation} \label{e:h1}
	\rho(B) \gtrsim 1 \quad \mbox{for all $B \in \sS_k, k \ge 1$.}
	\end{equation}
	However, the function $\rho(\cdot)$ defined in \eqref{e:firstdef} need not satisfy  the estimate \eqref{e:h1}. This requires us to increase $\rho$ further if necessary. We define the `diameter function'
	\begin{equation} \label{e:pi}
	\pi(B)= \prod_{k=0}^n \rho(B_i), \quad \mbox{for all $B \in \sS_n$,}
	\end{equation}
	where $B_i \in \sS_i$ for all $i=0,\ldots,n$, $B_n=B$ and there is a vertical edge between $B_i$ and $B_{i+1}$ for all $i=0,\ldots,n-1$. If $\rho$ is given by \eqref{e:reldiam}, then clearly $\pi(B)=\diam(B,\theta)/\diam(X,\theta)$. 
	By quasisymmetry, $\diam(B,\theta) \asymp \diam(B',\theta)$ whenever $B$ and $B'$ share a horizontal edge. This suggests the following condition on $\rho$:
	\begin{equation} \label{e:h3} 
	\pi(B) \asymp \pi(B'),\quad \mbox{whenever  $B$ and $B'$ share a horizontal edge,}
	\end{equation}
	where $\pi$ is defined as given in \eqref{e:pi}. 
	Similarly, for constructing measure, we need to ensure that $\rho$ satisfies
	\begin{equation} \label{e:h4}
	\pi(B)^\beta \Capa(B,(2\cdot B)^c) \asymp \sum_{B' \in \sD_n(B)} \pi(B')^\beta \Capa(B',(2\cdot B')^c), 
	\end{equation}
	for all $B \in \sS_k$ and $n > k \ge 0$, where $\sD_n(B)$ denotes the descendants of $B$ in $\sS_n$.
	The conditions \eqref{e:h3} and \eqref{e:h4} are rather delicate because the value of $\pi$ can change drastically if we change $\rho$ by a bounded multiplicative factor. This is due to the fact that the multiplicative `errors' in $\rho$ accumulate as we move to finer and finer scales.
	This suggests that we need to control the constants very carefully. 
	
	To achieve this we need to consider $\beta>2$ (instead of $\beta=2$ considered above).
	We choose $\rho$ defined in \eqref{e:firstdef} uniformly small by picking a function $u$ that satisfies \eqref{e:chooseu}
	along with an additional scale invariant H\"older continuity estimate (see \eqref{e:sufp3},\eqref{e:sufp4} and \eqref{e:sufp5}). Then using
	\begin{align*}
	\lefteqn{\sum_{B' \in \sD_{k+1}(B)} \rho_B(B')^\beta \Capa(B', (2\cdot B')^c) }\\ &\le \norm{\rho}_\infty^{\beta -2}\sum_{B' \in \sD_{k+1}(B)} \rho_B(B')^2  \Capa(B', (2\cdot B')^c)  \lesssim \norm{\rho}_\infty^{\beta -2} \Capa(B, (2\cdot B)^c),
	\end{align*}
	we obtain enough control on the constants in \eqref{e:s2} to ensure \eqref{e:h1} and \eqref{e:h3} after further modification of $\rho$. By the H\"older continuity estimate on $u$, $\norm{\rho}_\infty$ can be made arbitrarily small by increasing the vertical parameter $a$. 
	
	\section{Metric and measure via hyperbolic filling} \label{sec:HypFill}
	Given a metric space, it is often useful to view the space as the boundary of a Gromov hyperbolic space. 
	Such a viewpoint is prevalent but often implicit in various multi-scale arguments in analysis and probability. 
	Roughly speaking, a metric space  viewed simultaneously at different locations and  scales has a natural hyperbolic structure. A nice introduction to this viewpoint can be found in \cite{Sem01}.
	This will be made precise by the notion of hyperbolic filling in Subsection \ref{s:hf}.
	The main tool for the construction of metric is Theorem \ref{t:qscomb}-(a), and the construction of measure uses Lemma \ref{l:meas}.
	To describe the construction, we recall the definition of hyperbolic space in the sense of Gromov.
	\subsection{Gromov hyperbolic spaces and their boundary} \label{s:gromov}
	We briefly  recall the basics of Gromov hyperbolic spaces and refer the reader to \cite{CDP,GH90,Gro,Vai} for a detailed exposition.
	
	Let $(Z,d)$ be a metric space. Given three points $x,y,p \in Z$, we define the \emph{Gromov product} of $x$ and $y$ with respect to the base point $w$ as
	\begin{equation} \label{e:gromovproduct}
	(x\vert y)_w= \frac{1}{2} (d(x,w)+d(y,w)-d(x,y)).
	\end{equation}
	By the triangle inequality, Gromov product is always non-negative.
	We say that a metric space $(Z,d)$ is $\delta$-\emph{hyperbolic}, if for any four points $x,y,z,w \in Z$, we have
	\[
	(x\vert z)_w \ge (x\vert y)_w \wedge (y\vert z)_w - \delta.
	\]
	We say that  $(Z,d)$ is hyperbolic (or $d$ is a hyperbolic metric), if $(Z,d)$ is hyperbolic for some $\delta \in [0,\infty)$.
	If the above condition is satisfied for a fixed base point $w \in Z$, and arbitrary $x,y,z \in Z$, then $(Z,d)$ is $2 \delta$-hyperbolic \cite[Proposition 1.2]{CDP}.
	
	Next, we recall the notion of the boundary of a hyperbolic space. Let $(Z,d)$ be a hyperbolic space and let $w \in Z$.
	A sequence of points $\giveset{x_i} \subset Z$ is said to \emph{converge at infinity}, if
	\[
	\lim_{i,j \to \infty} (x_i\vert x_j)_w = \infty.
	\] 
	The above notion of convergence at infinity does not depend on the choice of the base point $w \in Z$, because by the triangle inequality $\abs{(x\vert y)_w- (x\vert y)_{w'}} \le d(w,w')$.

	Two sequences $\giveset{x_i}, \giveset{y_i}$ that converge at infinity are said to be \emph{equivalent}, if 
	\[
	\lim_{i \to \infty} (x_i\vert y_i)_w = \infty.
	\]
	This defines an equivalence relation among all sequences that converge at infinity.
	As before, is easy to check that the notion of equivalent sequences does not depend on the choice of the base point $w$.
	The \emph{boundary} $\partial Z$ of $(Z,d)$
	is defined as the set of equivalence classes of sequences 
	converging at infinity under the above equivalence relation.
	If there are multiple hyperbolic metrics on the same set $Z$, to avoid confusion, we denote the boundary of $(Z,d)$ by $\partial(Z,d)$ (see Lemma \ref{l:bilip}-(d)).
	The notion of Gromov product can be defined on the boundary as follows: for all $a, b \in \partial Z$
	\[
	(a\vert b)_w = \sup \giveset{ \liminf_{i \to \infty} (x_i\vert y_i)_w : \giveset{x_i} \in a, \giveset{y_i} \in b},
	\]
	and similarly, for $a \in \partial Z, y \in Z$, we define
	\[
	(a\vert y)_w= \sup \giveset{ \liminf_{i \to \infty} (x_i\vert y)_w : \giveset{x_i} \in a}.
	\]
	The boundary $\partial Z$ of the hyperbolic space $(Z,d)$ carries a family of metrics. Let $w \in Z$ be a base point. A metric $\rho:\partial Z \times \partial Z \to [0,\infty)$ on $\partial Z$ is said to be a \emph{visual metric} with \emph{visual parameter} $\alpha >1$  if there exists $k_1,k_2 >0$ such that
	\[
	k_1 \alpha^{-(a\vert b)_w} \le \rho(a,b) \le k_2 \alpha^{-(a\vert b)_w} 
	\]
	Note that visual metrics depend  on the choice of the base point, and on the visual parameter $\alpha$. If a visual metric with base point $w$ and visual parameter $\alpha$ exists, then it can be chosen to be
	\[
	\rho_{\alpha,w}(a,b):= \inf \sum_{i=1}^{n-1}\alpha^{-(a_i\vert a_{i+1})_w},
	\]
	where the infimum is over all finite sequences $\giveset{a_i}_{i=1}^n \subset \partial Z, n \ge 2$ such that $a_1=a, a_n=b$.

	Visual metrics exist as we recall now. A metric space $(Z,d)$ is said to be \emph{proper} if all closed balls are compact.
	For any  $\delta$-hyperbolic space $(Z,d)$, there exists $\alpha_0 >1$ ($\alpha$ depends only on $\delta$) such that if $\alpha \in (1,\alpha_0)$, then there exists a visual metric with parameter $\alpha$ \cite[Chapitre 7]{GH90}, \cite[Lemma 6.1]{BS}. 
	It is well known that quasi-isometry between almost geodesic hyperbolic spaces induces a quasisymmetry on their boundaries (the notion of almost geodesic space is given in Definition \ref{d:almostgeo}). Since this plays a central role in our construction of metric, we recall the relevant definitions and results below.
	
	We say that a map (not necessarily continuous) $f:(X_1,d_1) \to (X_2,d_2)$ between two metric spaces is a \emph{quasi-isometry} if there exists constants $A,B>0$ such that
	\[
	A^{-1} d_1(x,y) - A \le d_2(f(x),f(y)) \le A d_1(x,y) +A,
	\]
	for all $x,y \in X_1$, and
	\[
	\sup_{x_2 \in X_2} d(x_2,f(X_1)) = \sup_{x_2 \in X_2} \inf_{x_1 \in X_1} d(x_2,f(x_1)) \le B.
	\]
	\begin{definition}\label{d:qs}
			A \emph{distortion function} is a homeomorphism of $[0,\infty)$ onto itself. 
			Let $\eta$ be a distortion function. 
			A map $f:(X_1,d_1) \to (X_2,d_2)$ between metric spaces is said to be
			\emph{$\eta$-quasisymmetric}, if $f$ is a homeomorphism and
			\begin{equation} \label{e:qs}
			\frac{d_2(f(x),f(a))}{d_2(f(x),f(b))} \le \eta\left(\frac{d_1(x,a)}{d_1(x,b)}\right)  
			\end{equation}
			for all triples of points $x,a,b \in X_1, x \neq b$. We say $f$ is a \emph{quasisymmetry} if it is $\eta$-quasisymmetric for some distortion function $\eta$.
			We say that metric spaces $(X_1,d_1)$ and $(X_2,d_2)$ are quasisymmetric, if there exists a quasisymmetry $f:(X_1,d_1) \to (X_2,d_2)$.
			We say that  metrics $d_1$ and $d_2$  on $X$ are \emph{quasisymmetric} (or, $d_1$ is quasisymmetric to $d_2$), if the identity map $\operatorname{Id}:(X,d_1) \to (X,d_2)$ is a quasisymmetry.
			A quasisymmetry $f:(X_1,d_1) \to (X_2,d_2)$ is said to be a \emph{power quasisymmetry}, if there exists $\alpha>0, \lambda \ge 1$ such that $f$ is $\eta_{\alpha,\lambda}$-quasisymmetric, where
			\[
			\eta_{\alpha,\lambda}(t)= \begin{cases}
			\lambda t^{1/\alpha}, & \mbox{if $0 \le t <1$},\\
			\lambda t^\alpha, & \mbox{if $t\ge 1$.}
			\end{cases}
			\]
			Recall from Definition \ref{d:cgauge} that the \emph{conformal gauge} of a metric space $(X,d)$ is defined as
			\[
			\sJ(X,d):= \{ \theta \colon X \times X \to [0,\infty) \mid \mbox{$\theta$ is a metric on $X$, $d$ is quasisymmetric to $\theta$} \}.
			\]
			Bi-Lipschitz maps are the simplest examples of quasi-symmetric maps. Recall that a map $f:(X_1,d_1) \to (X_2,d_2)$ is said to be \emph{bi-Lipschitz}, if there exists $C \ge 1$, 
			\[
			C^{-1} d_1(x,y) \le d_2(f(x),f(y)) \le C d_1(x,y), \quad \mbox{for all $x,y \in X_1$.}
			\]
			Two metrics $d_1,d_2:X \times X \to [0,\infty)$ on $X$ are said to be \emph{bi-Lipschitz equivalent} if the identity map $\on{Id}:(X,d_1) \to (X,d_2)$ is bi-Lipschitz.
	\end{definition}
	We collect a few useful facts about quasisymmetric maps.
	\begin{proposition}[{\cite[Lemma 1.2.18]{MT10} and \cite[Proposition 10,8]{Hei}}] \label{p:QS}
		Let the identity map $\operatorname{Id}:(X,d_1) \to (X,d_2)$ be an $\eta$-quasisymmetry 
		for some distortion function $\eta$. By $B_i(x,r)$ we denote the open ball in $(X,d_i)$ with center $x$ and radius $r>0$, for $i=1,2$.
		\begin{enumerate}[label=\textup{(\alph*)},align=left,leftmargin=*,topsep=5pt,parsep=0pt,itemsep=2pt]
			\item For all $A \ge 1,x \in X, r>0$, there exists $s>0$ such that 
			\begin{equation} \label{e:ann1}
			B_2(x,s) \subset B_1(x,r) \subset B_1(x,Ar) \subset B_2(x, \eta(A)s).
			\end{equation}
			In \eqref{e:ann1}, $s$ can be defined as
			\begin{equation} \label{e:defs}
			s= \sup \giveset{0\le s_2 < 2\diam(X,d_2): B_2(x,s_2) \subset B_1(x,r) }
			\end{equation}
			Conversely, for all $A >1,x \in X,r >0$, there exists $t>0$ such that
			\begin{equation} \label{e:ann2}
			B_1(x,r) \subset B_2(x,t) \subset B_2(x,At) \subset B_1(x,A_1r),
			\end{equation}
			where $A_1 = \tilde{\eta}(A)$ and $\tilde{\eta}(t)$ is the distortion function given by $\tilde{\eta}(t)=1/\eta^{-1}(t^{-1})$. In \eqref{e:ann2}, $t$ can be defined as
			\begin{equation} \label{e:deft}
			A t=  \sup \giveset{0 \le r_2 < 2   \diam (X,d_2): B_2(x,Ar_2) \subset B_1(x,A_1 r)}.
			\end{equation}
			\item If $A \subset B \subset X$ such that $0< \diam(A,d_1) \le \diam(B,d_1)<\infty$, then  $0< \diam(A,d_2) \le \diam(B,d_2)<\infty$ and
			\begin{equation} \label{e:diamQS}	
			\frac{1}{2 \eta\left( \frac{\diam(B,d_1)}{\diam(A,d_1)}\right)} \le \frac{\diam(A,d_2)}{\diam(B,d_2)} \le \eta \left(\frac{2 \diam(A,d_1)}{\diam(B,d_1)}\right).
			\end{equation}
		\end{enumerate}
	\end{proposition}
	
	\begin{definition} \label{d:almostgeo}
			A metric space $(X,d)$ is $k$-\emph{almost geodesic}, if
			for every $x, y \in X$ and every $t \in [0, d(x,y)]$, there is some $z \in X$ with
			$\abs{d(x,z) - t} \le k$ and $\abs{d(y,z) - (d(x,y) - t)} \le k$. We say that a metric space is \emph{almost geodesic} if it is $k$-almost geodesic for some $k\ge 0$.
			We recall that quasi-isometry between almost geodesic hyperbolic spaces induces a quasisymmetry between their boundaries.
	\end{definition}
	
	\begin{proposition}[{\cite[Theorem 6.5 and Proposition 6.3]{BS}}] \label{p:bdymap}
		Let $(Z_1,d_1)$ and $(Z_2,d_2)$ be two almost geodesic, $\delta$-hyperbolic metric spaces. Let $f:(Z_1,d_1) \to (Z_2,d_2)$ denote quasi-isometry. 
		\begin{enumerate}[label=\textup{(\alph*)},align=left,leftmargin=*,topsep=5pt,parsep=0pt,itemsep=2pt]
			\item If $\giveset{x_i} \subset Z_1$ converges at infinity, then $\giveset{f(x_i)} \subset Y$ converges at infinity.
			If  $\giveset{x_i}$ and $\giveset{y_i}$ are equivalent sequences in $X$ converging at infinity, then
			$\giveset{f(x_i)}$ and $\giveset{f(y_i)}$ are also equivalent.
			\item If $a \in \partial Z_1$ and $\giveset{x_i} \in a$, let $b \in \partial Z_2$ be the equivalence class of $\giveset{ f(x_i)}$. Then $\partial f: \partial Z_1 \to \partial Z_2$ is well-defined, and is a bijection.
			\item Let $p_1 \in Z_1$ be a base point in $Z_1$, and let $f(p_1)$ be a corresponding base point in $Z_2$. Let $\rho_1, \rho_2$ denote  visual metrics (with not necessarily the same visual parameter) on $\partial Z_1, \partial Z_2$ with base points $p_1, f(p_1)$ respectively. Then the induced boundary map $\partial f: (\partial Z_1, \rho_1) \to (\partial Z_2, \rho_2)$ is a power quasisymmetry. 
		\end{enumerate}
	\end{proposition}
	\begin{remark} \label{r:const}
		The distortion function $\eta$ for the quasisymmetry $\partial f$ in (c) above can be chosen to depend only on the constants associated with the quasi-isometry $f:Z_1 \to Z_2$ and the constants associated with the properties of being almost geodesic and Gromov hyperbolic for $Z_1,Z_2$.
	\end{remark}
	
	\subsection{Hyperbolic filling of a compact metric space} \label{s:hf}
	Given a compact metric space $(X,d)$, one can construct an almost geodesic, hyperbolic space whose boundary equipped with a visual metric can be identified with $(X,d)$.
	We assume further that $(X,d)$ is doubling and uniformly perfect.
	Recall that a metric space $(X,d)$ is said to be
	\emph{$K_D$-doubling} if any ball $B(x,r)$ can be covered by $K_D$ balls of radius $r/2$,
	and be \emph{doubling} or satisfy the \emph{metric doubling property}
	if it is $K_D$-doubling for some $K_D \in \mathbb{N}$. A metric space $(X,d)$
	is said to be \emph{$K_P$-uniformly perfect} if $B(x,r)\setminus B(x,r/K_P) \neq \emptyset$ for any ball $B(x,r)$ such that $X \setminus B(x,r) \neq \emptyset$,
	and \emph{uniformly perfect} if it is $K_P$-uniformly perfect for some $K_P \in (1,\infty)$.
	
	We recall the notion of hyperbolic filling due to Bourdon and Pajot \cite{BP}, based on a similar construction due to Elek \cite{Ele}.
	We recall the definition in \cite{Car13}.
	Let $(X,d)$ be a compact, doubling, uniformly perfect, metric space. 
	For a ball $B=B(x,r)$ and $\alpha>0$, by $\alpha\cdot B$ we denote the ball $B(x,\alpha r)$. It is possible that balls with different centers and/or radii denote the same set. For this reason, we adopt the convention that a ball $B$ comes with a predetermined center and radius denoted by $x_B$ and $r_B$ respectively. We fix two parameter $a > 8$ and $\lambda \ge 3$. The parameters $a$ and $\lambda$ are respectively called the \emph{vertical} and \emph{horizontal} parameters of the hyperbolic filling. 
	For each $n \ge 0$, let $\sS_n$ denote a finite covering of $X$ by open balls such that for all $B \in \sS_n$, there exists a center $x_B \in X$ such that
	\begin{equation} \label{e:hyp1} 
	B=B(x_B,2 a^{-n}),
	\end{equation}
	and for any distinct pair $B \neq B'$ in $\sS_n$, we have
	\begin{equation} \label{e:hyp2}
	B(x_B, a^{-n}/2) \cap B(x_{B'}, a^{-n} /2) = \emptyset.
	\end{equation}
	We assume that 
	\begin{equation} \label{e:singleton}
	\sS_0= \giveset{X}
	\end{equation}
	is a singleton (by scaling the metric if necessary). 
	We remark that the assumption  \eqref{e:singleton} is just for convenience. 
	For arbitrary (but finite) diameter, we choose $n_0 \in \bZ$ such that
	$a^{-n_0}>\diam(X,d) \ge a^{-n_0-1}$. For the general compact case  we  replace $0$ with $n_0$, so that we have coverings $\sS_k$ for all $k \ge n_0$ such that $\sS_k$ is a covering by `almost' pairwise disjoint balls of radii roughly $a^{-k}$ as given in \eqref{e:hyp1} and \eqref{e:hyp2}.

	We construct a graph whose vertex set is $\sS= \coprod_{n=0}^\infty \sS_n$.
	Next, we construct a tree structure of \emph{vertical edges} on $\sS$. For each $n \ge 0$, we partition $\sS_{n+1}$ into pairwise disjoint sets $\giveset{T_n(B): B \in \sS_n}$ indexed by $\sS_n$, with $\sS_{n+1}= \coprod_{B \in \sS_n} T_n(B)$ satisfying the following property:
	\begin{equation} \label{e:hyp3}
	\mbox{if }B' \in T_n(B) \mbox{ for some } B \in \sS_n, B' \in \sS_{n+1}, \mbox{ then } d(x_{B'},A_n)= d(x_{B'},x_B),
	\end{equation}
	where $A_n= \{x_B : B \in \mathcal{S}_n\}$ denotes the set of centers of the balls in $\mathcal{S}_n$.
	In other words, if $B' \in T_n(B)$, then $x_B \in A_n$ is a minimizer to the distance between $x_{B'}$ and $A_n$. Since such a minimizer always exists, there exists a (not necessarily unique) partition  $\giveset{T_n(B): B \in \sS_n}$ of $\sS_{n+1}$ for all $n \ge 0$.
	We call the elements of $T_n(B)$ as the \emph{children} of $B$.
	From now on, let us fix one such partition $\giveset{T_n(B): B \in \sS_n}$ for each $n \ge 0$. 
	We say that there exist a \emph{vertical edge} between two sets $B,B' \in \sS$, if there exists $n \ge 0$ such that either $B \in \sS_n, B' \in \sS_{n+1} \cap T_n(B)$ or $B' \in \sS_{n+1}, B \in \sS_{n+1} \cap T_n(B')$; in other words, one of them is a child of the other. 
	Note that the vertical edges form a tree on the vertex set $\sS$, with base point (or root) $w$, where $\sS_0= \giveset{w}$. The unique path from the base point to a vertex $B \in \sS$ denotes the genealogy $g(B)$. More precisely, we define the genealogy $g(B)$ as $(B)$  as
	\[
	g(B)= \begin{cases}
	(B), &\mbox{if $B \in \sS_0$,}\\
	(B_0,B_1,\ldots,B_{n}), &\mbox{if $B=B_n \in \sS_n, n \ge 1$,  and }\\
	  & \quad \mbox{$B_{i+1} \in T(B_i)$, for $i=0,\ldots,n-1$.}
	\end{cases} 
	\]
	In the above definition, if $ 0 \le i <n$, we denote the vertex $B_i \in \sS_i$ by $g(B)_i$.
	If $B\in \sS_n$, and  $l >n$, we define $\sD_l(B)$ as the descendants of $B$ in the generation $l$
	\begin{equation} \label{e:defD}
	\sD_l(B) := \giveset{B' \in \sS_l: g(B')_n=B}.
	\end{equation}
	For $B \in \sS_n$, we denote $\cup_{l \ge n+1} \sD_l(B)$ by $\sD(B)$ which are the descendants of $B$.
	
	Using the horizontal parameter $\lambda \ge 3$, we define another family of edges on the vertex set $\sS$ call the \emph{horizontal edges}.
	We say $B \sim B'$ if there exists $n \ge 0$ such that $B,B' \in \sS_n$ and $\lambda\cdot B \cap \lambda\cdot B' \neq \emptyset$.
	We say that there is a \emph{horizontal edge} between $B,B' \in \sS$, if $B \sim B'$ and they are distinct (so as to avoid self-loops). 
	
	\begin{definition}[Hyperbolic filling] \label{d:hypfilling}
			Let $\sS_d=(\sS,E)$ denote the graph with  vertices in $\sS$ and whose edges $E$ are obtained by the taking the union of horizontal and vertical edges. With a slight abuse of notation, we often view $\sS_d$ as a metric space equipped with the (combinatorial) graph distance, which we denote by $D_\sS: \sS\times \sS \to \bZ_{\ge 0}$. The metric space 
			$\sS_d= (\sS, D_\sS)$ is almost geodesic and hyperbolic \cite[Proposition 2.1]{BP}. The metric space $\sS_d$ is said to be a \emph{hyperbolic filling} of $(X,d)$.
	\end{definition}
	We refer to Subsection \ref{s:construction} for a construction of hyperbolic filling.
	Note that the hyperbolic filling is not unique as we make an arbitrary choice of covering. Even if the covering is fixed, the choice of children $T_n(B)$ is not necessarily unique.  Nevertheless, any two hyperbolic fillings (with possibly different parameters) of a metric space are quasi-isometric to each other \cite[Corollaire 2.4]{BP}.

	We fix the base point of $\sS_d$ to be $w \in \sS$, where $\giveset{w}=\sS_0$.
	We now define a map $p: X \to \partial \sS_d$ that identifies $X$ with the boundary of $\sS_d$ as follows. For each $x \in X$, choose a sequence $\giveset{B_i}$ with $x \in B_i \in \sS_i, i \in \bN$. Then it is easy to see that the sequence $\giveset{B_i}$ converges at infinity. Let $p(x) \in \partial \sS_d$ denote the equivalence class containing $\giveset{B_i}$. 
	
	The map $p$ is a bijection and its inverse  $p^{-1}:\partial \sS_d \to X$ can be described as follows. For any $a \in \partial \sS_d$, and for any $\giveset{B_i} \in a$,  the corresponding sequence of centers $\giveset{x_{B_i}}$ is a convergent sequence in $X$, and the limit is $p^{-1}(a)= \lim_{i \to \infty} x_{B_i}$. The map $p^{-1}$ is well-defined; that is, if $\giveset{B_i}$ and $\giveset{B'_i}$ are equivalent sequences that converge at infinity, then $\lim_{i \to \infty} x_{B_i} = \lim_{i \to \infty} x_{B'_i}$.

	We summarize the properties of the hyperbolic filling $\sS_d$ and its boundary $\partial \sS_d$ as follows:
	\begin{proposition}[{\cite[Proposition 2.1]{BP}}] \label{p:filling}
		Let $(X,d)$ denote a compact, doubling, uniformly perfect metric space. Let $\sS_d$ denote a hyperbolic filling with vertical parameter $a>1$, and horizontal parameter $\lambda \ge 3$. Then $\sS_d$ is almost geodesic Gromov hyperbolic space.
		The map $p:X \to \partial \sS_d$ is a homeomorphism between $X$ and $\partial \sS_d$. If we choose the base point $w \in \sS_d$ as the unique vertex in $\sS_0$, then there exists $K>1$ such that 
		\[
		K^{-1} a^{-(p(x)\vert p(y))_w} \le d(x,y) \le K a^{-(p(x)\vert p(y))_w}
		\]
		for all $x,y \in X$.
	\end{proposition}
	By the above proposition we can recover the metric space $(X,d)$ from its hyperbolic filling $\sS_d$ with horizontal parameter $\lambda$ and vertical parameter $a$ (up to  bi-Lipschitz equivalence) as the boundary $\partial \sS_d$ equipped with a visual metric with base point $w$ and visual parameter $a$. There is a minor gap in \cite{BP} as pointed out in \cite[Section 4]{BoSa} and \cite{BBS}. We remark that the horizontal parameter $\lambda$ was chosen to be $1$ in \cite{BP}. If $\lambda=1$, then the hyperbolic filling need not be Gromov hyperbolic \cite[Example 8.8]{BBS}. As pointed out in \cite{BoSa}, if $\lambda>1$ such problems do not arise.
	
	For technical reasons following \cite[(2.8)]{Car13}, we will often assume that
	\begin{equation}\label{e:tech}
	\lambda \ge 32, \hspace{3mm} a \ge 24 (\lambda \vee K_P),
	\end{equation}
	where $K_P$ is such that $(X,d)$ is $K_P$-uniformly perfect.
	
	\subsection{Construction of hyperbolic fillings} \label{s:construction}
	Since the metric spaces we deal with need not be compact, we need a suitable
	substitute for hyperbolic fillings. 
	To circumvent this difficulty, we view the metric space as an increasing union of compact spaces and construct a sequence of hyperbolic fillings. Quasisymmetric maps and doubling measures have nice compactness properties that persist under such limits.
	
	We recall the notion of net in a metric space.
	\begin{definition} \label{d:net}
		Let $(X,d)$ be a metric space and let $\varepsilon>0$.
		A subset $N$ of $X$ is called an \emph{$\varepsilon$-net} in $(X,d)$
		if the following two conditions are satisfied:
		\begin{enumerate}[label=\textup{\arabic*.},align=left,leftmargin=*,topsep=5pt,parsep=0pt,itemsep=2pt]
			\item (Separation) $N$ is \emph{$\varepsilon$-separated} in $(X,d)$, i.e., $d(x,y) \geq \varepsilon$ for any $x,y \in N$ with $x \not= y$.
			\item (Maximality) If $N\subset M\subset X$ and $M$ is $\varepsilon$-separated in $(X,d)$, then $M=N$.
		\end{enumerate}
	\end{definition}
	In the lemma below, we recall a standard construction of hyperbolic filling and some of its properties. 
	\begin{lemma}[{Cf.\ \cite[Lemma 2.2]{Car13} and \cite[Theorem 2.1]{KRS}}] \label{l:con}
		Let $(X,d)$ be a complete, $K_P$-uniformly perfect, $K_D$-doubling metric space  such that either $\diam(X,d)=\frac 1 2$ or $\infty$.
		Let $a >8$ and let $x_0 \in X$. Let $N_0$ be a $1$-net in $(X,d)$ such that $x_0 \in N_0$. Define inductively the sets $N_k$ for $k \in \bN$ such that
		\[
		\mbox{$N_{k-1} \subset N_k$, and $N_{k}$ is $a^{-k}$-net in $(X,d)$, for all $k \in \bN$,}.
		\]
		For $k <0$ and $k \in \bZ$, we define $N_k$ to be a $a^{-k}$-net in $(N_{k+1},d)$ such that $x_0 \in N_k$ for all $k \in \bZ$ (Note that $N_k$ need not be $a^{-k}$-net in $(X,d)$ for $k<0$). For each $x \in N_k$ and $k \in \bZ$, we pick a \emph{predecessor} $y \in N_{k-1}$ such that $y$ is a closest point to $x$ in $N_{k-1}$ (by making a choice if there is more than one closest point); that is $y \in N_{k-1}$ satisfies
		\[
		d(x,y)= \min_{z \in N_{k-1}} d(x,z).
		\]
		For any $x \in N_k, k \in \bZ$, we denote its predecessor as defined above by $P(x) \in N_{k-1}$.
		\begin{enumerate}[label=\textup{(\alph*)},align=left,leftmargin=*,topsep=5pt,parsep=0pt,itemsep=2pt]
			\item For all $k \in \bZ$, and for any two distinct points $x,y \in N_k$, we have
			\begin{equation} \label{e:con1}
			B(x,a^{-k}/2) \cap B(y,a^{-k}/2) = \emptyset.
			\end{equation}
			We have the following covering property:
			\begin{align}\label{e:con2}
			\cup_{x \in N_k} B(x,a^{-k}) &=X, \quad \mbox{for all $k \ge 0$,}\\
			\cup_{x \in N_k} B(x, (1-a^{-1})^{-1}a^{-k}) &=X,  \quad \mbox{for all $k \in \bZ$.} \label{e:con3}
			\end{align}
			In particular, if $\diam(X,d)=\frac 1 2$, the coverings $\sS_n= \giveset{B(y, (1-a^{-1})^{-1}a^{-n}) \mid y \in N_n}$ for all $n \ge 0$ is a covering that satisfies \eqref{e:hyp1} and \eqref{e:hyp2}. For any $n \ge 0$ and for any $B=B(x_B,a^{-n}) \in \sS_n$, the sets
			\[
			T_n(B)= \giveset{ B(y,a^{-n-1}) \mid y \in N_{k+1} \mbox{such that } x_B=P(y) }
			\]
			forms a partition of $\sS_{n+1}$ as required by \eqref{e:hyp3}.
			
			\item Let $a,\lambda$ satisfy \eqref{e:tech}. Let $y \le N_{k+1}$ be such that 
			\[ 
			B(y,(1-a^{-1})^{-1} a^{-k-1}) \cap B(P(y),a^{-k}/3) \neq \emptyset.
			\]
			Then for any $z \in N_{k+1}$ such that $d(y,z)<2 (1-a^{-1})^{-1}a^{-k-1}$, we have
			$P(y)=P(z)$. (In other words, $y$ corresponds to the center of a non-peripheral  ball in $\sS_{k+1}$ as given in Definition \ref{defE}).
			\item  Let $k \in \bZ$ and $y \in N_k$. Let $D_k(y)$ denote the set of descendants of $y$ defined by
			\begin{equation} \label{e:descendants}
			D_k(y)= \giveset{y} \cup \giveset{ z \in N_l \mid \mbox{such that $l>k$ and $P^{l-k}(z)=y$}}.
			\end{equation}
			Then
			\begin{equation} \label{e:con4}
			\overline{B}(y,(1-a^{-1})^{-1} a^{-k}) \supset \overline{D_k(y)} \supset B(y, (2^{-1}- (a-1)^{-1}) a^{-k}).
			\end{equation}
			The space $\overline{D_k(y)}$ with the restricted metric $d$ is $K_D^2$-doubling and $K_P'$-uniformly perfect, where $K_P'=2a K_P (1-a^{-1})^{-1}(2^{-1}-(a-1)^{-1})^{-1}$.
		\end{enumerate}
	\end{lemma}
	
	\begin{proof}
	\begin{enumerate}[label=\textup{(\alph*)},align=left,leftmargin=*,topsep=5pt,parsep=0pt,itemsep=2pt]
		\item The properties \eqref{e:con1} and \eqref{e:con2} follow from the separation and maximality properties of the $a^{-k}$-net $N_k$ in $X$ respectively.
		We use the notation $P^k(y)$ denote the $k$-predecessor of $y$ (for example, $P^2(y)=P(P(y))$).
		To show \eqref{e:con3}, by \eqref{e:con2} it suffices to consider the case $k<0$. By \eqref{e:con2}, for any $y \in X$ there exists $y_0 \in N_0$ such that $d(y_0,y) <1$. Define $y_{l}=P^{-l}(y)$ for all $l <0$. 
		Since $d(y_l,y_{l+1})< a^{-l}$ for all $l<0$, we have
		\begin{align} \label{e:conp1}
		d(y_k,y) &\le d(y_0,y)+\sum_{l=-1}^{k} d(y_l,y_{l+1}) \nonumber \\
		&< \sum_{l=0}^k a^{-l}= (1-a^{-1})^{-1} (a^{-k}-a^{-1}) <(1-a^{-1})^{-1}  a^{-k}.
		\end{align}
		Since $y \in X$ is arbitrary and $y_k \in N_k$, we have \eqref{e:con3}.
		\item By the triangle inequality, we have
		\begin{align*}
		d(z,P(y)) &\le d(z,y)+d(y,P(y)) \le (2 \lambda+1) (1-a^{-1})^{-1}a^{-k-1}+ \frac 1 3 a^{-k}  \\
		&< a^{-k}/2 \quad \mbox{(by \eqref{e:tech})}
		\end{align*}
		By \eqref{e:con2} and $d(z,P(y))<a^{-k}/2$, we conclude that $P(z)=P(y)$.
		\item  By \eqref{e:con2} and triangle inequality, we have
		\begin{equation} \label{e:conp2} 
		d(y,z) \ge a^{-k}/2, \quad \mbox{for all $z \in N_{k+1} \setminus D(y)$ and for all $y \in N_k$}.
		\end{equation}
		By \eqref{e:conp1}, we have
		\begin{equation} \label{e:conp3} 
		\overline{D(z)} \subset \overline{B}(z,(1-a^{-1})^{-1}a^{-k-1}), \quad \mbox{for all $z \in N_{k+1}$}.
		\end{equation}
		Since $\bigcup_{w \in N_l} \overline{D(w)}$ is dense and closed (by the doubling property), we have
		\begin{equation} \label{e:conp4}
		\bigcup_{w \in N_l} \overline{D(w)}= X, \quad \mbox{for all $l \in \bZ$}.
		\end{equation}
		Combining \eqref{e:conp2}, \eqref{e:conp3}, \eqref{e:conp4} and using triangle inequality, we obtain \eqref{e:con4}.
		
		Next, we show that $\overline{D(y)}$ is $K_D$-doubling.
		More generally, we show that any subset $Y \subset X$ is $K_D^2$-doubling. Let $B(x,r) \cap Y, x \in Y$ be an arbitrary ball in $Y$. Since $(X,d)$ is $K_D$-doubling, the ball $B(x,r)$ can be covered by $N$ balls $B(x_i,r/4), i =1,\ldots, N$, where $N \ge K_D^2$. If $B(x_i ,r/4) \cap Y \neq \emptyset$, we choose $y_i \in B(x_i ,r/4) \cap Y$, so that $B(x_i,r/4) \subset B(y_i,r/2)$.
		Hence all such balls $B(y_i,r/2) \cap Y$ cover $B(x,r) \cap Y$. 
		
		Let $B(x,r) \cap \overline{D(y)}$ be an arbitrary ball in $\overline{D(y)}$ such that $x \in \overline{D(y)}$ and $B(x,r ) \cap \overline{D(y)} \neq  \overline{D(y)}$. Let $n \in \bZ$ be the unique integer such that 
		\[
		(1-a^{-1})^{-1}a^{-n} < r \le 	(1-a^{-1})^{-1}a^{-n+1}. 
		\]
		Since $\overline{D(y)}= \cup_{z \in N_n \cap D(y)} \overline{D(z)}$ for all $n \ge k$, by \eqref{e:con4}, there exists $z \in D(y) \cap N_n$ such that 
		\begin{equation} \label{e:up1}
		d(z,x) \le (1-a^{-1})^{-1} a^{-n} <r.
		\end{equation}
		Since $(X,d)$ is $K_P$-uniformly perfect, and using \eqref{e:con4} and $B(z,(2^{-1}-(a-1)^{-1})  a^{-n}) \neq X$, there exists $w \in \overline{D(y)}$ such that
		\begin{equation} \label{e:up2}
		(2^{-1}-(a-1)^{-1})^{-1} a^{-n} > d(w,z) \ge \frac{1}{K_P}(2^{-1}-(a-1)^{-1})  a^{-n}.
		\end{equation}
		We consider two cases, depending on whether or not $d(z,x)< \frac{1}{2} a^{-n}$. If $d(z,y) \ge \frac{1}{2} a^{-n}$, then
		\begin{equation} \label{e:up3}
		r>	d(z,x) \ge \frac{1}{2} a^{-n} \ge \frac{r}{2a((1-a^{-1})^{-1}}.
		\end{equation}
		On the other hand, if $d(z,x) < \frac{1}{2} a^{-n}$, then
		\begin{align*}
		d(w,x) \vee d(z,x) &\le d(z,w)+ d(z,x) \\
		& < (2^{-1}-(a-1)^{-1})^{-1} a^{-n} + \frac{1}{2} a^{-n} < a^{-n} <r.
		\end{align*}
		Hence by \eqref{e:up2},  if $d(z,x) < \frac{1}{2} a^{-n}$, we have 
		\begin{align} \label{e:up4} 
		d(w,x) \vee d(z,x) &\ge \frac{1}{2} d(w,z) \ge \frac{1}{2K_P}(2^{-1}-(a-1)^{-1})  a^{-n} \nonumber \\
		& \ge \frac{r}{2a K_P (1-a^{-1})^{-1}(2^{-1}-(a-1)^{-1})^{-1} }.
		\end{align}
		By \eqref{e:up3} and \eqref{e:up4}, $\overline{D(y)}$ is $2a K_P (1-a^{-1})^{-1}(2^{-1}-(a-1)^{-1})^{-1}$-uniformly perfect.
	\end{enumerate} 
	\end{proof}
	
	\begin{definition}[Extended hyperbolic filling] \label{d:exhyp}
		Let $(X,d)$ be a complete, $K_P$-uniformly perfect doubling metric space such that either $\diam(X,d)=\frac 1 2$ or $\infty$. Let $a>8$, $\lambda \ge 32$ be constants that satisfy \eqref{e:tech}. Let $x_0 \in X$ and consider the sets $N_k, k \in \bZ$ as defined in Lemma \ref{l:con}. Define
		\[
		\sS_k= \giveset{B(x, 2a^{-k}) : x \in N_k},  k\in \bZ.
		\]
		For any $k \in \bZ$ and for any  pair of distinct balls $B, B' \in \sS_k$, we say that there is a \emph{horizontal edge} between $B$ and $B'$ (denoted as $B \sim B'$) if and only if $\lambda \cdot B \cap \lambda \cdot B' \neq \emptyset$.  For any $k \in \bZ$ and for any  $B(x,2a^{-k}) \in \sS_k, B(y,2a^{-k-1}) \in \sS_{k+1}$, we say that there is a \emph{vertical edge} between $B(x,2a^{-k})$ and  $B(y,2a^{-k-1})$, if $x$ is the predecessor of $y$ (as defined in Lemma \ref{l:con}). 
		We define a graph $(V,E)$ with vertex set $V= \coprod_{k \in \bZ} \sS_k$ and the edge set $E$ defined by the union of horizontal and vertical edges. This graph is called the \emph{extended hyperbolic filling} of $(X,d)$ with horizonal parameter $\lambda$ and vertical parameter $a$.
		
		If $(X,d)$ is compact, the subgraph of the extended hyperbolic filling induced by $\sS= \coprod_{k \in \bZ_{\ge 0}} \sS_k$ forms a hyperbolic filling as given in Definition \ref{d:hypfilling}. 
		
		On the other hand, if $(X,d)$ is non-compact, we view $X$ as an increasing limit of compact spaces $\overline{D_l(x_0)}$  as $l \to -\infty$, where $D_k(x_0)$ is as defined in \eqref{e:descendants}.
		For any $k,l \in \bZ, l \le 0, k \ge l$, we define 
		\[
		\sS_k^l= \giveset{ B(x,2 a^{-k}) \cap \overline{D_l(x_0)}: x \in N_k \cap D_l(x_0)}. 
		\]
		We define a graph with vertex set $\sS^l= \coprod_{k \ge l, k \in \bZ} \sS_k^l$, whose edges are the union of horizonal and vertical edges. In this case, the vertical edges are defined using predecessor relation as above and the horizontal edges are defined with respect to the space $\overline{D_l(x_0)}$. That is $B \cap \overline{D_l(x_0)}, B' \cap \overline{D_l(x_0)} \in 	\sS_k^l$ share a horizontal edge if and only if $\lambda \cdot B \cap \lambda \cdot B' \cap \overline{D_l(x_0)} \neq \emptyset$. This graph with vertex set $\sS^l$ can be viewed as a hyperbolic filling of the compact space $\overline{D_l(x_0)}$. In the non-compact case, we think of  the sequence of hyperbolic fillings defined with vertex set $\sS^l$  as `converging' to the extended hyperbolic filling defined above as $l \to -\infty$.
	\end{definition}
	\subsection{Combinatorial description of the conformal gauge}
	The purpose of this section is to recall a combinatorial description of the conformal gauge essentially due to M.~Carrasco Piaggo \cite{Car13}. In this section, we fix a compact, doubling, uniformly perfect metric space $(X,d)$ and a hyperbolic filling $\sS_d= (\sS,D_\sS)$ with horizontal parameter $\lambda \ge 8$ and vertical parameter $a>1$ that satisfies \eqref{e:singleton}.
	
	Propositions \ref{p:bdymap} and \ref{p:filling} suggest the following strategy to construct
	metrics that are in the conformal gauge of $(X,d)$. By changing the metric of the hyperbolic filling $\sS_d$ to another metric that is almost geodesic and bi-Lipschitz  (in particular, quasi-isometric), every visual metric of its boundary is changed to a metric in the conformal gauge of $(X,d)$. Perhaps surprisingly, all metrics in the conformal gauge can be obtained in this manner (up to a bi-Lipschitz map) as explained in Theorem \ref{t:qscomb}.

	The change of metric in a hyperbolic filling is done using a \emph{weight function} 
	$\rho: \sS \to (0,1)$ on its vertex set. We define 
	\begin{equation}\label{e:defpi}
	\pi(B)= \prod_{B' \in g(B)} \rho(B').
	\end{equation}
	A \emph{path} $\gamma= \giveset{B_i}_{i=1}^N$ in $\sS_d$ is a sequence of vertices such that there is an edge between $B_i$ and $B_{i+1}$ for all $i=1,\ldots,N-1$. In this case, we say that $\gamma$ is a path from $B_1$ to $B_N$. 
	A path is said to be \emph{simple}, if no two vertices in the path are the same.
	A path is said to be \emph{horizontal} (resp. \emph{vertical}), if all the edges in the path are horizontal (resp. vertical).
	We define the $\rho$-length of a path $\gamma=\giveset{B_i}_{i=1}^N$ by 
	\begin{equation} \label{e:defLrho}
	L_\rho(\gamma)= \sum_{i=1}^N \pi(B_i), 
	\end{equation}
	where $\pi$ is as defined in \eqref{e:defpi}.
	For points $x,y \in X$ and $n \in \bN$, the set of paths $\Gamma_n(x,y)$ is defined as
	\begin{equation} \label{e:Gamman}
	\Gamma_n(x,y)= \biggl\{ \gamma=\{B_i\}_{i=1}^k \biggm\vert \begin{minipage}{150pt}$\gamma$ is a path from $B_1$ to $B_k$, $x \in B_1$, $y \in B_k$, $B_1 \in \sS_n$, $B_k \in \sS_n$\end{minipage}\biggr\}.
	\end{equation}
	We remark that a path $\gamma \in \Gamma_n(x,y)$ need not be a horizontal path.\\
	For two distinct points $x,y \in X$ and $\alpha \ge 2$, we define
	\begin{align}
	m_\alpha(x,y)&= \max \giveset{k : B \in \sS_k, x\in \alpha\cdot B,y \in \alpha\cdot B},\nonumber\\
	c_\alpha(x,y)&= \giveset{ B \in \sS_{k} : k=m_\alpha(x,y), x\in \alpha\cdot B,y \in \alpha\cdot B}, \nonumber\\
	\pi(c_\alpha(x,y))&= \max_{B \in c_\alpha(x,y)} \pi(B). \label{e:defpic}
	\end{align}

	\begin{assumption} \label{a:h1-h3}
		A weight function $\rho: \sS \to (0,1)$ may satisfy some of the following hypotheses:
		\begin{enumerate}[label=\textup{(H\arabic*)},align=left,leftmargin=*,topsep=5pt,parsep=0pt,itemsep=2pt]
			\item[\hypertarget{H1}{$\on{(H1)}$}](Quasi-isometry) There exist $0<\eta_-\le \eta_+<1$ so that $\eta_- \le \rho(B) \le \eta_+$ for all $B \in \sS$.
			\item[\hypertarget{H2}{$\on{(H2)}$}](Gromov product) There exists a constant $K_0\ge 1$ such that for all $B,B' \in \sS$ with $B \sim B' \in \sS$, we have
			\[
			\pi(B) \le K_0 \pi(B'),
			\]
			where $\pi$ is as defined in \eqref{e:defpi}.
			\item[\hypertarget{H3}{$\on{(H3)}$}](Visual parameter)
			There exists $\alpha \in [2,\lambda/4]$ and a constant $K_1\ge 1$ such that for any pair of points $x,y\in X$, there exists $n_0 \ge 1$ such that if $n \ge n_0$ and $\gamma$ is a path in $\Gamma_n(x,y)$, then 
			\[
			L_\rho(\gamma) \ge K_1^{-1} \pi(c_\alpha(x,y)),
			\]
			where $\Gamma_n(x,y),L_\rho, \pi(c_\alpha(x,y))$ are as defined in \eqref{e:Gamman}, \eqref{e:defLrho}, and \eqref{e:defpic} respectively. 
		\end{enumerate}
	\end{assumption}
	The following observation concerns the stability of the above assumption under `finite perturbations'.
	\begin{remark}
		Let $\rho, \rho': \sS\to (0,1)$ be two different weight functions such that the set $\giveset{ B \in \sS : \rho(B) \neq \rho'(B)}$ is finite.
		Then if $\rho$ satisfies the hypotheses \hyperlink{H1}{$\on{(H1)}$},  \hyperlink{H2}{$\on{(H2)}$}, and \hyperlink{H3}{$\on{(H3)}$}, then so does $\rho'$ (with possibly different constants).
	\end{remark}

	The weight function $\rho$ can be used to define a metric on $\sS$ that is bi-Lipschitz equivalent to $D_\sS$ as we recall below. We summarize the properties of the metric below.
	\begin{lemma}[{\cite[Lemma 2.3 and Proposition 2.4]{Car13}}] \label{l:bilip}
		Let $(X,d)$ be a compact, doubling, uniformly perfect metric space with $\diam(X,d)=\frac 1 2$, and let $\sS_d=(\sS,D_\sS)$ denote a hyperbolic filling with parameters $\lambda,a$ satisfying \eqref{e:tech} and \eqref{e:singleton}. Let $\rho:\sS\to (0,1)$ be a weight function that satisfies  \hyperlink{H1}{$\on{(H1)}$} and  \hyperlink{H2}{$\on{(H2)}$}. Then there exists a metric $D_\rho$ on $\sS$ such that:
		\begin{enumerate}[label=\textup{(\alph*)},align=left,leftmargin=*,topsep=5pt,parsep=0pt,itemsep=2pt]
			\item  $D_\rho$ is bi-Lipschitz equivalent to $D_\sS$; that is there exists $\Lambda \ge 1$ such that
			\[
			\Lambda^{-1} D_\sS(B,B') \le  D_\rho(B,B') \le \Lambda D_\sS(B,B'), \hspace{3mm}\mbox{for all $B,B'\in\sS$;} 
			\]
			\item any simple vertical path $\gamma= \giveset{B_i}_{i=1}^n$ joining $B \in \sS_m$ and $B' \in \sS_{m'}$ satisfies
			\[
			D_\rho(B,B')=\sum_{i=1}^{n-1} D_\rho(B_i,B_{i+1})= \abs{ \log \frac 1 {\pi(B)} - \log \frac 1 {\pi(B')}};
			\]
			\item $(\sS,D_\rho)$ is almost geodesic and Gromov hyperbolic.
			\item The identity map $\on{Id}:(\sS,D_\sS) \to (\sS,D_\rho)$ induces the identity map on their boundaries as described in Proposition \ref{p:bdymap}. That is, a sequence $\giveset{B_i}$ converges at infinity in $(\sS,D_\sS)$ if and only if it converges  at infinity in $ (\sS,D_\rho)$, and any two sequences that converge at infinity in $(\sS,D_\sS)$ are equivalent if and only if they are equivalent in $(\sS,D_\rho)$. In particular, the bijection
			$p: X \to \partial (\sS,D_\sS)$ described before Proposition \ref{p:bdymap} can be viewed as a bijection $\tilde{p}:X \to \partial (\sS,D_\rho)$ by composing with the induced identity map above.
			\item Assume in addition that  \hyperlink{H3}{$\on{(H3)}$} is also satisfied. Let $(\cdot\vert\cdot)_\rho$ denote the  Gromov product on $(\sS,D_\rho)$ with base point $w \in \sS_0$ extended to its boundary.
			Define
			$\tilde{\theta}_\rho: \partial(\sS,D_\rho) \times \partial(\sS,D_\rho) \to [0,\infty)$  as
			\begin{equation} \label{e:deftiltheta}
			\tilde{\theta}_\rho(\tilde{p}(x),\tilde{p}(y)) = \inf \sum_{i=1}^{n-1}  e^{-(\tilde{p}(x_i)\vert\tilde{p}(x_{i+1}))_\rho},
			\end{equation}
			where the infimum is over all finite sequence of points $\giveset{x_i}_{i=1}^n$ in $X$ such that $n \in \bN$, $x_1= x$, and  $x_n= y$.
			Then $\tilde{\theta}_\rho$ is a visual metric on $\partial(\sS,D_\rho)$ with visual parameter $e$. Moreover, 
			there exists $K>1$  such that
			\begin{align*}
			K^{-1} e^{-(\tilde{p}(x)\vert\tilde{p}(y))_\rho} &\le \tilde{\theta}_\rho(\tilde{p}(x),\tilde{p}(y)) \le K  e^{-(\tilde{p}(x)\vert\tilde{p}(y))_\rho},\\
			K^{-1} \pi(c_\alpha(x,y)) &\le \tilde{\theta}_\rho(\tilde{p}(x),\tilde{p}(y)) \le K \pi(c_\alpha(x,y)).
			\end{align*}
		\end{enumerate}
	\end{lemma}
	
	\begin{proof}[Sketch of the proof]
	We briefly recall the construction of the metric $D_\rho$. Let $E$ denote the edge set of the hyperbolic filling and let $\eta_-,\eta_+, K_0$ denote the constants in hypotheses \hyperlink{H1}{$\on{(H1)}$}, and \hyperlink{H2}{$\on{(H2)}$}.
	Define a function $\ell_\rho: E \to (0,\infty)$ as
	\[
	\ell_\rho(e)= \begin{cases}
	2 \max \giveset{- \log (\eta_+), - \log(\eta_-), \log (K_0)}, &\mbox{ if $e$ is a horizontal edge,}\\
	\abs{ \log\frac{\pi(B')}{\pi(B)}}, &\mbox{ if $e=(B',B)$ is  vertical.}
	\end{cases}
	\]
	Then the distance $D_\rho: \sS \times \sS \to [0,\infty)$ is defined 
	as
	\[
	D_\rho(B,B') = \inf_{\gamma} \sum_{i=1}^{N-1} \ell_\rho(e_i),
	\]
	where the infimum is taken over all paths $\gamma=\giveset{B_i}_{i=1}^N$ where $N$ varies over $\bN$, $B_1=B, B_N=B'$ and $e_i=(B_i,B_{i+1})$ is an edge for each $i=1,\ldots,N-1$.
	
	Part (a) is immediate from the definition of $D_\rho$. Part (b) and (c) are proved in \cite[Lemma 2.3]{Car13}. Part (d) follows from (a),(c) and Proposition \ref{p:bdymap}. Part (e) follows  from \cite[Proposition 2.4]{Car13}.
	\end{proof}
	
	The following theorem provides a combinatorial description of the conformal gauge $\sJ(X,d)$.
	In \cite[Theorem 1.1]{Car13}, Carrasco Piaggio has provided a combinatorial description of the \emph{Ahlfors regular} conformal gauge
	\[
	\sJ_{\on{AR}}(X,d)= \{ \theta \in \sJ(X,d) \mid \textrm{an Ahlfors regular measure $\mu$ on $(X,\theta)$ exists} \}.
	\]
	In  \cite[Theorem 1.1]{Car13} the hypotheses  \hyperlink{H1}{$\on{(H1)}$},  \hyperlink{H2}{$\on{(H2)}$},  \hyperlink{H3}{$\on{(H3)}$} correspond to a combinatorial description of $\sJ(X,d)$, whereas the hypothesis (H4) corresponds to the existence of an Ahlfors regular measure. This theorem is essentially contained in \cite{Car13}.
	\begin{theorem}[{Cf.\ \cite[Theorem 1.1]{Car13}}] \label{t:qscomb}
		Let $(X,d)$ be a compact, doubling, uniformly perfect metric space.
		\begin{enumerate}[label=\textup{(\alph*)},align=left,leftmargin=*,topsep=5pt,parsep=0pt,itemsep=2pt]
			\item Let $\sS_d=(\sS,D_\sS)$ denote a hyperbolic filling with parameters $\lambda,a$ satisfying \eqref{e:tech} and \eqref{e:singleton}. Let $\rho:\sS \to (0,1)$ be a weight function that satisfies the conditions  \hyperlink{H1}{$\on{(H1)}$},  \hyperlink{H2}{$\on{(H2)}$}, and  \hyperlink{H3}{$\on{(H3)}$}. 
			Define  the  metric $\theta _\rho:X \times X \to [0,\infty)$  as 
			\begin{equation} \label{e:deftheta}
			\theta_\rho(x,y)= \tilde{\theta}_\rho(\tilde{p}(x),\tilde{p}(y))\quad \mbox{ for $x,y \in X$,}
			\end{equation}
			where $\tilde{\theta}_\rho$ is as defined in \eqref{e:deftiltheta}. 
			Then  $\theta_\rho$ satisfies the following properties:
			\begin{enumerate}[label=\textup{(\roman*)},align=left,leftmargin=*,topsep=5pt,parsep=0pt,itemsep=2pt]
				\item $\theta_\rho \in \sJ(X,d)$; that is $\theta_\rho$ is quasisymmetric to $d$.
				\item there exists $C>0$ such that
				\begin{equation} \label{e:dist}
				C^{-1} \pi(c_\alpha(x,y)) \le \theta_\rho(x,y) \le C \pi(c_\alpha(x,y)),
				\end{equation}
				where $\alpha$ is the constant in  \hyperlink{H3}{$\on{(H3)}$}. Furthermore, there exists $K>1$ such that
				\begin{equation} \label{e:diam}
				K^{-1}\pi(B) \le \diam(B,\theta_\rho) \le K \pi(B) \qquad \mbox{ for all $B \in \sS$.}
				\end{equation}
				\item $\theta_\rho$ is a visual metric of the hyperbolic space $(\sS,D_\rho)$ constructed in Lemma \ref{l:bilip} in the following sense: there exists $C>0$ such that
				\[
				C^{-1} \theta_\rho(x,y)\le e^{-(\tilde{p}(x)\vert\tilde{p}(y))_\rho} \le C \theta_\rho(x,y),
				\]
				where $\tilde{p}: X \to \partial(S,D_\rho)$ is the bijection described in Lemma \ref{l:bilip}(d), and 
				$(\cdot\vert\cdot)_\rho$ denotes the Gromov product (extended to the boundary) on the hyperbolic space $(\sS,D_\rho)$ with base point $w \in \sS_0$.
				\item The distortion function $\eta$ of the power quasisymmetry $\on{Id}:(X,d) \to (X,\theta)$   can be chosen to depend only on the constants in  \hyperlink{H1}{$\on{(H1)}$},  \hyperlink{H2}{$\on{(H2)}$}, and  \hyperlink{H3}{$\on{(H3)}$}.
			\end{enumerate}
			\item  Conversely, let $\theta \in \sJ(X,d)$ be any metric in the conformal gauge. Then there exists a hyperbolic filling $\sS_d=(\sS,D_\sS)$ of $(X,d)$ with horizontal parameter $\lambda$, vertical parameter $a$, and a weight function $\rho:\sS \to (0,1)$ that satisfies the hypotheses  \hyperlink{H1}{$\on{(H1)}$},  \hyperlink{H2}{$\on{(H2)}$},  \hyperlink{H3}{$\on{(H3)}$}, and such that the metric $\theta_\rho$ defined in \eqref{e:deftheta} is bi-Lipschitz equivalent to $\theta$. 
		\end{enumerate}
	\end{theorem}
	
	\begin{proof}
	We begin with the proof of (b).\\
	\noindent (b) Let $\on{Id}: (X,d) \to (X,\theta)$ be an $\eta$-quasisymmetry for some distortion function $\eta$.
	
	The definition of the weight function $\rho$ in \cite{Car13} uses an Ahlfors regular measure. Since there is no such measure available in our setting, the following definition is more suited for our purposes. We normalize the metric $\theta$, so that $\diam(X,\theta)=\frac 1 2$.
	We will define the weight function $\rho:\sS \to (0,1)$ so that
	\[
	\pi(B)= \diam_\theta(B),
	\]
	for all $B \in \sS$, where $\sS$ is a hyperbolic filling of $(X,d)$ with parameters $\lambda,a$. Fix any $\lambda \ge 32$. The vertical parameter $a >1$ will be determined later in the proof.
	Hence we define $\rho:\sS \to (0,1)$ as
	\[
	\rho(B)= \begin{cases}
	\frac 1 2 & \mbox{if $B \in \sS_0$,}\\
	\frac{\diam_\theta(B)}{ \diam_\theta(g(B)_{n-1})} & \mbox{if $B \in \sS_n, n \ge 1$.}
	\end{cases}
	\]
	First, we show  \hyperlink{H2}{$\on{(H2)}$}.
	Let $B \sim B'$ with $B,B' \in \sS_n$. Then choose $y \in \lambda \cdot B \cap \lambda \cdot B'$. By triangle inequality,
	\[
	B  \subset B_d(x_B,2a^{-n}) \subset B_d(y, (\lambda+2)a^{-n}), \hspace{3mm} B' \subset  B_d(y, (\lambda+2)a^{-n}).
	\]
	By uniform perfectness, and triangle inequality, for any $r< \frac 1 2$, $ r/K_P \le \diam_d(B_d(x,r)) \le 2r$. Therefore by \eqref{e:diamQS}, we obtain
	\[
	\frac{1}{2 \eta(4 (\lambda +2)K_P)}\le \frac{\diam_\theta (B)}{\diam_\theta(B_d(y,(\lambda+2)a^{-n}))} \le \eta(8K_P/(\lambda+2)).
	\]
	Since the same inequality holds with $B$ replaced with $B'$, we have  \hyperlink{H2}{$\on{(H2)}$} with constant
	\[
	K_0=2 \eta(4 (\lambda +2)K_P)\eta(8K_P/(\lambda+2)),
	\]
	that depends only on the distortion function $\eta$,  the constant $K_P$ of uniform perfectness, and the horizontal parameter $\lambda$ (in particular, does not depend on the vertical parameter $a$).

	Next, we show  \hyperlink{H1}{$\on{(H1)}$}, which again relies on \eqref{e:diamQS}.
	We will choose $a >2(\lambda +1)$ large enough so that $\eta_+= \frac 1 2$ in  \hyperlink{H1}{$\on{(H1)}$}. Clearly this choice works when $B \in \sS_0$. If $B= \sS_n, n \ge 1$, and by denoting $B'= g(B)_{n-1}$, we have $x_B \in B'$. For $n \ge 2$, we  write (the case $n=1$ is easier and left to the reader)
	\[
	\rho(B)= \frac{ \diam_\theta(B)} {\diam_\theta(B')} = \frac{ \diam_\theta(B)} {\diam_\theta((4a) \cdot B)}  \frac{\diam_\theta((4a)\cdot B)} {\diam_\theta(B')}.
	\]
	Each of the terms can be estimated (from above and below) using \eqref{e:diamQS}, since by the triangle inequality and $d(x_B,x_{B'}) < 2 a^{-n+1}$ we have
	$B \subset (4a)\cdot B$, and $(4a)\cdot B \supset B(x_B, 2a^{-n-1})\supset B'$.
	Hence, we obtain
	\begin{align*}
	\rho(B) &\le  2  \eta\left( \frac{ 2 \diam_d(B) }{\diam_d((4a)\cdot B)}\right)  \eta\left( \frac{ \diam_d((4a)\cdot B) }{\diam_d(B')}\right)  \\
	& \le  2  \eta\left( K_p/a  \right)  \eta\left( 16 K_P \right) \\
	\rho(B) &\ge  \left[ 2  \eta\left( \frac {\diam_d((4a)\cdot B)}{  \diam_d(B) }\right)  \eta\left( 2\frac{\diam_d(B')}{ \diam_d((4a)\cdot B) }\right) \right]^{-1} \\
	&\ge  \left[ 2  \eta\left( 8 K_P\right)  \eta\left( 2K_P /a\right) \right]^{-1}
	\end{align*}
	First we choose $a$ large enough so that  $ 2  \eta\left( K_p/a  \right)  \eta\left( 16 K_P \right)  \le \frac 1 2$ and \eqref{e:tech} are satisfied. We set $\eta_-= \left[ 2  \eta\left( 8 K_P\right)  \eta\left( 2K_P /a\right) \right]^{-1}$. Hence we obtain  \hyperlink{H1}{$\on{(H1)}$}.

	For  \hyperlink{H3}{$\on{(H3)}$}, we once again use \eqref{e:diamQS}, to see that $\pi(B)= \diam_\theta(B)$ is comparable to $\diam_\theta( \lambda\cdot B)$ for all $B \in \sS$. More precisely, we have
	\[
	\diam_\theta (B) \le \diam_\theta (\lambda\cdot B) \le 2 \eta(2 \lambda K_P) \diam_\theta (B)
	\]
	for all $B \in \sS$.
	For any path $\gamma =\giveset{B_i}_{i=1}^m \in \Gamma_n(x,y)$, we choose points $x_i \in \lambda \cdot B_i \cap \lambda \cdot B_{i+1}, i =1,\ldots,m-1$, $x_0=x,x_{m}=y$ so that
	\begin{align} \label{e:cdcg1}
	\theta(x,y) &\le \sum_{i=0}^{m-1} \theta(x_i,x_{i+1}) \le  \sum_{i=1}^{m} \diam_\theta (\lambda\cdot B_i) \nonumber \\
	 &\le \sum_{i=1}^{m} \diam_\theta ( B_i) = 2 \eta(2 \lambda K_P) L_\rho(\gamma).
	\end{align}
	Fixing $\alpha=2$, and let $C \in c_2(x,y)$ such that $\pi(c_2(x,y))= \pi(C)$. Let $m=m_2(x,y)$. Let $B \in \sS_{m+1}$ be such that $x \in B$. By definition of $m_2(x,y)$, $y \notin 2 \cdot B$. Therefore $d(x,y) \ge d(x_B,y) -d(x_B,x) \ge a^{-m-1}$.
	By \eqref{e:diamQS}, and $\pi(c_2(x,y)) \le \diam_\theta(2\cdot C)$ we have
	\begin{align}\label{e:cdcg2}
	\pi(c_2(x,y)) &\le  2 \eta \left( \frac{\diam_d(2 \cdot C)}{d(x,y)}\right)\theta(x,y) \nonumber \\ 
	&\le 2 \eta \left( \frac{8 a^{-m}}{a^{-m-1}}\right)\theta(x,y)=2 \eta(8 a) \theta(x,y).
	\end{align}
	Combining \eqref{e:cdcg1} and \eqref{e:cdcg2} yields  \hyperlink{H3}{$\on{(H3)}$} with $\alpha=2$.
	\smallskip \\
	(a) This part is essentially contained in \cite{Car13}.
	The hypotheses  \hyperlink{H1}{$\on{(H1)}$} and  \hyperlink{H2}{$\on{(H2)}$} are used to construct a metric $D_\rho$ on $\sS$ as given in Lemma \ref{l:bilip}.
	If $\theta_\rho$ were defined using \eqref{e:deftheta}, it clearly satisfies the symmetry $\theta_\rho(x,y)=\theta_\rho(y,x)$, and triangle inequality. 
	The role of  \hyperlink{H3}{$\on{(H3)}$} is to show that $\theta_\rho(x,y)$ is at least  $e^{-(\tilde{p}(x)\vert\tilde{p}(y))_\rho}$ (up to a constant factor) as explained in Lemma \ref{l:bilip}-(e). 
	The fact that $\theta_\rho$ is quasisymmetric to $d$ follows from Lemma \ref{l:bilip}, Propositions \ref{p:filling} and \ref{p:bdymap}-(c). The statement about the dependence of distortion function $\eta$ on the constants follow from Remark \ref{r:const}.
	
	The estimate \eqref{e:diam} is also implicitly contained in \cite{Car13} and is a consequence of \eqref{e:dist}. Choose $x,y \in B$ such that $d(x,y) \ge \diam(B,d)/2$. Since $\on{Id}: (X,d) \to (X,\theta_\rho)$ is an $\eta$-quasisymmetry, by \eqref{e:diamQS}, there exists $C_1>0$ such that
	\[
	\theta_\rho(x,y) \le  \diam(B,\theta_\rho) \le C_1 \theta_\rho(x,y).
	\]
	Since $d(x,y) \ge \diam(B)/2$, $B$ is at a bounded distance in $(\sS,D_\sS)$ from any set $C \in c_\alpha(x,y)$. Combining these estimates along with \eqref{e:dist} and \eqref{e:diamQS}, we obtain
	\[
	\diam(B,\theta_\rho) \asymp \theta_\rho(x,y) \asymp \pi(c_\alpha(x,y)).
	\]
	\end{proof}
	
	\subsection{Construction of measure using the hyperbolic filling}
	
	As in Subsection \ref{s:hf}, we fix a compact, doubling, uniformly perfect metric space $(X,d)$ with $\diam(X,d)=\frac{1}{2}$, and a hyperbolic filling $\sS_d= (\sS,D_\sS)$ with horizontal parameter $\lambda$ and vertical parameter $a$ that satisfy \eqref{e:tech}. 
	\begin{definition}[gentle function] \label{d:gentle}
			Let $\sC \colon \sS \to (0,\infty)$ and $K \ge 1$. We say that $\sC$ is $K$-\emph{gentle} if
			\[
			K^{-1} \sC(B') \le \sC(B) \le K \sC(B),
			\]
			whenever there is an edge between $B$ and $B'$. We say that $\sC \colon \sS \to (0,\infty)$ is \emph{gentle} if it is $K$-gentle for some $K \ge 1$.
			The notion of $K$-gentle function extends to any function $f \colon V \to (0,\infty)$ on a graph $G=(V,E)$. In other words, we say that a function $f \colon V \to (0,\infty)$ is $K$-gentle if $\log f$ is $(\log K)$-Lipschitz with respect to the graph distance metric.
			
			We sometimes need to distinguish between the horizontal and vertical edges (see Theorem \ref{t:suff}). We say that $\sC \colon \sS \to (0,\infty)$
			is $(K_h,K_v)$-\emph{gentle} if 
			\[
			K_h^{-1} \sC(B') \le \sC(B) \le K_h \sC(B'),
			\]
			whenever $B$ and $B'$ share a horizontal edge, and 
			\[
			K_v^{-1} \sC(B') \le \sC(B) \le K_v \sC(B'),
			\]
			whenever $B$ and $B'$ share a vertical edge. Therefore every $(K_h,K_v)$-gentle function is $(K_h \vee K_v)$-gentle.
	\end{definition}
	Given a hyperbolic filling $\sS$, we need to approximate a ball $B(x,r)$ by a ball in the filling $\sS$. We introduce this notion in the following definition.
	\begin{definition} \label{d:AS}
			Let $(X,d)$ be a doubling metric space. Let $(\sS,D_\sS)$ be a hyperbolic filling of $(X,d)$ with parameters $a,\lambda$ that satisfy \eqref{e:tech} as constructed in Lemma \ref{l:con}-(a). By Lemma \ref{l:con}-(a), given a ball $B(x,r) \neq X$, there exists 
			$n \in \bZ$ and $B \in \sS_n$ such that 
			\begin{equation}  \label{e:as1}
			2 a^{-n-1} \le r < 2 a^{-n}, \quad \mbox{and} \quad d(x_B,x) <2 a^{-n}.
			\end{equation}
			We define 
			\begin{equation} \label{e:defAS}
			\sA_\sS(B(x,r))= \giveset{ B \in \sS: \mbox{$n \in \bZ_{\ge 0}$ and $B \in \sS_n$ satisfy \eqref{e:as1}}}.
			\end{equation}
		We remark that if $B, B' \in \sA_\sS(B(x,r))$, then $x \in B \cap B' \neq \emptyset$ and hence $B$ and $B'$ share a horizontal edge.
	\end{definition}

	Often, the measures in this work will satisfy the following volume doubling and reverse volume doubling properties.
	\begin{definition}[Volume doubling and Reverse volume doubling properties] \label{d:vd-rvd}
		Let $(X,d)$ be a metric space and let $\mu$ be a Borel measure on $X$.
		\begin{enumerate}[label=\textup{(\alph*)},align=left,leftmargin=*,topsep=5pt,parsep=0pt,itemsep=2pt]
		\item We say that $\mu$ satisfies the \emph{volume doubling property} \ref{VD},
			or $\mu$ is a doubling measure on $(X,d)$, or $(X,d,\mu)$ is \ref{VD}, if there exists $C_D \in (1,\infty)$ such that
			\begin{equation} \label{VD} \tag*{$\on{VD}$}
			0 < \mu(B(x,2r)) \le C_D \mu(B(x,r)) <\infty \quad \mbox{for all $x \in X$, $r \in (0,\infty)$.}
			\end{equation}
		\item We say that $\mu$ satisfies the \emph{reverse volume doubling property} \ref{RVD},
			or $(X,d,\mu)$ is \ref{RVD}, if there exist $C_1, C_2  \in (1,\infty), \alpha \in (0,\infty)$ such that
			\begin{equation} \label{RVD} \tag*{$\on{RVD}$}
			\mu(B(x,R)) \ge C_1^{-1} \left(\frac{R}{r}\right)^\alpha \mu(B(x,r)) 
			\end{equation}
		for all $x\in X$, $0<r \le R < \diam_{d}(X)/C_2$.
		\end{enumerate}
	\end{definition}
	\begin{remark}\label{r:doubling}
		We recall the following connections between the doubling and uniform perfectness properties of a metric space $(X,d)$
		and the volume doubling and reverse volume doubling properties.
			\begin{enumerate}[label=\textup{(\alph*)},align=left,leftmargin=*,topsep=5pt,parsep=0pt,itemsep=2pt]
				\item\label{it:VDimprove}If $\mu$ satisfies \ref{VD} on $(X,d)$, then $(X,d)$ is a doubling metric space. Conversely, every complete doubling metric space admits a measure that satisfies \ref{VD} \cite[Theorem 13.3]{Hei}. The constant $2$ in the definition of \ref{VD} is essentially arbitrary, as \ref{VD} implies
				\begin{equation} \label{e:VDimprove}
				\frac{\mu(B(x,R))}{\mu(B(x,r))} \le C_D \left(\frac{R}{r}\right)^\alpha, \mbox{for all $x \in X, 0<r \le R$, where $\alpha=\log_2 C_D$.}
				\end{equation}
				\item\label{it:RVD-uniformly-perfect}Let $\mu$ be a measure that satisfies \ref{VD} on $(X,d)$. Then $\mu$ satisfies \ref{RVD} if and only if $(X,d)$ is uniformly perfect \cite[Exercise 13.1]{Hei}.
			\end{enumerate}
	\end{remark}

	We introduce a hypothesis on a weight function $\rho:\sS \to (0,\infty)$ that plays an important role in the construction of a measure.
	\begin{assumption} \label{a:meas}
			Let $\sC:\sS \to (0,\infty)$ be a gentle function, and let $\beta >0$.
			A weight function $\rho: \sS \to (0,1)$ is said to be $(\beta,\sC)$-\emph{compatible} if there exists $K_2 \ge 1$ such that for all $B \in \sS_m$, and $n >m$,
			\[
			K_2^{-1} \pi(B)^\beta \sC(B) \le \sum_{B' \in \sD_n(B)} \pi(B')^\beta \sC(B') \le K_2 \pi(B)^\beta\sC(B),
			\] 
			where $\sD_n(B)$ denotes the descendants of $B$ of generation $n$ as defined in \eqref{e:defD}.
	\end{assumption}
	The above assumption is similar to (H4) in \cite{Car13}. 
	The following lemma is an analogue of \cite[Lemma 2.7]{Car13}. 
	\begin{lemma} \label{l:meas}
			Let $(\sS,D_\sS)$ be a hyperbolic filling of a doubling, $K_P$-uniformly perfect, compact metric space $(X,d)$ as given in Lemma \textup{\ref{l:con}-(a)}. Let $\rho:\sS \to (0,1)$ be a weight function that satisfies  \hyperlink{H1}{$\on{(H1)}$} and  \hyperlink{H2}{$\on{(H2)}$}. 
			Let $\sC:\sS \to (0,\infty)$ be a gentle function, and let $\beta >0$, such that $\rho$ is $(\beta,\sS)$-compatible. For $n \ge 0$, denote
			\[
			\mu_n:= \sum_{B \in \sS_n} \pi(B)^\beta  \sC(B) \delta_{x_B},
			\]
			where $\delta_{x_B}$ denotes the Dirac measure at $x_B$. Let $\mu$ be any weak* subsequential limit of $\mu_n$.
			Then there exists $C_1>1$ such that, for all $x \in X, r\le \diam(X,d)/2$, and for all $B \in \sA_\sS(B(x,r))$, we have
			\begin{equation} \label{e:volest}
			C_1^{-1} \pi(B)^\beta \sC(B) \le \mu((B(x,r)) \le C_1 \pi(B)^\beta \sC(B),
			\end{equation}
			where $\sA_\sS$ is as given in Definition \textup{\ref{d:AS}}.
			Furthermore, $\mu$ satisfies \ref{VD} on $(X,d)$.
	\end{lemma} 
	
	\begin{proof}[Sketch of the proof]
	We only sketch the proof and skip the details as it follows from almost the same argument as  \cite[Lemma 2.7]{Car13}.
	
	Let $x \in X, r\le \diam(X)/2, B \in \sA_\sS(B(x,r))$ and $B=B(x_B,2a^{-m})$. Choose $B_1 \in \sS_{m+2}$ such that $x \in B_1=B(x_{B_1},2a^{-m-2})$.
	By \cite[(2.10)]{Car13}, the centers of all the descendants of $B_1$ belong to $\overline B(x,r/2)$. This along with $(\beta,\sC)$-compatibility implies that
	\begin{align*}
	\mu(B(x,r)) &\ge \mu(\overline B(x,r/2)) \\
	&\ge \liminf_{n \to \infty} \mu_n\left(\giveset{x_C: C \in \sD_n(B_1)}\right) \\
	&\quad \mbox{(since $\overline B(x,r/2) \supset \giveset{x_C: C \in \sD_n(B_1)}$)}\\
	&\gtrsim \liminf_{n \to \infty} \sum_{B' \in \sD_n(B_1)} \pi(B')^\beta \sC(B') \\
	&\gtrsim \pi(B_1)^\beta \sC(B_1) \gtrsim \pi(B)^\beta \sC(B)\\
	&\quad \mbox{(by $(\beta,\sC)$-compatibility and gentleness of $\sC$)}. 
	\end{align*}
	By the argument in \cite[proof of Lemma 2.7]{Car13}, for any $B' \in \sS_n, n \ge m$ satisfying $x_{B'} \in B(x,r+a^{-m})$, we have 
	$g(B')_n \sim g(B_1)_m$, where $B_1$ is as  defined above.
	For the upper bound, for any $B' \in \sS_n, n \ge m$ such that $x_{B'} \in B(x,r+a^{-m})$, we have that $g(B')_m \sim B$. Therefore, we estimate
	\begin{align*}
	\mu(B(x,r)) &\le \mu(B(x,r+a^{-m})) \\
	&\le \limsup_{n \to \infty} \mu_n(B(x,r+a^{-m})) \\
	&\le \sum_{C \sim g(B_1)_m} \sum_{B' \in \sD_n(C)}  \pi(B')^\beta \sC(B')  \\
	&\lesssim   \sum_{C \sim g(B_1)_m} \pi(C)^\beta \sC(C).
	\end{align*}
	Since $C \sim g(B_1)_m$ and $B \sim g(B_1)_m$, by gentleness of $\sC$, we have $\sC(C) \asymp \sC(B)$. By \hyperlink{H2}{$\on{(H2)}$}, we have $\pi(C) \asymp \pi(B)$.
	Furthermore, by doubling the number of such $C \in \sS_m$ such that  $C \sim g(B_1)_m$ is bounded by a constant that depends only on the parameters of the filling. Combining these estimates, we obtain the desired upper bound $\mu(B(x,r)) \lesssim \pi(B)^\beta \sC(B)$. This completes the proof of \eqref{e:volest}.
	
	The conclusion that $\mu$ satisfies \ref{VD} follows from \eqref{e:volest} and the gentleness of $\sC$.
	\end{proof}
	
	In the following proposition, we express the measure in Lemma \ref{l:meas} using the metric in Theorem \ref{t:qscomb}-(a).
	\begin{proposition}\label{p:measure}
		Let $(X,d)$ be a compact, doubling, uniformly perfect metric space. Let $(\sS,D_\sS)$ be a hyperbolic filling with parameters $\lambda,a$ satisfying \eqref{e:tech}, \eqref{e:singleton} as given in Lemma \textup{\ref{l:con}-(a)}.
		Let $\sC: \sS \to (0,\infty)$ be a gentle function and let $\beta>0$.
		Let $\rho:\sC\to (0,1)$ be a weight function that satisfies \hyperlink{H1}{$\on{(H1)}$}, \hyperlink{H2}{$\on{(H2)}$},  \hyperlink{H3}{$\on{(H3)}$}, and $(\beta,\sC)$-compatibility.
		Let $\theta=\theta_\rho \in \sJ(X,d)$ denote the metric in Theorem \textup{\ref{t:qscomb}-(a)} and $\mu$ denote the measure on $X$ constructed in Lemma \textup{\ref{l:meas}}.
		Then, there exist $C_1>1$ such that
		\begin{equation} \label{e:meas} 
		C_1^{-1} r^\beta \sC(B) \le \mu(B_\theta(x,r))  \le   C_1 r^\beta \sC(B),
		\end{equation}
		for all $x \in X, r < \diam(X,\theta) ,B \in \sA_\sS(B_d(x,s))$, where $s$ is the largest number in $[0,2 \diam(X,d)]$ such that $B_d(x,s) \subset B_\theta(x,r)$ (as defined in \eqref{e:defs}) and $\sA_\sS(B_d(x,s))$ is as given in Definition \textup{\ref{d:AS}}.
	\end{proposition}
	
	\begin{proof}
	By an easy covering argument using the metric doubling property, it suffices to consider the case $r < \diam(X,\theta)/2$, so that $B_\theta(x,r) \neq X$.
	
	Let $x \in X$, $0<r<\diam(X,\theta)/2$ and let $s=\sup\{ s_1>0: B_d(x,s_1) \subset B_\theta(x,r)\}$.
	By Lemma \ref{l:meas}, $\mu$ satisfies \ref{VD} in $(X,d)$.
	By \eqref{e:ann1} and in  \eqref{e:VDimprove},  $\mu$ satisfies \ref{VD} in $(X,\theta)$ and there exists $C_2>1$ such that 
	\begin{equation} \label{e:mms1}
	C_2^{-1} \mu(B_d(x,s)) \le \mu(B_\theta(x,r))  \le  C_2 \mu(B_d(x,s)).
	\end{equation}
	By \eqref{e:ann1}, there exists $A_1>1$ such that $B_\theta(x,r) \subset A_1 \cdot B$ and $B \subset B_\theta(x,A_1r)$ for all $B \in  \sA_\sS(B_d(x,s))$.
	Hence by \eqref{e:diamQS}  and uniform perfectness, there exists $C_3>1$ such that
	\begin{equation} \label{e:mms2}
	C_3^{-1}  r  \le \diam(B,\theta)  \le  C_3  r, \quad \mbox{for all $B \in  \sA_\sS(B_d(x,s))$. }
	\end{equation}
	By \eqref{e:diam}, \eqref{e:mms1}, and \eqref{e:mms2}, we obtain \eqref{e:meas}.
	\end{proof}
	
	\subsection{Simplified hypotheses for construction of metric and measure}
	The goal of this section is to present an analogue of \cite[Theorem 1.2]{Car13} that will be used in the construction of metric measure space.
	Some of the main ideas in the proof of \cite[Theorem 1.2]{Car13} are  inspired by the `weight-loss program' of Keith and Laakso \cite[\textsection 5.2]{KL}.

	We continue to consider  a compact, doubling, uniformly perfect metric space $(X,d)$, and a hyperbolic filling $\sS_d= (\sS,D_\sS)$ with horizontal parameter $\lambda \ge 8$ and vertical parameter $a>1$ that satisfy \eqref{e:singleton}. 
	We consider $\beta>0$, $\sC:\sS \to (0,\infty)$ such that $\sC$ is gentle.
	Theorem \ref{t:suff} provides simpler sufficient conditions (S1), (S2) that allows us to construct a weight function that satisfies  \hyperlink{H1}{$\on{(H1)}$},  \hyperlink{H2}{$\on{(H2)}$},  \hyperlink{H3}{$\on{(H3)}$}, and is $(\beta,\sC)$-compatible.
	To state the sufficient conditions, we recall the following definition. 
	\begin{definition} \label{d:Gamk}
			For $B \in \sS_k, k \ge 0$, we define $\Gamma_{k+1}(B)$ 
			as the set of horizontal paths $\gamma=\giveset{B_i}_{i=1}^N, N \ge 2$ such that $B_i \in \sS_{k+1}$ for all $i=1,\ldots,N$,  $B_i \sim B_{i+1}$ for all $i=1,\ldots,N-1$, $x_{B_1} \in B$, $x_{B_N} \notin 2\cdot B$,  and $x_{B_i} \in 2 \cdot B$ for all $i=1,\ldots,N-1$.
	\end{definition}
	We introduce a subadditive estimate based on \cite[Proposition 3.15]{BM}.
	\begin{definition} \label{defE}
			We say that $B \in \sS_k, k \ge 1$ is {\em non-peripheral} if every horizontal neighbour of $B$ descends from the same parent. More precisely, $B \in \sS_k, k \ge 1$ is non-peripheral if
			\[
			B \sim B' \mbox{ implies that } g(B)_{k-1} = g(B')_{k-1}.
			\]
			By $\sN$ we denote the set of all non-peripheral vertices in $\sS$. We say that a function $\sC:\sS \to (0,\infty)$ satisfies \hypertarget{E}{$\on{(E)}$} if it obeys the following estimate:
			\begin{enumerate}[label=\textup{(E)},align=left,leftmargin=*,topsep=5pt,parsep=0pt,itemsep=2pt]
				\item[(E)] there exists $\delta\in(0,1)$ such that 
				\[
				\sC(B) \le (1-\delta) \sum_{B' \in \sN \cap \sD_{k+1}(B)} \sC(B') 
				\]
				for all $B \in \sS_{k}, k \ge 1$.
			\end{enumerate}
			In particular, the condition (E) implies $\sN \cap \sD_{k+1}(B) \neq \emptyset$ for all $k \ge 1, B \in \sS_k$.
	\end{definition}
	The following result is an analogue of \cite[Theorem 1.2]{Car13}.
	\begin{theorem}  \label{t:suff}
		Let $(X,d)$ be a compact, $K_D$-doubling, $K_P$-uniformly perfect metric space  and let $\beta>0$. 
		Consider a hyperbolic filling $\sS_d= (\sS,D_\sS)$ with horizontal parameter $\lambda \ge 8$ and vertical parameter $a>1$ that satisfies \eqref{e:singleton}. 
		Let $\sC: \sS \to (0,\infty)$ be a $(K_h,K_v)$-gentle function that satisfies  \hyperlink{E}{$\on{(E)}$}.  Then, there exists $\eta_0 \in (0,1)$ that depends only on $\beta, K_D,K_h, \lambda$ (but not on the vertical constants $a, K_v$ or uniform perfectness constant $K_P$) such that
		the following is true.
		If there exists a function $\sigma:\sS\to \left[0,\frac 1 4\right)$ that satisfies:
		\begin{enumerate}[label=\textup{(S\arabic*)},align=left,leftmargin=*,topsep=5pt,parsep=0pt,itemsep=2pt]
			\item[\hypertarget{S1}{$\on{(S1)}$}]for all $B \in \sS_k, k \ge 0$, if $\gamma=\giveset{B_i: 1 \le i \le N}$ is a path in $\Gamma_{k+1}(B)$ (as given in Definition \textup{\ref{d:Gamk}}), then
			\[
			\sum_{i=1}^N \sigma(B_i) \ge 1,
			\]
			\item[\hypertarget{S2}{$\on{(S2)}$}]and for all $k \ge 0$, and all $B \in \sS_k$, we have
			\[
			\sum_{B' \in \sD_{k+1}(B)}\sigma(B')^{\beta}\sC(B') \le \eta_0 \sC(B),
			\]
		\end{enumerate}
		then there exists a weight function $\rho: \sS \to (0,1)$ that satisfies  \hyperlink{H1}{$\on{(H1)}$},  \hyperlink{H2}{$\on{(H2)}$},  \hyperlink{H3}{$\on{(H3)}$}, and is $(\beta,\sC)$-compatible.
	\end{theorem}
	We recall some results from \cite{Car13} that goes into the proof of Theorem \ref{t:suff}.
	
	Let $\rho:\sS \to (0,\infty)$ be a function, we define $\rho^*:\sS \to (0,\infty)$ as 
	\[
	\rho^*(B)= \min_{B' \sim B} \rho(B), \hspace{3mm} \mbox{for $B \in \sS$.}
	\]
	We recall that $B \sim B'$ if there exists $k \ge 0$ such that $B, B' \in \sS_k$ and $(\lambda\cdot B) \cap(\lambda \cdot B') \neq \emptyset$.
	If $\gamma=\giveset{B_i}_{i=1}^N$ is a horizontal path, we define
	\[
	L_h(\gamma,\rho)= \sum_{j=1}^{N-1} \rho^*(B_j) \wedge \rho^*(B_{j+1}).
	\]
	\begin{proposition}[{\cite[Proposition 2.9]{Car13}}] \label{p:car2.9}
		Let $(X,d)$ be a compact, doubling, and uniformly perfect metric space. Let $\sS$ be a hyperbolic filling with parameters $a$ and $\lambda$ satisfying \eqref{e:tech}.
		Assume that $\rho: \sS \to (0,1)$ satisfies  \hyperlink{H1}{$\on{(H1)}$},  \hyperlink{H2}{$\on{(H2)}$}, and
		also the condition
		\begin{enumerate}[label=\textup{(H\arabic*')},align=left,leftmargin=*,topsep=5pt,parsep=0pt,itemsep=2pt]
			\item[\hypertarget{H3'}{$\on{(H3')}$}]for all $k \ge 1$, all $B \in \sS_k$ and all $\gamma \in \Gamma_{k+1}(B)$, it holds that $L_h(\gamma,\rho) \ge 1$.
		\end{enumerate}
		Then the function $\rho$ also satisfies  \hyperlink{H3}{$\on{(H3)}$}.
	\end{proposition}
	\cite[Propostion 2.9]{Car13} also assumes an additional assumption (H4) which was not used in the proof. In \cite{Car13}, the condition  \hyperlink{H3'}{$\on{(H3')}$} was stated for $k \ge 0$ but it is equivalent to the above condition because $\Gamma_1(B) = \emptyset$ for $B \in \sS_0$.
	\begin{lemma}[{\cite[Lemma 2.13]{Car13}}] \label{l:car2.13}
		Suppose we have a function $\pi_0 :\sS_k \to (0,\infty)$ such that 
		\[
		\forall B \sim B' \in \sS_k, \hspace{2mm}\frac 1 K \le \frac{\pi_0(B)}{\pi_0(B')} \le K,
		\]
		where $K \ge 1$ is a constant. Suppose that we have a function $\pi_1:\sS_{k+1} \to (0,\infty)$ which satisfies the following property:
		\[
		\forall B \in \sS_{k+1}, \exists A \in \sS_k \mbox{ with } d(x_B,x_A) \le 4a^{-k} \mbox{ and } 1\le \frac{\pi_0(A)}{\pi_1(B)} \le K.
		\]
		Define $\hat{\pi}_1:\sS_{k+1} \to (0,\infty)$ as 
		\[
		\hat{\pi}_1(B') = \pi_1(B') \vee \left(\frac 1 K \max \giveset{\pi_1(B): B \sim B'}\right).
		\]
		Then, for all $B \sim B' \in \sS_{k+1}$,  we have
		\[
		\frac 1 K \le \frac{\hat\pi_1(B)}{\hat\pi_1(B')} \le K.
		\]
		
	\end{lemma} 
	The following is a slight modification of \cite[Lemma 2.14]{Car13}.
	\begin{lemma}[{Cf.\ \cite[Lemma 2.14]{Car13}}] \label{l:car-2.14}
			Let $G=(V,E)$ be a graph whose vertices have  degree bounded by $D$. Let $\sC:V \to (0,\infty)$ be $K$-gentle; that is, 
			\[
			K^{-1} \le \frac{\sC(z)}{ \sC(z')} \le K , \mbox{ whenever there is an edge between $z$ and $z'$.}
			\]
			Let $\Gamma$ be a family of paths in $G$ and let $\beta>0$. Suppose that  $\tau:V \to (0,\infty)$ is a function satisfying
			\begin{equation} \label{e:adm}
			\sum_{i=1}^{N-1}\tau(z_1) \ge 1, \mbox{ for all paths $\gamma=\giveset{z_i}_{i=1}^N \in \Gamma$.}
			\end{equation}
			Define   $\tilde{\tau}:V \to (0,\infty)$  as
			\[
			\hat{\tau}(x)= 2 \max\giveset{ \tau(y): y \in V_2(x)},
			\]
			where $V_2(x)$ denotes the set of all vertices whose graph distance from $x$ is less than or equal to $2$.
			Then $\tilde{\tau}$ satisfies
			\[
			\sum_{i=1}^{N-1} \tilde{\tau}^*(z_i) \wedge \tilde{\tau}^*(z_{i+1}) \ge 1 \mbox{ for all paths  $\gamma=\giveset{z_i}_{i=1}^N \in \Gamma$,}
			\]
			where $\tilde{\tau}^*(x)= \min \giveset{\tilde{\tau}(y): y \sim x}$, and such that
			\begin{equation} \label{e:cont}
			\sum_{z \in V} \tilde{\tau}(z)^\beta \sC(z) \le 2^\beta  D^2 K^2 \sum_{z \in V} \tau(z)^\beta \sC(z).
			\end{equation}
	\end{lemma}
	
	\begin{proof}
	As shown in \cite[Lemma 2.14]{Car13} the function $\tilde{\tau}$ satisfies \eqref{e:adm}.
	
	Since  $\sC$ is $K$-gentle and $\sup_{x \in V} \abs{V_2(x)} \le D^2$, we obtain
	\begin{align*}
	\sum_{x \in V} \tilde{\tau}(x)^\beta \sC(x) &\le 2^\beta \sum_{x \in V} \sum_{z \in V_2(x)} \tau(z)^\beta \sC(x) \\
	& \le 2^\beta  K^2 \sum_{x \in V} \sum_{z \in V_2(x)} \tau(z)^\beta \sC(z) = 2^\beta  \sum_{z \in X}\sum_{x \in V_2(z)} \tau(z)^\beta \sC(z)\\
	&\le 2^\beta  K^2 D^2 \sum_{z \in X}\tau(z)^\beta \sC(z).
	\end{align*}
	\end{proof}
	
	\begin{proof}[Proof of Theorem \textup{\ref{t:suff}}]
	Let $D_h$ be such that
	\begin{equation} \label{e:Dh}
	D_h \ge  \sup_{k \ge 0} \max_{B \in \sS_k} \abs{\giveset{B' \in \sS_k: B'\sim B}}
	\end{equation}
	By $K_D$-doubling, $D_h$ and can be chosen to depend only on $\lambda$ and $K_D$ \cite[Exercise 10.17]{Hei}. 
	Similarly, the number of children can be bounded by a constant $D_v$ that depends only on $a$ and $K_D$ with
	\begin{equation} \label{e:Dv}
	D_v \ge \sup_{k \ge 0} \max_{B \in \sS_k}   \abs{\sD_{k+1}(B)}.
	\end{equation}
	Take $\eta_0 \in (0,1)$ be a constant which will be fixed later, and set
	\[
	\eta_-:= \left(\eta_0 K_v^{-1} D_v^{-1}\right)^{1/\beta} \wedge \frac 1 4.
	\]
	
	Let $\sigma:\sS \to [0,\frac 1 4)$ satisfy  \hyperlink{S1}{$\on{(S1)}$} and  \hyperlink{S2}{$\on{(S2)}$}.
	Define $\tau= \sigma \vee \eta_-$.
	Then
	\[
	\sum_{B' \in \sD_{k+1}(B)} \tau(B')^\beta \sC(B') \le \sum_{B' \in \sD_{k+1}(B)} \sigma(B')^\beta \sC(B') +\eta_-^\beta D_v K_v \sC(B) \le  2 \eta_0 \sC(B).
	\]
	For $B \in \sS_k$ define $V_{2,k}(B)= \giveset{B' \in \sS_k: \exists B'' \in \sS_k \mbox{ such that } B \sim B'' \sim B}$. Then by Lemma \ref{l:car-2.14}, the function
	\[
	\tilde{\tau}(B)=2 \max\giveset{\tau(B'): B' \in V_{2,k}(B)}, \hspace{2mm} \mbox{ for all $B \in \sS_k$}
	\]
	satisfies  \hyperlink{H3'}{$\on{(H3')}$} and
	\begin{align} \label{e:sf1}
	\sum_{B' \in \sD_{k+1}(B)} \tilde{\tau}(B')^\beta \sC(B') &\le 2^\beta K_h^2 D_h^2 \sum_{B' \in \sD_{k+1}(B)} {\tau}(B')^\beta \sC(B') \nonumber \\
	&\le 2^{\beta +1} K_h^2 D_h^2 \eta_0 \sC(B),  
	\end{align}
	for all $B \in \sS_k$.
	
	We construct $\hat{\rho}:\sS \to (0,1)$ satisfying
	\begin{enumerate}[label=\textup{(\arabic*)},align=left,leftmargin=*,topsep=5pt,parsep=0pt,itemsep=2pt]
		\item\label{it:t:suff-1} $\hat \rho \ge \tilde \tau$. In particular $\hat \rho$ satisfies  \hyperlink{H3'}{$\on{(H3')}$} and $\hat \rho(B) \ge \eta_-$ for all $B \in \sS$.
		\item\label{it:t:suff-2} \hyperlink{H2}{$\on{(H2)}$} with constant $K$, where $K=\eta_-^{-1}$.
		\item\label{it:t:suff-3} $\hat \rho (B) \le \max \giveset{\tilde \tau(B'): B' \sim B}$. In particular, $\hat \rho (B) \le \frac 1 2$ for all $B \in \sS$.
	\end{enumerate}
	We briefly recall the construction in \cite{Car13}. Set $\hat \rho(w)= \frac 1 2$, where $w \in \sS_0$. Note that $\tilde{ \tau } \le \frac 1 2 \le 1$ (since $\eta_- \le \frac 1 2$ and $\sigma \le \frac 1 4$). We construct $\hat \rho$ inductively on $\sS_k$.
	Suppose we have constructed $\hat \rho_i$ for $i=1,\ldots,k$. 
	We construct $\hat{\rho}_{k+1}$ 
	using Lemma \ref{l:car2.13}. We denote
	\[
	\pi_0(A)= \prod_{i=0}^k \hat{\rho}_i(g(A)_i) \mbox{ for $A \in \sS_k$ and, } \pi_1(B)= \tilde{\tau}(B) \pi_0(g(B)_j)\mbox{ for $B \in \sS_{k+1}$}.
	\]
	By the induction hypothesis along with Lemma \ref{l:car2.13}, we obtain a function $\hat{\pi}_1:\sS_{j+1} \to (0,\infty)$ that satisfies $K^{-1} \hat \pi_1(B') \le \hat \pi_1(B) \le K \hat \pi_1(B')$ for all $B \sim B' \in \sS_{j+1}$.
	We define $\hat \rho: \sS_{j+1} \to (0,\infty)$ as
	\[
	\hat \rho_{k+1}(B)= \frac{\hat \pi_1(B)}{\pi_0(g(B)_j)}.
	\]
	Carrasco Piaggio's proof of \cite[Theorem 1.2]{Car13} shows that  $\hat \rho$ satisfies properties \ref{it:t:suff-1}, \ref{it:t:suff-2}, and \ref{it:t:suff-3} above.
	For any $B \in \sS_k, k \ge 0$, using \eqref{e:sf1}  we estimate 
	\begin{align}
	\lefteqn{ \sum_{B' \in \sD_{k+1}} \hat \rho (B')^\beta \sC(B')} \nonumber\\
	&\le \sum_{B' \in \sD_{k+1}(B)} \sum_{B'' \sim B'} \tilde{ \tau}(B'')^\beta \sC(B') \mbox{ (by property \ref{it:t:suff-3} above)} \nonumber\\ 
	& \le K_h \sum_{B' \in \sD_{k+1}(B)} \sum_{B'' \sim B'} \tilde{ \tau}(B'')^\beta \sC(B'') \nonumber\\
	&\le K_h D_h \sum_{C \sim B} \sum_{B'' \in \sD_{k+1}(C)}  \tilde{ \tau}(B'')^\beta \sC(B'') \mbox{ ($\because B'' \sim B'$ implies $g(B'')_{k} \sim B$)} \nonumber\\
	&\le 2^{\beta +1} K_h^3 D_h^3 \eta_0  \sum_{C \sim B} \sC(C) \le 2^{\beta +1} K_h^4 D_h^4 \eta_0  \sC(B).  \label{e:sf2}
	\end{align}
	Now choose $\eta_0$ 
	\begin{equation} \label{e:defeta0}
	2^{\beta +1} K_h^3 D_h^3 \eta_0 =\frac 1 2,
	\end{equation}
	so that \eqref{e:sf2} yields
	\begin{equation} \label{e:sf3}
	\sum_{B' \in \sD_{k+1}} \hat \rho (B')^\beta \sC(B') \le \frac 1 2 \sC(B) \hspace{2mm} \mbox{ for all $B \in \sS_k, k \ge 0$}.
	\end{equation}
	Note from \eqref{e:defeta0} that $\eta_0$ depends only on $\beta, K_h, K_D, \lambda$ but not on  constants $K_v, a, K_P$.

	Next, we modify $\hat \rho$ so that it becomes $(\beta,\sC)$-compatible.
	For each $B \in \sS_{k}, k \ge 0$, we choose $\omega_B \ge 0$ such that 
	\begin{equation} \label{e:modify}
	\rho(B')= \begin{cases}
	\omega_B \vee \hat \rho(B') & \mbox{if $B' \in   \sD_{k+1}(B)\cap\sN$,} \\
	\hat \rho(B') &\mbox{if $B' \in  \sD_{k+1}(B) \setminus \sN$.}
	\end{cases}
	\end{equation}
	satisfies
	\begin{equation} \label{e:sf4}
	\sum_{B' \in \sD_{k+1}(B)} \rho(B')^\beta \sC(B')= \sC(B).
	\end{equation}
	The existence of an $\omega_B \in (0,\infty)$ that satisfies \eqref{e:sf4} follows from the intermediate value theorem. In particular, we use \eqref{e:sf3}, the continuity of the map
	\[
	\omega_B \mapsto \sum_{B' \in \sD_{k+1} \cap \sN} (\omega_B \vee \rho (B'))^\beta \sC(B') + \sum_{B' \in \sD_{k+1} \setminus \sN} \hat \rho (B')^\beta \sC(B'),
	\]
	along with the fact that $\sD_{k+1} \cap \sN$ is non-empty.
	The equality \eqref{e:sf4} implies that $\rho$ is $(\beta,\sC)$-compatible since
	\[
	\sum_{B' \in \sD_n(B)} \pi(B')^\beta \sC(B')= \pi(B)^\beta\sC(B)
	\]
	for all $B \in \sS_k$ and for all $n \ge k$. 
	
	It remains to show that $\rho$ satisfies  \hyperlink{H1}{$\on{(H1)}$},  \hyperlink{H2}{$\on{(H2)}$} and  \hyperlink{H3}{$\on{(H3)}$}.
	We start by verifying  \hyperlink{H1}{$\on{(H1)}$}. Clearly $\rho(B) \ge \hat \rho (B) \ge \eta_-$ for all $B \in \sS_k$. On the other hand, \hyperlink{E}{$\on{(E)}$} implies that $\omega_B \le (1- \delta)^{1/\beta}$, since
	\begin{align*}
	\sum_{B' \in \sD_{k+1}(B) \cap \sN} \omega_B^\beta \sC(B') &\le \sum_{B' \in \sD_{k+1}(B)} \rho(B')^\beta \sC(B')= \sC(B)\\
	&\le (1-\delta) \sum_{B'  \in \sD_{k+1}(B) \cap \sN} \sC(B').
	\end{align*}
	This combined with $\sigma \le \frac 1 4$ and property \ref{it:t:suff-3} of $\hat \rho$ implies that
	\[
	\eta_- \le \rho(B) \le (1/2)\vee (1-\delta)^{1/\beta}.
	\]
	By setting $\eta_+=(1/2)\vee (1-\delta)^{1/\beta} \in (0,1)$, we obtain   \hyperlink{H1}{$\on{(H1)}$}.
	
	Since $\rho \ge \hat \rho$, $\rho$ satisfies  \hyperlink{H3'}{$\on{(H3')}$}. Therefore by Proposition \ref{p:car2.9} it suffices to show  \hyperlink{H2}{$\on{(H2)}$}.
	Let $B \sim B' \in \sS_k, k \ge 1$. We consider two cases.\\
	{\em Case 1}: $g(B)_{k-1}=g(B')_{k-1}$. Then $\frac{\pi(B)}{\pi(B')} = \frac{ \rho(B)}{\rho(B')}$ which thanks to  \hyperlink{H1}{$\on{(H1)}$} satisfies
	\[
	\eta_- \le \frac{\pi(B)}{\pi(B')}\le \eta_-^{-1}.
	\]
	{\em Case 2}: $g(B)_{k-1} \neq g(B')_{k-1}$. Let $n \ge 0$ be the maximal integer such that $g(B)_n=g(B')_n$. In this case for $i=n+1, \ldots,k$, we have $g(B)_i \sim g(B')_i$. Hence for $i=n+2,\ldots,k$, $g(B)_i$ and $g(B')_i$ must both be peripheral (belong to $\sN^c$). Therefore
	\begin{align*}
	\frac{\pi(B)}{\pi(B')} &= \frac{\rho(g(B)_{n+1})}{ \rho(g(B')_{n+1})} \prod_{i=n+2}^k  \frac{\rho(g(B)_{i})}{ \rho(g(B')_{i})} = \frac{\rho(g(B)_{n+1})}{ \rho(g(B')_{n+1})} \prod_{i=n+2}^k  \frac{\hat \rho(g(B)_{i})}{ \hat \rho(g(B')_{i})}\\
	&=\frac{\rho(g(B)_{n+1})}{ \rho(g(B')_{n+1})} \frac{ \hat \pi(B)} {\hat \pi (B')} \frac{ \hat \rho(g(B')_{n+1})}{\hat \rho(g(B)_{n+1})}
	\end{align*}
	By combining property \ref{it:t:suff-2} of $\hat \rho$ to estimate $\frac{ \hat \pi(B)} {\hat \pi (B')} $ and $\eta_- \le \hat \rho \le \rho \le 1$ for the remaining terms, we obtain
	\[
	\eta_-^3\le \frac{\pi(B)}{\pi(B')} \le \eta_-^{-3}.
	\] 
	Combining the two cases, we obtain  \hyperlink{H2}{$\on{(H2)}$} with constant $K_0=\eta_-^{-3}$. 
	\end{proof}

	\begin{remark} \label{r:diff}
		One of the key differences between the construction in \cite{Car13} and our work is the proof of Theorem \ref{t:suff}.
		In the construction in \cite{Car13}, a similar modification as defined in \eqref{e:modify} was done but $\sN$ was chosen to be a singleton set. However, that choice does not work in our context because we need to ensure that $\omega_B \le \eta_+$, where $\eta_+ \in (0,1)$. This is because $\sC(B')$ can be strictly smaller than $\sC(B)$. The construction in \cite{Car13} can be interpreted as the particular case $\sC(B)=1$ for all $B \in \sS$.
		The requirement $\eta_+ \in (0,1)$  is the motivation behind the notion of non-peripheral vertices and the enhanced subadditive estimate \hyperlink{E}{$\on{(E)}$}.
	\end{remark}
	
	The following `patching lemma' allows us to combine functions that satisfy local versions of  \hyperlink{S1}{$\on{(S1)}$} and  \hyperlink{S2}{$\on{(S2)}$} into a global one. This is an adaptation of the construction in \cite[pp.\ 533--534]{Car13}. 
	\begin{lemma}[Patching lemma] \label{l:patch}
		Let $\sS$ denote a hyperbolic filling of a $K_D$-doubling, uniformly perfect, compact metric space, and let $\beta, \eta_1>0$. Let $\sS_d= (\sS,D_\sS)$ be a hyperbolic filling  with horizontal parameter $\lambda \ge 8$ and vertical parameter $a>1$ that satisfies \eqref{e:singleton}. 
		Let $\sC: \sS \to (0,\infty)$ be a $(K_h,K_v)$-gentle function.
		Assume that for all $B \in \sS_k, k \ge 1$, there exists $\sigma_B: \sS_{k+1} \to  [0,\frac 1 4)$. such that
		\begin{enumerate}[label=\textup{(\alph*)},align=left,leftmargin=*,topsep=5pt,parsep=0pt,itemsep=2pt]
			\item if we set $V_B=\giveset{ B' \in \sS_{k+1}: B' \cap 3 \cdot B \neq \emptyset}$, then $\sigma_B(B') =0$ for all $B' \in \sS_{k+1} \setminus V_B$.
			\item for any path $\gamma=\giveset{B_i}_{i=1}^N \in \Gamma_{k+1}(B)$, we have
			\[
			\sum_{i=1}^N \sigma_B(B_i) \ge 1,
			\]
			\item and $\sum_{B' \in \sS_{k+1}} \sigma_B (B')^\beta \sC(B') \le \eta_1 \sC(B)$.
		\end{enumerate}
		
		Let $\sigma:\sS \to [0,\frac 1 4)$  be  defined as
		\[
		\sigma(B') = \max \giveset{\sigma_A(B'): A \in \sS_k} 
		\]
		for all $B' \in \sS_{k+1}$ and for all $k \ge 1$, and $\sigma(B')=0$ for all $B'  \in \sS_0 \cup \sS_1$. Then there exists $C_{\ref{l:patch}} \ge 1$ that depends only on $K_D, K_h$ such that $\sigma$ satisfies  \hyperlink{S1}{$\on{(S1)}$} and the estimate
		\[
		\sum_{B' \in \sD_{k+1}(B)} \sigma(B')^\beta \sC(B') \le  C_{\ref{l:patch}} \eta_1 \sC(B).
		\]
	\end{lemma}
	
	\begin{proof}
	For any path $\gamma=\giveset{B_i}_{i=1}^N \in \Gamma_{k+1}(B), B \in \sS_k$, we have 
	$\sum_{i=1}^N \sigma(B_i) \ge \sum_{i=1}^N \sigma_B(B_i) \ge 1$. Therefore $\sigma$ satisfies  \hyperlink{S1}{$\on{(S1)}$}.
	
	For any $B \in \sS_{k+1}$, we have
	\begin{align*}
	\sum_{B' \in \sD_{k+1}(B)} \sigma(B')^\beta \sC(B') &=	\sum_{B' \in \sD_{k+1}(B)} \max \giveset{\sigma_A(B')^\beta : A \in \sS_k}  \sC(B')\\
	& \le 	\sum_{B' \in \sD_{k+1}(B)} \sum_{A: B' \in V_A} \sigma_A(B')^\beta \sC(B') \\
	& \le \sum_{A: V_B \cap V_A \neq \emptyset} \sum_{B' \in V_A} \sigma_A(B')^\beta \sC(B')\\ 
	&\le \sum_{A: V_B \cap V_A \neq \emptyset} \eta_1 \sC(A) \\
	& \le \sum_{A: A \sim B} \eta_1 \sC(A) \mbox{ \hspace{2mm}($\because V_B \cap V_A \neq \emptyset \implies A \sim B$)}\\
	& \le \sum_{A: A \sim B} \eta_1 K_h \sC(B)
	\le D_h K_h \eta_1 \sC(B),
	\end{align*}
	where $D_h$ is chosen as \eqref{e:Dh}.
	\end{proof}
	
	\section{Universality of the conformal walk dimension} \label{sec:dcw-proof}
	
	\subsection{Consequences of Harnack inequalities} \label{ssec:conseq-Harnack}
	In this subsection, we recall some previous results concerning the   elliptic and parabolic Harnack inequalities.
	We start with recalling the definition of the heat kernel and its sub-Gaussian estimates.
	
	\begin{definition}[\hypertarget{hke}{$\on{HKE}(\beta)$}]\label{d:HKE}
		Let $(X,d,m,\mathcal{E},\mathcal{F})$ be an MMD space, and let $\giveset{P_t}_{t>0}$
		denote its associated Markov semigroup. A family $\giveset{p_t}_{t>0}$ of non-negative
		Borel measurable functions on $X \times X$ is called the
		\emph{heat kernel} of $(X,d,m,\mathcal{E},\mathcal{F})$, if $p_t$ is the integral kernel
		of the operator $P_t$ for any $t>0$, that is, for any $t > 0$ and for any $f \in L^2(X,m)$,
		\[
		P_{t} f(x) = \int_X p_{t}(x,y) f (y)\, dm (y) \qquad \mbox{for $m$-a.e.\ $x \in X$.}
		\]
		Let $\beta \in (1,\infty)$.
		We say that $(X,d,m,\mathcal{E},\mathcal{F})$ satisfies the \emph{heat kernel estimates}
		\hyperlink{hke}{$\on{HKE}(\beta)$} with walk dimension $\beta$, if there exist
		$C_{1},c_{1},c_{2}, \delta\in(0,\infty)$ and a heat kernel $\giveset{p_t}_{t>0}$ such that for any $t>0$,
		\begin{align}\label{e:uhke}
		p_{t}(x,y) &\leq \frac{C_{1}}{m\bigl(B(x,t^{1/\beta})\bigr)} \exp \biggl( -c_{1} \Bigl( \frac{d(x,y)^{\beta}}{t} \Bigr)^{\frac{1}{\beta-1}} \biggr)
		\qquad \mbox{for $m$-a.e.\ $x,y \in X$,}\\
		p_{t}(x,y) &\ge \frac{c_{2}}{m\bigl(B(x,t^{1/\beta})\bigr)}
		\qquad \mbox{for $m$-a.e.\ $x,y\in X$ with $d(x,y) \leq \delta t^{1/\beta}$.}
		\label{e:nlhke}
		\end{align}
	\end{definition}
	
	The following condition is the key to establishing the parabolic Harnack inequality
	in the presence of the elliptic Harnack inequality.
	
	\begin{definition}[Capacity estimate] \label{d:cap}
		Let $\beta \in (1,\infty)$. We say that an MMD space $(X,d,m,\sE,\sF)$ satisfies the \emph{capacity estimate} \hypertarget{cap}{$\operatorname{cap}(\beta)$} 	
		if there exist $C_1,A_1,A_2>1$ such that for all $R\in (0,\diam(X,d)/A_2)$, $x \in X$ 
		\begin{equation} \label{eq:capbeta} \tag*{$\operatorname{cap}(\beta)$}
		C_1^{-1} \frac{m(B(x,R))}{R^\beta} \le \Capa(B(x,R),B(x,A_1R)^c) \le C_1 \frac{m(B(x,R))}{R^\beta}.
		\end{equation}
	\end{definition}
Poincar\'e and cutoff Sobolev inequalties are important functional inequalties for obtaining the stability of Harnack inequalties, which we recall below. 
\begin{definition} \label{d:pi-cs}
	Let $(X,d,m,\sE,\sF)$ be an MMD space and let $\beta \in (1,\infty)$.
	\begin{enumerate}[label=\textup{(\roman*)},align=left,leftmargin=*,topsep=5pt,parsep=0pt,itemsep=2pt]
	\item
	We say that $(X,d,m,\sE,\sF)$ satisfies the \emph{Poincar\'e inequality} 
	\hypertarget{pi}{$\operatorname{PI}(\beta)$}, if there exist constants $C>0$ and
	$A\geq 1$ such that for all $x\in X$, $R \in (0,\infty)$ and $f \in \sF$,
	\begin{equation}  \tag*{$\operatorname{PI}(\beta)$}
	\int_{B(x,R)} (f - f_{B(x,R)})^2 \,dm  \le C R^\beta \int_{ B(x,A R)} d \Gamma(f,f) , 
	\end{equation}
	where $f_{B(x,R)} := \frac1{\mu(B(x,R))} \int_{B(x,R)} f\, dm$. 
	\item Let $B_1 \subset B_2$ be open subsets of $X$. We say that $\varphi \in \sF$ is a \emph{cutoff function}
	for $B_1 \subset B_2$ if $0 \leq \varphi \leq 1$ $m$-a.e., $\varphi \equiv 1$ $m$-a.e.\ on $B_1$ and $\supp_{m}[\varphi] \subset B_2$.
	\item We say that $(X,d,m,\sE,\sF)$ satisfies the 
	\emph{cutoff Sobolev inequality} \hypertarget{cs}{$\operatorname{CS}(\beta)$}, 
	if there exist $C_1,C_2,C_3, \eta>0$ and $A>1$ such that the following holds.
	For all $x \in X$, $R > 0$ with $B_1=B(x,R)$, $B_2=B(x,A R)$,
	there exists a cutoff function $\varphi \in \sF \cap \contfunc(X)$ for $B_1 \subset B_2$ such that for any $u \in \sF$,
	\begin{equation} \tag*{$\operatorname{CS}(\beta)$}
	\int_{X} u^2 \,d \Gamma(\varphi,\varphi)  \le C_1  \int_{B_2 \setminus B_1} \Gamma(u,u)
	+ \frac{C_2}{R^\beta} \int_{B_2 \setminus B_1} u^2 \,dm,
	\end{equation}
	and such that the following scale invariant H\"older continuity estimate holds:
	\begin{equation} \label{e:Holder}
	\lvert \varphi(x_1)-\varphi(x_2) \rvert \le C_3 \biggl(\frac{d(x_1,x_2)}{R} \biggr)^\eta
	\end{equation}	for all $x_1,x_2 \in X$.
	\item We say that $(X,d,m,\sE,\sF)$ satisfies the 
	\emph{weak cutoff Sobolev inequality} \hypertarget{csweak}{$\operatorname{CS}_{\operatorname{weak}}(\beta)$}, 
	if \hyperlink{cs}{$\on{CS(\beta)}$} with ``$\varphi \in \sF \cap \contfunc(X)$'' replaced by
	``$\varphi \in \sF$'' and with the H\"older continuity estimate \eqref{e:Holder} dropped holds.
	\end{enumerate}
\end{definition}

	The following lemma shows that, under the above Poincar\'e and cutoff Sobolev inequalities,
	the extended Dirichlet space $\sF_e$ is contained in the space $\sF_{\on{loc}}$ as defined
	in \eqref{e:Floc} of functions locally in the domain $\sF$ of the Dirichlet form $(\sE,\sF)$.
	\begin{lemma} \label{l:extended}
		Let $(X,d,m,\sE,\sF)$ be an MMD space that satisfies \hyperlink{pi}{$\on{PI}(\beta_{1})$} and
		\hyperlink{csweak}{$\on{CS}_{\on{weak}}(\beta_{2})$} for some $\beta_{1},\beta_{2} \in (1,\infty)$.
		Then $\sF_e \subset \sF_{\on{loc}}$.
	\end{lemma}
	
	\begin{proof}
	Let $g \in \sF_e$. Then there exists an $\sE$-Cauchy sequence $\{g_n\}_{n} \subset \sF$
	such that $g_n$ converges to $g$ $m$-a.e. Let $B=B(x,R)$ be any ball.
	By the Poincar\'e inequality \hyperlink{pi}{$\on{PI}(\beta_{1})$}
	the sequence $g_n - (g_n)_B$ is $L^2(B,m)$-Cauchy. Since $g_n$ converges to $g$ $m$-a.e.\ and
	$g_n - (g_n)_B$ is $L^2(B,m)$-Cauchy, we have that $\lim_{n \to \infty} (g_n)_B= g_B$ and that $g_n$ converges to $g$ in $L^2(B,m)$.
	
	Let $A > 1$ be as in \hyperlink{csweak}{$\on{CS}_{\on{weak}}(\beta_{2})$} and let $\varphi$ be a cutoff function
	for $B = B(x,R) \subset B(x,A R)$ as in \hyperlink{csweak}{$\on{CS}_{\on{weak}}(\beta_{2})$}.
	By \cite[Theorem 2.1.7]{FOT} we may assume that $g_n$ is bounded, so that
	$g_n \varphi \in \sF$ by \cite[Theorem 1.4.2-(ii)]{FOT}.
	Noting that $\varphi^{2}\leq 1$ q.e.\ by \cite[Lemma 2.1.4]{FOT},
	by the Leibniz rule \cite[Lemma 3.2.5]{FOT} and the Cauchy--Schwarz inequality
	\cite[Lemma 5.6.1]{FOT} for $\Gamma(\cdot,\cdot)$, we obtain 
	\[
	\frac{1}{2} \sE(\varphi(g_n-g_m), \varphi(g_n-g_m)) \le \sE(g_n-g_m,g_n-g_m) + \int_{X} (g_n-g_m)^2 \,d\Gamma(\varphi,\varphi),
	\]
	which together with \hyperlink{csweak}{$\on{CS}_{\on{weak}}(\beta_{2})$} for $u=g_n-g_m$
	and the previous paragraph with $B(x,A R)$ in place of $B=B(x,R)$
	shows that $g_n \varphi$ is a Cauchy sequence in $(\sF, \sE_{1})$
	converging in $L^{2}(X,m)$ to $g \varphi$. Thus $g \varphi \in \sF$
	by the completeness of $(\sF, \sE_{1})$, and since $g \varphi = g$ in $B$ and
	$B=B(x,R)$ is an arbitrary ball in $(X,d)$, we conclude that $g \in \sF_{\on{loc}}$.
	\end{proof}
	
	We record the following theorem which relates the elliptic and parabolic Harnack inequalities.
	The equivalence of \ref{it:phichar-phi}, \ref{it:phichar-EHI-cap} and \ref{it:phichar-HKE} is due to
	Grigor'yan and Telcs \cite[Theorem 3.1]{GT02} in the context of random walks on graphs.
	This was later extended to the MMD space setting by several authors.
	The equivalence between \ref{it:phichar-phi} and \ref{it:phichar-PI-CS} is due to Barlow and Bass \cite{BB04}
	for random walks on graphs and was extended to the current setting in \cite{BBK}.
	\begin{theorem} \label{t:phichar}
		Let $(X,d,m,\sE,\sF)$ be an MMD space
		and let $\beta \in (1,\infty)$. Then the following are equivalent:
		\begin{enumerate}[label=\textup{(\alph*)},align=left,leftmargin=*,topsep=5pt,parsep=0pt,itemsep=2pt]
			\item\label{it:phichar-phi}$(X,d,m,\sE,\sF)$ satisfies \ref{PHI}.
			\item\label{it:phichar-EHI-cap}$(X,d,m,\sE,\sF)$ satisfies \ref{VD}, \ref{EHI} and \hyperlink{cap}{$\on{cap(\beta)}$}. 
			\item\label{it:phichar-HKE}$(X,d,m,\sE,\sF)$ satisfies \ref{VD} and \hyperlink{hke}{$\on{HKE(\beta)}$}. 
			\item\label{it:phichar-PI-CS}$(X,d,m,\sE,\sF)$ satisfies \ref{VD}, \hyperlink{pi}{$\on{PI(\beta)}$} and \hyperlink{cs}{$\on{CS(\beta)}$}.
			\item\label{it:phichar-PI-CSweak}$(X,d,m,\sE,\sF)$ satisfies \ref{VD}, \hyperlink{pi}{$\on{PI(\beta)}$} and \hyperlink{csweak}{$\on{CS}_{\on{weak}}(\beta)$}.
		\end{enumerate}
		Moreover, if $(X,d,m,\sE,\sF)$ satisfies any one of \ref{it:phichar-phi},
		\ref{it:phichar-EHI-cap}, \ref{it:phichar-HKE}, \ref{it:phichar-PI-CS} and \ref{it:phichar-PI-CSweak},
		then $(X,d)$ is arcwise connected and uniformly perfect and $(X,d,m)$ is \ref{RVD}.
	\end{theorem}
	
	\begin{proof}
	First, by Remark \ref{rmk:harnack}, \cite[Proposition 5.6]{GH14} and \cite[Lemma 5.2-(c),(b)]{BCM}
	(see also \cite[Proof of Theorem 5.4, $\textrm{(b)}\Rightarrow\textrm{(a)}$]{BCM}),
	\ref{it:phichar-EHI-cap} implies that $(X,d)$ is arcwise connected and uniformly perfect
	and that $(X,d,m)$ is \ref{RVD}. Then since \ref{it:phichar-EHI-cap} and \ref{RVD}
	together imply \ref{it:phichar-HKE} by \cite[Theorem 1.2]{GHL15},
	it follows that \ref{it:phichar-EHI-cap} implies \ref{it:phichar-HKE}.
	Next, \ref{it:phichar-HKE} implies \ref{it:phichar-phi} by \cite[Theorem 3.1]{BGK},
	\hyperlink{cs}{$\on{CS(\beta)}$} by \cite[Section 3]{BBK}%
	\footnote{We note that the proof of the H\"older continuity estimate of the resolvent kernel
	stated in \cite[Lemma 3.3]{BBK}, from which \eqref{e:Holder} follows, has a gap
	which has been resolved in the arXiv version. The proof of \hyperlink{cs}{$\on{CS(\beta)}$}
	in \cite[Section 3]{BBK} works also in the compact setting with minor modifications
	and it does not use the assumption that the metric $d$ is geodesic.},
	\hyperlink{pi}{$\on{PI(\beta)}$} by \cite[Proof of Theorem 1.2]{GHL15} or
	\cite[Proof of Theorem 3.2]{Lie} (see also \cite[Remark 2.9-(b)]{KM}),
	and thus \ref{it:phichar-PI-CS}. It is obvious that \ref{it:phichar-PI-CS}
	implies \ref{it:phichar-PI-CSweak}. Conversely, since the conjunction of \ref{VD}
	and \hyperlink{pi}{$\on{PI(\beta)}$} implies \ref{RVD} by \cite[Corollary 2.3]{Mur1}
	and Remark \ref{r:doubling}-\ref{it:RVD-uniformly-perfect}, it follows from
	\cite[Theorem 1.1]{GHL15} along with \cite[Proposition 5.11 and Remark 5.12]{BM}
	that \ref{it:phichar-PI-CSweak} implies \ref{it:phichar-EHI-cap}; in \cite[Theorem 1.1]{GHL15}
	the condition \ref{EHI} is stated and proved only for $h \in \mathcal{F}$, but
	by using Lemma \ref{l:extended} and the relative compactness in $X$ of all balls in $(X,d)$
	we obtain our version of \ref{EHI} in Definition \ref{d:harnack}, where $h \in \mathcal{F}_{e}$.
	
	It remains to prove that \ref{it:phichar-phi} implies \ref{it:phichar-HKE}. Since
	the conjunction of \ref{VD} and \ref{PHI} implies \hyperlink{hke}{$\on{HKE(\beta)}$}
	by \cite[Theorem 3.1]{BGK} (see also \cite[Proof of Theorem 3.2]{Lie}),
	it suffices to show that \ref{PHI} implies \ref{VD}. 
	
	The implication from \ref{PHI} to \ref{VD} follows from \cite[Theorem 3.2]{BGK} under the additional assumption that the metric $d$ is geodesic.
	However, this additional assumption is not necessary and we modify the proof in \cite{BGK} as follows.
	By \cite[Lemma 4.6]{BGK}, there exists a heat kernel $p_t(x,y)$ such that $(t,x,y) \mapsto p_t(x,y)$ is continuous on $(0,\infty) \times X \times X$.
	By \cite[(4.52)]{BGK}, there exists $c_1,c_2>0$ such that
	\begin{equation} \label{e:eqv1}
	\sup_{x,y \in B(x_0,r)} p_t(x,y) \ge \frac{c_1}{m(B(x_0,r))} \exp \left(- \frac{c_2 t}{r^\beta}\right) \quad \mbox{for all $x_0 \in X,r >0, t>0$.}
	\end{equation}
	Let $0<C_1<C_2<C_3<C_4$, $\delta\in (0,1)$ and $C_5 >1$ denote the constants in \ref{PHI}. Define $K =  \frac{C_3+C_4}{C_1 + C_2} \in (1,\infty)$.
	
	Let  $x_0 \in X, r>0$ be arbitrary. Fix $t >0$ such that $t=(C_1+C_2) \delta^{-\beta}r^\beta/2$. Using \eqref{e:eqv1}, we choose  $y \in B(x_0,r)$ such that $\sup_{x \in B(x_0,r)} p_t(x,y) \ge \frac{1}{2}  \frac{c_1}{2 m(B(x_0,r))} \exp \left(- \frac{c_2 t}{r^\beta}\right)$. By \ref{PHI}, we obtain
	\begin{equation} \label{e:eqv2}
	p_{Kt}(x_0,y) \ge \frac{C_5^{-1} c_1}{2 m(B(x_0,r))} \exp \left(- \frac{c_2 t}{r^\beta}\right) \quad \mbox{for some $y \in B(x_0,r)$.}
	\end{equation}
	By \ref{PHI} for the caloric function $(t,z) \mapsto p_t(x_0,z)$ on the cylinder $(0,C_5 \delta^{-\beta} r_1^\beta) \times B(x,\delta^{-1}r_1)$, where $r_1 >0$ satisfies $(C_1+C_2) \delta^{-\beta} r_1^\beta/2 = Kt$ (or equivalently, $r_1= K^{1/\beta}r$) and \eqref{e:eqv2}, we obtain
	\begin{equation} \label{e:eqv3}
	p_{K^2 t}(x_0,z) \ge  \frac{C_5^{-2} c_1}{2 m(B(x_0,r))}  \exp \left(- \frac{c_2 t}{r^\beta}\right) \quad \mbox{for all $z \in B(x_0,K^{1/\beta} r)$.}
	\end{equation}
	Using $\int_X p_{K^2 t} (x_0,z) \,m(dz) \le 1$ and  $t=(C_1+C_2) \delta^{-\beta}r^\beta/2$ and \eqref{e:eqv3}, there exists $C_6>1$ such that
	\[
	\frac{m(B(x_0,K^{1/\beta} r))}{m(B(x_0,r))} \le C_6, \quad \mbox{for all $x_0 \in X,r>0$.}
	\]
	By iterating the above estimate $ \lceil \beta \log 2/ \log K\rceil$ times, we obtain the volume doubling property \ref{VD}.
	\end{proof}
	
	\begin{remark}
		Theorem \ref{t:phichar} can be generalized to the case where the space-time scaling function
		$\Psi(r)=r^\beta$ is replaced with a homeomorphism $\Psi:[0,\infty) \to [0,\infty)$ satisfying the following estimates:
		there exist $C_{1}, \beta_{1},\beta_{2} \in (0,\infty)$ with $1 < \beta_{1} \leq \beta_{2}$ such that
		\[
		C_1^{-1} \left( \frac R r \right)^{\beta_1} \le \frac{\Psi(R)}{\Psi(r)} \le  	C_1 \left( \frac R r \right)^{\beta_2}
		\quad \mbox{for all $r,R \in (0,\infty)$ with $r\leq R$.}
		\]
		The generalized version of the relevant properties like \ref{PHI}
		and \hyperlink{cap}{$\on{cap(\beta)}$} for such space-time scale functions can be found in \cite{BGK,GHL15}.	
	\end{remark}
	
	Combining Theorem \ref{t:phichar} with the main result of \cite{Mur1}, we have
	the following alternative proof that \hyperlink{phi}{$\on{PHI(\beta)}$} cannot
	hold for $\beta \in (0,2)$ and thereby that \eqref{e:dcw-lower} holds.
	\begin{lemma} \label{l:dcw-lowerbound}
	If $\beta>0$ and an MMD space $(X,d,m,\sE,\sF)$ satisfies \hyperlink{phi}{$\on{PHI(\beta)}$},
	then $\beta \ge 2$.
	\end{lemma}
	\begin{proof}
	Assume to the contrary that $\beta<2$, so that $\rho:=d^{\beta/2}$ would be
	a metric on $X$ and $(X,\rho,m,\sE,\sF)$ would satisfy \hyperlink{phi}{$\on{PHI(2)}$}.
	For $\varepsilon>0$, we define the $\varepsilon$-chain metric as
	\[
	d_\varepsilon(x,y) = \inf \sum_{i=0}^{N-1} d(x_i,x_{i+1}),
	\] where the infimum is taken over all finite collection of points $\{ x_i \}_{i=0}^{N} \subset X$, $N \in \mathbb{N}$,
	such that $x_0=x$, $x_N=y$ and $d(x_i,x_{i+1})<\varepsilon$ for all $i=0,1,\ldots, N-1$.
	We define $\rho_\varepsilon(x,y)$ analogously for any $x,y \in X$.
	Then since $(X,\rho,m,\sE,\sF)$ would satisfy \hyperlink{cap}{$\on{cap(2)}$} and \hyperlink{pi}{$\on{PI(2)}$}
	by Theorem \ref{t:phichar}, it follows from \cite[Theorem 1.6 and Remark 1.7(a)]{Mur1} that there would exist $C>0$ such that
	\begin{equation} \label{e:dcwlb1}
	\rho_\varepsilon(x,y) \le C \rho(x,y) \quad \mbox{for all $x,y \in X$ and $\varepsilon>0$.}
	\end{equation}
	On the other hand, since $d(x_i,x_{i+1}) < \varepsilon$ is equivalent to $\rho(x_i,x_{i+1}) \le \varepsilon^{\beta/2}$, we would have
	\begin{equation}\label{e:dcwlb2}
	d(x,y) \le 	d_\varepsilon(x,y) \le \varepsilon^{1 - \beta/2} \rho_{\varepsilon^{2/\beta}}(x,y) \quad \mbox{for all $x,y \in X$ and $\varepsilon>0$.}
	\end{equation}
	Combining \eqref{e:dcwlb1} and \eqref{e:dcwlb2}, we would obtain
	\[
	d(x,y) \le C \varepsilon^{1 - \beta/2} \rho(x,y) = C \varepsilon^{1-\beta/2} d(x,y)^{\beta/2}
	\]
	for all $x,y \in X$ and $\varepsilon>0$, and letting $\varepsilon \downarrow 0$ would yield $\diam(X,d)=0$,
	which would contradict our standing assumption that $X$ contains at least two elements.
	\end{proof}

	We next collect some properties of \ref{EHI} in relation to time changes and
	quasisymmetric changes of the metric. The following lemma was observed first by
	Jun Kigami and taught through \cite{Kig16} to the authors of \cite{BM} and of the present paper.
	
	\begin{lemma}[{\cite{Kig16}, cf.\ \cite[Lemma 5.3]{BM}}] \label{l:EHI-TC-QS}
	Let $(X,d,m,\sE,\sF)$ be an MMD space, let $\mu \in\sA(X,d,m,\sE,\sF)$ and let $\theta \in \sJ(X,d)$.
	Then $(X,d,m,\sE,\sF)$ satisfies \ref{EHI} if and only if $(X,\theta,\mu,\sE^\mu,\sF^\mu)$ satisfies \ref{EHI}.
	\end{lemma}
	
	\begin{proof}
	$((\sF^{\mu})_{e},\sE^{\mu})=(\sF_{e},\sE)$ by \cite[Corollary 5.2.12]{CF},
	$\sF^{\mu} \cap \contfunc_{\mathrm{c}}(X) = \sF_e \cap \contfunc_{\mathrm{c}}(X) = \sF \cap \contfunc_{\mathrm{c}}(X)$
	by \eqref{e:defTC} and the equality $\sF_e \cap L^{2}(X,m) = \sF$ from \cite[Theorem 1.1.5-(iii)]{CF},
	and therefore for each open subset $U$ of $X$ we have
	\begin{equation} \label{e:EHI-TC-QS}
	\{h \in (\sF^{\mu})_e \mid \mbox{$h$ is $\sE^{\mu}$-harmonic on $U$} \}
		= \{h \in \sF_e \mid \mbox{$h$ is $\sE$-harmonic on $U$} \}.
	\end{equation}
	Note also that by \cite[Theorem 5.2.11]{CF} we have the following (see \cite[Section 2.1]{FOT}
	and \cite[Sections 1.2, 1.3 and 2.3]{CF} for the definition and basic properties of quasi-continuous functions):
	\begin{enumerate}[label=\textup{(\roman*)},align=left,leftmargin=*,topsep=5pt,parsep=0pt,itemsep=2pt]
	\item A subset $N$ of $X$ has $1$-capacity zero with respect to $(X,d,m,\sE,\sF)$, i.e., satisfies $\Capa_1(N)=0$,
	if and only if $N$ has $1$-capacity zero with respect to $(X,\theta,\mu,\sE^\mu,\sF^\mu)$.
	In other words, the notion of holding \emph{q.e.}, i.e., holding outside a set of $1$-capacity zero,
	with respect to $(X,d,m,\sE,\sF)$ is equivalent to that with respect to $(X,\theta,\mu,\sE^\mu,\sF^\mu)$.
	\item A function $u:X\setminus N \to [-\infty,\infty]$ defined q.e.\ on $X$, where $N$ is a subset of $X$ with $\Capa_1(N)=0$,
	is quasi-continuous with respect to $(X,d,m,\sE,\sF)$ if and only if $u$ is quasi-continuous with respect to $(X,\theta,\mu,\sE^\mu,\sF^\mu)$.
	\end{enumerate}
	
	Now choose a distortion function $\eta$ so that $\operatorname{Id}:(X,\theta) \to (X,d)$ is an $\eta$-quasisymmetry,
	and suppose that $(X,d,m,\sE,\sF)$ satisfies \ref{EHI} with the constants $C > 1$ and $\delta \in (0,1)$.
	Let $B_{d}(x,r)$ and $B_{\theta}(x,r)$ denote open balls of radius $r$ centered at $x$ in $(X,d)$ and $(X,\theta)$,
	respectively. Let $x \in X$, $r > 0$, and let $h \in (\sF^{\mu})_e = \sF_e$ be $\sE^{\mu}$-harmonic on $B_{\theta}(x,r)$
	and satisfy $h \geq 0$ $\mu$-a.e.\ on $B_{\theta}(x,r)$, where we consider only
	quasi-continuous $m$-versions of $h \in \sF_e$ with respect to $(X,d,m,\sE,\sF)$.
	Then $h$ is $\sE$-harmonic on $B_{\theta}(x,r)$ by \eqref{e:EHI-TC-QS},
	$h \geq 0$ q.e.\ on $B_{\theta}(x,r)$ by \cite[Lemma 2.1.4]{FOT} applied to
	$(X,\theta,\mu,\sE^\mu,\sF^\mu)$, hence $h \geq 0$ $m$-a.e.\ on $B_{\theta}(x,r)$, and
	$B_{\theta}(x,\delta' r) \subset B_{d}(x,\delta t) \subset B_{d}(x,t) \subset B_{\theta}(x,r)$
	for some $t > 0$ by \eqref{e:ann2} with $\delta' r$ in place of $r$ and $A=\delta^{-1}$,
	where $\delta':=\eta^{-1}(\delta)$. Thus we obtain
	\begin{equation} \label{e:EHI-pointwise}
	m\mbox{-}\!\esssup_{B_{\theta}(x,\delta'r)} h \leq C\cdot m\mbox{-}\!\essinf_{B_{\theta}(x,\delta'r)} h
	\end{equation}
	by \ref{EHI} for $(X,d,m,\sE,\sF)$ applied to $B_{d}(x,t)$ and $h$.
	On the other hand, we also have
	\begin{equation*}
	m\mbox{-}\!\essinf_{B_{\theta}(x,\delta'r)} h \leq h(y) \leq m\mbox{-}\!\esssup_{B_{\theta}(x,\delta'r)} h,
	\end{equation*}
	first for $m$-a.e.\ $y \in B_{\theta}(x,\delta'r)$,
	then for q.e.\ $y \in B_{\theta}(x,\delta'r)$ by \cite[Lemma 2.1.4]{FOT}
	and hence for $\mu$-a.e.\ $y \in B_{\theta}(x,\delta'r)$, and therefore
	from \eqref{e:EHI-pointwise} we conclude that
	\begin{equation*}
	h(y) \leq m\mbox{-}\!\esssup_{B_{\theta}(x,\delta'r)} h
	\leq C\cdot m\mbox{-}\!\essinf_{B_{\theta}(x,\delta'r)} h \leq C h(z)
	\quad \mbox{for $\mu$-a.e.\ $y,z \in B_{\theta}(x,\delta'r)$,}
	\end{equation*}
	proving \ref{EHI} for $(X,\theta,\mu,\sE^\mu,\sF^\mu)$.
	The converse implication from \ref{EHI} for $(X,\theta,\mu,\sE^\mu,\sF^\mu)$
	to \ref{EHI} for $(X,d,m,\sE,\sF)$ is proved in exactly the same way
	by noting that $\operatorname{Id}:(X,d) \to (X,\theta)$ is an $\tilde{\eta}$-quasisymmetry
	with the distortion function $\tilde{\eta}$ given by $\tilde{\eta}(t):=1/\eta^{-1}(t^{-1})$.
	\end{proof}
	
	As mentioned in the introduction, Delmotte has constructed a space that satisfies \ref{EHI}
	but fails to satisfy \ref{VD} and hence fails to satisfy \ref{PHI} for any $\beta>0$ \cite{Del};
	see also \cite[Example 8.4]{BCM} for a similar construction.
	Nevertheless, it is possible to obtain \ref{PHI} after a time change and a change of the metric.
	We recall the characterization of \ref{EHI} in \cite{BM,BCM}. 
	
	\begin{theorem}[{\cite{BM,BCM}}] \label{t:dcwfinite}
		Let $(X,d,m,\sE,\sF)$ be an MMD space. Then the following are equivalent:
		\begin{enumerate}[label=\textup{(\alph*)},align=left,leftmargin=*,topsep=5pt,parsep=0pt,itemsep=2pt]
			\item\label{it:dcwfinite-EHI}$(X,d,m,\sE,\sF)$ satisfies the metric doubling property and \ref{EHI}.
			\item\label{it:dcwfinite-PHI}There exist $\gamma>2$, $\mu \in\sA(X,d,m,\sE,\sF)$ and $\theta \in \sJ(X,d)$ such that the time-changed MMD space $(X,\theta,\mu,\sE^\mu,\sF^\mu)$ satisfies \hyperlink{phi}{$\on{PHI(\gamma)}$}. In other words, $\dcw <\infty$.
		\end{enumerate}
		Moreover, either of these two conditions implies that
		$(X,d)$ is arcwise connected and uniformly perfect.
	\end{theorem}
	
	\begin{proof}
	The implication from \ref{it:dcwfinite-EHI} to \ref{it:dcwfinite-PHI} follows from
	Remark \ref{rmk:harnack}, \cite[Theorems 5.4 and 7.9]{BCM} and Theorem \ref{t:phichar}.
	On the other hand, if \ref{it:dcwfinite-PHI} holds, then we see from Theorem \ref{t:phichar}
	and Remark \ref{r:doubling}-\ref{it:VDimprove} that $(X,\theta)$ is arcwise connected,
	uniformly perfect and doubling and that $(X,\theta,\mu,\sE^\mu,\sF^\mu)$ satisfies
	\ref{EHI}, and therefore the same hold also for $(X,d)$ and $(X,d,m,\sE,\sF)$ by
	\cite[Theorem 10.18]{Hei} and Lemma \ref{l:EHI-TC-QS}, completing the proof.
	\end{proof}

	The following elementary lemma is used to verify that the function defined in \eqref{e:defsC} on a hyperbolic filling is gentle and satisfies the enhanced subadditive estimate \hyperlink{E}{$\on{(E)}$}.
	\begin{lemma} \label{l:gentle}
		Let $(X,d,m)$ be a metric measure space that satisfies \ref{VD} and let $\gamma>0$. For any ball $B(x,r)$, we define
		\begin{equation} \label{e:defsC}
		\sC(B(x,r))= \frac{m(B(x,r))}{r^{\gamma}}.
		\end{equation}
		\begin{enumerate}[label=\textup{(\alph*)},align=left,leftmargin=*,topsep=5pt,parsep=0pt,itemsep=2pt]
			\item Let $\lambda \ge 1$. There exists $C_1 >0$ (that depends only on the constant of \ref{VD} and $\lambda$) such that for any $x,y \in X$ satisfying $B(x,\lambda r) \cap B(y,\lambda r) \neq \emptyset$, we have
			\[
			\sC(B(x,r)) \le C_1 \sC(B(y,r)).
			\]
			\item Let $a>1$. There exists $C_2 \ge 1$ (that depends only on the constant of \ref{VD}, $\gamma$ and $\lambda$) such that for any $x,y \in X$ satisfying $y \in B(x,r)$, we have
			\[
			C_2^{-1} \sC(B(y,r/a)) \le	\sC(B(x,r)) \le C_2 \sC(B(y,r/a)).
			\]
			\item
			
			There exists $C_3>1$ such that  the following estimate holds: for all $a >1, x \in X, r>0$ and $z_1,\ldots,z_k, k \in \bN$ such that $d(z_i,z_j) \ge r/(2a)$ for all $1 \le i < j \le k$ and satisfying
			$\cup_{i=1}^k B(z_i,r/a) \supset B(x,r/6)$, we have that
			\begin{equation} \label{e:subadd}
			\sC(B(x,r)) \le C_3 a^{-\gamma} \sum_{i=1}^k \sC(B(z_i,r/a)).
			\end{equation}
		\end{enumerate}
	\end{lemma}
	
	\begin{proof}
	We denote $m(B(x,r))$ by $V(x,r)$ in this proof.
	\begin{enumerate}[label=\textup{(\alph*)},align=left,leftmargin=*,topsep=5pt,parsep=0pt,itemsep=2pt]
		\item Let $C_D \in (1,\infty)$ denote the constant associated with \ref{VD} and let $\alpha = \log_2 C_D$,
		so that by \eqref{e:VDimprove} we have
		\begin{equation} \label{e:gen1}
		\frac{V(x,R)}{V(x,r)} \le C_D \left( \frac{R}{r} \right)^\alpha, \quad \mbox{for all $0<r \le R$ and $x \in X$.}
		\end{equation}
		Let $z \in B(x,\lambda r) \cap B(y, \lambda r)$. By using $B(x,r) \subset B(z, (\lambda+1)r)$, $B(z,r) \subset B(y,(\lambda+1)r)$ and \eqref{e:gen1}, we obtain
		\[
		V(x,r) \le V(z,(\lambda+1)r) \le  C_D (\lambda+1)^\alpha V(z,r)  \le C_D^2  (\lambda+1)^{2\alpha} V(y,r).
		\]
		\item Since $B(x,r) \subset B(y,2r)$ and $B(y,r/a) \subset B(x,2r)$, by \eqref{e:gen1} we have
		\begin{align*}
		V(x,r) &\le V(y,2r) \le C_D (2a)^\alpha V(y,r/a),\\ V(y,r/a) &\le V(x,2r) \le C_D V(x,r).
		\end{align*}
		Therefore
		\[
		\sC(B(x,r)) \le C_D 2^\alpha a^{\alpha - \gamma} \sC(B(y,r/a)), \quad \sC(B(y,r/a)) \le C_D a^{\gamma} \sC(B(x,r)).
		\]
		\item  By \ref{VD} and $\cup_{i=1}^k B(z_i,r/a) \supset B(x,r/3)$, we have
		\[
		V(x,r) \le C_D^3 V(x,r/8)  \le C_D^2 \sum_{i=1}^k V(x_i,r/a).
		\]
		Dividing both sides by $r^\gamma$, we obtain \eqref{e:subadd} with $C_3=C_D^3$.
	\end{enumerate}
	\end{proof}

	The elliptic Harnack inequality implies that the capacities across annuli  with similar locations and scales are comparable as we recall below. 
	\begin{lemma}[{\cite[Lemmas 5.22 and 5.23]{BCM}}] \label{l:cap}
		Let $(X,d,m,\sE,\sF)$ be an MMD space that satisfies the metric doubling property and \ref{EHI}. Then for any $A_1,A_2  >1$,
		there exist $C_1,C_2>1$ and $\gamma>0$ such that for all $x,  \in  X$, and for any $0 < s \le r < \diam(X,d)/C_1,$, we have
		\[
		C_2^{-1} \left(\frac{r }{s}\right)^{-\gamma} 		 \le	\frac{ \Capa(B(x,r), B(x,A_2 r)^c)}{\Capa(B(x,s), B(x,A_1s)^c)} \le C_2 \left(\frac{r }{s}\right)^\gamma.	
		\]	
	\end{lemma}
	
	\begin{proof}
	This follows immediately from Remark \ref{rmk:harnack} and \cite[Theorem 5.4, Lemmas 5.22 and 5.23]{BCM}.
	\end{proof}
	
	Using this lemma, we obtain the following comparison of capacity across annuli under a quasisymmetric change of metric.
	\begin{proposition} \label{p:qscap}
		Let $(X,d,m,\sE,\sF)$ be an MMD space that satisfies \hyperlink{phi}{$\on{PHI(\gamma)}$}, where $\gamma \ge 2$. Let $\theta \in \sJ(X,d)$ and $a>1$. Then there exists $C,A>0$ such that the following property holds.
		For any $x, \widetilde{x} \in X, 0<r<\diam(X,\theta)/A ,s>0$, $n \in \bZ$ such that
		\begin{equation} \label{e:qsc1}
		s= \sup \giveset{0 <t <2 \diam(X,d) : B_d(x,t) \subset B_\theta (x,r)},
		\end{equation}
		and
		\[
		2 a^{-n-1} \le s < 2a^{-n}, d(\widetilde{x}, x )< 2a^{-n},
		\]
		we have
		\begin{equation} \label{e:qsc2}
		C^{-1} \frac{m(B_d(\widetilde{x}, 2a^{-n}))}{\left[ 2a^{-n}\right]^\gamma}\le \Capa(B_\theta(x,r),B_\theta(x,2r)^c ) \le C \frac{m(B_d(\widetilde{x}, 2a^{-n}))}{\left[ 2a^{-n}\right]^\gamma}.
		\end{equation}
	\end{proposition}
	
	\begin{proof}
	By Theorem \ref{t:phichar}, $(X,d)$ satisfies doubling, uniformly perfect metric space. By Proposition \ref{p:QS}, there exists  $A_1,A_2,A_3>1$ such that for all $x \in X, 0< r< \diam(X,\theta)$, 
	\begin{equation} \label{e:qscap1}
	B_d(x,s) \subset B_\theta(x,r) \subset B_d(x,A_1 s) \subset B_d (s,2 A_1 s) \subset B_\theta(x,A_2 r ) \subset B_d(x,A_3 s),
	\end{equation}
	where $s>0$ is as defined in \eqref{e:qsc1}.
	If $B_d(x,A_3s) \neq X$ in \eqref{e:qscap1}, we have
	\begin{align} \label{e:qscap2}
	\Capa(B_d(x,s),B_{d}(x,A_3 s)^c) &\le \Capa(B_\theta(x,r), B_\theta(x,A_2 r)^c) \nonumber \\
	&\le \Capa(B_d(x, A_1 s),B_{d}(x, 2A_1 s)^c).
	\end{align}
	By Lemma \ref{l:cap}, Proposition \ref{p:QS}-(b), and Theorem \ref{t:phichar}, there exist  $C_1,A>1$ such that for all $x \in X, 0< r< \diam(X,\theta)/A$, we have
	\begin{equation} \label{e:qscap3}
	C_1^{-1} \le \frac{\Capa(B_\theta(x,r), B_\theta(x, 2 r)^c)}{\Capa(B_d(x,s),B_{d}(x,A_3 s)^c)}
	 \le C_1,
	\end{equation}
	and
	\begin{equation} \label{e:qscap4}
	C_1^{-1} \frac{m(B_d(x,s))}{s^\gamma} \le \Capa(B_d(x,s),B_{d}(x,A_3 s)^c) \le C_1  \frac{m(B_d(x,s))}{s^\gamma},
	\end{equation}
	where $s>0$ is as given in \eqref{e:qsc1}.
	By \eqref{e:qscap3}, \eqref{e:qscap4} and \ref{VD},  we obtain \eqref{e:qsc2}.
	\end{proof}

	We will use Theorem \ref{t:suff} to construct metrics.
	The following proposition plays a central role in constructing a function on the hyperbolic filling that satisfies the hypotheses \hyperlink{S1}{$\on{(S1)}$} and \hyperlink{S2}{$\on{(S2)}$} in Theorem \ref{t:suff}. 
	\begin{proposition} \label{p:lowenergy}
		Let $(X,d,m,\sE,\sF)$ be an MMD space that satisfies \hyperlink{phi}{$\on{PHI(\gamma)}$} for some $\gamma >2$ and let $\lambda>1$. There exist constants $A,C_1,C_2>1, \eta>0$ (that depend only on $\lambda$ and the constants associated with \hyperlink{phi}{$\on{PHI(\gamma)}$}) such that for any $a>1, x \in X,  0<r< \diam(X,d)/A$, 
		and for any collection of balls $\sB=\giveset{B(y_i,r/a) : i \in I}$ such that $\cup_{i \in I} B(y_i,r/a) = X$ and $\giveset{B(y_i, r/(4a))}$ is pairwise disjoint, there exists a function $\sigma \colon \sB \to [0,\infty)$ that obeys the following properties (note that $\sigma$ depends on $x \in X, r>0$): 
		\begin{enumerate}[label=\textup{(S\arabic*')},align=left,leftmargin=*,topsep=5pt,parsep=0pt,itemsep=2pt]
			\item\label{it:S1prime}for any sequence of balls $\gamma= \giveset{B_i: 1 \le i \le N}$ in $\sB$ such that $x_{B_1} \in B(x,r), x_{B_N} \notin B(x,2r)$ and $\lambda\cdot B_i \cap \lambda\cdot B_{i+1} \neq \emptyset$ for all $i=1,\ldots,N-1$,  we have
			\begin{equation} \label{e:suff1}
			\sum_{i=1}^N \sigma(B_i) \ge 1,
			\end{equation}
			and
			\begin{equation} \label{e:suff2}
			\sigma(B)=0, \quad \mbox{ for any ball $B \in \sB$ such that $x_B \notin B(x,2r)$}.
			\end{equation} 
			\item\label{it:S2prime}$\sigma \colon \sB \to (0,\infty)$ satisfies the following estimates
			\begin{equation} \label{e:suff3}
			\sum_{B \in \sB} \sigma(B)^2  \frac{m(B)}{(r/a)^\gamma} \le C_1 \frac{m(B(x,r))}{r^{\gamma}}
			\end{equation}
			and
			\begin{equation} \label{e:suff4}
			\sup_{B \in \sB} \sigma(B) \le C_2 a^{-\eta}.
			\end{equation}
			In particular, for any $\beta>2$, we have
			\begin{equation} \label{e:suff5}
			\sum_{B \in \sB} \sigma(B)^\beta  \frac{m(B)}{(r/a)^\gamma} \le C_1 C_2^{\beta-2} a^{-(\beta-2)\eta} \frac{m(B(x,r))}{r^{\gamma}}.
			\end{equation}
		\end{enumerate}
	\end{proposition}
	
	\begin{proof}
	For a function $u \in \sF \cap \contfunc(X)$ and $B \in \sB$, we define its `discretization' $u_d:\sB \to \bR$ as
	\begin{equation} \label{e:sufp1}
	u_{d}(B) :=  \fint_{B} u\,dm= \frac{1}{m(B)} \int_B u\,dm,
	\end{equation}
	and its `discrete gradient' $\sigma_u \colon \sB \to [0,\infty)$
	\begin{equation} \label{e:sufp2}
	\sigma_u(B) := \sum_{B' \in \sB: \lambda \cdot B' \cap \lambda\cdot B\neq \emptyset} \abs{u_\sB(B')-u_\sB(B)}.
	\end{equation}
	Our construction of $\sigma$ is the discrete gradient $\sigma_u$ of a well chosen function $u$.
	In particular, we choose a function $u \in \sF \cap \contfunc_{\mathrm{c}}(X)$ that satisfies the following properties: there exists $C_3>1,\eta>0$ (that depends only on the constant associated with  \hyperlink{phi}{$\on{PHI(\gamma)}$}) such that for all $x \in X, r < \diam(X,d)/A$, we have
	\begin{align} \label{e:sufp3}
	u \equiv 1 \mbox{ on $B(x,1.1r)$ } &\mbox{and $u \equiv 0$ on $B(x,1.9 r)$,}\\
	\sE(u,u) &\le C_3 \frac{m(B(x,r))}{r^\gamma},  \label{e:sufp4}\\
	\abs{u(y)-u(z)} &\le C_3 \biggl( \frac{d(y,z)}{r} \biggr)^\eta \quad \mbox{for all $y,z \in X$.}  \label{e:sufp5}
	\end{align}
	The existence of a function $u \in \sF \cap \contfunc_{\mathrm{c}}(X)$ satisfying the above properties follows from the cutoff Sobolev inequality \hyperlink{cs}{$\on{CS(\beta)}$}, Theorem \ref{t:phichar} and a standard covering argument as we recall below.
	By Theorem \ref{t:phichar} we have that $m$ is a doubling measure on $(X,d)$ and hence $(X,d)$ is a $K_D$-doubling metric space for some $K_D>1$.
	Therefore there exists $N_D \in \bN$ that depends only on $K_D$ and  $y_1,\ldots,y_{N_D} \in B(x,1.1 r)$ such that $\cup_{i=1}^{N_D} B(y_i,r/10) \supset B(x,1.1 r)$.
	By the construction of cutoff functions in \cite[Section 3]{BBK}, there exists $C_4>0,\eta>0$ such that for each $i=1,\ldots,N_D$ satisfies
	\begin{align*}
	\phi_i\equiv 1 \mbox{ on $B(y_i,r/10)$},& \quad \phi_i\equiv 0 \mbox{ on $B(y_i,r/5)^c$},\\
	\sE(\phi_i,\phi_i) &\le C_4 \frac{m(B(y_i,r/10))}{r^\gamma}, \\
	  \abs{\phi_i(y)-\phi_i(z)} &\le C_4 \biggl( \frac{d(y,z)}{r} \biggr)^\eta \mbox{for all $y,z \in X$.}
	\end{align*}
	By choosing $u= \max_{1\le i \le N_D}\phi_i$ and using the above estimates along with triangle inequality, $\sE(u,u)\le \sum_{i=1}^{N_D} \sE(\phi_i,\phi_i),\abs{u(y)-u(z)}\le \max_{1 \le i \le N_D} \abs{\phi_i(y)-\phi_i(z)}$, we obtain the desired properties \eqref{e:sufp3}, \eqref{e:sufp4} and \eqref{e:sufp5}.

	Let us show that the function $\sigma=\sigma_u$ as defined by \eqref{e:sufp1} and \eqref{e:sufp2} satisfies the desired conditions \ref{it:S1prime} and \ref{it:S2prime}.
	To this end, we note the following properties of $u_d \colon \sB \to \bR$:
	\begin{align*}
	u_d (B(x_B,r/a)) &=1 \quad \mbox{for any $x_B \in B(x,r/10)$ }\\
	& \mbox{(since $B(x_B,r/a) \subset B(x,1.1 r)$),}\\
	u_d(B(y,r/a)) &=0 \quad \mbox{for any $y \in X$ such that $d(y,y') \le 2 \lambda r/a$,}\\
		& \mbox{ where $y' \in B(x,2r)^c$  ($B(y,r/a) \subset B(x,1.9 r)^c$ }\\
		& \mbox{ because $(2 \lambda +1)r/a < 0.1 r$ by \eqref{e:tech}),}\\
	\sum_{i=1}^N \sigma_u(B_i) &\ge \sum_{i=1}^{N-1} \abs{u_d(B_i)-u_d(B_{i+1})}\ge \abs{u_d(B_1)-u_d(B_N)}
	\end{align*}
	for any sequence of balls $B_1,\ldots,B_N \in \sB$ such that $\lambda\cdot B_i \cap \lambda\cdot B_{i+1}\neq \emptyset$ for all $i=1,\ldots,N-1$.
	The above equations immediately imply \eqref{e:suff1} and \eqref{e:suff2}.

	Since the balls $B(y_i,r/(4a)), i \in I$ are disjoint, by doubling property of $(X,d)$, there exists $C_5>1$ that depends only on $\lambda$ and the doubling constant (but not on $a$) such that
	\begin{equation} \label{e:sufp7}
	\# \giveset{B' \in \sB: \lambda\cdot B \cap \lambda \cdot B' \neq \emptyset} \le C_5, \quad \mbox{for all $B \in \sB$.}
	\end{equation} 
	For any two balls $B, B' \in \sB$ such that $\lambda\cdot B \cap \lambda \cdot B' \neq \emptyset$, by \eqref{e:sufp5} we have
	\begin{equation} \label{e:sufp8}
	\abs{u_d(B)-u_d(B')} \le \sup_{y,z \le 2(\lambda+1)r/a} \abs{u(y)-u(z)} \le  C_3 (2(\lambda+1))^\eta a^{-\eta}.
	\end{equation} 
	Combining \eqref{e:sufp2}, \eqref{e:sufp7} and \eqref{e:sufp8}, we obtain \eqref{e:suff4} for $\sigma=\sigma_u$.
	
	It remains to show that $\sigma=\sigma_u$ satisfies \eqref{e:suff4}. To this end, we recall the following
	Poincar\'e inequality \hyperlink{pi}{$\on{PI(\gamma)}$ implied by \hyperlink{phi}{$\on{PHI(\gamma)}$} and Theorem \ref{t:phichar}}:
	there exist $C_P,A>1$ such that
	\begin{align} \label{e:sufp9}
	&\frac{1}{2 m(B(y,s))} \int_{B(y,s)} \int_{B(y,s)} (f(z)-f(w))^2 \,dm(z)\,dm(w) \nonumber \\
	&= \int_{B(y,s)}(f-f_{B(y,s)})^{2}\,dm \le C_P s^{\gamma} \int_{B(y,As)} d\Gamma(f,f), 
	\end{align}
	for any $f \in \sF,y \in X,s>0$. On the basis of \eqref{e:sufp9},
	the following comparison estimate between discrete and continuous energies is standard \cite{CS,BB04}.  
	Similar to \textsection \ref{s:hf}, for any two balls $B,B' \in \sB$ by $B' \sim B$
	we mean that $\lambda\cdot B \cap \lambda\cdot B' \neq \emptyset$.  We obtain \eqref{e:suff3} by the following estimates:
	\begin{align*}
	\MoveEqLeft{\sum_{B \in \sB} \sigma_u^2(B) \frac{m(B)}{(r/a)^\gamma}}\\
	&\lesssim \sum_{B, B' \in \sB, B' \sim B} \abs{u_d(B')-u_d(B)}^2 \frac{m(B)}{(r/a)^\gamma} \\
	& \quad \mbox{ (by \eqref{e:sufp7} and the Cauchy-Schwarz inequality)}\\
	&\lesssim \sum_{B, B' \in \sB, B' \sim B} \frac{1}{(r/a)^\gamma m(B')} \int_B \int_{B'}(u(y)-u(z))^2\,dm(y)\,dm(z)\\
	& \quad \mbox{ (by Jensen's inequality)}\\
	&\lesssim \sum_{B \in \sB} \frac{1}{m((2\lambda+1)\cdot B) (r/a)^\gamma} \int_B \int_{(2\lambda +1)\cdot B}(u(y)-u(z))^2\,dm(y)\,dm(z)\\
	& \qquad \mbox{ (by \ref{VD})}\\
	&\lesssim \sum_{B \in \sB} (2\lambda+1)^\gamma  \int_{A(2\lambda+1)\cdot B}  d\Gamma(u,u)\qquad \mbox{ (by \eqref{e:sufp9})} \\
	&\lesssim  \sE(u,u) \\
	&\quad \mbox{(since $(X,d)$ is $K_D$-doubling, we have $\sum_{B \in \sS_n} \one_{(2A\lambda+A) \cdot B} \lesssim 1$)}\\
	&\lesssim \frac{m(B(x,r))}{r^\gamma} \quad \mbox{(by \eqref{e:sufp4})}.
	\end{align*}
	Finally, \eqref{e:suff5} follows from \eqref{e:suff3}, \eqref{e:suff4}, and
	\[
	\sum_{B \in \sB} \sigma(B)^\beta  \frac{m(B)}{(r/a)^\gamma}  \le  \left( \sup_{B \in \sB} \sigma(B)\right)^{\beta-2}\sum_{B \in \sB} \sigma(B)^2  \frac{m(B)}{(r/a)^\gamma}. 
	\] 	
	\end{proof}

	The following proposition provides a convenient sufficient condition for a measure $\mu$ to be smooth and have full quasi-support.
	\begin{proposition}  \label{p:smooth}
		Let $(X,d,m,\sE,\sF)$ be an MMD space that satisfies \hyperlink{phi}{$\on{PHI(\gamma)}$} for some $\gamma \ge 2$ and let $\theta \in \sJ(X,d)$.
		Let $\beta>2$ and $\mu$ be a Borel measure on $X$ that satisfies the following estimate: there exist $C_1, A>1$ such that
		for any $x \in X$, $0<r<\diam(X,\theta)/A$, we have
		\begin{equation*}
		C_1^{-1} \frac{\mu(B_\theta (x,r))}{r^\beta} \le \Capa(B_\theta(x,r),B_\theta (x,2r)^c) \le C_1 \frac{\mu(B_\theta(x,r))}{r^\beta}.
		\end{equation*}
		Then $\mu \in \sA(X,d,m,\sE,\sF)$, i.e., $\mu$ is a smooth Radon measure on $X$
		with full quasi-support. Furthermore $(X,\theta,\mu)$ is \ref{VD} and \ref{RVD}.
	\end{proposition}
	
	\begin{proof}
	By Theorem \ref{t:phichar} and \cite[Lemma 5.3]{BM}
	(see also Lemma \ref{l:EHI-TC-QS} above), the MMD space $(X,\theta,m,\sE,\sF)$ satisfies \ref{EHI}
	and $(X,d)$ is a doubling, uniformly perfect metric space,
	so that $(X,\theta)$ is also doubling and uniformly perfect by \cite[Theorem 10.18 and Exercise 11.2]{Hei}.
	The volume doubling property \ref{VD} of $\mu$ in $(X,\theta)$ follows from
	Remark \ref{rmk:harnack}, Lemma \ref{l:cap} and \cite[Lemma 6.3]{BCM}, and
	the \ref{RVD} property of $\mu$ in $(X,\theta)$ follows from \ref{VD} of $(X,\theta,\mu)$,
	the uniform perfectness of $(X,\theta)$ and Remark \ref{r:doubling}-\ref{it:RVD-uniformly-perfect}.
	That $\mu$ is a smooth Radon measure on $X$ follows from
	Remark \ref{rmk:harnack} and \cite[Proposition 6.13]{BCM},
	and $\mu$ has full quasi-support by Remark \ref{rmk:harnack} and \cite[Theorem 5.4 and Proposition 6.16]{BCM}.
	\end{proof}
	
	\subsection{Completion of the proof of Theorem \ref{t:ehichar}} \label{ssec:prf-dcw-2}
	
	Now we are in the stage of completing the proof of our first main theorem (Theorem \ref{t:ehichar}).
	
	\begin{proof}[Proof of Theorem \textup{\ref{t:ehichar}}]
	By Theorem \ref{t:dcwfinite}, it suffices to show that \ref{it:ehichar-a} implies \ref{it:ehichar-c}.
	By Lemma \ref{l:dcw-lowerbound}, it suffices to show that $\dcw \le 2$.
	 To this end, we fix an arbitrary $\beta>2$.
	We shall construct a metric $\theta \in \sJ(X,d)$ and a measure $\mu \in \sA(X,d,m,\sE,\sF)$ such that the time-changed MMD space $(X,\theta,\mu,\sE^\mu,\sF^\mu)$ satisfies \hyperlink{phi}{$\on{PHI(\beta)}$}.

	By Theorem \ref{t:dcwfinite}, \eqref{e:cgauge}, \eqref{e:admiss}, and by changing the metric and measure if necessary,
	we may assume that $(X,d,m,\sE,\sF)$ satisfies \hyperlink{phi}{$\on{PHI(\gamma)}$} for some $\gamma>2$. By Theorem \ref{t:phichar},
	$(X,d,m,\sE,\sF)$ satisfies \ref{VD}, \ref{RVD}, \ref{EHI} and \hyperlink{cap}{$\on{cap(\gamma)}$}.

	If $(X,d)$ is bounded, we scale the metric so that  $\diam (X,d)=\frac 1 2$. By completeness and the metric doubling property, we recall that $(X,d)$ is compact \cite[Definition 10.15 and Exercise 10.17]{Hei}.

	Fix $\lambda \ge 32$  and let $a$ be an arbitrary constant that satisfies \eqref{e:tech}. The choice of $a$ will be made later in the proof. Let $x_0 \in X$.
	Let $\sS = \coprod_{k \in \bZ_{\ge 0}} \sS_k$ denote the vertex set of the hyperbolic filling as defined in Definition \ref{d:exhyp}, where $\sS_k= \giveset{B(x,2 a^{-k}): x \in N_k}$, where $N_k, k \in \bZ$ is a sequence of $a^{-k}$-separated sets such that
	$N_k \subset N_{k+1}$ and $x_0 \in N_k$ for all $k \in \bZ$ (recall from Lemma \ref{l:con} this is  a hyperbolic filling in the sense of Definition \ref{d:hypfilling}).

	We define a function $\sC: \coprod_{k \in \bZ} \sS_k \to (0,\infty)$ on the extended hyperbolic filling by
	\begin{equation} \label{e:univ1}
	\sC(B(x, 2a^{-k})) = \frac{m(B(x,2a^{-k}))}{(2a^{-k})^\gamma}, \mbox{for any $k \in \bZ$ and for any $B(x,2a^{-k}) \in \sS_k$.}
	\end{equation}
	Let us verify that $\sC$ is gentle and satisfies the enhanced subadditivity property  \hyperlink{E}{$\on{(E)}$}.
	By Lemma \ref{l:gentle}-(a),(b), there exist $K_h, K_v$ such that  
	$K_v$ depends only on $a$ and the constant associated with \ref{VD}, $K_h$ depends only on $\lambda$ and the constant associated with \ref{VD} such that
	\begin{align} \label{e:univ2}
	\sC(B_1) &\le K_h \sC(B_2), \quad \mbox{whenever $B_1$ and $B_2$ share a horizontal edge, } \nonumber\\
	\sC(B_1) &\le K_v \sC(B_2), \quad \mbox{whenever $B_1$ and $B_2$ share a vertical edge.} 
	\end{align} 
	Recall that for every ball $B \in \sS_k, k \in \bZ$, there exists an unique ball $g(B)_{k-1} \in \sS_{k-1}$ such that there is a vertical edge between $g(B)_{k-1}$.
	Note that is $B \sim B'$ and $B,B' \in \sS_{k+1}$, then by \eqref{e:tech}, we have $d(x_B,x_{B'}) \le 2 \lambda a^{-k-1} \le  \frac{1}{12}a^{-k}$.
	We denote the set of all non-peripheral elements of $\coprod_{k \in \bZ} \sS_k$ by $\sN$ as given in Definition \ref{defE}.
	Hence if $C= g(B)_k$ and $d(x_C,x_B) < \frac{1}{6} (2a^{-k}) + 2a^{-k-1}$, then $d(x_C,x_{B'}) < (3^{-1}+ 2 a^{-1}+12^{-1}) a^{-k} < \frac 1 2 a^{-k}$, and hence $B \in \sN$.
	This along with Lemma \ref{l:gentle}-(c) imply that, there exists $C_1>0$ such that
	\begin{equation} \label{e:univ3}
	\sC(B) \le  C_1 a^{- \gamma} \sum_{B' \in \sN \cap \sD_{k+1}(B)} \sC(B'), \quad \mbox{for all $B \in \sS_k, k \in \bZ$.}
	\end{equation}
	By \eqref{e:univ2} and \eqref{e:univ3}, we conclude that  $\sC$ is gentle and satisfies the enhanced subadditivity property  \hyperlink{E}{$\on{(E)}$}.

	For $k \in \bZ$ and $B \in \sS_k$, let $\Gamma_{k+1}(B)$ 
	denote the set of horizontal paths $\gamma=\giveset{B_i}_{i=1}^N, N \ge 2$ as given in Definition \ref{d:Gamk}. 
	If  $\diam(X,d)=\frac 1 2$, we note that $\Gamma_k(B)= \emptyset$ for all $k \le 0$.
	If $\diam(X,d)=\frac 1 2$, we assume that 
	\begin{equation} \label{e:univ4}
	a > 2A,
	\end{equation}
	where $A$ is the constant in Proposition \ref{p:lowenergy}.
	If either $k \in \bZ, \diam(X,d)=\infty$ or if $k \in \bN, \diam(X,d)=\frac 1 2$,
	for any $B \in \sS_k$, we define $\sigma_B: \sS_{k+1} \to (0,\infty)$ as the function defined in Proposition \ref{p:lowenergy}, that satisfies
	\[
	\sum_{i=1}^N \sigma_B(B_i) \ge 1, \quad \mbox{for any $\giveset{B_i}_{i=1}^N \in \Gamma_{k+1}(B)$.}
	\]
	Otherwise if $\diam(X,d)=\frac 1 2$ and $k \ge 0$, we simply define $\sigma_B: \sS_{k+1} \to [0,\infty)$ as $\sigma_B \equiv 0$ for all $B \in \sS_k$.
	For any $k \in \bZ$ and for any $B \in \sS_{k}$, we define
	\begin{equation} \label{e:univ5}
	\sigma(B)= \max_{C \in \sS_{k-1}} \sigma_C(B), \quad \mbox{for any $k \in \bZ, B \in \sS_k$.}
	\end{equation}
	Evidently, by Proposition \ref{p:lowenergy}, we have
	\begin{equation} \label{e:univ6}
	\sum_{i=1}^N \sigma(B_i) \ge 1, \quad \mbox{for any $\giveset{B_i}_{i=1}^N \in \Gamma_{k+1}(B)$ and for any $k \in \bZ, B \in \sS_k$.}
	\end{equation}
	In the compact case, the above statement is vacuously true for $k \le 0$. 
	By Proposition \ref{p:lowenergy} and the argument in Lemma \ref{l:patch},
	there exist $C_2, \eta>0$ such that
	\begin{equation} \label{e:univ7}
	\sum_{B' \in \sD_{k+1}(B)} \sigma(B')^\beta \sC(B') \le C_2 a^{-(\beta-2)\eta}  \sC(B), \quad \mbox{for any $k \in \bZ, B \in \sS_k$,}
	\end{equation}
	where $\sD_{k+1}(B)$ denote the set of descendants of $B$ in $\sS_{k+1}$ (that is,  $\sD_{k+1}(B)$ is the set of elements in $\sS_{k+1}(B)$ that share a vertical edge with $B$).
	
	We consider two cases. \\
	\textbf{Case 1}: $(X,d)$ is bounded.
	Let $\sS= \coprod_{k \ge 0} \sS_k$ denote the vertex set of the hyperbolic filling.
	In this case by \eqref{e:univ3}, we can ensure the enhanced subadditivity estimate \hyperlink{E}{(E)} by choosing $a$ large enough. Similarly  by \eqref{e:univ6} and \eqref{e:univ7}, the function $\sigma$ defined above satisfies the hypotheses  \hyperlink{S1}{$\on{(S1)}$} and \hyperlink{S2}{$\on{(S2)}$} of Theorem \ref{t:suff} for all large enough $a$. Therefore by Theorems \ref{t:suff} and \ref{t:qscomb}, and Proposition \ref{p:measure}, there exist a metric $\theta \in \sJ(X,d)$ and a measure $\mu$ on $X$ that satisfies the following estimate: there exists $C_3>0$ such that
	\begin{equation} \label{e:univ8} 
	C_3^{-1} r^\beta \sC(B) \le \mu(B_\theta(x,r))  \le   C_3 r^\beta \sC(B) 
	\end{equation}
for all $x \in X, r < \diam(X,\theta) ,B \in \sA_\sS(B_d(x,s))$,
	where $s$ is the largest number in $(0,2 \diam(X,d)]$ such that $B_d(x,s) \subset B_\theta(x,r)$ (as defined in \eqref{e:defs}) and $\sA_\sS(B_d(x,s))$ is as given in Definition \ref{d:AS}.
	Combining \eqref{e:univ8} and Proposition \ref{p:qscap}, we see that there exist $A_1,C_4>0$ such that
	\begin{equation} \label{e:univ9}
	C_4^{-1}   \frac{ \mu(B_\theta(x,r))}{ r^\beta} \le \Capa(B_\theta(x,r),B_\theta(x,2r)^c ) \le C_4 \frac{ \mu(B_\theta(x,r))}{ r^\beta},
	\end{equation}
	for any $x \in X, 0<r< \diam(X,d)/A_1$. By \eqref{e:univ9} and Proposition \ref{p:smooth}, $\mu \in \sA(X,d,m,\sE,\sF)$,
	$(X,\theta,\mu ,\sE^\mu,\sF^\mu)$ satisfies \ref{VD}, \ref{RVD} and \hyperlink{cap}{$\on{cap(\beta)}$},
	and by Lemma \ref{l:EHI-TC-QS} it also satisfies \ref{EHI}.
	Thus by Theorem \ref{t:phichar} the MMD space $(X,\theta,\mu ,\sE^\mu,\sF^\mu)$ satisfies \hyperlink{phi}{$\on{PHI(\beta)}$}.
	Since $\beta> 2$ is arbitrary, we conclude that the conformal walk dimension is two.
	
	\noindent\textbf{Case 2}: $(X,d)$ is unbounded.
	
	The approach in the unbounded case is to construct metrics and measures on an increasing sequence of compact sets that cover $X$, and to take suitable sub-sequential limit. Let $x_0 \in X$ be the point such that $x_0 \in N_k$ for all $k \in \bZ$ as given in Definition \ref{d:exhyp}. We consider the sequence of  subsets
	\begin{equation} \label{e:nuniv1}
	X_n= \overline{D_{-n}(x_0)}, \quad \mbox {for any $n \in \bN$,}
	\end{equation}
	where $D_{-n}(x_0)$ is as defined in \eqref{e:descendants}. By Lemma \ref{l:con}(c), $X_n$ is compact and satisfies
	\begin{equation} \label{e:nuniv2}
	B(x_0, (2^{-1}- (a-1)^{-1}) a^{n}) \subset X_n \subset  \overline{B}(x_0,(1-a^{-1})^{-1} a^{n}).
	\end{equation}
	For any $k \ge -n$, we define
	\begin{equation*}
	\sS_k^{(n)}= \giveset{ B(x,2 a^{-k}) \cap X_n : x \in N_k \cap D_{-n}(x_0)}. 
	\end{equation*}
	By Lemma \ref{l:con}, $\sS^{(n)}:=\coprod_{k \in \bZ, k \ge -n} \sS_k^{(n)}$ is a hyperbolic filling of the compact space with the same vertical edges induced from the extended hyperbolic filling. Similarly, $B \cap X_n, B' \cap X_n \in \sS_k^{(n)}$ share a horizontal edge if and only if $(\lambda \cdot B \cap X_n)\cap (\lambda \cdot B' \cap X_n) \neq \emptyset$. 
	We define $\sC_n:\sS^{(n)} \to (0,\infty), \sigma_n: \sS^{(n)} \to [0,\infty)$ as
	\[
	\sC_n( B \cap X_n)= \sC(B), \quad \sigma_n(B\cap X_n)= \sigma(B),
	\]
	for any $B \in B \in \coprod_{k \ge -n} \sS_k$, where $\sC$ is as given in \eqref{e:univ1} and $\sigma$ is a s given in \eqref{e:univ5}. 
	Similar to the compact case, by choosing $a>1$ large enough, by  \eqref{e:univ3}, we obtain 
	the enhanced subadditivity estimate \hyperlink{E}{$\on{(E)}$} for $\sC_n$ uniformly over $n$ (that is, the constant $\delta$ in associated with  \hyperlink{E}{$\on{(E)}$} does not depend on $n$) 
	Similarly, by increasing $a$ is necessary, and  by \eqref{e:univ6} and \eqref{e:univ7}, the function $\sigma_n$ defined above satisfies the hypotheses  \hyperlink{S1}{$\on{(S1)}$} and \hyperlink{S2}{$\on{(S2)}$} of Theorem \ref{t:suff}  uniformly in $n$ (that is, the constant $\eta_0$ in associated with  \hyperlink{E}{$\on{(S2)}$} does not depend on $n$). 
	
	Similar to the compact case,  by Theorems  \ref{t:suff} and \ref{t:qscomb}, there exist metrics $\theta_n \in \sJ(X_n,d)$ for each $n \in \bN$, and  a distortion function $\eta:[0,\infty) \to [0,\infty)$ such that the identity map
	\begin{equation} \label{e:nuniv4}
	\on{Id}: (X_n,d) \to (X_n,\theta_n) \mbox{ is an $\eta$-quasisymmetry for each $n \in \bN$.}
	\end{equation}
	By Proposition \ref{p:measure}, there exist
	measures $\mu_n$ on $X_n$ for each $n \in \bN$, constant $C_5>0$ such that 
	\begin{equation} \label{e:nuniv5} 
	C_5^{-1} r^\beta \sC_n(B) \le \mu_n(B_{\theta_n}(x,r))  \le   C_5 r^\beta \sC_n(B),
	\end{equation}
	for all $x \in X_n, r < \diam(X_n,\theta) ,B \in \sA_\sS(B_d(x,s))$,
	where $s$ is the largest number in $[0,2 \diam(X_n,d)]$ such that $B_d(x,s) \cap X_n \subset B_\theta(x,r)$ (as defined in \eqref{e:defs}) and $\sA_\sS(B_d(x,s) \cap X_n)$ is as given in Definition \ref{d:AS} corresponding to the hyperbolic filling $\sS^{(n)}$ of $X_{n}$.
	
	Next, normalize the metrics and measures by choosing a pair of sequences $\beta_n, \gamma_n >0$ such that $\widehat{\theta}_n=\beta_n \theta_n, \widehat{\mu}_n= \gamma_n \mu_n$ satisfy
	\begin{equation} \label{e:nuniv6}
	\diam(B_d(x_0,1), \widehat{\theta}_n)=1, \quad \widehat{\mu}_n(B_d(o,1))= 1.
	\end{equation}
	By \eqref{e:nuniv2} and \eqref{e:tech}, we note that  $B(x_0,1) \subset X_n$ for all $n \in \bN$. 
	
	We choose $p\in X$ such that $d(x_0,p) = \frac 1 2$. Since $\on{Id}:(X_n,d) \to (X_n, \widehat{\theta}_n)$ is a $\eta$-quasisymmetry, by comparing the ratio of the diameter of the sets $\giveset{x_0,p} \subset X_1$ in the metrics $d$ and $\theta_n$ using \eqref{e:diamQS}, there exists $C_{x_0,p} >1$ such that
	\begin{equation*}
	C_{x_0,p}^{-1} \le \widehat{\theta}_n(o,p) \le C_{x_0,p}^{-1}, \mbox{ for all $n \in \bN$.}
	\end{equation*}
	We estimate $\widehat{\theta}_n(x,y)/ \widehat{\theta}_n(x_0,p)$ by writing it as $\frac{\theta_n(x,y)}{\widehat\theta(x_0,x)} \frac{\theta_n(x_0,x)}{\widehat\theta(x_0,p)}$ 
	and using $\eta$-quasisymmetry to estimate each of the factors and their reciprocals. This yields the following estimate; for any $x,y \in X_n, n \in \bN$, there exists $C_{x,y} >1$ ($C_{x,y}$ depends only on $d(x_0,x), d(x,y)$ and $\eta$) such that
	\begin{equation} \label{e:nuniv7}
	C_{x,y}^{-1} \le \widehat{\theta}_n(x,y) \le C_{x,y}^{-1}, \mbox{ for all $n \in \bN$ such that $\giveset{x,y} \subset X_n$.}
	\end{equation}
	By a similar computation, for any $(x,y) \in X \times X$ and for any $\varepsilon >0$, there exists $\delta>0$ such that for any $(x',y') \in X \times X$ with $d(x,x') \vee d(y,y') < \delta$, we have
	\begin{equation} \label{e:nuniv8}
	\abs{\widehat{\theta}_n(x,y) - \widehat{\theta}_n(x',y')} \le  \widehat{\theta}_n(x,x')+\widehat{\theta}_n(y,y') < \varepsilon, 
	\end{equation}
	for all $n \in \bN$ such that $x,x',y,y' \in X_n$.
	By \eqref{e:nuniv8}, \eqref{e:nuniv7} and Arzel\`{a}-Ascoli theorem on the product space $X_n \times X_n$ equipped with the product metric $d_\infty((x,y),(x',y'))= d(x,x') \vee d(y,y')$,
	the sequence of functions $\widehat{\theta}_m, m \ge n$, has a subsequence that converges uniformly to a metric $\theta$ in $X_n$.
	By a diagonalization argument, we obtain a subsequence of  $\giveset{\widehat{\theta}_n : n \in \bN}$, that converges uniformly in compact subsets of $X \times X$.
	The limit metric $\theta \in \sJ(X,d)$ and $\on{Id}:(X,d) \to (X,\theta)$ is an $\eta$-quasisymmetry.
	
	The measures $\widehat{\mu}_n$ constructed using Lemma \ref{l:meas} are uniformly doubling in the following sense: there exists $C_D \ge 1$ such that 
	for all $n \in \bN, x\in X_n, 0 <r <\diam(X_n, \widehat{\theta}_n)$, we have
	\[
	\frac{\widehat{\mu}_n\left(B_{\widehat{\theta}_n}(x,2r)\right)}{\widehat{\mu}_n\left(B_{\widehat{\theta}_n}(x,r)\right)} \le C_D.
	\]
	By the argument in \cite[Theorem 1]{LuS}, by a further diagonalization argument using weak*-compactness of $\giveset{\widehat{\mu}_m : m \ge n}$ on $X_n$ for all $n \in \bN$, we obtain a measure $\mu$ on $X$.
	By \eqref{e:nuniv5} and Proposition \ref{p:qscap}, there exist constants $C_6>1,A_2 >1$ such that
	\begin{equation}
	C_6^{-1}   \frac{ \mu(B_\theta(x,r))}{ r^\beta} \le \Capa(B_\theta(x,r),B_\theta(x,2r)^c ) \le C_4 \frac{ \mu(B_\theta(x,r))}{ r^\beta},
	\end{equation}
	for any $x \in X, r>0$.
	The remainder of the proof is exactly same as the compact case.
	Hence, we conclude that the conformal walk dimension is two.
	\end{proof}

	\section{The attainment and Gaussian uniformization problems} \label{sec:GUP}
	
	In this section, we introduce the attainment problem for the conformal walk dimension and the Gaussian uniformization problem. Then we discuss partial progress towards them.
	
	Let $(X,d,m,\sE,\sF)$ be an MMD space that satisfies the metric doubling property and \hyperlink{ehi}{$\on{EHI}$}. 
	Recall that the Gaussian uniformization problem ask for a description of all metrics $\theta \in \sJ(X,d)$ and measures $\mu \in \sA(X,d,m,\sE,\sF)$ such that the corresponding time-changed MMD space $(X,\theta,\mu,\sE^\mu,\sF^\mu)$ satisfies 
	\hyperlink{phi}{$\on{PHI(2)}$}.
	For any $\beta > 0$, we define
	\begin{equation} \label{e:subgaussadmissible}
	\sG_\beta(X,d,m,\sE,\sF) := \biggl\{ \mu \biggm\vert
	\begin{minipage}{200pt}
		$\mu \in \sA(X,d,m,\sE,\sF)$, there exists $\theta \in \sJ(X,d)$ such that
		$(X,\theta,\mu,\sE^\mu,\sF^\mu)$ satisfies \hyperlink{phi}{$\on{PHI(\beta)}$}
	\end{minipage}
	\biggr\}.
	\end{equation}
	We define the set of \emph{Gaussian admissible measures} as
	\begin{equation} \label{e:gauss}
	\sG(X,d,m,\sE,\sF):= \sG_2(X,d,m,\sE,\sF).
	\end{equation}
	By Theorem \ref{t:ehichar}, we have
	\[
	\sG_\beta(X,d,m,\sE,\sF) = \emptyset \mbox{ for any $\beta < 2$, and } \sG_\beta(X,d,m,\sE,\sF) \neq \emptyset \mbox{ for any $\beta>2$.}
	\]
	This raises the following questions:
	\begin{enumerate}[label=\textup{\arabic*.},align=left,leftmargin=*,topsep=5pt,parsep=0pt,itemsep=2pt]
		\item \textbf{Attainment problem}: Is $\sG(X,d,m,\sE,\sF) \neq \emptyset$? Or equivalently, is the infimum in \eqref{e:dcw} attained?
		\item \textbf{Gaussian uniformization problem}: Describe all measures in the set $\sG(X,d,m,\sE,\sF)$.
	\end{enumerate}
	
	By Proposition \ref{p:metmeas}-\ref{it:conseq-PHI2-met}, the Gaussian admissible measures can be described as
	\begin{equation} \label{e:gaussianadmissible}
	\sG(X,d,m,\sE,\sF) = \biggl\{ \mu \biggm\vert
		\begin{minipage}{160pt}
		$\mu \in \sA(X,d,m,\sE,\sF)$, $d_{\on{int}}^\mu \in \sJ(X,d)$,
		$(X,d_{\on{int}}^\mu,\mu,\sE^\mu,\sF^\mu)$ satisfies \hyperlink{phi}{$\on{PHI(2)}$}
		\end{minipage}
	\biggr\}, 
	\end{equation}
	where $d_{\on{int}}^\mu$ denotes the intrinsic metric of the
	MMD space $(X,d,\mu,\sE^\mu,\sF^\mu)$ (recall Definition \ref{d:dint}).

	In this section, we prove Theorem \ref{t:ainf} and discuss its consequences for the Gaussian uniformization problem. In particular, Theorem \ref{t:ainf} shows that any two measures $\mu_1,\mu_2 \in \sG(X,d,m,\sE,\sF)$ must be $A_\infty$-related in $(X,d)$ (provided such measures exist). 
	
	\subsection{Consequences of PHI(2)} \label{ssec:conseq-PHI2}
	We begin with the proof of Proposition \ref{p:metmeas}, which is essentially contained in \cite[Section 4]{KM}.
	This states that any measure on $\sG(X,d,m,\sE,\sF)$ must be a minimal energy-dominant measure and that
	the metric must be bi-Lipschitz equivalent to the intrinsic metric of the time-changed MMD space.
	
	\begin{proof}[Proof of Proposition \textup{\ref{p:metmeas}}]
	By \hyperlink{phi}{$\on{PHI}(2)$} and Theorem \ref{t:phichar}, we have 
	\ref{VD} and all of the equivalent conditions in \cite[Theorem 1.2]{GHL15}.
	\begin{enumerate}[label=\textup{(\alph*)},align=left,leftmargin=*,topsep=5pt,parsep=0pt,itemsep=2pt]
		\item[\ref{it:conseq-PHI2-met}]We use \cite[Theorem 1.6 and Remark 1.7-(a)]{Mur1} and \cite[Proposition 4.8]{KM} to obtain \ref{it:conseq-PHI2-met}.
		\item[\ref{it:conseq-PHI2-meas}]This follows from \cite[Propositions 4.5 and 4.7]{KM}.
		\qedhere
	\end{enumerate}
	\end{proof}
	
	\begin{definition} \label{d:lip}
			Let $(X,d)$ be a metric space and let $u \colon X \to \mathbb{R}$.
			We define the \emph{pointwise Lipschitz constant} $\on{Lip}u(x)$ of $u$ at $x \in X$ as
			\[
			\on{Lip}u(x) := \limsup_{y \to x} \frac{\abs{u(x)-u(y)}}{d(x,y)},
			\]
			and $\on{Lip}(X)$ denotes the collection of all functions $u \colon X \to \mathbb{R}$ with
			\[
			\norm{u}_{\on{Lip}(X)} := \sup_{x,y \in X, x \neq y} \frac{ \abs{u(x)-u(y)}}{d(x,y)} < \infty.
			\]
			When it is necessary, we also write $\on{Lip}$ as $\on{Lip}_d$ to specify the metric $d$.
	\end{definition}
	We recall the notion of upper gradient and its variants. We refer the reader to \cite{HKST, Hei} for a comprehensive account.
	\begin{definition} \label{d:uppergrad}
			Let $(X,d,m)$ be a metric measure space and let $u \colon X \to \bR$ be a Borel measurable function.  
			A non-negative Borel measurable function $g$ is called an \emph{upper gradient}\footnote{This notion is called \emph{very weak gradient} in \cite{HK}. Our terminology is borrowed from \cite{HKST}.} if 
			\[
			\abs{u(x)-u(y)} \le \int_{\gamma} g\,ds,
			\]	
			for every rectifiable curve  $\gamma$ between $x$ and $y$.
			A non-negative Borel measurable function $g$ is called a $p$-\emph{weak upper gradient of} $u$ with $p \in [1,\infty)$ if
			\[
			\abs{u(x)-u(y)} \le \int_{\gamma} g\,ds,
			\]
			for all $\gamma\in \Gamma_{\on{rect}} \setminus \Gamma_0$, where $x$ and $y$ are the endpoints of $\gamma$, $\Gamma_{\on{rect}}$ denotes the collection of non-constant compact rectifiable curves and $\Gamma_0$ has $p$-modulus zero in the sense that
			\[
			\inf \biggl\{ \norm{\rho}^p_{L^p(X)} \biggm\vert \begin{minipage}{165pt}$\rho\colon X\to[0,\infty]$, $\rho$ is Borel measurable, $\int_{\gamma} \rho\,ds \ge 1$ for all $\gamma \in \Gamma_0$\end{minipage}\biggr\}=0.
			\]
			We denote by $N^{1,p}(X)$ the collection of functions $u \in L^p(X)$ that have a $p$-weak upper gradient $g \in L^p(X)$, and define $\norm{u}_{N^{1,p}(X)}= \norm{u}_{L^p(X)} + \inf_g \norm{g}_{L^p(X)}$, where $g$ is taken over all $p$-weak upper gradients of $u$. 
			We denote by $N_{\on{loc}}^{1,p}(X)$ the class of functions $u \in L^p_{\on{loc}}(X)$ that have a $p$-weak upper gradient that belongs to $L^p(B)$ for each ball $B$.
			If necessary, we denote the spaces $N^{1,p}(X)$ and $N_{\on{loc}}^{1,p}(X)$  by 
			$N^{1,p}(X,d,m)$ and $N_{\on{loc}}^{1,p}(X,d,m)$ respectively.
\end{definition}
\begin{definition} \label{d:p-poincare}
			We say that $(X,d,m)$ supports a $(1,p)$-\emph{Poincar\'e inequality} with $p \in [1,\infty)$ if there exists constants $K \ge 1, C >0$ such that for all $u \in \on{Lip}(X), x \in X$ and $r>0$,
			\[
			\fint_{B(x,r)} \abs{u-u_{B(x,r)}} \, dm \le C r \left[\fint_{B(x,Kr)}\left( \on{Lip}(u) \right)^p\,dm \right]^{1/p},
			\]
			where $\fint_A f\,dm$ denotes $\frac{1}{m(A)} \int_A f\,dm$ and $u_{B(x,r)}= \fint_{B(x,r)} u\,dm$.
			It is known that $(X,d,m)$ supports a $(1,p)$-Poincar\'e inequality if and only if 
			there exists constants $K \ge 1, C >0$ such that for every function $u$ that is integrable on balls and for any upper gradient $g$ of $u$ in $X$, $x \in X$ and $r>0$,
			\[
			\fint_{B(x,r)} \abs{u-u_{B(x,r)}} \, dm \le C r \left[\fint_{B(x,Kr)} g^p\,dm \right]^{1/p},
			\]
			where $u_{B(x,r)}$ is as above \cite[Theorem 8.4.2]{HKST}. 
	\end{definition}
	We need the following self-improvement of Poincar\'e inequality.
	\begin{proposition} \label{p:self-i}
		Let $(X,d,m,\sE,\sF)$ be an MMD space that satisfies \hyperlink{phi}{$\on{PHI}(2)$}. Then we have the following:
		\begin{enumerate}[label=\textup{(\alph*)},align=left,leftmargin=*,topsep=5pt,parsep=0pt,itemsep=2pt]
			\item \textup{(Cf. \cite[Theorem 2.2]{KZ12})} $\sF=N^{1,2}(X)$ with equivalent norms, $\on{Lip}(X) \cap \contfunc_{\mathrm{c}}(X)$ is dense in $\sF$ and $\sF_{\on{loc}} = N^{1,2}_{\on{loc}}(X)$.
			\item \textup{(Cf. \cite[Theorem 1.0.1]{KZ08})} $(X,d,m)$ satisfies $(1,p)$-Poincar\'e inequality for some $p \in [1,2)$.
		\end{enumerate}
	\end{proposition}
	
	\begin{proof}
	(a)
	Let $d_{\on{int}}$ denote the intrinsic metric corresponding to the MMD space \hyperlink{phi}{$\on{PHI}(2)$}.
	Since \ref{VD} and \hyperlink{pi}{$\on{PI(2)}$} are preserved under a bi-Lipschitz change of the metric (cf.\ \cite[Lemma 8.3.18]{HKST}), by Proposition \ref{p:metmeas}-\ref{it:conseq-PHI2-met},
	the MMD space $(X,d_{\on{int}},m,\sE,\sF)$ also satisfies \ref{VD} and \hyperlink{pi}{$\on{PI(2)}$}.
	Therefore by \cite[Theorem 2.2]{KZ12}, $\sF=N^{1,2}(X,d_{\on{int}},m)$ with equivalent norms and
	$\on{Lip}_{d_{\on{int}}}(X) \cap \contfunc_{\mathrm{c}}(X)$ is dense in $\sF$ and $\sF_{\on{loc}} = N^{1,2}_{\on{loc}}(X,d_{\on{int}},m)$.
	Since $d$ and $d_{\on{int}}$ are bi-Lipschitz equivalent, $\on{Lip}_{d_{\on{int}}}(X) = \on{Lip}_{d}(X)$ and $\on{Lip}_d(u)$ is comparable to  $\on{Lip}_{d_{\on{int}}}(u)$; that is there exists $C>0$ such that
	\begin{equation} \label{e:si1}
	C^{-1}\on{Lip}_d(u)(x) \le \on{Lip}_{d_{\on{int}}}(u)(x) \le C \on{Lip}_d(u)(x) \quad \mbox{for all $x \in X, u \in \on{Lip}_{d}(X)$.}
	\end{equation}
	
	(b) By (a), \cite[Proposition 2.1]{KZ08} and \cite[Lemma 8.3.18]{HKST}, $(X,d,m)$ satisfies the $(1,2)$-Poincar\'e inequality. By the self-improving property of \cite[Theorem 1.0.1]{KZ08},  $(X,d,m)$ satisfies $(1,p)$-Poincar\'e inequality for some $p \in [1,2)$. 
	\end{proof}
	
	\subsection{$A_\infty$-weights and the Gaussian uniformization problem} \label{ssec:Ainfty}
	\begin{definition}[$A_\infty$-relation] \label{d:ainf}
			Let $(X,d,m)$ be a complete metric measure space such that $m$ is a doubling measure. Let $m'$ be another doubling Borel measure on $X$. Then $m'$ is said to be \emph{$A_\infty$-related} to $m$ if for each $\varepsilon>0$ there exists $\delta>0$ such that
			\[
			m(E) < \delta m(B)\quad \mbox{implies}\quad m'(E) < \varepsilon m'(B)
			\]
			whenever $E$ is a measurable subset of a ball $B$. Evidently, if $m'$ is $A_\infty$-related to $m$, then $m'$ is absolutely continuous with respect to $m$, so that $dm' = w\,dm$ for some nonnegative locally integrable weight function $w$. It turns out that being $A_\infty$-related is a symmetric relation among doubling measures; that is, if $m'$ is $A_\infty$-related to $m$, then $m$ is $A_\infty$-related to $m'$ \cite[Chapter I]{ST}.
			
			Consider the following \emph{reverse H\"older inequality}:
			there is a locally $m$-integrable function $w$ in $X$ together with constants $C \ge 1$ and $p>1$ such that $dm'=w\,dm$ and
			\begin{equation} \label{e:rh}
			\left(\fint_{B} w^p \,dm \right)^{1/p} \le C \fint_B w \, dm
			\end{equation}
			whenever $B$ is a ball in $(X,d)$. It is well known that a doubling measure $m'$ is $A_\infty$-related to $m$ if and only if $m'$ is non-zero and the reverse H\"older inequality \eqref{e:rh} is satisfied \cite[Chapter I]{ST}.
	\end{definition}
	The $A_\infty$ relation among doubling measures is preserved under quasisymmetric change of metric as we show below.
	\begin{lemma} \label{l:bilip-ainf}
		Let $d_1,d_2$ two quasisymmetric metrics on 
		$X$ such that the metrics $d_1,d_2$ are uniformly perfect. Let $m_1,m_2$ be two doubling Borel measures with respect to $d_1$ such that $m_1$ and $m_2$ are $A_\infty$-related with respect to the metric $d_1$. Then $m_1$ and $m_2$ are $A_\infty$-related with respect to $d_2$.
	\end{lemma}
	
	\begin{proof}
	Let $B_i(x,r)$ denote the open ball in metric $d_i$ for $i=1,2$.
	By \cite[Lemma 1.2.18]{MT10}, there exists $C>0$ such that the following holds: for each $x \in X, r>0$ and $i \in \giveset{1,2}$, there exists $s>0$ such that
	\begin{equation} \label{e:la1}
	B_{3-i}(x,C^{-1}s)\subseteq B_{i}(x,r)\subseteq B_{3-i}(x,Cs), \quad \mbox{for all $x \in X, r>0$.}
	\end{equation}
	Note that since $m_1$ and $m_2$ are doubling with respect to $d_1$, they are also doubling on $(X,d_2)$. Therefore, there exists $C_1>0$ such that
	\begin{equation} \label{e:la3} 
	\frac{m_i(B_j(x,Cr))}{m_i(B_j(x,r))} \le C_1,\quad \mbox{for all $x \in X, r>0$, and $i,j \in \giveset{1,2}$.}
	\end{equation}
	Since $m_1$ and $m_2$ are $A_\infty$-related in $(X,d_1)$, we have $m_2 \ll m_1$, $dm_2= w\,dm_1$, where $w \ge 0$ is a Borel measurable function that satisfies the following reverse H\"older inequality: there exists $C_R \ge 1, p >1$ such that
	\begin{equation} \label{e:la2}
	\left(\fint_{B_1(x,r)} w^p \,dm_1 \right)^{1/p} \le C_R \fint_{B_1(x,r)} w \, dm_1,\quad \mbox{for all $x \in X, r>0$.}
	\end{equation} 
	For all $x \in X$, $r>0$, we estimate
	\begin{align*}
	\left(\fint_{B_2(x,r)} w^p \,dm_1 \right)^{1/p}  &\le 
	\left(\fint_{B_1(x,Cs)} w^p \,dm_1 \right)^{1/p} \left( \frac{m_1(B_1(x,Cs))}{m_1(B_1(x,C^{-1}s))}\right)^{1/p} \\
	&\qquad \mbox{(by \eqref{e:la1})} \\
	&\le C_R C_1^{2/p} \fint_{B_1(x,Cs)} w \, dm_1 \quad \mbox{(by \eqref{e:la2} and \eqref{e:la3})} \\
	&= C_R C_1^{2/p} \frac{m_2(B_1(x,Cs))}{m_1(B_1(x,Cs))} \quad \mbox{(by $dm_2=w\,dm_1$)} \\
	&\le C_R C_1^{2/p}\frac{m_2(B_1(x,Cs))}{m_1(B_2(x,r))}\quad   \mbox{(by \eqref{e:la1})}\\
	&\le C_R C_1^{4/p}\frac{m_2(B_1(x,C^{-1}s))}{m_1(B_2(x,r))}  \quad \mbox{(by \eqref{e:la3})}\\
	&\le C_R C_1^{4/p}\frac{m_2(B_2(x,r))}{m_1(B_2(x,r))}  \quad \mbox{(by \eqref{e:la1})}\\
	&\le   C_R C_1^{4/p}\fint_{B_2(x,r)} w\,dm_1, \quad \mbox{(since $dm_2=w \, dm_1$).}
	\end{align*}
	\end{proof}
	
	
	Let $f \colon (X_1,d_1) \to (X_2,d_2)$ be a homeomorphism between two metric spaces.
	For all $x \in X, r>0$, we define
	\begin{equation} \label{e:Lf}
	L_f(x,r)= \sup \giveset{ d_2(f(x),f(y)): d_1(x,y) \le r}, \quad L_f(x)= \limsup_{r \to 0}\frac{L_f(x,r)}{r}.
	\end{equation}
	For $\varepsilon>0$, we define
	\begin{equation} \label{e:Lfeps}
	L_f^\varepsilon(x)= \sup_{0 < r \le \varepsilon} \frac{L_f(x,r)}{r}.
	\end{equation}
	Clearly, $L_f$ decreases as $\varepsilon$ decreases and
	\begin{equation*}
	\lim_{\varepsilon \downarrow 0} L_f^\varepsilon(x)= L_f(x), \qquad \mbox{for all $x \in X$.}
	\end{equation*}
	\begin{lemma}[{\cite[Lemma 7.16]{HK}}] \label{l:wgrad}
		Let $f \colon (X_1,d_1)\to (X_2,d_2)$ be a $\eta$-quasisymmetry. Let $x_0 \in X$ and $0<R<\diam(X,d_1),$. There is a constant $C$ (that depends only on $\eta$) such that for all $\varepsilon>0$, the function $CL^\varepsilon_f$ is an upper gradient of the function $u(x)= d_2(f(x),f(x_0))$ in $B(x_0,R)$.
	\end{lemma}
	We introduce the notion of $C$-approximation to compare balls in different metrics.
	\begin{definition}
			Let $d_1$ and $d_2$ be two metrics on $X$ such that the identity map $\on{Id}:(X,d_1) \to (X,d_2)$ is a $\eta$-quasisymmetry. Let $C \ge 1$ be a constant. We say that a ball  $B_{d_2}(x_2,r_2)$ is a $C$-approximation of  $B_{d_1}(x_1,r_1)$ if
			\begin{align*}
			d_1(x_1,x_2) \le C r_1, \qquad d_2(x_1,x_2) \le C r_2 \\
			B_{d_1}(x_2, C^{-1} r_1) \subset B_{d_2}(x_2,r_2) \subset B_{d_1}(x_2, Cr_1), \\
			B_{d_2}(x_1, C^{-1} r_2) \subset B_{d_1}(x_1,r_1) \subset B_{d_2}(x_1, Cr_2)
			\end{align*}
	\end{definition}
	By the same argument as Proposition  \ref{p:qscap}, we obtain the following comparison of capacities.
	\begin{lemma} \label{l:Cap-approx}
		Let $(X,d_i,m_i,\sE,\sF_i), i =1,2,$ be two MMD spaces that satisfy \hyperlink{phi}{$\on{PHI(2)}$}
		such that the identity map $\on{Id} \colon (X,d_1) \to (X,d_2)$ is a quasisymmetry.
		Let $B_i(x,r)$ denote a open ball of radius $r$ and center $x$, for $i=1,2$. Let $C_1 \ge 1$ and $A_1, A_2>1$.  There exists $C_2, A_3>1$ such that 
		\[
		\Capa(B_1(x_1,r_1), B_1(x_1,A_1r_1)^c) \le  C_2	\Capa(B_2(x_2,r_2), B_2(x_2,A_2r_2)^c) 
		\]
		for all balls $B_1(x_1,r_1)$ and $B_2(x_2,r_2)$ such that $r_1 < \diam(X,d_1)/A_3, r_2 < \diam(X,d_2)/A_3$ and that $B_1(x_1,r_1)$ is a $C_1$-approximation of $B_2(x_2,r_2)$.
	\end{lemma}
	
	The following is an analogue of \cite[Lemma 7.19]{HK}
	\begin{lemma} \label{l:weakl2}
		Let $(X,d_i,m_i,\sE,\sF_e \cap L^2(m_i)), i=1,2$ be two MMD spaces that satisfy \hyperlink{phi}{$\on{PHI}(2)$} and are time changes of each other with full quasi-support.
		Let the identity map $f \colon (X,d_1)\to (X,d_2)$ be an $\eta$-quasisymmetry. Then the function $L^\varepsilon_f$ defined in \eqref{e:Lfeps} is in weak $L^2$ for any $\varepsilon<R/10$ and for any ball $B_{d_1}(x_0,R), R<\diam(X,d_1)$. Furthermore, there exists $C \ge 1$ such that $L^\varepsilon_f$ satisfies the estimate
		\begin{equation} \label{e:weakl2}
		m_1\left(\giveset{x \in B_{d_1}(x_0,R): L^\varepsilon_f(x)>t}\right) \le C t^{-2} m_2 \left(B_{d_1}(x_0,R)\right), 
		\end{equation}
		for all $t>0, 0<R<\diam(X,d_1),x_0 \in X$. 
		Here $C \ge 1$ depends only on $\eta$ and the constants associated with the MMD spaces $(X,d_i,m_i,\sE,\sF_e \cap L^2(m_i)), i=1,2$.
	\end{lemma}
	
	\begin{proof}
	Let $E_t$ denote the set 
	\[
	E_t:=\giveset{x \in B_{d_1}(x_0,R): L^\varepsilon_f(x)>t}.
	\]
	Then by the $5B$-covering lemma  \cite[Theorem 1.2]{Hei} there exists a countable collection of disjoint balls $B_i=B_{d_1}(x_i,r_i), i \in I$ such that $0 <r_i \le \varepsilon$,
	\begin{equation} \label{e:wk0}
	\frac{L_f(x_i,r_i)}{r_i} >t
	\end{equation}
	and
	\[
	E_t \subset \cup_i 5B_i \subset 2 B.
	\]\\
	Note that  the metrics $d_1,d_2$ are uniformly perfect by Proposition \ref{p:metmeas}-\ref{it:conseq-PHI2-met}.
	Define \[B'_i := B_{d_2}(x_i,L_f(x_i,r_i)/\eta(1)).\] 
	Roughly speaking, the balls $B_i'$ in $d_2$-metric approximate the balls $B_i$ in the $d_1$-metric for each $i \in I$. 
	More precisely, since $f$ is a $\eta$-quasisymmetry and $d_1,d_2$ are uniformly perfect, 
	there exists $C \ge 1$ such that  $B_i$ is a $C$-approximation of $B_i'$ for all $i \in I$. 
	In particular,
	\begin{equation} \label{e:wk1}
	C^{-1}B_i' \subset B_i \subset CB_i', \quad C^{-1} B_i \subset B_i' \subset C B_i, \mbox{ for all $i \in I$.}
	\end{equation}
	Since $(X,d_i,m_2,\sE,\sF_e \cap L^2(m_i)), i=1,2$ satisfies \hyperlink{phi}{$\on{PHI}(2)$}, there exists $A_1,A_2 > 1$ such that 
	\begin{equation} \label{e:wk2}
	\Capa(B_i, (A_1 B_i)^c) \asymp \frac{r_i^2}{m_1(B_i)}, \quad \Capa(B_i', (A_2 B_i')^c) \asymp \frac{L_f(x_i,r_i)^2}{m_2(B_i')} \quad \mbox{for all $i \in I$.}
	\end{equation}
	Furthermore, since by \eqref{e:wk1} and Lemma \ref{l:Cap-approx}, we have
	\begin{equation} \label{e:wk3}
	\Capa(B_i, (A_1 B_i)^c) \asymp \Capa(B_i', (A_2 B_i')^c)\quad \mbox{for all $i \in I$.}
	\end{equation}
	We combine the above estimates, to obtain \eqref{e:weakl2} as follows:
	\begin{align*}
	m_1(E_t) &\le \sum_{i} m_1(5B_i) \lesssim \sum_{i} m(B_i) \quad \mbox{(since $m_1$ is a doubling )}\\
	&\lesssim \sum_i r_i^2 \Capa(B_i, (A_1B_i)^c) \quad  \mbox{(by \eqref{e:wk2})} \\
	&\lesssim  t^{-2} \sum_i L_f(x_i,r_i)^2 \Capa(B_i, (A_1B_i)^c) \quad \mbox{(by \eqref{e:wk0})} \\
	&\lesssim t^{-2} \sum_i  L_f(x_i,r_i)^2  \Capa(B'_i, (A_2 B'_i)^c) \quad \mbox{(by \eqref{e:wk3})} \\
	&\lesssim t^{-2} \sum_i m_2(B_i')  \quad  \mbox{(by \eqref{e:wk2})} \\
	&\lesssim t^{-2} \sum_i m_2(C^{-1}B_i') \quad \mbox{(by \ref{VD} of $(X,d_2,m_2)$)}\\
	& \lesssim t^{-2} \sum_i m_2(B_i)  \quad \mbox{(since $C^{-1}B_i' \subset B_i$)}\\
	&\lesssim t^{-2} m_2 \left(B_{d_1}(x_0,2R)\right) \\
	&\qquad \mbox{(since $B_i$'s are disjoint and $\cup_i B_i \subset B_{d_1}(x_0,2R)$)} \\
	&\lesssim t^{-2} m_2 \left(B_{d_1}(x_0,R)\right)\quad \mbox{(by \ref{VD} of $(X,d_1,m_2)$)}.
	\end{align*}
	The claimed dependence of the constant $C$ in \eqref{e:weakl2} follows from the above argument.
	\end{proof}
	
	\begin{corollary}[{\cite[Corollary 7.21]{HK}}] \label{c:weakl2} 
		Let $(X,d_i,m_i,\sE,\sF_e \cap L^2(m_i)), i=1,2$ be two MMD spaces that satisfy the assumptions of Lemma \textup{\ref{l:weakl2}}. Let $L_f^\varepsilon$ denote the function defined in \eqref{e:Lfeps}. For all $s \in [1,2)$ and $x_0 \in X, 0<\varepsilon<R/10$, $R < \diam(X,d_1)$, the function $L_f^\varepsilon$ is in $L^s(B_{d_1}(x_0,R),m_1)$ with 
		\[
		\left( \int_{B_{d_1}(x_0,R)} \abs{L_f^\varepsilon}^s\, dm_1 \right)^{1/s} \le C m_1(B(x_0,R))^{(2-s)/(2s)} m_2\left(B(x_0,R)\right)^{1/s},
		\]
		where $C$ only depends only on $s, \eta$ and the constants associated with the two MMD spaces. By letting $\varepsilon \downarrow 0$, a similar statement is true for $L_f$.
	\end{corollary}
	
	\begin{proof}[Proof of Theorem \textup{\ref{t:ainf}}]
	By Proposition \ref{p:metmeas}-\ref{it:conseq-PHI2-meas}, both $m_1$ and $m_2$ are minimal energy dominant measures. Therefore, $m_1$ and $m_2$ are mutually absolute continuous.
	By Proposition \ref{p:metmeas}-\ref{it:conseq-PHI2-met}, both $d_1$ and $d_2$ are bi-Lipschitz equivalent to intrinsic metrics, and therefore by Lemma \ref{l:bilip-ainf}, we may assume that $d_1$ and $d_2$ are intrinsic metrics with respect to the symmetric measures $m_1$ and $m_2$ respectively.
	
	Let $f \colon (X,d_1) \to (X,d_2)$ denote the identity map, which is an $\eta$-quasisymmetry.
	Then by the Lebesgue--Radon--Nikodym theorem, the volume derivative 
	\begin{equation} \label{e:ai1}
	\mu_f(x)= \lim_{r \downarrow 0} \frac{m_2\left(B_{d_1}(x,r)\right)}{m_1\left(B_{d_1}(x,r)\right)}
	\end{equation}
	exists and is finite for $m_1$-almost every $x \in X$. Since $m_2 \ll m_1$, we have
	$dm_2= \mu_f\,dm_1$; that is $m_2(E)= \int_E \mu_f \, dm_1$ for all measurable sets $E$.

	Since $(X,d_i,m_i,\sE,\sF \cap L^2(m_i))$ satisfies \hyperlink{phi}{$\on{PHI}(2)$} for $i=1,2$, there exists constants $A_1, A_2,C_1,C_2$ such that
	\begin{equation} \label{e:ai2}
	\Capa(B_{d_i}(x,r), B_{d_i}(x,A_ir)^c) \asymp \frac{r^2}{m_i\left(B_{d_i}(x,r)\right)}, 
	\end{equation}
for all $x \in X$, $r < \diam(X,d_i)/C_i,i=1,2.$ 
	Similar to \eqref{e:wk1}, there exists $C \ge 1$ such that
	for all $r < \diam(X,d_1)$, $x \in X$, $B_{d_1}(x,r)$ is a $C$-approximation of $B_{d_2}(x,L_f(x,r))$. That is, for all $r < \diam(X,d_1)$, $x \in X$, $B_{d_1}(x,r)$,
	\begin{align}\label{e:ai3}
	B_{d_2}(x,C^{-1}L_f(x,r)) &\subset   B_{d_1}(x,r)\subset  B_{d_2}(x, CL_f(x,r)), \nonumber \\  B_{d_1}(x,C^{-1}r)&\subset  B_{d_2}(x,L_f(x,r)) \subset  B_{d_1}(x,Cr).
	\end{align}
	By \eqref{e:ai3}, the $\eta$-quasisymmetry of $f$, and the same argument as Proposition  \ref{p:qscap}, there exists $C_3>1$ such that
	\begin{equation} \label{e:ai4}
	\Capa(B_{d_1}(x,r), B_{d_1}(x,A_1r)^c) \asymp \Capa(B_{d_2}(x,L_f(x,r)), B_{d_2}(x,A_2L_f(x,r))^c),
	\end{equation}
	for all $x \in X, r< \diam(X,d_1)/C_3)$.
	Combining \eqref{e:ai2} and \eqref{e:ai4} shows that
	\begin{equation} \label{e:ai4.5}
	\frac{L_f(x,r)^2}{r^2} \asymp \frac{m_2(B_{d_2}(x,L_f(x,r)))}{m_1(B_{d_1}(x,r))} \asymp \frac{m_2(B_{d_1}(x,r))}{m_1(B_{d_1}(x,r))}, 
	\end{equation}
	for all $x \in X, r<\diam(X,d_1)$.
	Therefore
	\begin{equation} \label{e:ai5}
	\mu_f(x) \asymp L_f(x)^2 \quad \mbox{for almost every $x \in X$.}
	\end{equation}
	Let $p \in [1,2)$ be the constant in Proposition \ref{p:self-i}-(b) so that
	$(X,d_1,m_1)$ satisfies $(1,p)$-Poincar\'e inequality.
	We shall show that $L_f$ satisfies the reverse H\"older inequality
	\begin{equation} \label{e:ai6}
	\left(\fint_{B_{d_1}(x_0,r)} L_f^2 \, dm_1\right)^{1/2} \le C \left(\fint_{B_{d_1}(x_0,r)} L_f^p\, dm_1\right)^{1/p},  
	\end{equation}
	for all $x_0 \in X, r < \diam(X,d_1)$.
	Then by Gehring's lemma \cite[Lemma 7.3]{HK}, H\"older inequality and \eqref{e:ai6}, we obtain the following reverse H\"older inequality for the function $\mu_f$: there exists $\varepsilon>0$ such that
	\[
	\left(\fint_{B_{d_1}(x_0,r)} \mu_f^{1+\varepsilon} \, dm_1\right)^{1/(1+\varepsilon)} \le C \fint_{B_{d_1}(x_0,r)} \mu_f\, dm_1, 
	\]
	for all $x_0 \in X, r < \diam(X,d_1)$.
	By the equivalence between reverse H\"older inequality and $A_\infty$-relation as explained in Definition \ref{d:ainf}, it suffices to show \eqref{e:ai6}.
	
	Since $(X,d_1,m_1)$ satisfies the $(1,p)$-Poincar\'e inequality, by Lemma \ref{l:wgrad} $CL_f^\varepsilon$ is an upper gradient of $u(x)= d_2(x,x_0)$ in $B_{d_1}(x_0,r)$.
	Therefore by the Poincar\'e inequality we have
	\[
	\fint_{B_{d_1}(x_0,K^{-1}r)} \abs{u-u_{B_{d_1}(x_0,K^{-1}r)}} \,dm_1 \lesssim r \left(\fint_{B_{d_1}(x_0,r)} (L_f^{\varepsilon})^p \,dm_1\right)^{1/p}.
	\]
	We let $\varepsilon \downarrow 0$ and use Corollary \ref{c:weakl2} and
	the dominated convergence theorem to obtain
	\begin{equation} \label{e:ai7}
	\fint_{B_{d_1}(x_0,K^{-1}r)} \abs{u-u_{B_{d_1}(x_0,K^{-1}r)}} \,dm_1 \lesssim r \left(\fint_{B_{d_1}(x_0,r)} \abs{L_f}^p \,dm_1\right)^{1/p}.
	\end{equation}
	By the uniform perfectness of $(X,d_1)$ and the volume doubling property, there exists $K_1$ such that $m_1\left(B(x_0,K^{-1}r) \setminus B(x_0,K_1^{-1}K^{-1}r)\right) \gtrsim m_1(B(x_0,r))$. Using the quasisymmetry of $f$, we obtain
	\begin{align*}
	u_{B_{d_1}(x_0,K^{-1}r)} &= \fint_{B_{d_1}(x_0,K^{-1}r)}  d_2(x,x_0) \, m_1(dx) \\
	&\ge \frac{1}{m_1(B_{d_1}(x_0,r))} \int_{B(x_0,K^{-1}r) \setminus B(x_0,K_1^{-1}K^{-1}r)}d_2(x,x_0) \, m_1(dx)\\
	& \gtrsim L_f(x_0,r) \frac{m_1\left(B(x_0,K^{-1}r) \setminus B(x_0,K_1^{-1}K^{-1}r)\right)}{ m_1(B(x_0,r))}\\
	& \ge C_1^{-1} L_f(x_0,r),
	\end{align*}
	because 
	\[
	L_f(x_0,r) \lesssim d_2(x,x_0) \quad\mbox{for all $x \in B(x_0,K^{-1}r) \setminus B(x_0,K_1^{-1}K^{-1}r)$}
	\]
	by the quasisymmetry of $f$ and the uniform perfectness of $(X,d_1)$.
	For sufficiently small $\delta>0$, we similarly have 
	\[
	u(x)=d_2(x,x_0) \le \eta(\delta K_2) L_f(x_0,r) \le (2C_1)^{-1}L_f(x_0,r)
	\]
	for all $x \in B(x_0,\delta K^{-1}r)$. Consequently, using the above estimates and the volume doubling property, we obtain
	\begin{align} \label{e:ai8}
	\fint_{B_{d_1}(x_0,K^{-1}r)} \abs{u-u_{B_{d_1}(x_0,K^{-1}r)}} \,dm_1 &\gtrsim \fint_{B_{d_1}(x_0,\delta K^{-1}r)} \abs{u-u_{B_{d_1}(x_0,K^{-1}r)}} \,dm_1 \nonumber \\
	&\gtrsim L_f(x_0,r).
	\end{align}
	It follows from the above estimates that
	\begin{align*}
	\left(\fint_{B_{d_1}(x_0,r)} L_f^2 \, dm_1\right)^{1/2}  &\lesssim \left(\fint_{B_{d_1}(x_0,r)} \mu_f \, dm_1\right)^{1/2} \quad \mbox{(by \eqref{e:ai5})}\\
	& \lesssim \left( \frac{m_2(B_{d_1}(x_0,r))}{m_2\left(B_{d_1}(x_0,r)\right)} \right)^{1/2} \quad \mbox{(since $dm_2=\mu_f\,dm_1$)}\\
	&\lesssim \frac{L_f(x_0,r)}{r} \quad \mbox{(by \eqref{e:ai4.5})}\\
	& \lesssim \left(\fint_{B_{d_1}(x_0,r)} \abs{L_f}^p \,dm_1\right)^{1/p}\quad \mbox{(by \eqref{e:ai8} and \eqref{e:ai7})}.
	\end{align*}
	This completes the proof of \eqref{e:ai6}, and therefore of Theorem \ref{t:ainf}.
	\end{proof}
	
	Let $(X,d,m,\sE,\sF)$ be an MMD space that satisfies \hyperlink{ehi}{$\on{EHI}$}, where $(X,d)$ is a doubling metric space.
	If $\mu \in \sG(X,d,m,\sE,\sF)$, then by Theorem \ref{t:ainf}
	\begin{equation} \label{e:ainf}
	\sG(X,d,m,\sE,\sF) \subseteq \giveset{ \widetilde{\mu} : \mbox{$\widetilde{\mu}$ is $A_\infty$-related to $\mu$} }.
	\end{equation}
	One might ask if the inclusion in \eqref{e:ainf} is strict. For the Brownian motion on $\bR^n$, the above inclusion is strict if and only if $n \ge 2$ (see Theorem \ref{t:1dgaussian} and Example \ref{x:ainf}).
	We need the definition of a maximal semi-metric.
	\begin{definition} \label{d:maxmetric}
		A function $r \colon X \times X \to [0,\infty)$ is said to be a \emph{semi-metric}, if it satisfies all the properties of a metric except possibly the property that $r(x,y)=0$ implies $x=y$.	
		
		Let $h \colon X \times X \to [0,\infty)$ be an arbitrary function.
		Then there exists a unique maximal semi-metric $d_h \colon X \times X \to [0,\infty)$ such that $d_h(x,y)\le h(x,y)$ for all $x,y \in X$ \cite[Lemma 3.1.23]{BBI}.
		We call $d_h$ the \emph{maximal semi-metric induced by $h$}.
		Equivalently, $d_h$ can be defined as follows. Let $\widetilde{h}(x,y)= \min(h(x,y),h(y,x))$. Then 
		\begin{equation} \label{e:maxmetric}
		d_h(x,y) = \inf \giveset{ \sum_{i=0}^{N-1} \widetilde{h}(x_i,x_{i+1}): N \in \bN, x_0=x, x_N=y}.
		\end{equation}
	\end{definition}
	
	We provide a necessary condition for a measure to be in $\sG(X,d,m,\sE,\sF)$.
	Using this necessary condition, below we obtain examples for which the inclusion \eqref{e:ainf} is strict.
	
	\begin{lemma} \label{l:neccessary}
		Let $(X,d,m,\sE,\sF)$ satisfy \hyperlink{phi}{$\on{PHI(\gamma)}$} for some $\gamma \ge 2$. Let $\mu \in \sG(X,d,m,\sE,\sF)$. Define 
		\begin{equation} \label{e:defh}
		h(x,y) = \sqrt{\frac{\mu(B_d(x,d(x,y))) d(x,y)^\gamma}{m(B_d(x,d(x,y)))}}, \quad \mbox{for any $x, y \in X$ with $x \neq y$,}
		\end{equation}
		and $h(x,x)=0$ for any $x \in X$. Let $d_h$ denote the maximal semi-metric induced by $h$. Then there exists $C>0$ such that
		\[
		h(x,y) \le C d_h(x,y) \quad \mbox{for all $x,y \in X$}.
		\]
	\end{lemma}
	
	\begin{proof}
	Let $\theta \in \sJ(X,d)$ be such that the MMD space $(X,\theta,\mu,\sE^\mu,\sF^\mu)$ satisfies \hyperlink{phi}{$\on{PHI(2)}$}.
	It suffices to show the existence of $C_1>0$ such that
	\begin{equation} \label{e:ns1}
	C_1^{-1} \theta(x,y) \le	h(x,y) \le C_1 \theta(x,y) \quad \mbox{for all $x,y \in X$}.
	\end{equation}
	In particular, due to the triangle inequality for $\theta$,
	\eqref{e:ns1} implies a similar inequality with $h(x,y)$ replaced by $d_h(x,y)$,
	which immediately implies that $h$ is comparable to $d_h$.

	By Theorem \ref{t:phichar}, $m$ and $\mu$ satisfy \ref{VD} on $(X,d)$ and $(X,\theta)$.
	By using Proposition \ref{p:QS} and \ref{VD}, there exists $C_2>0$ such that
	\begin{equation} \label{e:ns2}
	C_2^{-1} \mu(B_\theta(x,\theta(x,y))) \le \mu(B_d(x,d(x,y))) \le  C_2 \mu(B_\theta(x,\theta(x,y)))
	\end{equation}
	for all $x,y \in X$, where $x \neq y$.
	By an argument similar to the proof of Proposition \ref{p:qscap}
	using Lemma \ref{l:cap} and Proposition \ref{p:QS}, there exist $C_3,A>0$ such that
	\begin{equation} \label{e:ns3}
	C_3^{-1 }\le \frac{\Capa(B_\theta(x,\theta(x,y)), B_\theta(x,2\theta(x,y))^c)}{\Capa(B_d(x,d(x,y)), B_d(x,2d(x,y))^c)} \le C_3, 
	\end{equation}
	for any pair $x,y \in X$ such that $0< \theta(x,y) < \diam(X,\theta)/A$.
	By Theorem \ref{t:phichar}, Lemma \ref{l:cap} and Proposition \ref{p:QS}
	and by increasing $A$ if necessary, there exists $C_4>0$ such that
	\begin{align} \label{e:ns4}
	\Capa(B_\theta(x,\theta(x,y)), B_\theta(x,2\theta(x,y))^c) & \ge C_4^{-1} \frac{\mu(B_\theta(x,\theta(x,y)))}{\theta(x,y)^2}, \nonumber \\
	\Capa(B_\theta(x,\theta(x,y)), B_\theta(x,2\theta(x,y))^c) &\le C_4  \frac{\mu(B_\theta(x,\theta(x,y)))}{\theta(x,y)^2},
	\end{align}
	and
	\begin{align} \label{e:ns5}
	\Capa(B_d(x,d(x,y)), B_d(x,2d(x,y))^c) &\ge C_4^{-1} \frac{m(B_d(x,d(x,y)))}{d(x,y)^\gamma}, \nonumber \\
	\Capa(B_d(x,d(x,y)), B_d(x,2d(x,y))^c) &\le C_4  \frac{m(B_d(x,d(x,y)))}{d(x,y)^\gamma},
	\end{align}
	for any  pair $x,y \in X$ such that $0< \theta(x,y) < \diam(X,\theta)/A$.
	Combining \eqref{e:ns2}, \eqref{e:ns3}, \eqref{e:ns4} and \eqref{e:ns5}
	shows that there exists $C_5>0$ such that
	\begin{equation} 
	\label{e:ns6}
	C_5^{-1} \theta(x,y) \le	h(x,y) \le C_5 \theta(x,y),
	\end{equation}
	for all $x,y \in X$ such that $0 \le \theta(x,y) < \diam(X,\theta)/A$.
	Since $(X,\theta)$ is uniformly perfect, by replacing $y$ with a closer point $\widetilde{y}$,
	and using \eqref{e:ns6} and \ref{VD}, we obtain \eqref{e:ns1}.
	\end{proof}

	\begin{example} \label{x:ainf}
			Let $n \ge 2$ and let $(X,d,m,\sE,\sF)$ denote the Dirichlet form corresponding to Brownian motion on $\bR^n$. If $w(x)=\abs{x_1}^t$, where $t \in \bR, x=(x_1,\ldots,x_n)$, then $w\,dm$ is $A_\infty$-related to $m$ if and only if $t > -1$ \cite[p.\ 222, Example (c)]{Sem93}. 
			If $t>0, x=(0,\ldots,0), y= (0,\ldots,0,1)$ and $h$ is as given in \eqref{e:defh} with $\gamma=2$, then using \eqref{e:maxmetric} it is easy to check that $d_h(x,y)=0$. 
			This can be seen by choosing equally spaced points $x_0,\ldots,x_n$ on the straight line joining $x$ and $y$ and letting $n \to \infty$ in \eqref{e:maxmetric}. Therefore, by Lemma \ref{l:neccessary} we obtain
			\[
			w\,dm  \in \giveset{ \mu: \mbox{$\mu$ is $A_\infty$-related to $m$} }\setminus \sG(X,d,m,\sE,\sF), \quad \mbox{for any $n \ge 2, t>0$.}
			\]
			In other words, the inclusion in \eqref{e:ainf} is strict for the Brownian motion on $\bR^n, n \ge 2$. 
	\end{example}
	The above example and Lemma \ref{l:neccessary} illustrate that if a measure is too small in the neighborhood of a curve,
	then it will fail to be in $\sG(X,d,m,\sE,\sF)$. As we will see in Subsubsection \ref{sssec:SGs} below,
	a similar (but more subtle) phenomenon happens in the higher-dimensional Sierpi\'{n}ski gaskets.
	
	We recall the definition of strong $A_\infty$-weights on $\bR^n$ introduced by David and Semmes in \cite{DS} and show its relevance to the Gaussian uniformization problem for the Brownian motion on $\bR^2$.
	The following definition is a slight reformulation of the one in \cite{DS} and the equivalence between the two definitions follows from \cite[Lemma 3.1]{Sem93}.
	\begin{definition} \label{d:strongainf}
			Let $d,m$ denote the Euclidean metric and Lebesgue measure on $\bR^n$, respectively. Let $\mu= w \,dm$ be $A_\infty$-related to $m$.
			Define 
			\[
			h(x,y)= \left( \mu(B_{x,y})\right)^{1/n},
			\]
			where $B_{x,y}$ is the Euclidean ball with center $z=(x+y)/2$ and radius $d(x,y)/2$.
			Let $d_h$ denote the maximal semi-metric induced by $h$.
			We say that $\mu$ is \emph{strong $A_\infty$-related} to $m$, if there exists $C>0$ such that
			\[
			d_h(x,y) \ge C^{-1} h(x,y), \quad \mbox{for all $x,y \in \bR^n$.}
			\]
	\end{definition}
	The following relates the Gaussian uniformization problem in $\bR^2$ in terms of strong $A_\infty$-weights.
	\begin{proposition} \label{p:strongainf}
		Let $(X,d,m,\sE,\sF)$ denote the MMD space corresponding to the Brownian motion on $\bR^2$. Then
		\begin{equation} \label{e:strongainf}
		\sG(X,d,m,\sE,\sF) = \giveset{ \mu : \mbox{$\mu$ is strong $A_\infty$-related to $m$} }.
		\end{equation}
	\end{proposition}
	
	\begin{proof}
	By Lemma \ref{l:neccessary} and Theorem \ref{t:ainf}, we have the inclusion
	\begin{equation} \label{e:sai1}
	\sG(X,d,m,\sE,\sF) \subseteq \giveset{ \mu : \mbox{$\mu$ is strong $A_\infty$-related to $m$} }.
	\end{equation}
	For the reverse inclusion, consider a measure $\mu=w\,dm$ that is $A_\infty$-related to $m$. Since $\mu$ is mutually absolutely continuous with respect to $m$, $\mu \in \sA(X,d,m,\sE,\sF)$.
	
	Let $\theta$ denote the metric $d_h$ in Definition \ref{d:strongainf}. 
	Since $\mu$ is $A_\infty$-related to $m$, $\mu$ satisfies \ref{VD} on $(X,d)$.  Since $\theta(x,y)$ is comparable to $\sqrt{\mu(B_{x,y})}$, where $B_{x,y}= B((x+y)/2,d(x,y)/2))$, by the \ref{VD} and \ref{RVD} for the measure $\mu$, we obtain that
	\begin{equation} \label{e:sai2}
	\theta \in \sJ(X,d).
	\end{equation}
	By \eqref{e:sai2}, Proposition \ref{p:QS} and \ref{VD}, there exists $C_1>0$ such that
	\begin{equation} \label{e:sai3}
	C_1^{-1} r^2 \le \mu(B_\theta(x,r)) \le C_1 r^2 \quad \mbox{for all $x \in X, r>0$.}
	\end{equation}
	By Lemma \ref{l:cap}, Proposition \ref{p:QS}, and \eqref{e:sai2}, there exist $C_2>0$ such that
	\begin{equation} \label{e:sai4}
	C_2^{-1} \le \Capa(B_\theta(x,r),B_\theta(x,2r)^c) \le C_2 \quad \mbox{for all $x \in X,r>0$.}
	\end{equation}
	By Lemma \ref{l:EHI-TC-QS}, the time-changed MMD space $(X,\theta,\mu,\sE^\mu,\sF^\mu)$ satisfies \hyperlink{ehi}{$\on{EHI}$}.
	Combining \eqref{e:sai2}, \eqref{e:sai3} and \eqref{e:sai4}, and using Theorem \ref{t:phichar}, we obtain that $\mu \in \sG(X,d,m,\sE,\sF)$.
	\end{proof}
	
	Proposition \ref{p:strongainf} along with known results on strong $A_\infty$-related measures
	leads to many further examples of measures in $\sG(X,d,m,\sE,\sF)$ for Brownian motion in $\bR^2$. For instance, Bessel potentials can be used to construct strong $A_\infty$-measures \cite[Theorem 3.1]{BHS}.
	
	Unlike the case of Brownian motion on $\mathbb{R}^{n}$ with $n \geq 2$ treated so far,
	it turns out that the inclusion in \eqref{e:ainf} is an equality for Brownian motion on $\mathbb{R}$,
	which we prove in the rest of this subsection as a complete answer to the Gaussian uniformization problem for Brownian motion on $\mathbb{R}$.
	For this purpose, we need the following lemma characterizing $A_\infty$-related weights in terms of a reverse H\"older inequality.
	
	\begin{lemma} \label{l:ainf}
		Let $(X,d)$ be a locally compact metric space and let $m$ be a non-zero Radon measure on $(X,d)$ satisfying
		\ref{VD} and such that the function $(0,\infty) \ni r \mapsto m(B(x,r))$ is continuous for each $x \in X$.
		Then for each $[0,\infty)$-valued $w\in L^{1}_{\on{loc}}(X,m)$, the following are equivalent:
		\begin{enumerate}[label=\textup{(\alph*)},align=left,leftmargin=*,topsep=5pt,parsep=0pt,itemsep=2pt]
			\item\label{it:ainf-ainf}$\mu= w\,dm$ is $A_\infty$-related to $m$.
			\item\label{it:ainf-rhol}$w \not\equiv 0$ and there exists $C>1$ such that the following reverse H\"older inequality holds:
			\[
			\fint_{B} w \, dm \le C \left(\fint_{B} \sqrt{w} \, dm \right)^2 \qquad \mbox{whenever $B$ is a ball in $(X,d)$.}
			\]
		\end{enumerate}
	\end{lemma}
	
	\begin{proof}
	\ref{it:ainf-rhol}$\implies$\ref{it:ainf-ainf}:
	By Gehring's lemma \cite[Lemma 7.3]{HK}, there exist $\varepsilon>0$ and $C_1>0$ such that
	\[
	\left(	\fint_{B} w^{1+\varepsilon} \, dm\right)^{1/(1+\varepsilon)} \le C_1  \left(\fint_{B} \sqrt{w} \, dm \right)^2 \le \fint_B w\, dm
	\]
	for all balls $B$ in $(X,d)$. By \cite[Theorem 18 in Chapter I]{ST}, $\mu$ is $A_\infty$-related to $m$.
	
	\noindent \ref{it:ainf-ainf}$\implies$\ref{it:ainf-rhol}: By  \cite[Theorem 18 in Chapter I]{ST}, there exist $r>1$ and $C_2>1$ such that
	\begin{equation} \label{e:ainf1}
	\left( \fint_{B} w^r \, dm\right)^{1/r} \le C_2 \fint_{B} w\, dm \qquad \mbox{for all balls $B$.}
	\end{equation}
	Choose $\theta \in (0,1)$ such that $\theta r^{-1} + 2(1-\theta)=1$. By H\"older's inequality and \eqref{e:ainf1},
	\begin{align*}
	\fint_B w\, dm &\le \left( \fint_{B} w^r \, dm\right)^{\theta/r} \left( \fint_{B} \sqrt{w} \, dm\right)^{2(1-\theta)}\\ &\le C_2^\theta \left( \fint_{B} w \, dm\right)^{\theta} \left( \fint_{B} \sqrt{w} \, dm\right)^{2(1-\theta)}
	\end{align*}
	for all balls $B$. This immediately implies \ref{it:ainf-rhol} with $C=C_2^{\theta/(1-\theta)}$.
	\end{proof}

	In the following result, we consider the case of Brownian motion in $\bR$; that is $(X,d,m,\sE,\sF)$ is given by $X=\bR$, $d$ is the Euclidean distance, $m$ is the Lebesgue measure, $\sF=W^{1,2}$ and $\sE(f,f)= \int \abs{f'}^2\,dm$.
	\begin{theorem} \label{t:1dgaussian}
		Let $(X,d,m,\sE,\sF)$ denote the MMD space corresponding to the Brownian motion on $\bR$.
		Then the family of Gaussian admissible measures is characterized by the reverse H\"older inequality
		as in Lemma \textup{\ref{l:ainf}-\ref{it:ainf-rhol}}, i.e.,
		\begin{align*}
		&\sG(X,d,m,\sE,\sF) \\
			&= \{ \mu \mid \mbox{$\mu$ is $A_\infty$-related to $m$} \} \\
			&= \{ w\,dm \mid \mbox{$w \in L^{1}_{\on{loc}}(X,m)$, $w$ is $[0,\infty)$-valued and satisfies Lemma \textup{\ref{l:ainf}-\ref{it:ainf-rhol}}} \}.
		\end{align*}
	\end{theorem}
	
	\begin{proof}
	Set
	\[
	\widetilde{\sG}= \biggl\{  g \,dm \biggm\vert
	\begin{minipage}{185pt}
	$g \not\equiv 0$, there exists $C \ge 1$ such that for all $a < b$,
	$(b-a)^{1/2} \bigl( \int_a^b g \, dm \bigr)^{1/2} \le C \int_a^b \sqrt{g}\,dm$
	\end{minipage}\biggr\}.
	\]
	By Theorem \ref{t:ainf} and Lemma \ref{l:ainf}, it suffices to show that $\widetilde{\sG}\subset {\sG}(X,d,m,\sE,\sF)$.
	Let $g \,dm \in \widetilde{\sG}$. Then consider the measures $\mu_1 = \sqrt{g}\,dm$ and $\mu= g\,dm$. For all $a < b$, an easy calculation shows that
	\begin{equation} \label{e:gup0}
	d_{\on{int}}^\mu(a,b)= \mu_1([a,b]) = \int_a^b \sqrt{g}\, dm.
	\end{equation}
	Since $\sqrt{g}$ satisfies a reverse H\"older inequality, by \cite[Lemma 12 and Theorem 17 in Chapter I]{ST}, $\mu_1$ is a doubling measure on $(X,d)$.
	By the correspondence between doubling measures and quasisymmetric maps on $\bR$ described in \cite[Remark 13.20-(b)]{Hei} and \eqref{e:gup0}, we have 
	\begin{equation} \label{e:gup1}
	d_{\on{int}}^\mu \in \sJ(X,d).
	\end{equation}
	By the reverse H\"older inequality assumption on $\sqrt{g}$, we have
	\[
	\mu_1([a,b]) \le (b-a)^{1/2} \left( \mu([a,b])\right)^{1/2} \le C \mu_1([a,b]).
	\]
	Since $\mu_1$ is a doubling measure on $(X,d)$, the above estimate shows that $\mu$ is also a doubling measure on $(X,d)$.
	Since $d_{\on{int}}^\mu \in \sJ(X,d)$, $\mu$ is a doubling measure on $(X,d_{\on{int}}^\mu)$ by \eqref{e:ann1}.
	
	By \cite[Theorem 18 of Chapter I]{ST} and Lemma \ref{l:ainf}, the measures $\mu$ and $m$ are mutually absolutely continuous. This implies that $\mu \in \sA(X,d,m,\sE,\sF)$.
	
	By \cite[Theorem 1.4]{CW}, for any interval $I=[a,b]$ and for all $f \in W^{1,2}$, we have a Poincar\'e inequality
	\[
	\int_I \left(f(x)- \frac{1}{\mu(I)} \int_I f\,d\mu \right)^2 \, \mu(dx)  \le K_{\mu,I}^2 \int_I \abs{f'(x)}^2 \, dm(x),
	\]
	where the optimal constant $K_{\mu,I}$ satisfies the two-sided estimate
	\begin{align}
	K_{\mu,I} &\asymp \frac{1}{\mu(I)} \left( \sup_{a < x< b}  \Biggl\{ \mu([x,b])^{1/2} \left(\int_a^x \mu([a,t])^2  \, dt\right)^{1/2} \Biggr\} \right.  \nonumber \\
	& \qquad \qquad \left. + \sup_{a < x< b}  \Biggl\{ \mu([a,x])^{1/2} \left(\int_x^b \mu([t,b])^2  \, dt\right)^{1/2} \Biggr\} \right). \label{e:gup2}
	\end{align}
	By using the bound $\mu(A) \le \mu(I)$ for all $A \subset I$, \eqref{e:gup2} and reverse H\"older inequality, we have
	\begin{equation} \label{e:gup3}
	K_{\mu,I}^2 \lesssim \mu(I) m(I)  \lesssim  \mu_1(I)^2.
	\end{equation}
	By \eqref{e:gup0}, $\mu_1(I)$ is the diameter of $I$ under the intrinsic metric $d_{\on{int}}^\mu$. Therefore by \eqref{e:gup2} and \eqref{e:gup3},
	we have the Poincar\'e inequality \hyperlink{pi}{$\on{PI}(2)$} for the MMD space 
	$(X,d_{\on{int}}^\mu,\mu,\sE^\mu,\sF^\mu)$. By \cite[Lemma 1]{Stu94}, the MMD space 	$(X,d_{\on{int}}^\mu,\mu,\sE^\mu,\sF^\mu)$ satisfies  \hyperlink{cs}{$\on{CS}(2)$}. 
	Since $\mu$ is a doubling measure on $(X,d_{\on{int}}^\mu)$, we have that $\mu \in \sG(X,d,m,\sE,\sF)$ by Theorem \ref{t:phichar}. Alternately, the claim that $\mu \in \sG(X,d,m,\sE,\sF)$ follows from \cite[Theorem 3.5]{Stu96}.
	\end{proof}
	
	\begin{remark}
			A major obstruction to determining the Gaussian admissible measures
			for multidimensional Brownian motion in $\bR^n, n\ge 2$ is that the intrinsic metric
			with respect to $\mu = g\,dm$ does not admit a simple description unlike the
			one-dimensional case where there is a simple formula \eqref{e:gup0}. As noted in
			Example \ref{x:ainf}, the conclusion of Theorem \ref{t:1dgaussian} fails in higher dimensions.
	\end{remark}

\section{The attainment problem for self-similar sets}\label{sec:pfcsss}

In this section, we study the attainment problem, that of whether the infimum in
\eqref{e:dcw} defining the conformal walk dimension $\dcw=2$ is attained,
in the case of a self-similar Dirichlet form
$(\mathcal{E},\mathcal{F})$ on a post-critically finite self-similar set $K$.
After introducing the framework of such a Dirichlet form
in Subsection \ref{ssec:pfcsss-preliminaries}, we prove in
Subsection \ref{ssec:attain-harmonic-func-pcf} that $\dcw=2$ is attained (if and)
only if $(K,\theta_{h},\Gamma(h,h),\mathcal{E},\mathcal{F})$ satisfies \hyperlink{phi}{$\on{PHI}(2)$}
for some \emph{harmonic function} $h\in\mathcal{F}$ and a metric $\theta_{h}$ on $K$
quasisymmetric to the resistance metric $R_{\mathcal{E}}$ of $(\mathcal{E},\mathcal{F})$,
where $\Gamma(h,h)$ denotes the energy measure of $h$ associated with $(\mathcal{E},\mathcal{F})$.
Then in Subsection \ref{ssec:examples} we present several examples, all of which are
shown NOT to attain $\dcw=2$ except for the two-dimensional standard Sierpi\'{n}ski
gasket, which is known to attain $\dcw=2$ by the results in \cite{Kig08,Kaj12}
as discussed in Theorem \ref{thm:SG2-attained} and its proof below.

The restriction of the framework to post-critically finite self-similar sets is mainly for
the sake of simplicity. In fact, all the results in Subsection \ref{ssec:attain-harmonic-func-pcf}
can be verified, with just slight modifications in the proofs, also for the canonical
self-similar Dirichlet form on any \emph{generalized Sierpi\'{n}ski carpet}
introduced in Subsection \ref{ssec:attain-harmonic-func-GSCs}, which forms
essentially the only class of examples of infinitely ramified self-similar fractals
where the theory of a canonical self-similar Dirichlet form has been established.
We treat the case of generalized Sierpi\'{n}ski carpets in
Subsection \ref{ssec:attain-harmonic-func-GSCs} and explain what changes
are needed for the arguments in Subsections \ref{ssec:pfcsss-preliminaries}
and \ref{ssec:attain-harmonic-func-pcf} to go through in this case.

\subsection{Preliminaries}\label{ssec:pfcsss-preliminaries}

In this subsection, we first introduce our framework of a post-critically finite self-similar
set and a self-similar Dirichlet form on it, for which we mainly follow the presentation of
\cite[Section 3]{Kaj14}, and then present preliminary facts.

Let us start with the standard notions concerning self-similar sets.
We refer to \cite[Chapter 1]{Kig01} for details.
Throughout this subsection, we fix a compact metrizable topological space $K$,
a finite set $S$ with $\#S\geq 2$ and a continuous injective map
$F_{i}\colon K\to K$ for each $i\in S$. We set $\mathcal{L}:=(K,S,\{F_{i}\}_{i\in S})$.
\begin{definition}\label{d:shift}
\begin{enumerate}[label=\textup{(\arabic*)},align=left,leftmargin=*,topsep=5pt,parsep=0pt,itemsep=2pt]
\item Let $W_{0}:=\{\emptyset\}$, where $\emptyset$ is an element
	called the \emph{empty word}, let
	$W_{n}:=S^{n}=\{w_{1}\dots w_{n}\mid w_{i}\in S\textrm{ for }i\in\{1,\dots,n\}\}$
	for $n\in\mathbb{N}$ and let $W_{*}:=\bigcup_{n\in\mathbb{N}\cup\{0\}}W_{n}$.
	For $w\in W_{*}$, the unique $n\in\mathbb{N}\cup\{0\}$ with $w\in W_{n}$
	is denoted by $\lvert w\rvert$ and called the \emph{length of $w$}.
	For $i\in S$ and $n\in\mathbb{N}\cup\{0\}$ we write $i^{n}:=i\dots i\in W_{n}$.
\item We set
	$\Sigma:=S^{\mathbb{N}}=\{\omega_{1}\omega_{2}\omega_{3}\ldots\mid \omega_{i}\in S\textrm{ for }i\in\mathbb{N}\}$,
	which is always equipped with the product topology of the discrete topology on $S$,
	and define the \emph{shift map} $\sigma\colon\Sigma\to\Sigma$ by
	$\sigma(\omega_{1}\omega_{2}\omega_{3}\dots):=\omega_{2}\omega_{3}\omega_{4}\dots$.
	For $i\in S$ we define $\sigma_{i}\colon\Sigma\to\Sigma$  by
	$\sigma_{i}(\omega_{1}\omega_{2}\omega_{3}\dots):=i\omega_{1}\omega_{2}\omega_{3}\dots$.
	For $\omega=\omega_{1}\omega_{2}\omega_{3}\ldots\in\Sigma$ and
	$n\in\mathbb{N}\cup\{0\}$, we write $[\omega]_{n}:=\omega_{1}\dots\omega_{n}\in W_{n}$.
\item For $w=w_{1}\dots w_{n}\in W_{*}$, we set
	$F_{w}:=F_{w_{1}}\circ\dots\circ F_{w_{n}}$ ($F_{\emptyset}:=\id{K}$),
	$K_{w}:=F_{w}(K)$, $\sigma_{w}:=\sigma_{w_{1}}\circ\dots\circ \sigma_{w_{n}}$
	($\sigma_{\emptyset}:=\id{\Sigma}$) and $\Sigma_{w}:=\sigma_{w}(\Sigma)$,
	and if $w\not=\emptyset$ then $w^{\infty}\in\Sigma$ is defined by
	$w^{\infty}:=www\dots$ in the natural manner.
\end{enumerate}
\end{definition}
\begin{definition}\label{d:sss}
$\mathcal{L}=(K,S,\{F_{i}\}_{i\in S})$ is called a \emph{self-similar structure}
if and only if there exists a continuous surjective map $\pi\colon\Sigma\to K$ such that
$F_{i}\circ\pi=\pi\circ\sigma_{i}$ for any $i\in S$.
Note that such $\pi$, if it exists, is unique and satisfies
$\{\pi(\omega)\}=\bigcap_{n\in\mathbb{N}}K_{[\omega]_{n}}$ for any $\omega\in\Sigma$.
\end{definition}
In the rest of this subsection we always assume that $\mathcal{L}$ is a self-similar structure.
\begin{definition}\label{d:V0Vstar}
\begin{enumerate}[label=\textup{(\arabic*)},align=left,leftmargin=*,topsep=5pt,parsep=0pt,itemsep=2pt]
\item We define the \emph{critical set} $\mathcal{C}_{\mathcal{L}}$ and the
	\emph{post-critical set} $\mathcal{P}_{\mathcal{L}}$ of $\mathcal{L}$ by
	\begin{equation}\label{e:C-P}\textstyle
		\mathcal{C}_{\mathcal{L}}:=\pi^{-1}\bigl(\bigcup_{i,j\in S,\,i\not=j}K_{i}\cap K_{j}\bigr)
		\qquad\textrm{and}\qquad
		\mathcal{P}_{\mathcal{L}}:=\bigcup_{n\in\mathbb{N}}\sigma^{n}(\mathcal{C}_{\mathcal{L}}).
	\end{equation}
	$\mathcal{L}$ is called \emph{post-critically finite}, or \emph{p.-c.f.}\ for short,
	if and only if $\mathcal{P}_{\mathcal{L}}$ is a finite set.
\item We set $V_{0}:=\pi(\mathcal{P}_{\mathcal{L}})$, $V_{n}:=\bigcup_{w\in W_{n}}F_{w}(V_{0})$
	for $n\in\mathbb{N}$ and $V_{*}:=\bigcup_{n\in\mathbb{N}\cup\{0\}}V_{n}$.
\end{enumerate}
\end{definition}
$V_{0}$ should be considered as the \emph{``boundary"} of the self-similar set $K$;
indeed, $K_{w}\cap K_{v}=F_{w}(V_{0})\cap F_{v}(V_{0})$ for any $w,v\in W_{*}$ with
$\Sigma_{w}\cap\Sigma_{v}=\emptyset$ by \cite[Proposition 1.3.5-(2)]{Kig01}.
According to \cite[Lemma 1.3.11]{Kig01}, $V_{n-1}\subset V_{n}$ for any $n\in\mathbb{N}$,
and if $V_{0}\not=\emptyset$ then $V_{*}$ is dense in $K$.
Also note that by \cite[Theorem 1.6.2]{Kig01}, $K$ is connected if and only if
for any $i,j\in S$ there exist $n\in\mathbb{N}$ and $\{i_{k}\}_{k=0}^{n}\subset S$
with $i_{0}=i$ and $i_{n}=j$ such that $K_{i_{k-1}}\cap K_{i_{k}}\not=\emptyset$
for any $k\in\{1,\dots,n\}$, and if $K$ is connected then it is arcwise connected.

In the remainder of this subsection our self-similar structure
$\mathcal{L}=(K,S,\{F_{i}\}_{i\in S})$ is always assumed to be post-critically finite
with $K$ connected, so that $2\leq\#V_{0}<\infty$, $K\not=\overline{V_{0}}=V_{0}$
and $V_{*}$ is countably infinite and dense in $K$.

Next we briefly recall the construction and basic properties of
a self-similar Dirichlet form on such $\mathcal{L}$;
see \cite[Chapter 3]{Kig01} for details.
Let $D=(D_{pq})_{p,q\in V_{0}}$ be a real symmetric matrix of size $\#V_{0}$
(which we also regard as a linear operator on $\mathbb{R}^{V_{0}}$) such that
\begin{enumerate}[label=\textup{(D\arabic*)},align=left,leftmargin=*,topsep=5pt,parsep=0pt,itemsep=2pt]
	\item \label{it:D1}$\{u\in\mathbb{R}^{V_{0}}\mid Du=0\}=\mathbb{R}\one_{V_{0}}$,
	\item \label{it:D2}$D_{pq}\geq 0$ for any $p,q\in V_{0}$ with $p\not=q$.
\end{enumerate}
We define
\begin{equation}\label{e:E0-pcf}
\mathcal{E}^{(0)}(u,v):=-\sum_{p,q\in V_{0}}D_{pq}u(q)v(p)
	=\frac{1}{2}\sum_{p,q\in V_{0}}D_{pq}(u(p)-u(q))(v(p)-v(q))
\end{equation}
for $u,v\in\mathbb{R}^{V_{0}}$, so that
$(\mathcal{E}^{(0)},\mathbb{R}^{V_{0}})$ is a Dirichlet form on $L^{2}(V_{0},\#)$.
Furthermore let $\mathbf{r}=(r_{i})_{i\in S}\in(0,\infty)^{S}$ and define
\begin{equation}\label{e:Em-pcf}
\mathcal{E}^{(n)}(u,v):=\sum_{w\in W_{n}}\frac{1}{r_{w}}\mathcal{E}^{(0)}(u\circ F_{w}\vert_{V_{0}},v\circ F_{w}\vert_{V_{0}}),
	\quad u,v\in\mathbb{R}^{V_{n}}
\end{equation}
for each $n\in\mathbb{N}$, where $r_{w}:=r_{w_{1}}r_{w_{2}}\dots r_{w_{n}}$
for $w=w_{1}w_{2}\dots w_{n}\in W_{n}$ ($r_{\emptyset}:=1$).
\begin{definition}\label{d:harmonic-str}
The pair $(D,\mathbf{r})$ of a real symmetric matrix $D=(D_{pq})_{p,q\in V_{0}}$
of size $\#V_{0}$ with the properties \ref{it:D1} and \ref{it:D2} and
$\mathbf{r}=(r_{i})_{i\in S}\in(0,\infty)^{S}$ is called a
\emph{harmonic structure} on $\mathcal{L}$ if and only if
$\mathcal{E}^{(0)}(u,u)=\inf_{v\in\mathbb{R}^{V_{1}},\,v\vert_{V_{0}}=u}\mathcal{E}^{(1)}(v,v)$
for any $u\in\mathbb{R}^{V_{0}}$; note that then
\begin{equation}\label{e:harmonic-str}
\mathcal{E}^{(n_{1})}(u,u)=\min_{v\in\mathbb{R}^{V_{n_{2}}},\,v\vert_{V_{n_{1}}}=u}\mathcal{E}^{(n_{2})}(v,v)
\end{equation}
for any $n_{1},n_{2}\in\mathbb{N}\cup\{0\}$ with $n_{1}\leq n_{2}$ and any $u\in\mathbb{R}^{V_{n_{1}}}$
by \cite[Proposition 3.1.3]{Kig01}. If $\mathbf{r}\in(0,1)^{S}$ in addition,
then $(D,\mathbf{r})$ is called \emph{regular}.
\end{definition}
In the rest of this subsection, we assume that $(D,\mathbf{r})$ is a regular
harmonic structure on $\mathcal{L}$. In this case,
$\{\mathcal{E}^{(n)}(u\vert_{V_{n}},u\vert_{V_{n}})\}_{n\in\mathbb{N}\cup\{0\}}$
is non-decreasing and hence has the limit in $[0,\infty]$ for any $u\in \contfunc (K)$.
Then we define a linear subspace $\mathcal{F}$ of $\contfunc (K)$ and a non-negative definite
symmetric bilinear form $\mathcal{E}\colon\mathcal{F}\times\mathcal{F}\to\mathbb{R}$ by
\begin{align}\label{e:DF-domain}
\mathcal{F}&:=\{u\in \contfunc (K)\mid\lim\nolimits_{n\to\infty}\mathcal{E}^{(n)}(u\vert_{V_{n}},u\vert_{V_{n}})<\infty\},\\
\mathcal{E}(u,v)&:=\lim\nolimits_{n\to\infty}\mathcal{E}^{(n)}(u\vert_{V_{n}},v\vert_{V_{n}})\in\mathbb{R},
	\quad u,v\in\mathcal{F},
\label{e:DF-form}
\end{align}
so that $(\mathcal{E},\mathcal{F})$ is easily seen to possess the following self-similarity
properties (note that $\mathcal{F}\cap \contfunc (K)=\mathcal{F}$ in the present setting):
\begin{align}\label{e:SSDF-domain}
\mathcal{F}\cap \contfunc (K)&=\{u\in \contfunc (K)\mid\textrm{$u\circ F_{i}\in\mathcal{F}$ for any $i\in S$}\},\\
\mathcal{E}(u,v)&=\sum_{i\in S}\frac{1}{r_{i}}\mathcal{E}(u\circ F_{i},v\circ F_{i}),
	\quad u,v\in\mathcal{F}\cap \contfunc (K).
\label{e:SSDF-form}
\end{align}
By \cite[Proposition 2.2.4, Lemma 2.2.5, Theorem 2.2.6, Lemma 2.3.9, Theorems 2.3.10 and 3.3.4]{Kig01},
$(\mathcal{E},\mathcal{F})$ is a resistance form on $K$ and its resistance metric
$R_{\mathcal{E}}\colon K\times K\to[0,\infty)$ is a metric on $K$ compatible with the original
topology of $K$; here $(\mathcal{E},\mathcal{F})$ being a \emph{resistance form} on $K$
means that it has the following properties
(see \cite[Definition 2.3.1]{Kig01} or \cite[Definition 3.1]{Kig12}):
\begin{enumerate}[label=\textup{(RF\arabic*)},align=left,leftmargin=*,topsep=5pt,parsep=0pt,itemsep=2pt]
\item \label{it:RF1}$\{u\in\mathcal{F}\mid\mathcal{E}(u,u)=0\}=\mathbb{R}\one_{K}$.
\item \label{it:RF2}$(\mathcal{F}/\mathbb{R}\one_{K},\mathcal{E})$ is a Hilbert space.
\item \label{it:RF3}$\{u\vert_{V}\mid u\in\mathcal{F}\}=\mathbb{R}^{V}$ for any non-empty finite subset $V$ of $K$.
\item \label{it:RF4}$R_{\mathcal{E}}(x,y):=\sup_{u\in\mathcal{F}\setminus\mathbb{R}\one_{K}}\lvert u(x)-u(y)\rvert^{2}/\mathcal{E}(u,u)<\infty$
	for any $x,y\in K$.
\item \label{it:RF5}$u^{+}\wedge 1\in\mathcal{F}$ and
	$\mathcal{E}(u^{+}\wedge 1,u^{+}\wedge 1)\leq\mathcal{E}(u,u)$
	for any $u\in\mathcal{F}$.
\end{enumerate}
See \cite[Chapter 2]{Kig01} and \cite[Part 1]{Kig12} for further details of resistance forms.

In the present framework, the notion of harmonic functions is defined as follows.

\begin{definition}\label{d:harmonic-pcf}
Let $n\in\mathbb{N}\cup\{0\}$. A continuous function $h\in \contfunc (K)$ is called
\emph{$\mathcal{E}$-harmonic} on $K\setminus V_{n}$, or \emph{$n$-harmonic} for short,
if and only if $h\in\mathcal{F}$ and
\begin{equation}\label{e:harmonic-pcf}
\begin{split}
\mathcal{E}(h,h)&=\inf_{v\in\mathcal{F},\,v\vert_{V_{n}}=h\vert_{V_{n}}}\mathcal{E}(v,v),\\
\mspace{-100mu}\textrm{or equivalently,}\quad
	\mathcal{E}(h,v)&=0\mspace{7.5mu}\textrm{for any $v\in\mathcal{F}^{K\setminus V_{n}}$,}
\end{split}
\end{equation}
where $\mathcal{F}^{K\setminus V_{n}}:=\{u\in\mathcal{F}\mid u\vert_{V_{n}}=0\}$.
We set $\mathcal{H}_{n}:=\{h\in \contfunc (K)\mid\textrm{$h$ is $n$-harmonic}\}$.
\end{definition}
It is obvious that $\mathcal{H}_{n}$ is a linear subspace of $\mathcal{F}$ and
$\mathbb{R}\one_{K}\subset\mathcal{H}_{n}\subset\mathcal{H}_{n+1}$ for any $n\in\mathbb{N}\cup\{0\}$.
Moreover, we easily have the following proposition by \cite[Lemma 2.2.2 and Theorem 3.2.4]{Kig01},
\eqref{e:Em-pcf}, \eqref{e:harmonic-str}, \eqref{e:SSDF-domain} and \eqref{e:SSDF-form}.

\begin{proposition}\label{p:harmonic-pcf}
Let $n\in\mathbb{N}\cup\{0\}$.
\begin{enumerate}[label=\textup{(\arabic*)},align=left,leftmargin=*,topsep=5pt,parsep=0pt,itemsep=2pt]
\item \label{it:harmonic-extension-pcf}For each $u\in\mathbb{R}^{V_{n}}$ there exists a unique
	$H_{n}(u)\in\mathcal{H}_{n}$ such that $H_{n}(u)\vert_{V_{n}}=u$. Moreover,
	$H_{n}\colon \mathbb{R}^{V_{n}}\to\mathcal{H}_{n}$ is linear (and hence it is a linear isomorphism).
\item \label{it:harmonic-pcf}It holds that
	\begin{align}\label{e:harmonic1-pcf}
	\mathcal{H}_{n}
		&=\{h\in\mathcal{F}\mid\mathcal{E}(h,h)=\mathcal{E}^{(n)}(h\vert_{V_{n}},h\vert_{V_{n}})\}\\
		&=\{h\in \contfunc (K)\mid\textrm{$h\circ F_{w}\in\mathcal{H}_{0}$ for any $w\in W_{n}$}\}.
	\label{e:harmonic2-pcf}
	\end{align}
	In particular, for each $w\in W_{*}$, a linear map $F_{w}^{*}\colon\mathcal{H}_{0}\to\mathcal{H}_{0}$
	is defined by $F_{w}^{*}h:=h\circ F_{w}$.
\end{enumerate}
\end{proposition}

Now we equip $K$ with a measure to turn $(\mathcal{E},\mathcal{F})$
into a Dirichlet form. Indeed, we have the following proposition.

\begin{proposition}\label{prop:RF-DF}
Let $\mu$ be a Radon measure on $K$ with full support. Then $(\mathcal{E},\mathcal{F})$
is an irreducible, strongly local, regular symmetric Dirichlet form on $L^{2}(K,\mu)$,
and its extended Dirichlet space $\mathcal{F}_{e}$ coincides with $\mathcal{F}$.
Moreover, the capacity $\Capa_{1}^{\mu}$ associated with $(K,R_{\mathcal{E}},\mu,\mathcal{E},\mathcal{F})$
satisfies $\inf_{x\in K}\Capa_{1}^{\mu}(\{x\})>0$, and in particular
\textup{(recall Definition \ref{d:admissible})}
\begin{equation}\label{e:admiss-pcf}
\mathcal{A}(K,R_{\mathcal{E}},\mu,\mathcal{E},\mathcal{F})
	=\{\nu\mid\textrm{$\nu$ is a Radon measure on $K$ with full support}\}.
\end{equation}
\end{proposition}

\begin{proof}
This proposition is well known to experts on Dirichlet forms on fractals,
but we include a complete proof of it for the reader's convenience.
$(\mathcal{E},\mathcal{F})$ is a regular symmetric Dirichlet form on $L^{2}(K,\mu)$
by \cite[Corollary 6.4 and Theorem 9.4]{Kig12}, strongly local by the same argument
as \cite[Proof of Lemma 3.12]{Hin05} on the basis of \eqref{e:SSDF-domain},
\eqref{e:SSDF-form} and $\mathcal{E}(\one_{K},\one_{K})=0$, and irreducible
by \ref{it:RF1} above and \cite[Theorem 2.1.11]{CF}. The equality
$\mathcal{F}_{e}=\mathcal{F}$ is immediate from \ref{it:RF1}, \ref{it:RF2} and
\ref{it:RF4}. We also easily see from \ref{it:RF4}, $\diam_{R_{\mathcal{E}}}(K)<\infty$
and $\mu(K)<\infty$ that $\inf_{x\in K}\Capa_{1}^{\mu}(\{x\})>0$, so that a subset of $K$
is quasi-closed with respect to $(K,R_{\mathcal{E}},\mu,\mathcal{E},\mathcal{F})$
if and only if it is closed in $K$. In particular, any Radon measure $\nu$ on $K$ is smooth with respect
to $(K,R_{\mathcal{E}},\mu,\mathcal{E},\mathcal{F})$ and $\nu$ having full quasi-support
with respect to $(K,R_{\mathcal{E}},\mu,\mathcal{E},\mathcal{F})$ means that the only closed
subset $F$ of $K$ with $\nu(K\setminus F)=0$ is $F=K$, which together imply \eqref{e:admiss-pcf}.
\end{proof}


Let $d_{\mathrm{H}}\in(0,\infty)$ be such that $\sum_{i\in S}r_{i}^{d_{\mathrm{H}}}=1$,
so that $d_{\mathrm{H}}\geq 1$ since
\begin{equation}\label{e:pcf-dH-geq-1}
\max_{x,y\in V_{0}}R_{\mathcal{E}}(x,y)
	\leq\sum_{i\in S}\max_{x,y\in V_{0}}R_{\mathcal{E}}(F_{i}(x),F_{i}(y))
	\leq\biggl(\sum_{i\in S}r_{i}\biggr)\max_{x,y\in V_{0}}R_{\mathcal{E}}(x,y)
\end{equation}
by the connectedness of $K$, \cite[Theorem 1.6.2 and Lemma 3.3.5]{Kig01} and hence
$\sum_{i\in S}r_{i}\geq 1=\sum_{i\in S}r_{i}^{d_{\mathrm{H}}}$.
Let $\measure$ be the \emph{self-similar measure} on $\mathcal{L}$ with weight $(r_{i}^{d_{\mathrm{H}}})_{i\in S}$,
i.e., the unique Borel measure on $K$ such that $\measure(K_{w})=r_{w}^{d_{\mathrm{H}}}$ for any $w\in W_{*}$.
The measure $\measure$ could be considered as the ``uniform distribution'' on $\mathcal{L}$,
and it is the most typical choice of the reference measure $\mu$ for $(\mathcal{E},\mathcal{F})$.
It is well known that $(K,R_{\mathcal{E}},\measure,\mathcal{E},\mathcal{F})$ satisfies
\hyperlink{phi}{$\on{PHI}(d_{\mathrm{H}}+1)$}; more precisely, the following lemma and proposition hold.

\begin{lemma}\label{l:volume-pcf}
There exist $c_{1},c_{2}\in(0,\infty)$ such that
for any $(x,s)\in K\times(0,\diam_{R_{\mathcal{E}}}(K)]$,
\begin{equation}\label{e:volume-pcf}
c_{1}s^{d_{\mathrm{H}}}\leq \measure(B_{R_{\mathcal{E}}}(x,s))\leq c_{2}s^{d_{\mathrm{H}}}.
\end{equation}
%
\end{lemma}

\begin{proof}
This is immediate from Lemma \ref{l:adapted-RE-pcf} below and \cite[Lemma 4.2.3]{Kig01}.
\end{proof}

\begin{proposition}\label{prop:PHI-pcf}
$(K,R_{\mathcal{E}},\measure,\mathcal{E},\mathcal{F})$ satisfies \hyperlink{phi}{$\on{PHI}(d_{\mathrm{H}}+1)$}.
\end{proposition}

\begin{proof}
Lemma \ref{l:volume-pcf} and \cite[Theorem 15.10]{Kig12} together imply that
$(K,R_{\mathcal{E}},\measure,\mathcal{E},\mathcal{F})$ satisfies 
\hyperlink{hke}{$\on{HKE}(d_{\mathrm{H}}+1)$} as well as \ref{VD} and \ref{RVD}, and therefore
it also satisfies \hyperlink{phi}{$\on{PHI}(d_{\mathrm{H}}+1)$} by Theorem \ref{t:phichar}.
\end{proof}

We conclude this subsection with the following Proposition, which is essentially due to
Kigami \cite{Kig09,Kig12} and gives a simple equivalent condition for the validity of
\hyperlink{phi}{$\on{PHI}(\beta)$} after quasisymmetric change of the metric and time change.

\begin{proposition}\label{p:condition-PHI-pcf}
Let $\theta \in \mathcal{J}(K,R_{\mathcal{E}})$ \textup{(recall \eqref{e:cgauge-dfn})},
let $\mu$ be a Radon measure on $K$ with full support and let $\beta\in(1,\infty)$.
Then the following conditions are equivalent:
\begin{enumerate}[label=\textup{(\alph*)},align=left,leftmargin=*,topsep=5pt,parsep=0pt,itemsep=2pt]
\item \label{it:condition-PHI-pcf-a}$(K,\theta,\mu,\mathcal{E},\mathcal{F})$ satisfies \hyperlink{phi}{$\on{PHI}(\beta)$}.
\item \label{it:condition-PHI-pcf-b}There exists $C\in(1,\infty)$ such that for any $w\in W_{*}$,
	\begin{equation}\label{e:DM2-pcf}
	C^{-1}(\diam_{\theta}(K_{w}))^{\beta}\leq r_{w}\mu(K_{w})\leq C(\diam_{\theta}(K_{w}))^{\beta}.
	\end{equation}
\end{enumerate}
Moreover, if either of these conditions holds, then $\mu(F_{w}(V_{0}))=0$
for any $w\in W_{*}$ and $\mu(\{x\})=0$ for any $x\in K$.
\end{proposition}

The rest of this subsection is devoted to the proof of Proposition \ref{p:condition-PHI-pcf},
which requires the following lemmas and definitions.

\begin{lemma}[{\cite[Lemma 3.3.5]{Kig01}, \cite[Theorem A.1]{Kig04}}]\label{l:RE-scaling-pcf}
There exists $c_{R_{\mathcal{E}}}\in(0,1]$ such that for any $w\in W_{*}$ and any $x,y\in K$,
\begin{equation}\label{e:RE-scaling-pcf}
c_{R_{\mathcal{E}}}r_{w}R_{\mathcal{E}}(x,y)
	\leq R_{\mathcal{E}}(F_{w}(x),F_{w}(y))\leq r_{w}R_{\mathcal{E}}(x,y).
\end{equation}
\end{lemma}

\begin{definition}\label{d:partition}
\begin{enumerate}[label=\textup{(\arabic*)},align=left,leftmargin=*,topsep=5pt,parsep=0pt,itemsep=2pt]
\item Let $w,v\in W_{*}$, $w=w_{1}\dots w_{n_{1}}$, $v=v_{1}\dots v_{n_{2}}$.
	We define $wv\in W_{*}$ by $wv:=w_{1}\dots w_{n_{1}}v_{1}\dots v_{n_{2}}$
	($w\emptyset:=w$, $\emptyset v:=v$). We write $w\leq v$ if and only if
	$w=v\tau$ for some $\tau\in W_{*}$; note that
	$\Sigma_{w}\cap\Sigma_{v}=\emptyset$ if and only if neither $w\leq v$ nor $v\leq w$.
\item\label{it:partition}A finite subset $\Lambda$ of $W_{*}$ is called a
	\emph{partition} of $\Sigma$ if and only if $\Sigma_{w}\cap\Sigma_{v}=\emptyset$
	for any $w,v\in\Lambda$ with $w\not=v$ and $\Sigma=\bigcup_{w\in\Lambda}\Sigma_{w}$.
\item Let $\Lambda_{1},\Lambda_{2}$ be partitions of $\Sigma$.
	We say that $\Lambda_{1}$ is a \textit{refinement of $\Lambda_{2}$},
	and write $\Lambda_{1}\leq\Lambda_{2}$, if and only if for each
	$w^{1}\in\Lambda_{1}$ there exists $w^{2}\in\Lambda_{2}$ such that $w^{1}\leq w^{2}$.
\end{enumerate}
\end{definition}
\begin{definition}\label{d:scale-pcf}
\begin{enumerate}[label=\textup{(\arabic*)},align=left,leftmargin=*,topsep=5pt,parsep=0pt,itemsep=2pt]
\item We define $\Lambda_{1}:=\{\emptyset\}$,
	\begin{equation}\label{e:scale}
	\Lambda_{s}:=\{w\mid\textrm{$w=w_{1}\dots w_{n}\in W_{*}\setminus\{\emptyset\}$,
		$r_{w_{1}\dots w_{n-1}}>s\geq r_{w}$}\}
	\end{equation}
	for each $s\in(0,1)$, and $\mathscr{S}:=\{\Lambda_{s}\}_{s\in(0,1]}$. We call $\mathscr{S}$
	the \emph{scale} on $\Sigma$ associated with $\mathbf{r}$.
\item For each $(s,x)\in(0,1]\times K$, we define
	$\Lambda_{s,x}:=\{w\in\Lambda_{s}\mid x\in K_{w}\}$,
	$K_{s}(x):=\bigcup_{w\in\Lambda_{s,x}}K_{w}$,
	$\Lambda^{1}_{s,x}:=\{w\in\Lambda_{s}\mid K_{w}\cap K_{s}(x)\not=\emptyset\}$
	and $U_{s}(x):=\bigcup_{w\in\Lambda^{1}_{s,x}}K_{w}$.
\end{enumerate}
\end{definition}
Clearly $\lim_{s\downarrow 0}\min\{\lvert w\rvert\mid w\in\Lambda_{s}\}=\infty$, and
it is easy to see that $\Lambda_{s}$ is a partition of $\Sigma$ for any $s\in(0,1]$
and that $\Lambda_{s_{1}}\leq\Lambda_{s_{2}}$ for any $s_{1},s_{2}\in(0,1]$ with
$s_{1}\leq s_{2}$. These facts together with \cite[Proposition 1.3.6]{Kig01}
imply that for any $x\in K$, each of $\{K_{s}(x)\}_{s\in(0,1]}$ and
$\{U_{s}(x)\}_{s\in(0,1]}$ is non-decreasing in $s$ and forms
a fundamental system of neighborhoods of $x$ in $K$. Moreover,
$\{U_{s}(x)\}_{(s,x)\in(0,1]\times K}$ can be used as a replacement for the metric balls
$\{B_{R_{\mathcal{E}}}(x,s)\}_{(x,s)\in K\times(0,\diam(K,R_{\mathcal{E}})]}$
in $(K,R_{\mathcal{E}})$ by virtue of the following lemma.

\begin{lemma}\label{l:adapted-RE-pcf}
There exist $\alpha_{1},\alpha_{2}\in(0,\infty)$ such that for any $(s,x)\in(0,1]\times K$,
\begin{equation}\label{e:adapted-RE-pcf}
B_{R_{\mathcal{E}}}(x,\alpha_{1}s)\subset U_{s}(x)\subset B_{R_{\mathcal{E}}}(x,\alpha_{2}s).
\end{equation}
\end{lemma}

\begin{proof}
This is mentioned in \cite[Subsection 4.1]{Kaj13},
but we include a complete proof of it for the reader's convenience.
By the upper inequality in \eqref{e:RE-scaling-pcf} we have
$\diam_{R_{\mathcal{E}}}(K_{w})\leq r_{w}\diam_{R_{\mathcal{E}}}(K)$ for any $w\in W_{*}$,
which implies the latter inclusion in \eqref{e:adapted-RE-pcf} with
$\alpha_{2}\in(2\diam_{R_{\mathcal{E}}}(K),\infty)$ arbitrary. On the other hand,
by \cite[Proof of Lemma 4.2.4]{Kig01} there exists $\alpha_{1}\in(0,\infty)$
such that $R_{\mathcal{E}}(x,y)\geq\alpha_{1}s$ for any $s\in(0,1]$, any
$w,v\in\Lambda_{s}$ with $K_{w}\cap K_{v}=\emptyset$ and any $(x,y)\in K_{w}\times K_{v}$,
which yields the former inclusion in \eqref{e:adapted-RE-pcf}.
\end{proof}

\begin{proof}[Proof of Proposition \textup{\ref{p:condition-PHI-pcf}}]
This equivalence can be easily concluded by combining Theorem \ref{t:phichar} and
results in \cite{Kig12,Kig09}, as follows. First, by Theorem \ref{t:phichar}
and \cite[Theorem 15.10]{Kig12}, under $\theta \in \mathcal{J}(K,R_{\mathcal{E}})$,
\ref{it:condition-PHI-pcf-a} is equivalent to the following condition \ref{it:condition-PHI-pcf-c}:
\begin{enumerate}[label=\textup{(\alph*)},align=left,leftmargin=*,topsep=5pt,parsep=0pt,itemsep=2pt]
\addtocounter{enumi}{2}
\item \label{it:condition-PHI-pcf-c}$(K,\theta,\mu)$ is \ref{VD} and
	there exists $C\in(1,\infty)$ such that for any $x,y\in K$ with $x\not=y$,
	\begin{equation}\label{e:DM2-RE-pcf}
	C^{-1}\theta(x,y)^{\beta}\leq R_{\mathcal{E}}(x,y)\mu\bigl(B_{\theta}(x,\theta(x,y))\bigr)\leq C\theta(x,y)^{\beta}.
	\end{equation}
\end{enumerate}

Next, by $\theta \in \mathcal{J}(K,R_{\mathcal{E}})$, \eqref{e:cgauge} and
\eqref{e:ann1} with $(d_{1},d_{2})\in\{(R_{\mathcal{E}},\theta),(\theta,R_{\mathcal{E}})\}$,
$(K,\theta,\mu)$ is \ref{VD} if and only if $(K,R_{\mathcal{E}},\mu)$ is \ref{VD},
which in turn is, by Lemma \ref{l:adapted-RE-pcf} and the compactness of $K$,
equivalent to the existence of $C\in(1,\infty)$ such that
\begin{equation}\label{e:scale-VD-pcf}
\mu(U_{s}(x))\leq C\mu(U_{s/2}(x))\qquad\textrm{for any $(s,x)\in(0,1]\times K$.}
\end{equation}
Then by \cite[Theorem 1.3.5]{Kig09} and the fact that $\mathscr{S}$
is \emph{locally finite} with respect to $\mathcal{L}$, i.e.,
\begin{equation}\label{e:LF-pcf}
\sup\nolimits_{(s,x)\in(0,1]\times K}\#\Lambda^{1}_{s,x}<\infty
\end{equation}
by \cite[Lemma 4.2.3]{Kig01} and \cite[Lemma 1.3.6]{Kig09}, we have \eqref{e:scale-VD-pcf}
if and only if there exists $C\in(1,\infty)$ such that the following hold:
\begin{gather}\label{e:ELm-pcf}
\mu(K_{wj})\geq C^{-1}\mu(K_{w})\quad\textrm{for any $(w,j)\in W_{*}\times S$,}\\
\mu(K_{w})\leq C\mu(K_{v})
	\quad\textrm{for any $s\in(0,1]$ and any $w,v\in\Lambda_{s}$ with $K_{w}\cap K_{v}\not=\emptyset$.}
\label{e:GE-pcf}
\end{gather}
Moreover, by \cite[Theorem 1.2.4]{Kig09}, \eqref{e:ELm-pcf} implies that
\begin{equation}\label{e:sss-meas-bdry-pt}
\mu(F_{w}(V_{0}))=0=\mu(\{x\}) \qquad \textrm{for any $w\in W_{*}$ and any $x\in K$.}
\end{equation}
(We remark that \eqref{e:sss-meas-bdry-pt} is part of the assumptions of \cite[Theorem 1.3.5]{Kig09}
but can be dropped; indeed, even without assuming \eqref{e:sss-meas-bdry-pt},
\cite[Proofs of Theorems 1.3.10 and 1.3.11]{Kig09} show that any one of
the three conditions \cite[Theorem 1.3.5-(1),(2),(3)]{Kig09} implies \eqref{e:ELm-pcf},
from which \eqref{e:sss-meas-bdry-pt} also follows by \cite[Theorem 1.2.4]{Kig09}.)

On the other hand, since the quasisymmetry of $R_{\mathcal{E}}$ to $\theta$
yields $\delta_{1},\delta_{2}\in(0,\infty)$ such that
$B_{\theta}(x,\delta_{1}\theta(x,y))\subset B_{R_{\mathcal{E}}}(x,R_{\mathcal{E}}(x,y))
	\subset B_{\theta}(x,\delta_{2}\theta(x,y))$
for any $x,y\in K$ with $x\not=y$ by \eqref{e:ann1},
under \ref{VD} of $(K,\theta,\mu)$ and $(K,R_{\mathcal{E}},\mu)$
we have \eqref{e:DM2-RE-pcf} if and only if
there exists $C\in(0,\infty)$ such that for any $x,y\in K$ with $x\not=y$,
\begin{equation}\label{e:DM2-volume-RE-pcf}
C^{-1}\theta(x,y)^{\beta}
	\leq R_{\mathcal{E}}(x,y)\mu\bigl(B_{R_{\mathcal{E}}}(x,R_{\mathcal{E}}(x,y))\bigr)
	\leq C\theta(x,y)^{\beta}.
\end{equation}
Therefore \ref{it:condition-PHI-pcf-c} is equivalent to the following condition \ref{it:condition-PHI-pcf-d}:
\begin{enumerate}[label=\textup{(\alph*)},align=left,leftmargin=*,topsep=5pt,parsep=0pt,itemsep=2pt]
\addtocounter{enumi}{3}
\item \label{it:condition-PHI-pcf-d}There exists $C\in(1,\infty)$ such that \eqref{e:ELm-pcf},
	\eqref{e:GE-pcf} and \eqref{e:DM2-volume-RE-pcf} hold.
\end{enumerate}

Thus it remains to show that \ref{it:condition-PHI-pcf-d} is equivalent to \ref{it:condition-PHI-pcf-b}. Indeed,
let $w\in W_{*}$ and take $x,y\in K_{w}$ with the property
$\diam_{R_{\mathcal{E}}}(K_{w})=R_{\mathcal{E}}(x,y)$, so that $w\in\Lambda_{r_{w}}$,
$R_{\mathcal{E}}(x,y)/\diam_{R_{\mathcal{E}}}(K)\in[c_{R_{\mathcal{E}}}r_{w},r_{w}]$
by Lemma \ref{l:RE-scaling-pcf}, and $\theta(x,y)^{\beta}=(\diam_{\theta}(\{x,y\}))^{\beta}\asymp(\diam_{\theta}(K_{w}))^{\beta}$
by $\theta \in \mathcal{J}(K,R_{\mathcal{E}})$ and \eqref{e:diamQS}.
If \ref{it:condition-PHI-pcf-d} holds, then since $(K,R_{\mathcal{E}},\mu)$ is \ref{VD} we easily see
from Lemma \ref{l:adapted-RE-pcf}, \eqref{e:LF-pcf} and \eqref{e:GE-pcf} that
\begin{equation}\label{e:meas-cell-REball-comparable}
\mu\bigl(B_{R_{\mathcal{E}}}(x,\diam_{R_{\mathcal{E}}}(K_{w}))\bigr)
	=\mu\bigl(B_{R_{\mathcal{E}}}(x,R_{\mathcal{E}}(x,y))\bigr)\asymp\mu(K_{w})
\end{equation}
and hence \eqref{e:DM2-volume-RE-pcf} implies \ref{it:condition-PHI-pcf-b}.
Conversely suppose that \ref{it:condition-PHI-pcf-b} holds.
Then for any $j\in S$, Lemma \ref{l:RE-scaling-pcf} yields
$\diam_{R_{\mathcal{E}}}(K_{wj})/\diam_{R_{\mathcal{E}}}(K_{w})
	\in[c_{R_{\mathcal{E}}}\min_{k\in S}r_{k},c_{R_{\mathcal{E}}}^{-1}\max_{k\in S}r_{k}]$,
hence
\begin{equation}\label{e:REqcd-diamKwjKw}
(\diam_{\theta}(K_{wj}))^{\beta}\asymp(\diam_{\theta}(K_{w}))^{\beta}
\end{equation}
by $\theta \in \mathcal{J}(K,R_{\mathcal{E}})$ and \eqref{e:diamQS},
and therefore \eqref{e:DM2-pcf} implies $\mu(K_{wj})\asymp\mu(K_{w})$,
i.e., \eqref{e:ELm-pcf} holds. Also for any $s\in(0,1]$ and any $v,\tau\in\Lambda_{s}$
with $K_{v}\cap K_{\tau}\not=\emptyset$,
$\diam_{R_{\mathcal{E}}}(K_{v})\asymp\diam_{R_{\mathcal{E}}}(K_{v}\cup K_{\tau})\asymp\diam_{R_{\mathcal{E}}}(K_{\tau})$
by Lemma \ref{l:RE-scaling-pcf} and hence
\begin{equation}\label{e:REqcd-diamKvKtau}
(\diam_{\theta}(K_{v}))^{\beta}\asymp(\diam_{\theta}(K_{v}\cup K_{\tau}))^{\beta}\asymp(\diam_{\theta}(K_{\tau}))^{\beta}
\end{equation}
by $\theta \in \mathcal{J}(K,R_{\mathcal{E}})$ and \eqref{e:diamQS},
which together with \eqref{e:DM2-pcf} implies $\mu(K_{v})\asymp\mu(K_{\tau})$,
proving \eqref{e:GE-pcf}. In particular, $(K,R_{\mathcal{E}},\mu)$ is \ref{VD}, and
now it follows from Lemma \ref{l:adapted-RE-pcf}, \eqref{e:LF-pcf} and \eqref{e:GE-pcf}
that \eqref{e:meas-cell-REball-comparable} holds, which together with
\eqref{e:DM2-pcf} yields \eqref{e:DM2-volume-RE-pcf}, proving \ref{it:condition-PHI-pcf-d}.
\end{proof}

\subsection{A necessary condition: attainment by the energy measure of some harmonic function}\label{ssec:attain-harmonic-func-pcf}

Throughout this subsection, we assume that $\mathcal{L}=(K,S,\{F_{i}\}_{i\in S})$
is a post-critically finite self-similar structure with $\#S\geq 2$ and $K$ connected
and that $(D,\mathbf{r})$ is a regular harmonic structure on $\mathcal{L}$, and
we follow the notation introduced in Subsection \ref{ssec:pfcsss-preliminaries}.

Proposition \ref{p:condition-PHI-pcf} with $\beta=2$ justifies introducing
the following set of pairs of metrics and Borel probability measures on $K$.
Recall Definition \ref{d:harmonic-pcf} for $\mathcal{H}_{0}$.

\begin{definition}\label{d:GetaC-pcf}
We set
\begin{align}\label{eq:HomeoPlus}
\homeo^{+}&:=\{\eta\mid\textrm{$\eta\colon[0,\infty)\to[0,\infty)$, $\eta$ is a homeomorphism}\},\\
\mathcal{P}(K)&:=\{\mu\mid\textrm{$\mu$ is a Borel probability measure on $K$}\},
\label{eq:PK-pcf}
\end{align}
which is equipped with the topology of weak convergence, and
for each $(\eta,C)\in\homeo^{+}\times(1,\infty)$ we define
\begin{equation}\label{e:GetaC-pcf}
\begin{split}
\attainsss(\eta&,C):=\attainsss_{\mathcal{L},(D,\mathbf{r})}(\eta,C)\\
&:=\biggl\{(\theta,\mu)\biggm\vert
	\begin{minipage}{250pt}
		$\theta$ is a metric on $K$ and $\eta$-quasisymmetric to $R_{\mathcal{E}}$,
		$\mu\in\mathcal{P}(K)$,
		$C^{-1}\leq r_{w}\mu(K_{w})/(\diam_{\theta}(K_{w}))^{2}\leq C$
		for any $w\in W_{*}$
	\end{minipage}
	\biggr\},
\end{split}
\end{equation}
which is considered as a subset of $\contfunc(K\times K)\times\mathcal{P}(K)$. We also set
\begin{equation}\label{e:GLDr}
\attainsss:=\attainsss_{\mathcal{L},(D,\mathbf{r})}
	:=\bigcup_{(\eta,C)\in\homeo^{+}\times(1,\infty)}\attainsss(\eta,C)
\end{equation}
and for each subset $\mathcal{Z}$ of $\attainsss$ define
$\mathcal{H}_{0}(\mathcal{Z})\subset\mathcal{H}_{0}$ and
$\widetilde{\mathcal{H}}_{0}(\mathcal{Z})\subset\mathcal{H}_{0}/\mathbb{R}\one_{K}$ by
\begin{align}\label{e:H0Z-pcf}
\mathcal{H}_{0}(\mathcal{Z})
	&:=\{h\in\mathcal{H}_{0}\mid\textrm{$(\theta_{h},\Gamma(h,h))\in\mathcal{Z}$ for some metric $\theta_{h}$ on $K$}\},\\
\widetilde{\mathcal{H}}_{0}(\mathcal{Z})
	&:=\{h+\mathbb{R}\one_{K}\mid h\in\mathcal{H}_{0}(\mathcal{Z})\}.
\label{e:H0tildeZ-pcf}
\end{align}
\end{definition}

Since $\mu\in\mathcal{A}(K,R_{\mathcal{E}},\measure,\mathcal{E},\mathcal{F})$ for any
$(\theta,\mu)\in\attainsss$ by \eqref{e:admiss-pcf} with $\measure$ in place of $\mu$ and \eqref{e:GetaC-pcf},
it follows from Proposition \ref{p:condition-PHI-pcf} with $\beta=2$ and \eqref{e:cgauge} that
\begin{equation}\label{e:condition-PHI-pcf-G}
\mathcal{G}(K,R_{\mathcal{E}},\measure,\mathcal{E},\mathcal{F})
	=\{a\mu\mid\textrm{$(\theta,\mu)\in\attainsss$, $a\in(0,\infty)$}\}.
\end{equation}
In particular,
\begin{equation}\label{e:condition-PHI-pcf-G-attained}
\begin{minipage}{280pt}
\emph{$\mathcal{G}(K,R_{\mathcal{E}},\measure,\mathcal{E},\mathcal{F})\not=\emptyset$,
i.e., the infimum in \eqref{e:dcw} is attained for $(K,R_{\mathcal{E}},\measure,\mathcal{E},\mathcal{F})$,
if and only if $\attainsss\not=\emptyset$, namely
$\attainsss(\eta,C)\not=\emptyset$ for some $(\eta,C)\in\homeo^{+}\times(1,\infty)$.}
\end{minipage}
\end{equation}
In fact, it turns out that in this case $\mathcal{H}_{0}(\attainsss)\not=\emptyset$, i.e.,
$(\theta_{h},\Gamma(h,h))\in\attainsss$ for some $h\in\mathcal{H}_{0}$ and some $\theta_{h}\in\mathcal{J}(K,R_{\mathcal{E}})$,
which is the main result of this subsection and stated as follows.
\emph{We let $c_{R_{\mathcal{E}}}$ be as in Lemma \textup{\ref{l:RE-scaling-pcf}},
take arbitrary $(\eta,C)\in\homeo^{+}\times(1,\infty)$,
define $\tilde{\eta}\in\homeo^{+}$ by $\tilde{\eta}(t):=1/\eta^{-1}(t^{-1})$
($\tilde{\eta}(0):=0$) and fix them throughout the rest of this subsection.}

\begin{theorem}\label{t:attain-harmonic-func}
If $\attainsss(\eta,C)\not=\emptyset$, then
$\mathcal{H}_{0}(\attainsss(c_{R_{\mathcal{E}}}^{-1}\eta,C))\not=\emptyset$, i.e.,
there exist $h\in\mathcal{H}_{0}$ and a metric $\theta_{h}$ on $K$
such that $(\theta_{h},\Gamma(h,h))\in\attainsss(c_{R_{\mathcal{E}}}^{-1}\eta,C)$.
\end{theorem}

Moreover, a slight addition to our proof of Theorem \ref{t:attain-harmonic-func}
also shows the following proposition, which is used in Subsubsection \ref{sssec:SGs}
to prove the non-attainment of the infimum in \eqref{e:dcw} for the $N$-dimensional
Sierpi\'{n}ski gasket with $N\geq 3$ (Theorem \ref{thm:SGN-NOT-attained}).
Recall \ref{it:RF2} for $(\mathcal{F}/\mathbb{R}\one_{K},\mathcal{E})$.

\begin{proposition}\label{p:H0GetaC-compact-pcf}
\begin{enumerate}[label=\textup{(\arabic*)},align=left,leftmargin=*,topsep=5pt,parsep=0pt,itemsep=2pt]
\item\label{it:H0GetaC-compact-pcf} $\widetilde{\mathcal{H}}_{0}(\attainsss(\eta,C))$ is compact in norm in $(\mathcal{F}/\mathbb{R}\one_{K},\mathcal{E})$.
\item\label{it:H0GetaC-compact-scaling-pcf} If $h\in\mathcal{H}_{0}(\attainsss(\eta,C))$, then
	$\mathcal{E}(h\circ F_{w},h\circ F_{w})^{-1/2}h\circ F_{w}\in\mathcal{H}_{0}(\attainsss(c_{R_{\mathcal{E}}}^{-1}\eta,C))$
	for any $w\in W_{*}$.
\end{enumerate}
\end{proposition}

We remark that in Proposition \ref{p:H0GetaC-compact-pcf}-\ref{it:H0GetaC-compact-scaling-pcf}
we have $\mathcal{E}(h\circ F_{w},h\circ F_{w})=r_{w}\Gamma(h,h)(K_{w})>0$ for any $w\in W_{*}$ by
Lemma \ref{l:scaling-energy-meas} below and the lower inequality in \eqref{e:GetaC-pcf} for $\mu=\Gamma(h,h)$.

The rest of this subsection is devoted to the proof of Theorem \ref{t:attain-harmonic-func}
and Proposition \ref{p:H0GetaC-compact-pcf}, which is reduced to proving a series of
propositions and lemmas concerning the set $\attainsss(\eta,C)$. We start with establishing
its compactness. Note that $\mathcal{P}(K)$ is a compact metrizable topological space
by \cite[Theorems 9.1.5 and 9.1.9]{Str} and hence that
$\contfunc(K\times K)\times\mathcal{P}(K)$ is also metrizable.

\begin{proposition}\label{p:GetaC-compact}
$\attainsss(\eta,C)$ is a compact subset of $\contfunc(K\times K)\times\mathcal{P}(K)$.
\end{proposition}

\begin{proof}
Let $\{(\theta_{n},\mu_{n})\}_{n\in\mathbb{N}}\subset\attainsss(\eta,C)$.
By the metrizability of $\contfunc(K\times K)\times\mathcal{P}(K)$ noted above,
it suffices to show that there exists a subsequence of
$\{(\theta_{n},\mu_{n})\}_{n\in\mathbb{N}}$ converging to some
$(\theta,\mu)\in\attainsss(\eta,C)$ in $\contfunc(K\times K)\times\mathcal{P}(K)$.

First, recalling that the compactness of $K$ implies that of $\mathcal{P}(K)$ by
\cite[Theorem 9.1.9]{Str}, we can choose $\mu\in\mathcal{P}(K)$ and a subsequence
$\{\mu_{n_{k}}\}_{k\in\mathbb{N}}$ of $\{\mu_{n}\}_{n\in\mathbb{N}}$ converging to
$\mu$ in $\mathcal{P}(K)$, and therefore by considering
$\{(\theta_{n_{k}},\mu_{n_{k}})\}_{k\in\mathbb{N}}$ instead of
$\{(\theta_{n},\mu_{n})\}_{n\in\mathbb{N}}$ we may assume that
$\{\mu_{n}\}_{n\in\mathbb{N}}$ itself converges to $\mu$ in $\mathcal{P}(K)$.

Next, $\diam_{\theta_{n}}(K)\in[C^{-1/2},C^{1/2}]$ for any $n\in\mathbb{N}$ by the inequalities
in \eqref{e:GetaC-pcf} and hence $\{\theta_{n}\}_{n\in\mathbb{N}}$ is uniformly bounded.
Moreover, for each $n\in\mathbb{N}$, since the $\eta$-quasisymmetry of $\theta_{n}$ to
$R_{\mathcal{E}}$ yields the $\tilde{\eta}$-quasisymmetry of $R_{\mathcal{E}}$ to $\theta_{n}$
by \cite[Proposition 10.6]{Hei}, it follows from \eqref{e:diamQS} that for any $x,y\in K$,
\begin{gather}\label{e:GetaC-RE-eta-d}
\frac{R_{\mathcal{E}}(x,y)}{\diam_{R_{\mathcal{E}}}(K)}
	=\frac{\diam_{R_{\mathcal{E}}}(\{x,y\})}{\diam_{R_{\mathcal{E}}}(K)}
	\leq\eta\Bigl(\frac{2\diam_{\theta_{n}}(\{x,y\})}{\diam_{\theta_{n}}(K)}\Bigr)
	\leq\eta\bigl(2C^{1/2}\theta_{n}(x,y)\bigr),\\
C^{-1/2}\theta_{n}(x,y)\leq\frac{\diam_{\theta_{n}}(\{x,y\})}{\diam_{\theta_{n}}(K)}
	\leq\tilde{\eta}\Bigl(\frac{2\diam_{R_{\mathcal{E}}}(\{x,y\})}{\diam_{R_{\mathcal{E}}}(K)}\Bigr)
	=\tilde{\eta}\Bigl(\frac{2R_{\mathcal{E}}(x,y)}{\diam_{R_{\mathcal{E}}}(K)}\Bigr),
\label{e:GetaC-d-eta-RE}
\end{gather}
which in turn implies that for any $x_{1},y_{1},x_{2},y_{2}\in K$,
\begin{equation}\label{e:GetaC-equicont}
\begin{split}
\lvert\theta_{n}(x_{1},y_{1})-\theta_{n}(x_{2},y_{2})\rvert
	&\leq \theta_{n}(x_{1},x_{2})+\theta_{n}(y_{1},y_{2})\\
&\leq C^{1/2}\biggl(\tilde{\eta}\Bigl(\frac{2R_{\mathcal{E}}(x_{1},x_{2})}{\diam_{R_{\mathcal{E}}}(K)}\Bigr)
	+\tilde{\eta}\Bigl(\frac{2R_{\mathcal{E}}(y_{1},y_{2})}{\diam_{R_{\mathcal{E}}}(K)}\Bigr)\biggr),
\end{split}
\end{equation}
so that $\{\theta_{n}\}_{n\in\mathbb{N}}\subset \contfunc(K\times K)$ is equicontinuous.
Thus by the Arzel\`{a}--Ascoli theorem (see, e.g., \cite[Theorem 11.28]{Rud})
there exist $\theta\in \contfunc(K\times K)$ and a subsequence $\{\theta_{n_{k}}\}_{k\in\mathbb{N}}$
of $\{\theta_{n}\}_{n\in\mathbb{N}}$ converging to $\theta$ in $\contfunc(K\times K)$,
so that $\{(\theta_{n_{k}},\mu_{n_{k}})\}_{k\in\mathbb{N}}$
converges to $(\theta,\mu)$ in $\contfunc(K\times K)\times\mathcal{P}(K)$.
Then $\diam_{\theta}(K)\in[C^{-1/2},C^{1/2}]$ and for any $x,y,z\in K$
we have \eqref{e:GetaC-RE-eta-d} with $\theta$ in place of $\theta_{n}$,
$\theta(x,x)=0$, $\theta(x,y)=\theta(y,x)\geq 0$ and $\theta(x,y)\leq \theta(x,z)+\theta(z,y)$
by the same properties of $\theta_{n_{k}}$ for $k\in\mathbb{N}$, whence $\theta$ is a metric on $K$.
Furthermore letting $k\to\infty$ in the $\eta$-quasisymmetry of $\theta_{n_{k}}$ to
$R_{\mathcal{E}}$ as defined in \eqref{e:qs} yields that of $\theta$ to $R_{\mathcal{E}}$.

To show the inequalities in \eqref{e:GetaC-pcf} for $(\theta,\mu)$, let $w\in W_{*}$,
choose $x=x_{w}\in K_{w}\setminus F_{w}(V_{0})$ and set $s:=r_{w}$, so that
$w\in\Lambda_{s}$ and $\Lambda_{s,x}=\{w\}$ by \cite[Proposition 1.3.5-(2)]{Kig01}
(recall Definition \ref{d:scale-pcf}).
Note that $K_{w}=F_{w}(K)$ is compact and hence closed in $K$, that
$K_{w}\setminus F_{w}(V_{0})=K\setminus\bigl(F_{w}(V_{0})\cup\bigcup_{v\in W_{\lvert w\rvert}\setminus\{w\}}K_{v}\bigr)$
by \cite[Proposition 1.3.5-(2)]{Kig01} and is thus open in $K$, and that
$K_{w}\subset U^{\circ}_{s}(x)$ with $U^{\circ}_{s}(x)$ the interior of
$U_{s}(x)$ in $K$ by \cite[Proposition 1.3.6]{Kig01}.
By using these facts and the convergence of $\{(\theta_{n_{k}},\mu_{n_{k}})\}_{k\in\mathbb{N}}$
to $(\theta,\mu)$ in $\contfunc(K\times K)\times\mathcal{P}(K)$ to let $k\to\infty$ in the
inequalities in \eqref{e:GetaC-pcf} for $(\theta_{n_{k}},\mu_{n_{k}})$, we obtain
\begin{gather}\label{e:GetaC-limit-lower}
r_{w}\mu(K_{w})
	\geq\limsup_{k\to\infty}r_{w}\mu_{n_{k}}(K_{w})
	\geq C^{-1}(\diam_{\theta}(K_{w}))^{2},\\
r_{w}\mu(K_{w}\setminus F_{w}(V_{0}))
	\leq\liminf_{k\to\infty}r_{w}\mu_{n_{k}}(K_{w}\setminus F_{w}(V_{0}))
	\leq C(\diam_{\theta}(K_{w}))^{2},\label{e:GetaC-limit-upper}\\
\begin{split}
r_{w}\mu(K_{w})\leq r_{w}\mu(U^{\circ}_{s}(x))
	&\leq\liminf_{k\to\infty}r_{w}\mu_{n_{k}}(U^{\circ}_{s}(x))
	\leq \sum_{v\in\Lambda^{1}_{s,x}}C\frac{r_{w}}{r_{v}}(\diam_{\theta}(K_{v}))^{2}\\
&\leq \frac{C}{\min_{k\in S}r_{k}}\sum_{v\in\Lambda^{1}_{s,x}}(\diam_{\theta}(K_{v}))^{2}
	\asymp(\diam_{\theta}(K_{w}))^{2},
\end{split}
\label{e:GetaC-limit-Usx}
\end{gather}
where the last step in \eqref{e:GetaC-limit-Usx} follows from \eqref{e:REqcd-diamKvKtau}
and \eqref{e:LF-pcf}. We now conclude from \eqref{e:GetaC-limit-lower},
\eqref{e:GetaC-limit-Usx} and \eqref{e:REqcd-diamKwjKw} that
$\mu(K_{wj})\asymp r_{wj}^{-1}(\diam_{\theta}(K_{wj}))^{2}
	\asymp r_{w}^{-1}(\diam_{\theta}(K_{w}))^{2}\asymp\mu(K_{w})$
for any $(w,j)\in W_{*}\times S$, which together with \cite[Theorem 1.2.4]{Kig09}
implies that $\mu(F_{w}(V_{0}))=0$ and hence
$\mu(K_{w}\setminus F_{w}(V_{0}))=\mu(K_{w})$ for any $w\in W_{*}$, so that
\eqref{e:GetaC-limit-lower} and \eqref{e:GetaC-limit-upper} yield the inequalities
in \eqref{e:GetaC-pcf} for $(\theta,\mu)$ and thus $(\theta,\mu)\in\attainsss(\eta,C)$.
\end{proof}

\begin{corollary}\label{c:GetaC-compact}
Let $\{(\theta_{n},\mu_{n})\}_{n\in\mathbb{N}}\subset\attainsss(\eta,C)$,
$\mu\in\mathcal{P}(K)$ and suppose that $\{\mu_{n}\}_{n\in\mathbb{N}}$ converges
to $\mu$ in $\mathcal{P}(K)$. Then there exists a metric $\theta$ on $K$ such that
$(\theta,\mu)\in\attainsss(\eta,C)$.
\end{corollary}

\begin{proof}
Since $\attainsss(\eta,C)$ is a compact subset of $\contfunc(K\times K)\times\mathcal{P}(K)$
by Proposition \ref{p:GetaC-compact}, there exist $(\theta,\nu)\in\attainsss(\eta,C)$
and a subsequence $\{(\theta_{n_{k}},\mu_{n_{k}})\}_{k\in\mathbb{N}}$ of
$\{(\theta_{n},\mu_{n})\}_{n\in\mathbb{N}}$ converging to $(\theta,\nu)$ in
$\contfunc(K\times K)\times\mathcal{P}(K)$, but then $\{\mu_{n_{k}}\}_{k\in\mathbb{N}}$
converges in $\mathcal{P}(K)$ to both $\mu$ and $\nu$, hence $\mu=\nu$ and thus
$(\theta,\mu)=(\theta,\nu)\in\attainsss(\eta,C)$.
\end{proof}

We next observe that the set $\attainsss(\eta,C)$ is almost invariant under
the operation of pulling back by $F_{w}$ followed by a suitable normalization,
as stated in the following lemma.

\begin{lemma}\label{l:GetaC-scaling}
Let $(\theta,\mu)\in\attainsss(\eta,C)$, $w\in W_{*}$ and define
$(\theta_{w},\mu_{w})\in \contfunc(K\times K)\times\mathcal{P}(K)$ by
\begin{equation}\label{e:GetaC-scaling}
\theta_{w}(x,y):=\frac{\theta(F_{w}(x),F_{w}(y))}{\sqrt{r_{w}\mu(K_{w})}}\qquad\textrm{and}\qquad
	\mu_{w}(A):=\frac{\mu(F_{w}(A))}{\mu(K_{w})}.
\end{equation}
Then $(\theta_{w},\mu_{w})\in\attainsss(c_{R_{\mathcal{E}}}^{-1}\eta,C)$.
\end{lemma}

\begin{proof}
It is immediate from $(\theta,\mu)\in\attainsss(\eta,C)$ that $\theta_{w}$ and $\mu_{w}$
can be defined by \eqref{e:GetaC-scaling} and are a metric and a Borel probability
measure on $K$, respectively, and that $(\theta_{w},\mu_{w})$ satisfies the inequalities
in \eqref{e:GetaC-pcf}. Moreover, for any $x,y,z\in K$ and $t\in(0,\infty)$ with
$\theta_{w}(x,y)\leq t\theta_{w}(x,z)$, we have $\theta(F_{w}(x),F_{w}(y))\leq t\theta(F_{w}(x),F_{w}(z))$
and hence it follows from the $\eta$-quasisymmetry of $\theta$ to $R_{\mathcal{E}}$ and
Lemma \ref{l:RE-scaling-pcf} that
\begin{equation*}
c_{R_{\mathcal{E}}}r_{w}R_{\mathcal{E}}(x,y)
	\leq R_{\mathcal{E}}(F_{w}(x),F_{w}(y))
	\leq\eta(t)R_{\mathcal{E}}(F_{w}(x),F_{w}(z))
	\leq\eta(t)r_{w}R_{\mathcal{E}}(x,z)
\end{equation*}
and thus that $R_{\mathcal{E}}(x,y)\leq c_{R_{\mathcal{E}}}^{-1}\eta(t)R_{\mathcal{E}}(x,z)$,
proving the $c_{R_{\mathcal{E}}}^{-1}\eta$-quasisymmetry of $\theta_{w}$ to $R_{\mathcal{E}}$.
\end{proof}

The operation as in \eqref{e:GetaC-scaling} of pulling back Borel measures on $K$ by
$F_{w}$ is compatible with the analogous operation on $\mathcal{F}$
(recall \eqref{e:SSDF-domain} and \eqref{e:SSDF-form}) in the following sense.

\begin{lemma}\label{l:scaling-energy-meas}
Let $u\in\mathcal{F}$ and $w\in W_{*}$. Then
$\Gamma(u,u)(F_{w}(A))=r_{w}^{-1}\Gamma(u\circ F_{w},u\circ F_{w})(A)$
for any Borel subset $A$ of $K$, and in particular
$\Gamma(u,u)(K_{w})=r_{w}^{-1}\mathcal{E}(u\circ F_{w},u\circ F_{w})$.
Moreover, if $\Gamma(u,u)(K_{w})>0$, then for any Borel subset $A$ of $K$,
\begin{equation}\label{e:scaling-energy-meas}
\frac{\Gamma(u,u)(F_{w}(A))}{\Gamma(u,u)(K_{w})}=\Gamma(u_{w},u_{w})(A),
	\quad\textrm{where}\quad u_{w}
	:=\mathcal{E}(u\circ F_{w},u\circ F_{w})^{-1/2}u\circ F_{w}.
\end{equation}
\end{lemma}

\begin{proof}
Since $F_{w}\colon K\to K_{w}$ is a homeomorphism, the first assertion is easily seen to be
equivalent to \cite[Lemma 4-(i)]{HN}, and the second follows by choosing $A=K$ in the first.
Furthermore if $\Gamma(u,u)(K_{w})>0$, then we see from the first and second assertions
and the bilinearity of $\Gamma(f,g)$ in $f,g\in\mathcal{F}$ that
for any Borel subset $A$ of $K$,
\begin{equation*}
\frac{\Gamma(u,u)(F_{w}(A))}{\Gamma(u,u)(K_{w})}
	=\frac{r_{w}^{-1}\Gamma(u\circ F_{w},u\circ F_{w})(A)}{r_{w}^{-1}\mathcal{E}(u\circ F_{w},u\circ F_{w})}
	=\Gamma(u_{w},u_{w})(A),
\end{equation*}
completing the proof.
\end{proof}

At this stage, we can already give the proof of Proposition \ref{p:H0GetaC-compact-pcf} as follows.

\begin{proof}[Proof of Proposition \textup{\ref{p:H0GetaC-compact-pcf}}]
Recall \eqref{e:H0Z-pcf} for $\mathcal{H}_{0}(\mathcal{Z})$ and
\eqref{e:H0tildeZ-pcf} for $\widetilde{\mathcal{H}}_{0}(\mathcal{Z})$.
\begin{enumerate}[label=\textup{(\arabic*)},align=left,leftmargin=*,topsep=5pt,parsep=0pt,itemsep=2pt]
\item[\ref{it:H0GetaC-compact-scaling-pcf}]
	Let $h\in\mathcal{H}_{0}(\attainsss(\eta,C))$ and set $\mu:=\Gamma(h,h)$, so that
	$(\theta,\mu)\in \attainsss(\eta,C)$ for some metric $\theta$ on $K$. Let $w\in W_{*}$,
	set $h_{w}:=\mathcal{E}(h\circ F_{w},h\circ F_{w})^{-1/2}h\circ F_{w}$ and define
	$(\theta_{w},\mu_{w})\in \contfunc(K\times K)\times\mathcal{P}(K)$ by \eqref{e:GetaC-scaling}.
	Then $h_{w}\in\mathcal{H}_{0}$ by Proposition \ref{p:harmonic-pcf}-\ref{it:harmonic-pcf},
	$\Gamma(h_{w},h_{w})=\mu_{w}$ by Lemma \ref{l:scaling-energy-meas},
	$(\theta_{w},\mu_{w})\in\attainsss(c_{R_{\mathcal{E}}}^{-1}\eta,C)$ by Lemma \ref{l:GetaC-scaling}
	and thus $h_{w}\in\mathcal{H}_{0}(\attainsss(c_{R_{\mathcal{E}}}^{-1}\eta,C))$.
\item[\ref{it:H0GetaC-compact-pcf}]Let $\{h_{n}\}_{n\in\mathbb{N}}\subset\mathcal{H}_{0}(\attainsss(\eta,C))$,
	so that $\{\Gamma(h_{n},h_{n})\}_{n\in\mathbb{N}}\subset\mathcal{P}(K)$ and hence
	$\{h_{n}\}_{n\in\mathbb{N}}\subset\{h\in\mathcal{H}_{0}\mid(\Gamma(h,h)(K)=)\,\mathcal{E}(h,h)=1\}$.
	Since $\mathcal{H}_{0}/\mathbb{R}\one_{K}$ is a finite-dimensional linear subspace of
	$(\mathcal{F}/\mathbb{R}\one_{K},\mathcal{E})$ by Proposition \ref{p:harmonic-pcf}-\ref{it:harmonic-extension-pcf},
	$\{h\in\mathcal{H}_{0}/\mathbb{R}\one_{K}\mid\mathcal{E}(h,h)=1\}$ is compact in norm in
	$(\mathcal{F}/\mathbb{R}\one_{K},\mathcal{E})$, and thus there exist $h\in\mathcal{H}_{0}$
	and a strictly increasing sequence $\{n_{k}\}_{k\in\mathbb{N}}\subset\mathbb{N}$ such that
	$\mathcal{E}(h,h)=1$ and $\lim_{k\to\infty}\mathcal{E}(h-h_{n_{k}},h-h_{n_{k}})=0$. Then
	$\Gamma(h,h)\in\mathcal{P}(K)$ and, noting that
	$\mathcal{F}\times\mathcal{F}\ni(u,v)\mapsto\Gamma(u,v)$ is bilinear, symmetric
	and non-negative definite, for any Borel subset $A$ of $K$ we have
	\begin{equation}\label{e:energy-meas-converge}
	\begin{split}
	\bigl\lvert\Gamma(h,h&)(A)^{1/2}-\Gamma(h_{n_{k}},h_{n_{k}})(A)^{1/2}\bigr\rvert^{2}
		\leq\Gamma(h-h_{n_{k}},h-h_{n_{k}})(A)\\
	&\leq\Gamma(h-h_{n_{k}},h-h_{n_{k}})(K)
		=\mathcal{E}(h-h_{n_{k}},h-h_{n_{k}})\xrightarrow{k\to\infty}0.
	\end{split}
	\end{equation}
	In particular, $\{\Gamma(h_{n_{k}},h_{n_{k}})\}_{k\in\mathbb{N}}$
	converges to $\Gamma(h,h)$ in $\mathcal{P}(K)$, and now it follows from
	$\{h_{n_{k}}\}_{k\in\mathbb{N}}\subset\mathcal{H}_{0}(\attainsss(\eta,C))$
	and Corollary \ref{c:GetaC-compact} that $h\in\mathcal{H}_{0}(\attainsss(\eta,C))$,
	which together with $\lim_{k\to\infty}\mathcal{E}(h-h_{n_{k}},h-h_{n_{k}})=0$
	proves that $\widetilde{\mathcal{H}}_{0}(\attainsss(\eta,C))$ is (sequentially)
	compact in norm in $(\mathcal{F}/\mathbb{R}\one_{K},\mathcal{E})$.
	\qedhere
\end{enumerate}
\end{proof}

We continue with the preparation for the proof of Theorem \ref{t:attain-harmonic-func}.
Recall that $\mu$ is a minimal energy-dominant measure of $(\mathcal{E},\mathcal{F})$
for any $(\theta,\mu)\in\attainsss$ by Proposition \ref{p:condition-PHI-pcf} and
Proposition \ref{p:metmeas}-\ref{it:conseq-PHI2-meas}, and hence in particular that $\Gamma(u,u)$ is
absolutely continuous with respect to $\mu$ for any $(\theta,\mu)\in\attainsss$
and any $u\in\mathcal{F}$. The following lemma is a special case of
the well-known Lebesgue differentiation theorem.

\begin{lemma}\label{l:Lebesgue-points}
Let $(\theta,\mu)\in\attainsss$, $u\in\mathcal{F}$ and set $f:=d\Gamma(u,u)/d\mu$.
Then $\mu$-a.e.\ $x\in K$ is an \emph{$(R_{\mathcal{E}},\mu)$-Lebesgue point}
for $f$, i.e., satisfies
\begin{equation}\label{e:Lebesgue-points}
\lim_{s\downarrow 0}\frac{1}{\mu(B_{R_{\mathcal{E}}}(x,s))}\int_{B_{R_{\mathcal{E}}}(x,s)}\lvert f(y)-f(x)\rvert\,d\mu(y)=0.
\end{equation}
\end{lemma}

\begin{proof}
We have $R_{\mathcal{E}}\in\mathcal{J}(K,\theta)$ by $(\theta,\mu)\in\attainsss$
and hence $\theta\in\mathcal{J}(K,R_{\mathcal{E}})$ by \eqref{e:cgauge}.
Thus \eqref{e:ann1} holds with $(d_{1},d_{2})=(R_{\mathcal{E}},\theta)$, and
$(K,\theta,\mu)$ is \ref{VD} by Proposition \ref{p:condition-PHI-pcf} and
Theorem \ref{t:phichar}, which together imply that $(K,R_{\mathcal{E}},\mu)$ is \ref{VD}.
Now since $\contfunc (K)$ is dense in $L^{1}(K,\mu)$ (see, e.g., \cite[Theorem 3.14]{Rud})
and $f\in L^{1}(K,\mu)$, the claim follows by Lebesgue's differentiation theorem
\cite[(2.8)]{Hei} for $(K,R_{\mathcal{E}},\mu)$, which requires $(K,R_{\mathcal{E}},\mu)$ to be \ref{VD}.
\end{proof}

\begin{lemma}\label{l:Lebesgue-pt-scaling}
Let $(\theta,\mu)\in\attainsss$, $u\in\mathcal{F}$, let $f\colon K\to[0,\infty)$
be a Borel measurable $\mu$-version of $d\Gamma(u,u)/d\mu$ and let $x\in K$ satisfy
\eqref{e:Lebesgue-points}. Then for any $\omega\in\pi^{-1}(x)$ and any $w\in W_{*}$,
\begin{equation}\label{e:Lebesgue-pt-scaling}
\lim_{n\to\infty}\frac{\Gamma(u,u)(K_{[\omega]_{n}w})}{\mu(K_{[\omega]_{n}w})}=f(x).
\end{equation}
\end{lemma}

\begin{proof}
Let $\omega\in\pi^{-1}(x)$, $w\in W_{*}$, $n\in\mathbb{N}\cup\{0\}$ and set
$s_{n}:=\diam_{R_{\mathcal{E}}}(K_{[\omega]_{n}})$. Then by \eqref{e:ELm-pcf},
\eqref{e:meas-cell-REball-comparable}, and \ref{VD} of $(K,R_{\mathcal{E}},\mu)$
noted in the above proof of Lemma \ref{l:Lebesgue-points}, we have
\begin{equation}\label{e:meas-Komegamw-REball-comparable}
\mu(K_{[\omega]_{n}w})\geq c^{\lvert w\rvert}\mu(K_{[\omega]_{n}})
	\geq c^{\lvert w\rvert}c'\mu(B_{R_{\mathcal{E}}}(x,2s_{n}))
\end{equation}
for some $c,c'\in(0,\infty)$ determined solely by $\mathcal{L},(D,\mathbf{r}),(\theta,\mu)$.
Now since $K_{[\omega]_{n}w}\subset K_{[\omega]_{n}}\subset B_{R_{\mathcal{E}}}(x,2s_{n})$
and $\lim_{n\to\infty}s_{n}=0$ by Lemma \ref{l:RE-scaling-pcf}, it follows from
\eqref{e:meas-Komegamw-REball-comparable} and \eqref{e:Lebesgue-points} that
\begin{align*}
\biggl\lvert\frac{\Gamma(u,u)(K_{[\omega]_{n}w})}{\mu(K_{[\omega]_{n}w})}&-f(x)\biggr\rvert
	=\biggl\lvert\frac{1}{\mu(K_{[\omega]_{n}w})}\int_{K_{[\omega]_{n}w}}(f(y)-f(x))\,d\mu(y)\biggr\rvert\\
&\leq\frac{1}{\mu(K_{[\omega]_{n}w})}\int_{K_{[\omega]_{n}w}}\lvert f(y)-f(x)\rvert\,d\mu(y)\\
&\leq\frac{(c^{\lvert w\rvert}c')^{-1}}{\mu(B_{R_{\mathcal{E}}}(x,2s_{n}))}\int_{B_{R_{\mathcal{E}}}(x,2s_{n})}\lvert f(y)-f(x)\rvert\,d\mu(y)
	\xrightarrow{m\to\infty}0,
\end{align*}
proving \eqref{e:Lebesgue-pt-scaling}.
\end{proof}

Taking an $(R_{\mathcal{E}},\mu)$-Lebesgue point $x\in K$ for $d\Gamma(u,u)/d\mu$ with
$(d\Gamma(u,u)/d\mu)(x)>0$ and considering the enlargements of infinitesimally small
cells containing $x$ to the original scale as in Lemmas \ref{l:GetaC-scaling}
and \ref{l:scaling-energy-meas}, we arrive at the following proposition.

\begin{proposition}\label{p:Lebesgue-pt-scaling}
Let $(\theta,\mu)\in\attainsss$, $u\in\mathcal{F}$, let $f\colon K\to[0,\infty)$
be a Borel measurable $\mu$-version of $d\Gamma(u,u)/d\mu$, let $x\in K$
satisfy \eqref{e:Lebesgue-points} and $f(x)>0$, and let $\omega\in\pi^{-1}(x)$.
For each $n\in\mathbb{N}\cup\{0\}$,
define $\mu_{n}:=\mu_{[\omega]_{n}}\in\mathcal{P}(K)$ by \eqref{e:GetaC-scaling}
with $w=[\omega]_{n}$ and, noting that $\Gamma(u,u)(K_{[\omega]_{n}})>0$ by
\eqref{e:Lebesgue-pt-scaling}, define $u_{n}:=u_{[\omega]_{n}}\in\mathcal{F}$
by \eqref{e:scaling-energy-meas} with $w=[\omega]_{n}$. If $v\in\mathcal{F}$ and
$\{n_{k}\}_{k\in\mathbb{N}}\subset\mathbb{N}$ is strictly increasing and satisfies
$\lim_{k\to\infty}\mathcal{E}(v-u_{n_{k}},v-u_{n_{k}})=0$, then $\Gamma(v,v)\in\mathcal{P}(K)$
and $\{\mu_{n_{k}}\}_{k\in\mathbb{N}}$ converges to $\Gamma(v,v)$ in $\mathcal{P}(K)$.
\end{proposition}

\begin{proof}
Let $w\in W_{*}$. Then we see from \eqref{e:GetaC-scaling}, \eqref{e:scaling-energy-meas},
\eqref{e:Lebesgue-pt-scaling} and $f(x)\in(0,\infty)$ that
\begin{equation}\label{e:Lebesgue-pt-scaling-Kw}
\frac{\Gamma(u_{n},u_{n})(K_{w})}{\mu_{n}(K_{w})}
	=\frac{\Gamma(u,u)(K_{[\omega]_{n}w})/\mu(K_{[\omega]_{n}w})}{\Gamma(u,u)(K_{[\omega]_{n}})/\mu(K_{[\omega]_{n}})}
	\xrightarrow{n\to\infty}\frac{f(x)}{f(x)}=1.
\end{equation}
On the other hand,
the same argument as \eqref{e:energy-meas-converge} based on
$\lim_{k\to\infty}\mathcal{E}(v-u_{n_{k}},v-u_{n_{k}})=0$ yields
$\lim_{n\to\infty}\Gamma(u_{n_{k}},u_{n_{k}})(K_{w})=\Gamma(v,v)(K_{w})$,
which together with \eqref{e:Lebesgue-pt-scaling-Kw} implies that
\begin{equation}\label{e:Lebesgue-pt-converge-Gammavv}
\lim_{k\to\infty}\mu_{n_{k}}(K_{w})=\Gamma(v,v)(K_{w})
\end{equation}
and in particular that $\Gamma(v,v)(K)=\lim_{k\to\infty}\mu_{n_{k}}(K)=1$,
namely $\Gamma(v,v)\in\mathcal{P}(K)$.
Note that $\mu_{n}(F_{w}(V_{0}))=\mu(K_{[\omega]_{n}})^{-1}\mu(F_{[\omega]_{n}w}(V_{0}))=0$
for any $n\in\mathbb{N}\cup\{0\}$ by Proposition \ref{p:condition-PHI-pcf} and that
$\Gamma(v,v)(F_{w}(V_{0}))=0$ by $\#F_{w}(V_{0})<\infty$ and \cite[Theorem 4.3.8]{CF},
and recall that $K_{w}=F_{w}(K)$ is closed in $K$ and $K_{w}\setminus F_{w}(V_{0})$ is open
in $K$ as noted in the last paragraph of the proof of Proposition \ref{p:GetaC-compact}.
By using these facts and the equality
$K\setminus V_{n}=\bigcup_{v\in W_{n}}(K_{v}\setminus F_{v}(V_{0}))$, with the union disjoint,
implied by \cite[Proposition 1.3.5-(2)]{Kig01} for any $n\in\mathbb{N}$, we easily see
that the validity of \eqref{e:Lebesgue-pt-converge-Gammavv} for any $w\in W_{*}$
is equivalent to the desired convergence of $\{\mu_{n_{k}}\}_{k\in\mathbb{N}}$
to $\Gamma(v,v)$ in $\mathcal{P}(K)$, which completes the proof.
\end{proof}

Now we can conclude the proof of the main result of this subsection (Theorem \ref{t:attain-harmonic-func}).

\begin{proof}[Proof of Theorem \textup{\ref{t:attain-harmonic-func}}]
By the assumption $\attainsss(\eta,C)\not=\emptyset$ we can take
$(\theta,\mu)\in\attainsss(\eta,C)$. Let $u\in\mathcal{H}_{0}\setminus\mathbb{R}\one_{K}$,
which exists by Proposition \ref{p:harmonic-pcf}-\ref{it:harmonic-extension-pcf} and $\#V_{0}\geq 2$, and let
$f\colon K\to[0,\infty)$ be a Borel measurable $\mu$-version of $d\Gamma(u,u)/d\mu$. Then
since $\mu\bigl(f^{-1}((0,\infty))\bigr)>0$ by $\int_{K}f\,d\mu=\Gamma(u,u)(K)=\mathcal{E}(u,u)>0$,
Lemma \ref{l:Lebesgue-points} implies that there exists $x\in K$
with the properties \eqref{e:Lebesgue-points} and $f(x)>0$.
Let $\omega\in\pi^{-1}(x)$, and for each $n\in\mathbb{N}\cup\{0\}$,
as in Proposition \ref{p:Lebesgue-pt-scaling} define
$(\theta_{n},\mu_{n}):=(\theta_{[\omega]_{n}},\mu_{[\omega]_{n}})\in \contfunc(K\times K)\times\mathcal{P}(K)$
by \eqref{e:GetaC-scaling} with $w=[\omega]_{n}$ and $u_{n}:=u_{[\omega]_{n}}\in\mathcal{F}$
by \eqref{e:scaling-energy-meas} with $w=[\omega]_{n}$, so that
$\{(\theta_{n},\mu_{n})\}_{n\in\mathbb{N}\cup\{0\}}\subset\attainsss(c_{R_{\mathcal{E}}}^{-1}\eta,C)$
by Lemma \ref{l:GetaC-scaling} and
$\{u_{n}\}_{n\in\mathbb{N}\cup\{0\}}\subset\{h\in\mathcal{H}_{0}\mid\mathcal{E}(h,h)=1\}$
by Proposition \ref{p:harmonic-pcf}-\ref{it:harmonic-pcf}. Noting that
$\mathcal{H}_{0}/\mathbb{R}\one_{K}$ is a finite-dimensional linear subspace of the Hilbert space
$(\mathcal{F}/\mathbb{R}\one_{K},\mathcal{E})$ by Proposition \ref{p:harmonic-pcf}-\ref{it:harmonic-extension-pcf}
and hence that $\{h\in\mathcal{H}_{0}/\mathbb{R}\one_{K}\mid\mathcal{E}(h,h)=1\}$
is compact in norm in $(\mathcal{F}/\mathbb{R}\one_{K},\mathcal{E})$, for some
$h\in\mathcal{H}_{0}$ and some strictly increasing sequence $\{n_{k}\}_{k\in\mathbb{N}}\subset\mathbb{N}$
we have $\mathcal{E}(h,h)=1$ and $\lim_{k\to\infty}\mathcal{E}(h-u_{n_{k}},h-u_{n_{k}})=0$. Then
$\Gamma(h,h)\in\mathcal{P}(K)$ and $\{\mu_{n_{k}}\}_{k\in\mathbb{N}}$ converges to $\Gamma(h,h)$
in $\mathcal{P}(K)$ by Proposition \ref{p:Lebesgue-pt-scaling}, and it follows from this convergence,
$\{(\theta_{n_{k}},\mu_{n_{k}})\}_{k\in\mathbb{N}}\subset\attainsss(c_{R_{\mathcal{E}}}^{-1}\eta,C)$ and
Corollary \ref{c:GetaC-compact} that $(\theta_{h},\Gamma(h,h))\in\attainsss(c_{R_{\mathcal{E}}}^{-1}\eta,C)$
for some metric $\theta_{h}$ on $K$.
\end{proof}

\subsection{Examples}\label{ssec:examples}

In this subsection, we show that the infimum in \eqref{e:dcw} defining the conformal walk dimension $\dcw=2$
\emph{fails} to be attained for some concrete examples of post-critically finite self-similar sets.
In view of \eqref{e:condition-PHI-pcf-G-attained} and Theorem \ref{t:attain-harmonic-func},
for the proof of the non-attainment $\mathcal{G}(K,R_{\mathcal{E}},\measure,\mathcal{E},\mathcal{F})=\emptyset$
it suffices to verify that the conclusion of Theorem \ref{t:attain-harmonic-func} cannot hold for any
$h\in\mathcal{H}_{0}$. We start with providing a further characterization of
$h\in\mathcal{H}_{0}$ as in the conclusion of Theorem \ref{t:attain-harmonic-func}.
In the following definition and proposition, we continue to assume the setting
specified in the first paragraph of Subsection \ref{ssec:attain-harmonic-func-pcf}.

\begin{definition}\label{d:dint-pcf}
Let $\mu$ be a Borel measure on $K$. Recalling Definition \ref{d:dint}, we define the
\emph{$\mu$-intrinsic metric} $d_{\on{int}}^{\mu}\colon K\times K\to[0,\infty]$ of $(\mathcal{E},\mathcal{F})$ by
\begin{equation}\label{e:dint-pcf}
d_{\on{int}}^{\mu}(x,y):=\sup\{u(x)-u(y)\mid\textrm{$u\in\mathcal{F}$, $\Gamma(u,u)\leq\mu$}\}.
\end{equation}
\end{definition}

\begin{proposition}\label{p:attain-harmonic-func}
\begin{enumerate}[label=\textup{(\arabic*)},align=left,leftmargin=*,topsep=5pt,parsep=0pt,itemsep=2pt]
\item\label{it:attain-harmonic-func1}There exists $C\in(1,\infty)$ such that for any $h\in\mathcal{H}_{0}$ and any $w\in W_{*}$,
	\begin{equation}\label{e:DM2-pcf-dint}
	C^{-1}\bigl(\diam(K_{w},d_{\on{int}}^{\Gamma(h,h)})\bigr)^{2}
		\leq r_{w}\Gamma(h,h)(K_{w})\leq C\bigl(\diam(K_{w},d_{\on{int}}^{\Gamma(h,h)})\bigr)^{2}.
	\end{equation}
\item\label{it:attain-harmonic-func2}Let $h\in\mathcal{H}_{0}$. Then $\Gamma(h,h)\in\mathcal{G}(K,R_{\mathcal{E}},\measure,\mathcal{E},\mathcal{F})$
	if and only if $d_{\on{int}}^{\Gamma(h,h)}\in\mathcal{J}(K,R_{\mathcal{E}})$.
\item\label{it:attain-harmonic-func3}$\mathcal{G}(K,R_{\mathcal{E}},\measure,\mathcal{E},\mathcal{F})=\emptyset$
	if and only if $d_{\on{int}}^{\Gamma(h,h)}\not\in\mathcal{J}(K,R_{\mathcal{E}})$
	for any $h\in\mathcal{H}_{0}$.
\end{enumerate}
\end{proposition}

\begin{proof}
\begin{enumerate}[label=\textup{(\arabic*)},align=left,leftmargin=*,topsep=5pt,parsep=0pt,itemsep=2pt]
\item[\ref{it:attain-harmonic-func1}]Let $h\in\mathcal{H}_{0}$ and $w\in W_{*}$. Since $\lvert h(x)-h(y)\rvert\leq d_{\on{int}}^{\Gamma(h,h)}(x,y)$
	for any $x,y\in K$ by \eqref{e:dint-pcf} with $\mu=\Gamma(h,h)$, we see from
	Lemma \ref{l:scaling-energy-meas}, $h\in\mathcal{H}_{0}$, \eqref{e:harmonic2-pcf},
	\eqref{e:harmonic1-pcf} and \eqref{e:E0-pcf} that
	\begin{align*}
	r_{w}\Gamma(h,h)(K_{w})
		=\mathcal{E}^{(0)}(h\circ F_{w}\vert_{V_{0}},h\circ F_{w}\vert_{V_{0}})
		\leq C\bigl(\diam(K_{w},d_{\on{int}}^{\Gamma(h,h)})\bigr)^{2},
	\end{align*}
	where $C:=\frac{1}{2}\sum_{p,q\in V_{0},\,p\not=q}D_{pq}$.
	On the other hand, setting $C':=\diam_{R_{\mathcal{E}}}(K)$,
	for any $x,y\in K_{w}$ and any $u\in\mathcal{F}$ with $\Gamma(u,u)\leq\Gamma(h,h)$,
	by \ref{it:RF4} and Lemma \ref{l:scaling-energy-meas} we have
	\begin{align*}
	\lvert u(x)-u&(y)\rvert^{2}
		=\bigl\lvert(u\circ F_{w})(F_{w}^{-1}(x))-(u\circ F_{w})(F_{w}^{-1}(y))\bigr\rvert^{2}\\
	&\leq C'\mathcal{E}(u\circ F_{w},u\circ F_{w})
		=C'r_{w}\Gamma(u,u)(K_{w})
		\leq C'r_{w}\Gamma(h,h)(K_{w}),
	\end{align*}
	therefore taking the supremum over such $u$ yields
	$d_{\on{int}}^{\Gamma(h,h)}(x,y)^{2}\leq C'r_{w}\Gamma(h,h)(K_{w})$
	by \eqref{e:dint-pcf} with $\mu=\Gamma(h,h)$ and thus
	$\bigl(\diam(K_{w},d_{\on{int}}^{\Gamma(h,h)})\bigr)^{2}\leq C'r_{w}\Gamma(h,h)(K_{w})$,
	proving \eqref{e:DM2-pcf-dint}.
\item[\ref{it:attain-harmonic-func2}]If $\Gamma(h,h)\in\mathcal{G}(K,R_{\mathcal{E}},\measure,\mathcal{E},\mathcal{F})$ then
	$d_{\on{int}}^{\Gamma(h,h)}\in\mathcal{J}(K,R_{\mathcal{E}})$ by Proposition \ref{p:metmeas}-\ref{it:conseq-PHI2-met}
	(recall \eqref{e:gaussianadmissible}), and conversely if
	$d_{\on{int}}^{\Gamma(h,h)}\in\mathcal{J}(K,R_{\mathcal{E}})$ then
	$\Gamma(h,h)\in\mathcal{G}(K,R_{\mathcal{E}},\measure,\mathcal{E},\mathcal{F})$
	by \eqref{e:DM2-pcf-dint} and Proposition \ref{p:condition-PHI-pcf} with $\beta=2$.
\item[\ref{it:attain-harmonic-func3}]This is immediate from \ref{it:attain-harmonic-func2}
	and the fact that, by \eqref{e:condition-PHI-pcf-G-attained},
	Theorem \ref{t:attain-harmonic-func} and \eqref{e:condition-PHI-pcf-G},
	$\mathcal{G}(K,R_{\mathcal{E}},\measure,\mathcal{E},\mathcal{F})=\emptyset$ if and only if
	$\Gamma(h,h)\not\in\mathcal{G}(K,R_{\mathcal{E}},\measure,\mathcal{E},\mathcal{F})$
	for any $h\in\mathcal{H}_{0}$.\qedhere
\end{enumerate}
\end{proof}

\subsubsection{The Vicsek set}\label{sssec:Vicsek}

\begin{example}[Vicsek set]\label{exmp:Vicsek}
Set $S:=\{0,1,2,3,4\}$, define $\{q_{i}\}_{i\in S}\subset\mathbb{R}^{2}$ by
$q_{0}:=(0,0)$, $q_{1}:=(1,1)$, $q_{2}:=(-1,1)$, $q_{3}:=(-1,-1)$ and $q_{4}:=(1,-1)$,
and define $f_{i}\colon\mathbb{R}^{2}\to\mathbb{R}^{2}$ for each $i\in S$ by
$f_{i}(x):=q_{i}+\frac{1}{3}(x-q_{i})$.
Let $K$ be the \textit{self-similar set} associated with $\{f_{i}\}_{i\in S}$, i.e.,
the unique non-empty compact subset of $\mathbb{R}^{2}$ such that
$K=\bigcup_{i\in S}f_{i}(K)$, which exists and satisfies $K\subsetneq[-1,1]^{2}$
thanks to $\bigcup_{i\in S}f_{i}([-1,1]^{2})\subsetneq[-1,1]^{2}$ by \cite[Theorem 1.1.4]{Kig01},
and set $F_{i}:=f_{i}\vert_{K}$ for each $i\in S$. Then $\mathcal{L}:=(K,S,\{F_{i}\}_{i\in S})$
is a self-similar structure by \cite[Theorem 1.2.3]{Kig01} and called the
\emph{Vicsek set} (see Figure \ref{fig:Vicsek-SGs} below), and it easily follows from
$K\subset[-1,1]^{2}$ that $\mathcal{P}_{\mathcal{L}}=\{1^{\infty},2^{\infty},3^{\infty},4^{\infty}\}$
and $V_{0}=\{q_{1},q_{2},q_{3},q_{4}\}$, so that $\mathcal{L}$ is post-critically finite
and $K$ is connected. Let $d\colon K\times K\to[0,\infty)$ be the Euclidean metric on $K$
given by $d(x,y):=\lvert x-y\rvert$.

Let $r\in\bigl(0,\frac{1}{2}\bigr)$, set $\mathbf{r}=(r_{i})_{i\in S}:=(1-2r,r,r,r,r)$,
and define $D=(D_{pq})_{p,q\in V_{0}}$ and $D'=(D'_{pq})_{p,q\in V_{0}\cup\{q_{0}\}}$ by
\begin{equation*}
D_{pq}:=
	\begin{cases}
	1 & \textrm{if $p\not=q$,}\\
	-3 & \textrm{if $p=q$,}
	\end{cases}
\qquad
D'_{pq}:=
	\begin{cases}
	4 & \textrm{if $p\not=q$ and $q_{0}\in\{p,q\}$,}\\
	0 & \textrm{if $p\not=q$ and $p,q\in V_{0}$,}\\
	-16 & \textrm{if $p=q=q_{0}$,}\\
	-4 & \textrm{if $p=q\in V_{0}$.}
	\end{cases}
\end{equation*}
Then setting $\mathcal{E}'^{(0)}(u,v):=-\sum_{p,q\in V_{0}\cup\{q_{0}\}}D'_{pq}u(q)v(p)$ for
$u,v\in\mathbb{R}^{V_{0}\cup\{q_{0}\}}$, we immediately see that
$\mathcal{E}^{(0)}(u,u)=\inf_{v\in\mathbb{R}^{V_{0}\cup\{q_{0}\}},\,v\vert_{V_{0}}=u}\mathcal{E}'^{(0)}(v,v)$
for any $u\in\mathbb{R}^{V_{0}\cup\{q_{0}\}}$, which in turn easily implies that
$(D,\mathbf{r})$ is a regular harmonic structure on $\mathcal{L}$.
\end{example}

Our concern is whether $\mathcal{G}(K,R_{\mathcal{E}},\measure,\mathcal{E},\mathcal{F})=\emptyset$
for the MMD space $(K,R_{\mathcal{E}},\measure,\mathcal{E},\mathcal{F})$ resulting from this
$\mathcal{L},(D,\mathbf{r})$. We first remark that the resistance metric $R_{\mathcal{E}}$
is quasisymmetric to the Euclidean metric $d$ and that it is bi-Lipschitz equivalent to $d$ when $r=\frac{1}{3}$.

\begin{lemma}\label{lem:Vicsek-d-RE-QS}
Let $\mathcal{L}=(K,S,\{F_{i}\}_{i\in S}),d,(D,\mathbf{r}=(1-2r,r,r,r,r))$ be as in
Example \textup{\ref{exmp:Vicsek}} and let $(K,R_{\mathcal{E}},\measure,\mathcal{E},\mathcal{F})$
be the MMD space resulting from $\mathcal{L},(D,\mathbf{r})$ as introduced in
Subsection \textup{\ref{ssec:pfcsss-preliminaries}}.
Then $R_{\mathcal{E}}\in\mathcal{J}(K,d)$. Moreover,
$R_{\mathcal{E}}$ is bi-Lipschitz equivalent to $d$ if $r=\frac{1}{3}$.
\end{lemma}

\begin{proof}
To see $R_{\mathcal{E}}\in\mathcal{J}(K,d)$, we verify that $R_{\mathcal{E}}$ and $d$
satisfy the assumptions of \cite[Theorem 3.6.6]{Kig20}, i.e., that they are adapted
in the sense of \cite[Definition 2.4.7 and Proposition 2.4.8]{Kig20} and exponential
in the sense of \cite[Definitions 3.1.15 and 3.6.2]{Kig20} and that $R_{\mathcal{E}}$
is gentle with respect to $d$ in the sense of \cite[Definition 3.3.1]{Kig20}.
Indeed, $R_{\mathcal{E}}$ is adapted by Lemma \ref{l:adapted-RE-pcf} and
exponential by Lemma \ref{l:RE-scaling-pcf}, $d$ is obviously exponential, and
$d$ is adapted since Lemma \ref{l:adapted-RE-pcf} is easily seen to hold also with
$d,(\frac{1}{3})_{i\in S}$ in place of $R_{\mathcal{E}},\mathbf{r}$. Moreover,
if $w,v\in W_{*}$ satisfy $w\not=v$, $\lvert w\rvert=\lvert v\rvert$ and $K_{w}\cap K_{v}\not=\emptyset$,
then for some $\tau\in W_{*}$, $n\in\mathbb{N}\cup\{0\}$ and $i,j\in S\setminus\{0\}$
with $\lvert i-j\rvert=2$ we have $\{w,v\}=\{\tau 0i^{n},\tau ij^{n}\}$ and hence
$r_{w}/r_{v}\in\{r_{0}/r_{i},r_{i}/r_{0}\}=\{(1-2r)/r,r/(1-2r)\}$,
which together with Lemma \ref{l:RE-scaling-pcf} shows that $R_{\mathcal{E}}$
is gentle with respect to $d$. Thus \cite[Theorem 3.6.6]{Kig20} is applicable to
$R_{\mathcal{E}}$ and $d$ and yields $R_{\mathcal{E}}\in\mathcal{J}(K,d)$.
Lastly, if $r=\frac{1}{3}$, then we have Lemma \ref{l:adapted-RE-pcf} with $d$
in place of $R_{\mathcal{E}}$, which together with Lemma \ref{l:adapted-RE-pcf}
implies that $R_{\mathcal{E}}$ is bi-Lipschitz equivalent to $d$.
\end{proof}

In the situation of Example \ref{exmp:Vicsek}, it turns out that
the energy measure $\Gamma(h,h)$ of any $\mathcal{E}$-harmonic function
$h\in\mathcal{H}_{0}$ has its support within the union of the diagonals of
$[-1,1]^{2}$, as stated in the following proposition.

\begin{proposition}\label{prop:Vicsek}
Let $\mathcal{L}=(K,S,\{F_{i}\}_{i\in S})$ and $(D,\mathbf{r})$ be as in Example \textup{\ref{exmp:Vicsek}}
and consider the MMD space $(K,R_{\mathcal{E}},\measure,\mathcal{E},\mathcal{F})$ resulting from
$\mathcal{L},(D,\mathbf{r})$ as introduced in Subsection \textup{\ref{ssec:pfcsss-preliminaries}}.
Set $I_{i,i+2}:=\{(1-t)q_{i}+tq_{i+2}\mid t\in[0,1]\}$ for $i\in\{1,2\}$. Then
$\Gamma(h,h)\bigl(K\setminus(I_{1,3}\cup I_{2,4})\bigr)=0$ for any $h\in\mathcal{H}_{0}$.
\end{proposition}

\begin{proof}
Let $U$ be a connected component of $K\setminus(I_{1,3}\cup I_{2,4})$. Then
it is immediate from $K=\bigcup_{i\in S}F_{i}(K)\subsetneq[-1,1]^{2}$ that
$\partial_{K}U$ consists of a unique element $q_{U}\in I_{1,3}\cup I_{2,4}$,
where $\partial_{K}U$ denotes the boundary of $U$ in $K$.
Note that $u\one_{U},u\one_{K\setminus U}\in\mathcal{F}$ and
$\mathcal{E}(u\one_{U},u\one_{K\setminus U})=0$ for any $u\in\mathcal{F}$
with $u(q_{U})=0$; indeed, this is clear by \cite[Exercise 1.4.1 and Theorem 1.4.2-(ii)]{FOT}
and the strong locality of $(\mathcal{E},\mathcal{F})$ if $q_{U}\not\in\supp_{m}[u]$,
and the general case follows by approximating $u$ by
$\{u-(-\frac{1}{n})\vee(u\wedge\frac{1}{n})\}_{n\in\mathbb{N}}$
on the basis of \cite[Theorem 1.4.2-(iv)]{FOT}. Therefore
for any $u\in\mathcal{F}$, the function $u_{U}\in \contfunc (K)$ defined by
$u_{U}\vert_{U}:=u(q_{U})\one_{U}$ and $u_{U}\vert_{K\setminus U}:=u\vert_{K\setminus U}$ satisfies
$u_{U}=(u-u(q_{U}))\one_{K\setminus U}+u(q_{U})\one_{K}\in\mathcal{F}$ and $\mathcal{E}(u,u)\geq\mathcal{E}(u_{U},u_{U})$,
where the equality holds if and only if $u=u_{U}$ by \ref{it:RF1}. In particular,
any $u\in\mathcal{F}$ with $u\not=u_{U}$ fails to be $0$-harmonic by $U\cap V_{0}=\emptyset$,
so that any $h\in\mathcal{H}_{0}$ satisfies $h\vert_{U}=h(q_{U})\one_{U}$,
hence $\Gamma(h,h)(U)=0$ by \cite[Corollary 3.2.1]{FOT} and thus
$\Gamma(h,h)\bigl(K\setminus(I_{1,3}\cup I_{2,4})\bigr)=0$ since $U$ is any one
of the countably many connected components of $K\setminus(I_{1,3}\cup I_{2,4})$.
\end{proof}

\begin{corollary}\label{cor:Vicsek}
Let $\mathcal{L}=(K,S,\{F_{i}\}_{i\in S})$ and $(D,\mathbf{r})$ be as in Example \textup{\ref{exmp:Vicsek}}
and consider the MMD space $(K,R_{\mathcal{E}},\measure,\mathcal{E},\mathcal{F})$ resulting from
$\mathcal{L},(D,\mathbf{r})$ as introduced in Subsection \textup{\ref{ssec:pfcsss-preliminaries}}.
Then $\mathcal{G}(K,R_{\mathcal{E}},\measure,\mathcal{E},\mathcal{F})=\emptyset$, i.e., the
infimum in \eqref{e:dcw} is not attained for $(K,R_{\mathcal{E}},\measure,\mathcal{E},\mathcal{F})$.
\end{corollary}

\begin{proof}
For any $h\in\mathcal{H}_{0}$ we have $\Gamma(h,h)\not\in\mathcal{A}(K,R_{\mathcal{E}},\measure,\mathcal{E},\mathcal{F})$
since $\Gamma(h,h)$ does not have full support by Proposition \ref{prop:Vicsek}, and
therefore $\Gamma(h,h)\not\in\mathcal{G}(K,R_{\mathcal{E}},\measure,\mathcal{E},\mathcal{F})$
for any $h\in\mathcal{H}_{0}$, which is equivalent to
$\mathcal{G}(K,R_{\mathcal{E}},\measure,\mathcal{E},\mathcal{F})=\emptyset$
by \eqref{e:condition-PHI-pcf-G-attained}, Theorem \ref{t:attain-harmonic-func}
and \eqref{e:condition-PHI-pcf-G} as already noted in the proof of
Proposition \ref{p:attain-harmonic-func}-\ref{it:attain-harmonic-func3}.
\end{proof}

\begin{figure}[t]\centering
	\includegraphics[height=95pt]{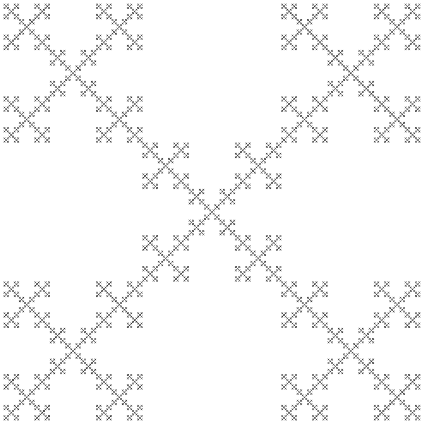}\hspace*{15pt}%
	\includegraphics[height=95pt]{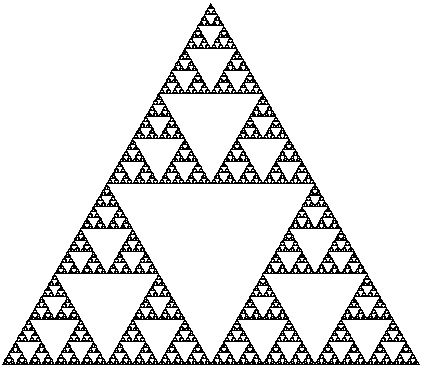}\hspace*{10pt}%
	\includegraphics[height=100pt]{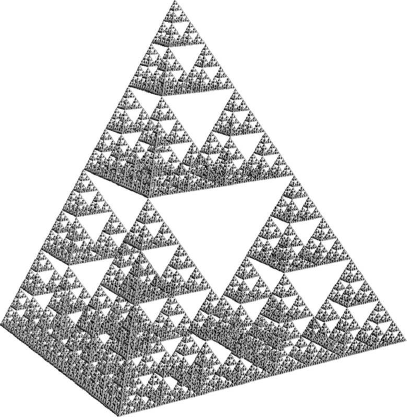}
	\caption{The Vicsek set and the $N$-dimensional Sierpi\'{n}ski gaskets ($N=2,3$)}
	\label{fig:Vicsek-SGs}
\end{figure}

\subsubsection{Higher-dimensional Sierpi\'{n}ski gaskets}\label{sssec:SGs}

\begin{example}[$N$-dimensional Sierpi\'{n}ski gasket]\label{exmp:SGs}
Let $N\in\mathbb{N}$, $N\geq 2$ and let $\{q_{k}\}_{k=0}^{N}\subset\mathbb{R}^{N}$
be the set of the vertices of a regular $N$-dimensional simplex $\triangle_{N}$, so that
$\triangle_{N}$ is the convex hull of $\{q_{k}\mid k\in\{0,1,\ldots,N\}\}$ in $\mathbb{R}^{N}$.
We further set $S:=\{0,1,\ldots,N\}$ and for each $i\in S$ define
$f_{i}\colon\mathbb{R}^{N}\to\mathbb{R}^{N}$ by $f_{i}(x):=q_{i}+\frac{1}{2}(x-q_{i})$.
Let $K$ be the self-similar set associated with $\{f_{i}\}_{i\in S}$,
which exists and satisfies $K\subsetneq\triangle_{N}$
thanks to $\bigcup_{i\in S}f_{i}(\triangle_{N})\subsetneq\triangle_{N}$
by \cite[Theorem 1.1.4]{Kig01}, and set $F_{i}:=f_{i}\vert_{K}$ for each $i\in S$.
Then $\mathcal{L}:=(K,S,\{F_{i}\}_{i\in S})$ is a self-similar structure by \cite[Theorem 1.2.3]{Kig01}
and called the \emph{$N$-dimensional (standard) Sierpi\'{n}ski gasket} (see Figure \ref{fig:Vicsek-SGs} above),
and it easily follows from $K\subset\triangle_{N}$ that $\mathcal{P}_{\mathcal{L}}=\{k^{\infty}\mid k\in S\}$
and $V_{0}=\{q_{k}\mid k\in S\}$, so that $\mathcal{L}$ is post-critically finite and $K$ is connected.
Let $d\colon K\times K\to[0,\infty)$ be the Euclidean metric on $K$ given by $d(x,y):=\lvert x-y\rvert$.

Define $D=(D_{pq})_{p,q\in V_{0}}$ by $D_{pp}:=-N$ and $D_{pq}:=1$ for $p,q\in V_{0}$ with
$p\not=q$. By the symmetry of $\mathcal{L}$ and $D$, there exists a unique $r\in(0,\infty)$
such that $(D,\mathbf{r}=(r_{i})_{i\in S})$ with $r_{i}:=r$ is a harmonic structure
on $\mathcal{L}$, and an elementary calculation shows that $r=\frac{N+1}{N+3}<1$,
so that $(D,\mathbf{r})$ is a regular harmonic structure on $\mathcal{L}$.
\end{example}

In the rest of this subsection, we fix the setting of Example \ref{exmp:SGs} and
consider the MMD space $(K,R_{\mathcal{E}},\measure,\mathcal{E},\mathcal{F})$ resulting from
$\mathcal{L},(D,\mathbf{r})$ as introduced in Subsection \ref{ssec:pfcsss-preliminaries}.
We first remark that the resistance metric $R_{\mathcal{E}}$ is bi-Lipschitz equivalent to
the power $d^{\log_{2}(1/r)}$ of the Euclidean metric $d$ and hence quasisymmetric to $d$.

\begin{lemma}\label{lem:SGs-d-RE-QS}
$R_{\mathcal{E}}$ is bi-Lipschitz equivalent to $d^{\log_{2}(1/r)}$.
In particular, $R_{\mathcal{E}}\in\mathcal{J}(K,d)$.
\end{lemma}

\begin{proof}
Lemma \ref{l:adapted-RE-pcf} with $d^{\log_{2}(1/r)}$ in place of $R_{\mathcal{E}}$
is easily seen to hold and, in combination with Lemma \ref{l:adapted-RE-pcf},
immediately implies the assertions.
\end{proof}

The following result, which is essentially due to Kigami \cite{Kig08},
was the starting point of the whole study of the present paper.

\begin{theorem}[\cite{Kig08,Kaj12}]\label{thm:SG2-attained}
Assume that $N=2$, and let $h_{1},h_{2}\in\mathcal{H}_{0}$ satisfy
$\mathcal{E}(h_{1},h_{1})=\mathcal{E}(h_{2},h_{2})=1$ and $\mathcal{E}(h_{1},h_{2})=0$. Then
$d_{\on{int}}^{\Gamma(h_{1},h_{1})+\Gamma(h_{2},h_{2})},d_{\on{int}}^{\Gamma(h_{1},h_{1})}\in\mathcal{J}(K,R_{\mathcal{E}})$ and
$\Gamma(h_{1},h_{1})+\Gamma(h_{2},h_{2}),\Gamma(h_{1},h_{1})\in\mathcal{G}(K,R_{\mathcal{E}},\measure,\mathcal{E},\mathcal{F})$.
\end{theorem}

\begin{proof}
Set $\mu_{h_{1},h_{2}}:=\Gamma(h_{1},h_{1})+\Gamma(h_{2},h_{2})$ and let
$\mu\in\{\mu_{h_{1},h_{2}},\Gamma(h_{1},h_{1})\}$.
As in the proof of Lemma \ref{lem:Vicsek-d-RE-QS} above, to see
$d_{\on{int}}^{\mu}\in\mathcal{J}(K,R_{\mathcal{E}})$ we apply \cite[Theorem 3.6.6]{Kig20}.
Note that the results in
\cite[Sections 5 and 6]{Kig08} and \cite[Sections 3 and 4]{Kaj12} are applicable to $d_{\on{int}}^{\mu}$
by virtue of the identification of $d_{\on{int}}^{\mu}$ given in \cite[Theorem 4.2]{Kaj12}.
By \cite[Theorem 5.11]{Kig08} for $\mu=\mu_{h_{1},h_{2}}$ and
\cite[Proposition 3.16-(1)]{Kaj12} for $\mu=\Gamma(h_{1},h_{1})$,
$d_{\on{int}}^{\mu}$ is a metric on $K$ compatible with the original topology of $K$ and
adapted in the sense of \cite[Definition 2.4.7 and Proposition 2.4.8]{Kig20}.
Also by \cite[Theorem 5.11 and Proof of Theorem 3.2]{Kig08} for $\mu=\mu_{h_{1},h_{2}}$ and
\cite[Proposition 3.16-(1) and Lemma 3.9]{Kaj12} for $\mu=\Gamma(h_{1},h_{1})$,
$d_{\on{int}}^{\mu}$ is exponential in the sense of \cite[Definitions 3.1.15 and 3.6.2]{Kig20}
and gentle with respect to $d$ in the sense of \cite[Definition 3.3.1]{Kig20}.
Clearly $d$ is exponential in the sense of \cite{Kig20}, and is also adapted in the
sense of \cite{Kig20} by Lemma \ref{l:adapted-RE-pcf} with $d^{\log_{2}(5/3)}$ in
place of $R_{\mathcal{E}}$ mentioned in the above proof of Lemma \ref{lem:SGs-d-RE-QS}.
Thus \cite[Theorem 3.6.6]{Kig20} is applicable to
$d_{\on{int}}^{\mu}$ and $d$ and shows, together with Lemma \ref{lem:SGs-d-RE-QS} and
\eqref{e:cgauge}, that $d_{\on{int}}^{\mu}\in\mathcal{J}(K,d)=\mathcal{J}(K,R_{\mathcal{E}})$.
Finally, $(K,d_{\on{int}}^{\mu},\mu,\mathcal{E},\mathcal{F})$ satisfies
\ref{VD} and \hyperlink{hke}{$\on{HKE(2)}$} by \cite[Theorems 6.2 and 6.3]{Kig08}
for $\mu=\mu_{h_{1},h_{2}}$ and by \cite[Theorem 3.19 and Corollary 4.3]{Kaj12}
for $\mu=\Gamma(h_{1},h_{1})$, and it thus satisfies \hyperlink{phi}{$\on{PHI(2)}$} by
Theorem \ref{t:phichar}, proving $\mu\in\mathcal{G}(K,R_{\mathcal{E}},\measure,\mathcal{E},\mathcal{F})$.
\end{proof}

One of the key observations for the validity of Theorem \ref{thm:SG2-attained} is that
the energy measures of harmonic functions are volume doubling with respect to the resistance
metric $R_{\mathcal{E}}$ (or equivalently, with respect to the Euclidean metric $d$ on $K$),
which in fact extends to the $N$-dimensional Sierpi\'{n}ski gasket with $N\geq 3$ as follows.

\begin{proposition}\label{prop:energy-meas-VD-SGs}
$\bigl(K,R_{\mathcal{E}},\Gamma(h,h)\bigr)$ is \ref{VD} for any $h\in\mathcal{H}_{0}\setminus\mathbb{R}\one_{K}$.
More generally, if $\{h_{n}\}_{n\in\mathbb{N}}\subset\mathcal{H}_{0}$ satisfies
$\sum_{n\in\mathbb{N}}\mathcal{E}(h_{n},h_{n})\in(0,\infty)$, then
$\bigl(K,R_{\mathcal{E}},\sum_{n\in\mathbb{N}}\Gamma(h_{n},h_{n})\bigr)$ is \ref{VD}.
\end{proposition}

\begin{proof}
We follow \cite[Proof of Theorem 3.2]{Kig08}. Let $h\in\mathcal{H}_{0}\setminus\mathbb{R}\one_{K}$.
As noted in the second paragraph of the proof of Proposition \ref{p:condition-PHI-pcf},
$\bigl(K,R_{\mathcal{E}},\Gamma(h,h)\bigr)$ is \ref{VD} if and only if
\eqref{e:ELm-pcf} and \eqref{e:GE-pcf} with $\mu=\Gamma(h,h)$ hold.
To verify \eqref{e:ELm-pcf} and \eqref{e:GE-pcf}, recalling Proposition \ref{p:harmonic-pcf}-\ref{it:harmonic-extension-pcf}
and \ref{it:RF1}, for each $j\in S$ we choose a basis
$\{h^{j}_{k}\}_{k=0}^{N}$ of the linear space $\mathcal{H}_{0}$ such that
$h^{j}_{0}=\one_{K}$, $h^{j}_{1}\vert_{V_{0}}=N^{-1/2}\one_{V_{0}\setminus\{q_{j}\}}$,
$h^{j}_{k}(q_{j})=0$ and $\sum_{q\in V_{0}}h^{j}_{k}(q)=0$ for any $k\in\{2,\ldots,N\}$ and
$\{h^{j}_{k}\}_{k=2}^{N}$ is orthonormal in $(\mathcal{H}_{0}/\mathbb{R}\one_{K},\mathcal{E})$.
Then $\{h^{j}_{k}\}_{k=1}^{N}$ is also orthonormal in $(\mathcal{H}_{0}/\mathbb{R}\one_{K},\mathcal{E})$
by \eqref{e:harmonic1-pcf}, and it easily follows by $\mathcal{H}_{0}\subset\mathcal{H}_{1}$,
\eqref{e:harmonic2-pcf} and solving \cite[(3.2.1)]{Kig01} that
$h^{j}_{1}\circ F_{j}=\frac{N+1}{N+3}h^{j}_{1}$ and that
$h^{j}_{k}\circ F_{j}=\frac{1}{N+3}h^{j}_{k}$ for any $k\in\{2,\ldots,N\}$.
In particular, for each $w\in W_{*}$, a linear map
$F_{w}^{*}\colon\mathcal{H}_{0}/\mathbb{R}\one_{K}\to\mathcal{H}_{0}/\mathbb{R}\one_{K}$
is defined by $F_{w}^{*}(u+\mathbb{R}\one_{K}):=u\circ F_{w}+\mathbb{R}\one_{K}$ and is invertible, and we set
$h_{w}:=\mathcal{E}(h\circ F_{w},h\circ F_{w})^{-1/2}h\circ F_{w}\in\mathcal{H}_{0}\setminus\mathbb{R}\one_{K}$.

Let $(w,j)\in W_{*}\times S$. Then
$h_{w}=\sum_{k=0}^{N}a_{k}h^{j}_{k}$ for some $(a_{k})_{k=0}^{N}\in\mathbb{R}^{N+1}$.
Since $h_{w}\circ F_{j}=a_{0}h^{j}_{0}+\frac{N+1}{N+3}a_{1}h^{j}_{1}+\frac{1}{N+3}\sum_{k=2}^{N}a_{k}h^{j}_{k}$
and $\sum_{k=1}^{N}a_{k}^{2}=\mathcal{E}(h_{w},h_{w})=1$, by Lemma \ref{l:scaling-energy-meas} we have
\begin{equation}\label{e:ELm-SGs-proof}
\begin{split}\mspace{-2.5mu}
\frac{\frac{N+1}{N+3}\Gamma(h,h)(K_{wj})}{\Gamma(h,h)(K_{w})}
	&=\mathcal{E}(h_{w}\circ F_{j},h_{w}\circ F_{j})\\
&=\frac{(N+1)^{2}a_{1}^{2}+\sum_{k=2}^{N}a_{k}^{2}}{(N+3)^{2}}
	\in\biggl[\frac{1}{(N+3)^{2}},\frac{(N+1)^{2}}{(N+3)^{2}}\biggr],
\end{split}
\end{equation}
proving \eqref{e:ELm-pcf} with $\mu=\Gamma(h,h)$.

Next, to show \eqref{e:GE-pcf}, let $w,v\in W_{*}$ satisfy $\lvert w\rvert=\lvert v\rvert$ and
$K_{w}\cap K_{v}\not=\emptyset$. We may assume $w\not=v$, so that there exist
$\tau\in W_{*}$, $n\in\mathbb{N}\cup\{0\}$ and $i,j\in S$ with $i\not=j$
such that $w=\tau ij^{n}$ and $v=\tau ji^{n}$.
Take $(b_{k})_{k=0}^{N},(c_{k})_{k=0}^{N}\in\mathbb{R}^{N+1}$ such that
$h_{\tau}\circ F_{i}=\sum_{k=0}^{N}b_{k}h^{j}_{k}$ and
$h_{\tau}\circ F_{j}=\sum_{k=0}^{N}c_{k}h^{i}_{k}$.
Then $b_{0}=h_{\tau}(F_{i}(q_{j}))=h_{\tau}(F_{j}(q_{i}))=c_{0}$
by $F_{i}(q_{j})=F_{j}(q_{i})$, $Nb_{0}+N^{1/2}b_{1}+Nc_{0}+N^{1/2}c_{1}=2Nb_{0}$
by the harmonicity of $h_{\tau}$ at $F_{i}(q_{j})$ (see \cite[(3.2.1)]{Kig01})
and thus $b_{1}=-c_{1}$. Moreover, by \eqref{e:ELm-SGs-proof},
\begin{equation}\label{e:Gammahtau}
\frac{\sum_{k=1}^{N}b_{k}^{2}}{\sum_{k=1}^{N}c_{k}^{2}}
	=\frac{\mathcal{E}(h_{\tau}\circ F_{i},h_{\tau}\circ F_{i})}{\mathcal{E}(h_{\tau}\circ F_{j},h_{\tau}\circ F_{j})}
	\leq\frac{(N+1)^{2}/(N+3)^{2}}{1/(N+3)^{2}}=(N+1)^{2}.
\end{equation}
Since $h_{\tau}\circ F_{ij^{n}}=b_{0}h^{j}_{0}+\bigl(\frac{N+1}{N+3}\bigr)^{n}b_{1}h^{j}_{1}+\bigl(\frac{1}{N+3}\bigr)^{n}\sum_{k=2}^{N}b_{k}h^{j}_{k}$
and $h_{\tau}\circ F_{ji^{n}}=c_{0}h^{i}_{0}+\bigl(\frac{N+1}{N+3}\bigr)^{n}c_{1}h^{i}_{1}+\bigl(\frac{1}{N+3}\bigr)^{n}\sum_{k=2}^{N}c_{k}h^{i}_{k}$,
we see from Lemma \ref{l:scaling-energy-meas}, $b_{1}=-c_{1}$ and \eqref{e:Gammahtau} that
\begin{align}
\frac{\Gamma(h,h)(K_{w})}{\Gamma(h,h)(K_{v})}
	&=\frac{\mathcal{E}(h_{\tau}\circ F_{ij^{n}},h_{\tau}\circ F_{ij^{n}})}{\mathcal{E}(h_{\tau}\circ F_{ji^{n}},h_{\tau}\circ F_{ji^{n}})}
	=\frac{\bigl(\frac{N+1}{N+3}\bigr)^{2n}b_{1}^{2}+\bigl(\frac{1}{N+3}\bigr)^{2n}\sum_{k=2}^{N}b_{k}^{2}}{\bigl(\frac{N+1}{N+3}\bigr)^{2n}c_{1}^{2}+\bigl(\frac{1}{N+3}\bigr)^{2n}\sum_{k=2}^{N}c_{k}^{2}} \nonumber \\
&=\frac{((N+1)^{2n}-1)b_{1}^{2}+\sum_{k=1}^{N}b_{k}^{2}}{((N+1)^{2n}-1)b_{1}^{2}+\sum_{k=1}^{N}c_{k}^{2}}\nonumber\\
&\leq\frac{((N+1)^{2n}-1)b_{1}^{2}+(N+1)^{2}\sum_{k=1}^{N}c_{k}^{2}}{((N+1)^{2n}-1)b_{1}^{2}+\sum_{k=1}^{N}c_{k}^{2}}
	\leq(N+1)^{2},
\label{e:GE-SGs-proof}
\end{align}
which proves \eqref{e:GE-pcf} with $\mu=\Gamma(h,h)$ and thereby that
$\bigl(K,R_{\mathcal{E}},\Gamma(h,h)\bigr)$ is \ref{VD}.

Finally, if $\{h_{n}\}_{n\in\mathbb{N}}\subset\mathcal{H}_{0}$ satisfies
$\sum_{n\in\mathbb{N}}\mathcal{E}(h_{n},h_{n})\in(0,\infty)$, then
$\sum_{n\in\mathbb{N}}\Gamma(h_{n},h_{n})$ is a Radon measure on $K$,
\eqref{e:ELm-pcf} and \eqref{e:GE-pcf} with $\mu=\sum_{n\in\mathbb{N}}\Gamma(h_{n},h_{n})$
hold by \eqref{e:ELm-SGs-proof} and \eqref{e:GE-SGs-proof}, and hence
$\bigl(K,R_{\mathcal{E}},\sum_{n\in\mathbb{N}}\Gamma(h_{n},h_{n})\bigr)$ is \ref{VD}.
\end{proof}

Despite Proposition \ref{prop:energy-meas-VD-SGs}, Theorem \ref{thm:SG2-attained} does
NOT extend to the case of $N\geq 3$, which is the main result of this subsection and
stated as follows; recall Proposition \ref{p:attain-harmonic-func}-\ref{it:attain-harmonic-func3}.

\begin{theorem}\label{thm:SGN-NOT-attained}
Assume that $N\geq 3$. Then $d_{\on{int}}^{\Gamma(h,h)}\not\in\mathcal{J}(K,R_{\mathcal{E}})$
for any $h\in\mathcal{H}_{0}$. Equivalently,
$\mathcal{G}(K,R_{\mathcal{E}},\measure,\mathcal{E},\mathcal{F})=\emptyset$, i.e., the
infimum in \eqref{e:dcw} is not attained for $(K,R_{\mathcal{E}},\measure,\mathcal{E},\mathcal{F})$.
\end{theorem}

Theorem \ref{thm:SGN-NOT-attained} is proved by deducing a contradiction to the
conclusion of the following proposition through taking scaling limits of
functions in $\mathcal{H}_{0}$ on the basis of Proposition \ref{p:H0GetaC-compact-pcf}.

\begin{proposition}\label{p:SGN-NOT-attained}
Let $i,j\in S$, $i\not=j$ and let $h^{i,j}\in\mathcal{H}_{0}$ be such that
$h^{i,j}\vert_{V_{0}}=\one_{\{q_{i}\}}-\one_{\{q_{j}\}}$. Then
$d_{\on{int}}^{\Gamma(h^{i,j},h^{i,j})}(x,y)=0$ for any
$k,l\in S\setminus\{i,j\}$ and any $x,y\in\{(1-t)q_{k}+tq_{l}\mid t\in[0,1]\}$.
Thus if $N\geq 3$, then $d_{\on{int}}^{\Gamma(h^{i,j},h^{i,j})}$ is not a metric on $K$ and
$\Gamma(h^{i,j},h^{i,j})\not\in\mathcal{G}(K,R_{\mathcal{E}},\measure,\mathcal{E},\mathcal{F})$.
\end{proposition}

\begin{proof}
In the same way as in the proof of Proposition \ref{prop:energy-meas-VD-SGs},
it easily follows by $\mathcal{H}_{0}\subset\mathcal{H}_{1}$,
\eqref{e:harmonic2-pcf} and solving \cite[(3.2.1)]{Kig01} that
$h^{i,j}\circ F_{k}=\frac{1}{N+3}h^{i,j}$ for any $k\in S\setminus\{i,j\}$.
Let $u\in\mathcal{F}$ satisfy $\Gamma(u,u)\leq\Gamma(h^{i,j},h^{i,j})$. Then
setting $C:=\diam_{R_{\mathcal{E}}}(K)$, we see from Lemma \ref{l:scaling-energy-meas}
that for any $w\in\bigcup_{n=1}^{\infty}(S\setminus\{i,j\})^{n}$ and any $x,y\in K_{w}$,
\begin{align}
\lvert u(x)-u(y)\rvert^{2}&=\bigl\lvert(u\circ F_{w})(F_{w}^{-1}(x))-(u\circ F_{w})(F_{w}^{-1}(y))\bigr\rvert^{2}\nonumber\\
&\leq C\mathcal{E}(u\circ F_{w},u\circ F_{w})
	=Cr_{w}\Gamma(u,u)(K_{w})\nonumber\\
&\leq Cr_{w}\Gamma(h^{i,j},h^{i,j})(K_{w})
	=C\mathcal{E}(h^{i,j}\circ F_{w},h^{i,j}\circ F_{w})\nonumber\\
&=(N+3)^{-2\lvert w\rvert}C\mathcal{E}(h^{i,j},h^{i,j}).
\label{eq:SGN-dint-varphi}
\end{align}
Now let $k,l\in S\setminus\{i,j\}$ and $x,y\in\{(1-t)q_{k}+tq_{l}\mid t\in[0,1]\}$.
Then for any $n\in\mathbb{N}$, taking $v^{(n)},\tau^{(n)}\in\{k,l\}^{n}$ such that
$x\in K_{v^{(n)}}$ and $y\in K_{\tau^{(n)}}$, from \eqref{eq:SGN-dint-varphi} we obtain
\begin{align*}
&\lvert u(x)-u(y)\rvert \\
&\leq\bigl\lvert u(x)-u(F_{v^{(n)}}(q_{k}))\bigr\rvert+\bigl\lvert u(F_{v^{(n)}}(q_{k}))-u(F_{\tau^{(n)}}(q_{l}))\bigr\rvert
	+\bigl\lvert u(F_{\tau^{(n)}}(q_{l}))-u(y)\bigr\rvert\\
&\leq 2(N+3)^{-n}C^{1/2}\mathcal{E}(h^{i,j},h^{i,j})^{1/2}+\sum_{w\in\{k,l\}^{n}}\bigl\lvert u(F_{w}(q_{k}))-u(F_{w}(q_{l}))\bigr\rvert\\
&\leq(2+2^{n})(N+3)^{-n}C^{1/2}\mathcal{E}(h^{i,j},h^{i,j})^{1/2}
	\xrightarrow{n\to\infty}0
\end{align*}
and thus $u(x)-u(y)=0$. Since $u\in\mathcal{F}$ satisfying $\Gamma(u,u)\leq\Gamma(h^{i,j},h^{i,j})$
is arbitrary, it follows from \eqref{e:dint-pcf} with $\mu=\Gamma(h^{i,j},h^{i,j})$ that
$d_{\on{int}}^{\Gamma(h^{i,j},h^{i,j})}(x,y)=0$. Finally, if $N\geq 3$, then we can choose
$k,l$ as above so that $k\not=l$, hence $x,y$ as above so that $x\not=y$,
thus $d_{\on{int}}^{\Gamma(h^{i,j},h^{i,j})}$ is not a metric on $K$
by $d_{\on{int}}^{\Gamma(h^{i,j},h^{i,j})}(x,y)=0$, and therefore
$\Gamma(h^{i,j},h^{i,j})\not\in\mathcal{G}(K,R_{\mathcal{E}},\measure,\mathcal{E},\mathcal{F})$
by Proposition \ref{p:attain-harmonic-func}-\ref{it:attain-harmonic-func2}.
\end{proof}

\begin{proof}[Proof of Theorem \textup{\ref{thm:SGN-NOT-attained}}]
Let $h\in\mathcal{H}_{0}$. By Proposition \ref{p:attain-harmonic-func}-\ref{it:attain-harmonic-func2},\ref{it:attain-harmonic-func3}
it suffices to show that $\Gamma(h,h)\not\in\mathcal{G}(K,R_{\mathcal{E}},\measure,\mathcal{E},\mathcal{F})$.
Since this is clear for $h\in\mathbb{R}\one_{K}$ by $\Gamma(h,h)(K)=\mathcal{E}(h,h)=0$,
in the rest of this proof we assume that $h\in\mathcal{H}_{0}\setminus\mathbb{R}\one_{K}$.
Take $j\in S$ such that $h(q_{j})=\min_{q\in V_{0}}h(q)$, let $\{h^{j}_{k}\}_{k=0}^{N}$
be the basis of $\mathcal{H}_{0}$ as introduced in the first paragraph of the
proof of Proposition \ref{prop:energy-meas-VD-SGs}, and set $\psi:=h^{j}_{1}$, so that
$\psi\circ F_{j}=\frac{N+1}{N+3}\psi$, for a unique $a\in\mathbb{R}$ we have
$(h-a\psi-h(q_{j})\one_{K})\circ F_{j}=\frac{1}{N+3}(h-a\psi-h(q_{j})\one_{K})$,
and $a>0$ by $h(q_{j})=\min_{q\in V_{0}}h(q)$. Setting
$h_{j^{n}}:=\mathcal{E}(h\circ F_{j^{n}},h\circ F_{j^{n}})^{-1/2}h\circ F_{j^{n}}$ for
$n\in\mathbb{N}\cup\{0\}$, we easily see from these observations and $\mathcal{E}(\psi,\psi)=1$
that $\lim_{n\to\infty}\mathcal{E}(\psi-h_{j^{n}},\psi-h_{j^{n}})=0$.

Now suppose that $\Gamma(h,h)\in\mathcal{G}(K,R_{\mathcal{E}},\measure,\mathcal{E},\mathcal{F})$,
which means, by \eqref{e:GLDr}, \eqref{e:condition-PHI-pcf-G} and \eqref{e:H0Z-pcf},
that $h_{j^{0}}\in\mathcal{H}_{0}(\attainsss(\eta,C))$ for some
$(\eta,C)\in\homeo^{+}\times(1,\infty)$. Then, with $c_{R_{\mathcal{E}}}$ as in Lemma \ref{l:RE-scaling-pcf},
$\{h_{j^{n}}\}_{n\in\mathbb{N}\cup\{0\}}\subset\mathcal{H}_{0}(\attainsss(c_{R_{\mathcal{E}}}^{-1}\eta,C))$
by Proposition \ref{p:H0GetaC-compact-pcf}-\ref{it:H0GetaC-compact-scaling-pcf} and hence
$\psi\in\mathcal{H}_{0}(\attainsss(c_{R_{\mathcal{E}}}^{-1}\eta,C))$ by
$\lim_{n\to\infty}\mathcal{E}(\psi-h_{j^{n}},\psi-h_{j^{n}})=0$, \eqref{e:H0tildeZ-pcf}
and Proposition \ref{p:H0GetaC-compact-pcf}-\ref{it:H0GetaC-compact-pcf}.
Further, letting $i\in S\setminus\{j\}$ and setting
$\varphi:=\mathcal{E}(\psi\circ F_{i},\psi\circ F_{i})^{-1/2}\bigl(\psi\circ F_{i}-N^{-1/2}\frac{N+2}{N+3}\one_{K}\bigr)$,
we would have $\varphi=(2N+2)^{-1/2}h^{i,j}$ with $h^{i,j}$ as in Proposition \ref{p:SGN-NOT-attained} by
$\mathcal{H}_{0}\subset\mathcal{H}_{1}$, \eqref{e:harmonic2-pcf} and solving \cite[(3.2.1)]{Kig01},
$\varphi\in\mathcal{H}_{0}(\attainsss(c_{R_{\mathcal{E}}}^{-2}\eta,C))$
by $\psi\in\mathcal{H}_{0}(\attainsss(c_{R_{\mathcal{E}}}^{-1}\eta,C))$
and Proposition \ref{p:H0GetaC-compact-pcf}-\ref{it:H0GetaC-compact-scaling-pcf}, and thus
$\Gamma(h^{i,j},h^{i,j})=(2N+2)\Gamma(\varphi,\varphi)\in\mathcal{G}(K,R_{\mathcal{E}},\measure,\mathcal{E},\mathcal{F})$
by \eqref{e:H0Z-pcf}, \eqref{e:GLDr} and \eqref{e:condition-PHI-pcf-G}. This would be
a contradiction to Proposition \ref{p:SGN-NOT-attained} and thereby proves that
$\Gamma(h,h)\not\in\mathcal{G}(K,R_{\mathcal{E}},\measure,\mathcal{E},\mathcal{F})$.
\end{proof}

We conclude this subsection with the following theorem, which is an easy
consequence of the conjunction of Proposition \ref{prop:energy-meas-VD-SGs},
\cite[Corollary 15.12]{Kig12} and Theorem \ref{thm:SGN-NOT-attained}.

\begin{theorem}\label{thm:SGN-NOT-HKE2}
Assume that $N\geq 3$, let $\{h_{n}\}_{n\in\mathbb{N}}\subset\mathcal{H}_{0}$
satisfy $\sum_{n\in\mathbb{N}}\mathcal{E}(h_{n},h_{n})\in(0,\infty)$ and set
$\mu:=\sum_{n\in\mathbb{N}}\Gamma(h_{n},h_{n})$. Then there does not exist a metric
$\theta$ on $K$ compatible with the original topology of $K$ such that the MMD space
$(K,\theta,\mu,\mathcal{E},\mathcal{F})$ satisfies \hyperlink{hke}{$\on{HKE}(2)$}.
\end{theorem}

\begin{proof}
Since $(K,R_{\mathcal{E}},\mu)$ is \ref{VD} by Proposition \ref{prop:energy-meas-VD-SGs},
$\mu$ is a Radon measure on $K$ with full support, hence
$\mu\in\mathcal{A}(K,R_{\mathcal{E}},\measure,\mathcal{E},\mathcal{F})$
by \eqref{e:admiss-pcf} with $\measure$ in place of $\mu$, and
\cite[Corollary 15.12]{Kig12} is applicable to $(K,R_{\mathcal{E}},\mu,\mathcal{E},\mathcal{F})$
and yields $\rho\in\mathcal{J}(K,R_{\mathcal{E}})$ and $\beta\in(1,\infty)$ such that
$(K,\rho,\mu,\mathcal{E},\mathcal{F})$ has a continuous heat kernel
$p^{\mu}=p^{\mu}_{t}(x,y)\colon(0,\infty)\times K\times K\to\mathbb{R}$ and satisfies
\ref{VD} and \hyperlink{hke}{$\on{HKE}(\beta)$} with ``$\mu$-a.e.\ $x,y$''
in \eqref{e:uhke} and \eqref{e:nlhke} replaced by ``any $x,y$''.

Now suppose that there existed a metric $\theta$ on $K$ compatible with the original topology of $K$
such that $(K,\theta,\mu,\mathcal{E},\mathcal{F})$ satisfied \hyperlink{hke}{$\on{HKE}(2)$}.
Then for each $t\in(0,\infty)$, we easily see from the continuity of $p^{\mu}_{t}$,
the lower semi-continuity of $\mu(B_{\theta}(\cdot,t^{1/2}))$,
the upper semi-continuity of $\mu\bigl(\overline{B}_{\theta}(\cdot,t^{1/2})\bigr)$
and \ref{VD} of $(K,\rho,\mu)$ that \hyperlink{hke}{$\on{HKE}(2)$} for
$(K,\theta,\mu,\mathcal{E},\mathcal{F})$ would also hold with the same heat kernel $p^{\mu}$
and with ``$\mu$-a.e.\ $x,y$'' in \eqref{e:uhke} and \eqref{e:nlhke} replaced by ``any $x,y$''.
Now for any $x,y\in K$ with $x\not=y$ we would obtain, first
$\mu(B_{\rho}(x,t^{1/\beta}))^{-1}\asymp p^{\mu}_{t}(x,x)\asymp\mu(B_{\theta}(x,t^{1/2}))^{-1}$
for any $t\in(0,\infty)$, then $\rho(x,y)^{\beta/2}\lesssim\theta(x,y)$
by combining \eqref{e:uhke} for $\rho,\mu,\beta$ and \eqref{e:nlhke} for $\theta,\mu,2$
with $t=(\theta(x,y)/\delta)^{2}$ for a constant $\delta\in(0,\infty)$, and
$\theta(x,y)\lesssim\rho(x,y)^{\beta/2}$ by combining \eqref{e:uhke} for $\theta,\mu,2$
and \eqref{e:nlhke} for $\rho,\mu,\beta$ with $t=(\rho(x,y)/\delta')^{\beta}$
for a constant $\delta'\in(0,\infty)$. Thus $\theta\asymp\rho^{\beta/2}$,
in particular $\theta\in\mathcal{J}(K,\rho)=\mathcal{J}(K,R_{\mathcal{E}})$
by $\rho\in\mathcal{J}(K,R_{\mathcal{E}})$ and \eqref{e:cgauge}, and
\ref{VD} of $(K,\rho,\mu)$ would imply \ref{VD} of $(K,\theta,\mu)$, whence
$(K,\theta,\mu,\mathcal{E},\mathcal{F})$ would satisfy \hyperlink{phi}{$\on{PHI(2)}$}
by Theorem \ref{t:phichar}. Therefore we would get
$\mu\in\mathcal{G}(K,R_{\mathcal{E}},\measure,\mathcal{E},\mathcal{F})$, which
would contradict Theorem \ref{thm:SGN-NOT-attained} and completes the proof.
\end{proof}

\subsection{The case of generalized Sierpi\'{n}ski carpets}\label{ssec:attain-harmonic-func-GSCs}

In this subsection, we treat the case of the canonical self-similar Dirichlet form
on an arbitrary generalized Sierpi\'{n}ski carpet and see that the arguments in
Subsection \ref{ssec:attain-harmonic-func-pcf} go through also in this case
with just slight modifications in the proofs.
We closely follow the presentation of \cite[Section 4]{Kaj14} for the
introduction of the framework of this subsection up to Theorem \ref{thm:BBKT-HK}
below, which we repeat here for the reader's convenience.

We fix the following setting throughout this subsection.
\begin{framework}\label{frmwrk:GSC}
Let $N,l\in\mathbb{N}$, $N\geq 2$, $l\geq 3$ and set $Q_{0}:=[0,1]^{N}$.
Let $S\subsetneq\{0,1,\ldots,l-1\}^{N}$ be non-empty, define
$f_{i}\colon\mathbb{R}^{N}\to\mathbb{R}^{N}$ by $f_{i}(x):=l^{-1}i+l^{-1}x$ for each $i\in S$
and set $Q_{1}:=\bigcup_{i\in S}f_{i}(Q_{0})$, so that $Q_{1}\subsetneq Q_{0}$.
Let $K$ be the self-similar set associated with $\{f_{i}\}_{i\in S}$,
i.e., the unique non-empty compact subset of $\mathbb{R}^{N}$ such that
$K=\bigcup_{i\in S}f_{i}(K)$, which exists and satisfies $K\subsetneq Q_{0}$ thanks to
$Q_{1}\subsetneq Q_{0}$ by \cite[Theorem 1.1.4]{Kig01}, and set $F_{i}:=f_{i}\vert_{K}$ for each
$i\in S$, so that $\GSC(N,l,S):=(K,S,\{F_{i}\}_{i\in S})$ is a self-similar structure
by \cite[Theorem 1.2.3]{Kig01}.
Let $d\colon K\times K\to[0,\infty)$ be the Euclidean metric on $K$ given by $d(x,y):=\lvert x-y\rvert$,
set $d_{\mathrm{f}}:=\log_{l}\# S$, and let $\measure$ be the self-similar measure on
$\GSC(N,l,S)$ with weight $(1/\#S)_{i\in S}$, i.e., the unique Borel probability measure
on $K$ such that $\measure(K_{w})=(\#S)^{-\lvert w\rvert}$ for any $w\in W_{*}$,
which exists by \cite[Propositions 1.5.8, 1.4.3, 1.4.4 and Corollary 1.4.8]{Kig01}.
\end{framework}
Recall that $d_{\mathrm{f}}$ is the Hausdorff dimension of $(K,d)$ and that
$\measure$ is a constant multiple of the $d_{\mathrm{f}}$-dimensional Hausdorff measure
on $(K,d)$; see, e.g., \cite[Proposition 1.5.8 and Theorem 1.5.7]{Kig01}.
Note that $d_{\mathrm{f}}<N$ by $S\subsetneq\{0,1,\ldots,l-1\}^{N}$.

The following definition is essentially due to Barlow and Bass \cite[Section 2]{BB99}. In what follows,
the interior of $A\subset\mathbb{R}^{N}$ in $\mathbb{R}^{N}$ is denoted by $\interior_{\mathbb{R}^{N}}(A)$.
\begin{definition}[Generalized Sierpi\'{n}ski carpet, {\cite[Subsection \textup{2.2}]{BBKT}}]\label{dfn:GSC}
$\GSC(N,l,S)$ is called a \emph{generalized Sierpi\'{n}ski carpet}
if and only if the following four conditions are satisfied:
\begin{enumerate}[label=\textup{(GSC\arabic*)},align=left,leftmargin=*,topsep=5pt,parsep=0pt,itemsep=2pt]
\item \label{it:GSC1}(Symmetry) $f(Q_{1})=Q_{1}$ for any isometry $f$ of $\mathbb{R}^{N}$ with $f(Q_{0})=Q_{0}$.
\item \label{it:GSC2}(Connectedness) $Q_{1}$ is connected.
\item \label{it:GSC3}(Non-diagonality)
	$\interior_{\mathbb{R}^{N}}\bigl(Q_{1}\cap \prod_{k=1}^{N}[(i_{k}-\varepsilon_{k})l^{-1},(i_{k}+1)l^{-1}]\bigr)$
	is either empty or connected for any $(i_{k})_{k=1}^{N}\in\mathbb{Z}^{N}$ and
	any $(\varepsilon_{k})_{k=1}^{N}\in\{0,1\}^{N}$.
\item \label{it:GSC4}(Borders included) $[0,1]\times\{0\}^{N-1}\subset Q_{1}$.
\end{enumerate}
\end{definition}
As special cases of Definition \ref{dfn:GSC}, $\GSC(2,3,S_{\mathrm{SC}})$ and
$\GSC(3,3,S_{\mathrm{MS}})$ are called the \emph{(two-dimensional standard) Sierpi\'{n}ski carpet}
and the \emph{Menger sponge}, respectively, where
$S_{\mathrm{SC}}:=\{0,1,2\}^{2}\setminus\{(1,1)\}$ and
$S_{\mathrm{MS}}:=\bigl\{(i_{1},i_{2},i_{3})\in\{0,1,2\}^{3}\bigm\vert\sum_{k=1}^{3}\one_{\{1\}}(i_{k})\leq 1\bigr\}$
(see Figure \ref{fig:GSCs} above).
\begin{figure}[t]\centering
	\includegraphics[height=80pt]{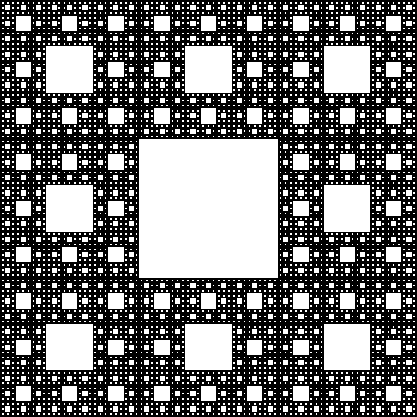}\hspace*{3pt}
	\includegraphics[height=80pt]{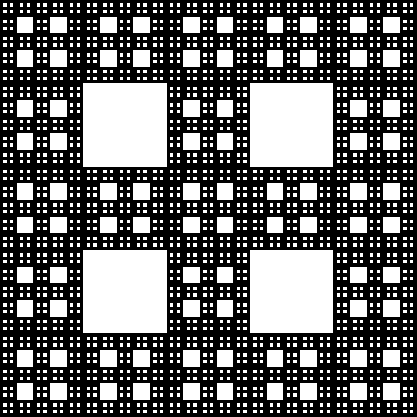}\hspace*{3pt}
	\includegraphics[height=80pt]{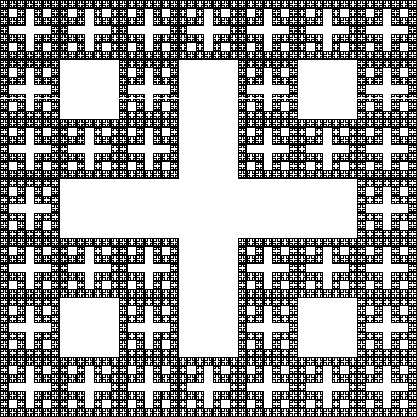}\hspace*{3pt}
	\includegraphics[height=83pt]{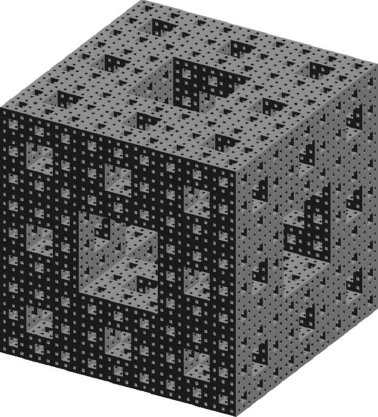}
	\caption{Sierpi\'{n}ski carpet, some other generalized Sierpi\'{n}ski
	carpets with $N=2$ and Menger sponge}\label{fig:GSCs}
\end{figure}

See \cite[Remark 2.2]{BB99} for a description of the meaning of each of the four
conditions \ref{it:GSC1}, \ref{it:GSC2}, \ref{it:GSC3} and \ref{it:GSC4}
in Definition \ref{dfn:GSC}. We remark that there are several equivalent ways of
stating the non-diagonality condition, as in the following proposition.
\begin{proposition}[{\cite[\textsection 2]{Kaj10}}]\label{prop:ND}
Set $\lvert x\rvert_{1}:=\sum_{k=1}^{N}\lvert x_{k}\rvert$ for $x=(x_{k})_{k=1}^{N}\in\mathbb{R}^{N}$.
Then \ref{it:GSC3} is equivalent to any one of the following three conditions:
\begin{enumerate}[label=\textup{(NDF)},align=left,leftmargin=*,topsep=2pt,parsep=0pt,itemsep=2pt]
\item[\hypertarget{NDN}{\textup{(ND)$_{\mathbb{N}}$}}]$\interior_{\mathbb{R}^{N}}\bigl(Q_{1}\cap \prod_{k=1}^{N}[(i_{k}-1)l^{-n},(i_{k}+1)l^{-n}]\bigr)$
	is either empty or connected for any $n\in\mathbb{N}$ and any $(i_{k})_{k=1}^{N}\in\{1,2,\ldots,l^{n}-1\}^{N}$.
\item[\hypertarget{ND2}{\textup{(ND)$_{2}$}}]$\interior_{\mathbb{R}^{N}}\bigl(Q_{1}\cap \prod_{k=1}^{N}[(i_{k}-1)l^{-2},(i_{k}+1)l^{-2}]\bigr)$
	is either empty or connected for any $(i_{k})_{k=1}^{N}\in\{1,2,\ldots,l^{2}-1\}^{N}$.
\item[\hypertarget{NDF}{\textup{(NDF)}}]For any $i,j\in S$ with $f_{i}(Q_{0})\cap f_{j}(Q_{0})\not=\emptyset$
	there exists $\{n(k)\}_{k=0}^{\lvert i-j\rvert_{1}}\subset S$ such that $n(0)=i$,
	$n(\lvert i-j\rvert_{1})=j$ and $\lvert n(k)-n(k+1)\rvert_{1}=1$ for any $k\in\{0,\ldots,\lvert i-j\rvert_{1}-1\}$.
\end{enumerate}
\end{proposition}
\begin{remark}\label{rmk:GSC}
\begin{enumerate}[label=\textup{(\arabic*)},align=left,leftmargin=*,topsep=5pt,parsep=0pt,itemsep=2pt]
\item Only the case of $n=1$ of \hyperlink{NDN}{\textup{(ND)$_{\mathbb{N}}$}} had been assumed
	in the\ original definition of generalized Sierpi\'{n}ski carpets in \cite[Section 2]{BB99},
	but Barlow, Bass, Kumagai and Teplyaev \cite{BBKT} later realized that it had been
	too weak for \cite[Proof of Theorem 3.19]{BB99} and had to be replaced by
	\hyperlink{NDN}{\textup{(ND)$_{\mathbb{N}}$}} (or equivalently, by \ref{it:GSC3}).
\item In fact, \cite[Subsection 2.2]{BBKT} assumes instead of \ref{it:GSC2}
	the seemingly stronger condition that $\interior_{\mathbb{R}^{N}}Q_{1}$ is connected,
	but it is implied by \ref{it:GSC2} and \ref{it:GSC3}
	in view of \hyperlink{NDF}{\textup{(NDF)}} in Proposition \ref{prop:ND} and is thus
	equivalent to \ref{it:GSC2} under the assumption of \ref{it:GSC3}.
\end{enumerate}
\end{remark}
In the rest of this subsection, we assume that
$\mathcal{L}:=\mathrm{GSC}(N,l,S)=(K,S,\{F_{i}\}_{i\in S})$ is a generalized
Sierpi\'{n}ski carpet. Then we easily see the following proposition and lemma.
\begin{proposition}\label{prop:GSC-V0}
Set $S_{k,\varepsilon}:=\{(i_{n})_{n=1}^{N}\in S\mid i_{k}=(l-1)\varepsilon\}$
for $k\in\{1,2,\dots,N\}$ and $\varepsilon\in\{0,1\}$. Then
$\mathcal{P}_{\mathcal{L}}=\bigcup_{k=1}^{N}(S_{k,0}^{\mathbb{N}}\cup S_{k,1}^{\mathbb{N}})$
and $\overline{V_{0}}=V_{0}=K\setminus(0,1)^{N}\not=K$.
\end{proposition}

\begin{lemma}\label{l:volume-GSC}
There exist $c_{1},c_{2}\in(0,\infty)$ such that
for any $(x,s)\in K\times(0,\diam_{d}(K)]$,
\begin{equation}\label{e:volume-GSC}
c_{1}s^{d_{\mathrm{f}}}\leq \measure(B_{d}(x,s))\leq c_{2}s^{d_{\mathrm{f}}}.
\end{equation}
\end{lemma}

We next recall some basics of the canonical self-similar Dirichlet form on $\GSC(N,l,S)$.
There are two established ways of constructing a non-degenerate $\measure$-symmetric
diffusion without killing inside on $K$, or equivalently, a non-zero, strongly local,
regular symmetric Dirichlet form on $L^{2}(K,\measure)$, one
by Barlow and Bass \cite{BB89,BB99} using the reflecting Brownian motions on the
domains approximating $K$, and the other by Kusuoka and Zhou \cite{KZ92} based on graph
approximations. It had been a long-standing open problem to prove that the constructions
in \cite{BB89,BB99} and in \cite{KZ92} give rise to the same diffusion on $K$,
which Barlow, Bass, Kumagai and Teplyaev \cite{BBKT} have finally solved
by proving the uniqueness of a non-zero conservative regular symmetric
Dirichlet form on $L^{2}(K,\measure)$ possessing certain local symmetry.
As a consequence of the results in \cite{BBKT}, after some additional arguments
in \cite{Hin13,Kaj13,Kaj22} we have the unique existence of a canonical
self-similar Dirichlet form $(\mathcal{E},\mathcal{F})$ on $L^{2}(K,\measure)$, the
heat kernel estimates \hyperlink{hke}{$\on{HKE}(d_{\mathrm{w}})$} with $d_{\mathrm{w}}>2$,
and hence \hyperlink{phi}{$\on{PHI}(d_{\mathrm{w}})$} by Lemma \ref{l:volume-GSC} and
Theorem \ref{t:phichar}, as follows.

\begin{definition}\label{dfn:GSC-isometry}
We define
\begin{equation}\label{eq:GSC-isometry}
\mathcal{I}_{0}:=\{f\vert_{K}\mid\textrm{$f$ is an isometry of $\mathbb{R}^{N}$, $f(Q_{0})=Q_{0}$}\},
\end{equation}
which forms a subgroup of the group of isometries of $(K,d)$ by virtue of \ref{it:GSC1}.
\end{definition}

\begin{theorem}[{\cite[Theorems \textup{1.2} and \textup{4.32}]{BBKT}}, {\cite[Proposition \textup{5.1}]{Hin13}}, {\cite[Proposition \textup{5.9}]{Kaj13}}]\label{thm:GSCDF}
There exists a unique (up to constant multiples of $\mathcal{E}$) regular symmetric Dirichlet
form $(\mathcal{E},\mathcal{F})$ on $L^{2}(K,\measure)$ satisfying $\mathcal{E}(u,u)>0$
for some $u\in\mathcal{F}$, $\one_{K}\in\mathcal{F}$, $\mathcal{E}(\one_{K},\one_{K})=0$,
\eqref{e:SSDF-domain}, \eqref{e:SSDF-form} for some $r\in(0,\infty)$, and the following:
\begin{enumerate}[label=\textup{(GSCDF)},align=left,leftmargin=*,topsep=4pt,parsep=0pt,itemsep=2pt]
\item\label{it:GSCDF}If $u\in \mathcal{F}\cap \contfunc (K)$ and $g\in\mathcal{I}_{0}$
	then $u\circ g\in\mathcal{F}$ and $\mathcal{E}(u\circ g,u\circ g)=\mathcal{E}(u,u)$.
\end{enumerate}
\end{theorem}

Throughout the rest of this subsection, we fix $(\mathcal{E},\mathcal{F})$ and $r$ as given in
Theorem \ref{thm:GSCDF}; note that $r$ is uniquely determined by $(\mathcal{E},\mathcal{F})$,
since $\mathcal{E}(u,u)>0$ for some $u\in\mathcal{F}\cap \contfunc (K)$ by the existence of such
$u\in\mathcal{F}$ and the denseness of $\mathcal{F}\cap \contfunc (K)$ in the Hilbert space
$(\mathcal{F},\mathcal{E}_{1}:=\mathcal{E}+\langle\cdot,\cdot\rangle_{L^{2}(K,\measure)})$.

\begin{definition}\label{dfn:GSCDF}
The regular symmetric Dirichlet form $(\mathcal{E},\mathcal{F})$ on $L^{2}(K,\measure)$ as in
Theorem \textup{\ref{thm:GSCDF}} is called the \emph{canonical Dirichlet form} on $\mathrm{GSC}(N,l,S)$,
and we set $d_{\mathrm{w}}:=\log_{l}(\# S/r)$. Note that $(\mathcal{E},\mathcal{F})$ is
also strongly local by the same argument as \cite[Proof of Lemma 3.12]{Hin05} based on
\eqref{e:SSDF-domain}, \eqref{e:SSDF-form} and $\mathcal{E}(\one_{K},\one_{K})=0$.
\end{definition}

\begin{theorem}[{\cite[Remarks 5.4-1.]{BB99}, \cite[Theorem 2.7]{Kaj22}}]\label{thm:dwSC}
$d_{\mathrm{w}}>2$.
\end{theorem}

\begin{theorem}[{\cite[Theorem 1.3]{BB99}, \cite[Theorem 4.30 and Remark 4.33]{BBKT}}]\label{thm:BBKT-HK}
There exists a (unique) continuous version $p=p_{t}(x,y)\colon(0,\infty)\times K\times K\to[0,\infty)$
of the heat kernel of $(K,d,\measure,\mathcal{E},\mathcal{F})$, and there exist
$c_{1},c_{2},c_{3},c_{4}\in(0,\infty)$ such that for any $(t,x,y)\in(0,1]\times K\times K$,
\begin{equation}\label{eq:HK-subGauss-GSC}
\frac{c_{1}}{t^{d_{\mathrm{f}}/d_{\mathrm{w}}}}\exp\biggl(-\Bigl(\frac{d(x,y)^{d_{\mathrm{w}}}}{c_{2}t}\Bigr)^{\frac{1}{d_{\mathrm{w}}-1}}\biggr)
	\leq p_{t}(x,y)
	\leq\frac{c_{3}}{t^{d_{\mathrm{f}}/d_{\mathrm{w}}}}\exp\biggl(-\Bigl(\frac{d(x,y)^{d_{\mathrm{w}}}}{c_{4}t}\Bigr)^{\frac{1}{d_{\mathrm{w}}-1}}\biggr).
\end{equation}
In particular, $(K,d,\measure,\mathcal{E},\mathcal{F})$ satisfies 
\hyperlink{hke}{$\on{HKE}(d_{\mathrm{w}})$} and \hyperlink{phi}{$\on{PHI}(d_{\mathrm{w}})$}
by Lemma \textup{\ref{l:volume-GSC}} and Theorem \textup{\ref{t:phichar}}.
\end{theorem}

The following lemma is an immediate consequence of Theorem \ref{thm:BBKT-HK}.
Recall that $\mathcal{F}$ is a Hilbert space under the inner product
$\mathcal{E}_{1}:=\mathcal{E}+\langle\cdot,\cdot\rangle_{L^{2}(K,\measure)}$.

\begin{lemma}\label{l:global-PI-GSC}
The inclusion map $\mathcal{F}\hookrightarrow L^{2}(K,\measure)$ is a compact linear
operator from $(\mathcal{F},\mathcal{E}_{1})$ to $L^{2}(K,\measure)$, and there exists
$C_{\mathrm{P}}\in(0,\infty)$ such that for any $u\in\mathcal{F}$,
\begin{equation}\label{e:global-PI}
\int_{K}\Bigl\lvert u-\int_{K}u\,d\measure\Bigr\rvert^{2}\,d\measure\leq C_{\mathrm{P}}\mathcal{E}(u,u).
\end{equation}
In particular, $\{u\in\mathcal{F}\mid\mathcal{E}(u,u)=0\}=\mathbb{R}\one_{K}$,
$(\mathcal{F}/\mathbb{R}\one_{K},\mathcal{E})$ is a Hilbert space, and the extended
Dirichlet space $\mathcal{F}_{e}$ of $(K,d,\measure,\mathcal{E},\mathcal{F})$
coincides with $\mathcal{F}$.
\end{lemma}

\begin{proof}
The compactness of the inclusion map $\mathcal{F}\hookrightarrow L^{2}(K,\measure)$
follows from Theorem \ref{thm:BBKT-HK} and \cite[Corollary 4.2.3 and Exercise 4.2]{Dav}.
The existence of $C_{\mathrm{P}}\in(0,\infty)$ satisfying \eqref{e:global-PI}
for any $u\in\mathcal{F}$ is implied by Theorems \ref{thm:BBKT-HK} and \ref{t:phichar}
as a special case of \hyperlink{pi}{$\on{PI}(d_{\mathrm{w}})$}, or more elementarily
by \cite[Theorems 4.5.1 and 4.5.3]{Dav} and the fact that
$\{u\in\mathcal{F}\mid\mathcal{E}(u,u)=0\}=\mathbb{R}\one_{K}$
by $\mathcal{E}(\one_{K},\one_{K})=0$, Theorem \ref{thm:BBKT-HK} and \cite[Theorem 2.1.11]{CF}.
The completeness of $(\mathcal{F}/\mathbb{R}\one_{K},\mathcal{E})$ is immediate from
that of $(\mathcal{F},\mathcal{E}_{1})$ and \eqref{e:global-PI}, and they also easily
imply the equality $\mathcal{F}_{e}=\mathcal{F}$ as proved in \cite[Proposition 2.9]{HiKu}.
\end{proof}

Introducing the following space of $\mathcal{E}$-harmonic functions is convenient
for our purpose in this subsection; recall \eqref{e:harmonic} for the definition
of $\mathcal{E}$-harmonicity of functions.

\begin{definition}\label{d:harmonic-GSC}
We define
\begin{equation}\label{e:harmonic-GSC}
\mathcal{H}_{0}:=\biggl\{h\in\mathcal{F}\biggm\vert
	\begin{minipage}{190pt}
	$h$ is $\mathcal{E}$-harmonic on $K\setminus V_{0}$, i.e., $\mathcal{E}(h,v)=0$
	for any $v\in\mathcal{F}\cap \contfunc (K)$ with $\supp\nolimits_{\measure}[v]\subset K\setminus V_{0}$
	\end{minipage}
	\biggr\},
\end{equation}
which is clearly a linear subspace of $\mathcal{F}$ and satisfies
$\mathbb{R}\one_{K}\subset\mathcal{H}_{0}$ by $\mathcal{E}(\one_{K},\one_{K})=0$.
Note that $\mathcal{H}_{0}$ is weakly closed in $(\mathcal{F},\mathcal{E}_{1})$ since
$\mathcal{E}(\cdot,v)$ is a bounded linear functional on $(\mathcal{F},\mathcal{E}_{1})$
for any $v\in\mathcal{F}$.
\end{definition}

As the counterpart of Proposition \ref{p:condition-PHI-pcf} for generalized Sierpi\'{n}ski carpets,
we have the following characterization of the pair $(\theta,\mu)$ of
$\theta\in\mathcal{J}(K,d)$ (recall \eqref{e:cgauge-dfn}) and
$\mu\in\mathcal{A}(K,d,\measure,\mathcal{E},\mathcal{F})$ (recall Definition \ref{d:admissible})
such that $(K,\theta,\mu,\mathcal{E}^{\mu},\mathcal{F}^{\mu})$ satisfies \ref{PHI};
see also \cite{Kig19} for related results.

\begin{theorem}\label{t:condition-PHI-GSC}
Let $\theta\in\mathcal{J}(K,d)$, let $\mu$ be a Radon measure on $K$ with full support and let $\beta\in(1,\infty)$.
Then the following conditions are equivalent:
\begin{enumerate}[label=\textup{(\alph*)},align=left,leftmargin=*,topsep=5pt,parsep=0pt,itemsep=2pt]
\item \label{it:condition-PHI-GSC-a}$\mu\in\mathcal{A}(K,d,\measure,\mathcal{E},\mathcal{F})$ and
	$(K,\theta,\mu,\mathcal{E}^{\mu},\mathcal{F}^{\mu})$ satisfies \ref{PHI}.
\item \label{it:condition-PHI-GSC-b}There exists $C\in(1,\infty)$ such that for any $w\in W_{*}$, with $r_{w}:=r^{\lvert w\rvert}$,
	\begin{equation}\label{e:DM2-GSC}
	C^{-1}(\diam_{\theta}(K_{w}))^{\beta}\leq r_{w}\mu(K_{w})\leq C(\diam_{\theta}(K_{w}))^{\beta}.
	\end{equation}
\end{enumerate}
Moreover, if either of these conditions holds, then
$\mu(F_{w}(V_{0}))=0$ for any $w\in W_{*}$ and $\mu(\{x\})=0$ for any $x\in K$.
\end{theorem}

Theorem \ref{t:condition-PHI-GSC} follows by repeating the same arguments as the proof
of Proposition \ref{p:condition-PHI-pcf}, on the basis of the following proposition concluded
from Theorem \ref{t:phichar} with the help of \cite[Lemma 5.22 and Proposition 6.16]{BCM}.

\begin{proposition}\label{p:condition-PHI-GSC}
Let $\theta\in\mathcal{J}(K,d)$, let $\mu$ be a Radon measure on $K$ with full support and let $\beta\in(1,\infty)$.
Then Theorem \textup{\ref{t:condition-PHI-GSC}-\ref{it:condition-PHI-GSC-a}} is equivalent to the following condition:
\begin{enumerate}[label=\textup{(\alph*)},align=left,leftmargin=*,topsep=5pt,parsep=0pt,itemsep=2pt]
\addtocounter{enumi}{2}
\item \label{it:condition-PHI-GSC-c}$(K,\theta,\mu)$ is \ref{VD} and
	there exists $C\in(0,\infty)$ such that for any $x,y\in K$ with $x\not=y$,
	\begin{equation}\label{e:DM2-d-GSC}
	C^{-1}\theta(x,y)^{\beta}\leq d(x,y)^{d_{\mathrm{w}}-d_{\mathrm{f}}}\mu\bigl(B_{\theta}(x,\theta(x,y))\bigr)\leq C\theta(x,y)^{\beta}.
	\end{equation}
\end{enumerate}
\end{proposition}

\begin{proof}
Note first that, if $\mu$ is admissible with respect to $(K,d,\measure,\mathcal{E},\mathcal{F})$,
then for any Borel subsets $A,B$ of $K$ with $B$ closed in $K$ and $\overline{A}\cap B=\emptyset$,
by \cite[Theorem 5.2.11]{CF} and \cite[Theorem 4.6.2 and Lemma 2.1.4]{FOT} we have
\begin{align}
&\Capa^{\mu}(A,B) \nonumber\\
&=\inf\{\mathcal{E}(f^{+}\wedge 1,f^{+}\wedge 1)\mid\textrm{$f\in\mathcal{F}^{\mu}$, $\overline{A}\cap\supp\nolimits_{\mu}[f-\one_{K}]=\emptyset=B\cap\supp\nolimits_{\mu}[f]$}\} \nonumber\\
&=\inf\biggl\{\mathcal{E}(f,f)\biggm\vert\begin{minipage}{160pt}$f\in\mathcal{F}^{\mu}$, $0\leq f\leq 1$ $\mu$-a.e.,\\$\overline{A}\cap\supp\nolimits_{\mu}[f-\one_{K}]=\emptyset=B\cap\supp\nolimits_{\mu}[f]$\end{minipage}\biggr\} \nonumber\\
&=\inf\{\mathcal{E}(f,f)\mid\textrm{$f\in\mathcal{F}(A,B)$, $0\leq f\leq 1$ $\measure$-a.e.}\} \nonumber\\
&=\inf\{\mathcal{E}(f^{+}\wedge 1,f^{+}\wedge 1)\mid f\in\mathcal{F}(A,B)\}=\Capa(A,B),
\label{e:rel-cap-TC}
\end{align}
where $\Capa^{\mu}(A,B)$ and $\Capa(A,B)$ denote the capacity between $A,B$ with respect to
$(K,\theta,\mu,\mathcal{E}^{\mu},\mathcal{F}^{\mu})$ and $(K,d,\measure,\mathcal{E},\mathcal{F})$,
respectively. Next, since $(K,d,\measure,\mathcal{E},\mathcal{F})$ satisfies
\eqref{e:volume-GSC}, \hyperlink{cap}{$\on{cap}(d_{\mathrm{w}})$} and \ref{EHI} by
Lemma \ref{l:volume-GSC}, Theorems \ref{thm:BBKT-HK} and \ref{t:phichar}, there exist
$C_{1},A_{1},A_{2}\in(1,\infty)$ with $A_{2}\geq 2$ such that for any $(x,s)\in K\times(0,\diam_{d}(K)/A_{2})$,
\begin{equation} \label{e:cap-GSC}
C_{1}^{-1}s^{d_{\mathrm{f}}-d_{\mathrm{w}}}
	\leq \Capa\bigl(B_{d}(x,s),K\setminus B_{d}(x,A_{1}s)\bigr)
	\leq C_{1}s^{d_{\mathrm{f}}-d_{\mathrm{w}}},
\end{equation}
and by the quasisymmetry of $\theta$ to $d$ and Lemma \ref{l:EHI-TC-QS}
we have \ref{EHI} for $(K,\theta,\measure,\mathcal{E},\mathcal{F})$, as well as for
$(K,\theta,\mu,\mathcal{E}^{\mu},\mathcal{F}^{\mu})$ provided
$\mu\in\mathcal{A}(K,d,\measure,\mathcal{E},\mathcal{F})$.

To prove the desired equivalence, assume Theorem \ref{t:condition-PHI-GSC}-\ref{it:condition-PHI-GSC-a},
so that $(K,\theta,\mu,\mathcal{E}^{\mu},\mathcal{F}^{\mu})$ satisfies \ref{VD} and
\ref{eq:capbeta} by Theorem \ref{t:phichar} and therefore
in view of \eqref{e:rel-cap-TC} there exist $C_{2},A_{3},A_{4}\in(1,\infty)$ such that
\begin{equation}\label{e:cap-TC-GSC}
C_{2}^{-1}\frac{\mu(B_{\theta}(x,s))}{s^{\beta}}
	\leq \Capa\bigl(B_{\theta}(x,s),K\setminus B_{\theta}(x,A_{3}s)\bigr)
	\leq C_{2}\frac{\mu(B_{\theta}(x,s))}{s^{\beta}}
\end{equation}
for any $(x,s)\in K\times(0,\diam_{\theta}(K)/A_{4})$.
To verify \eqref{e:DM2-d-GSC}, let $x,y\in K$ satisfy $x\not=y$.
By the quasisymmetry of $\theta$ to $d$, \eqref{e:ann1} and \eqref{e:ann2},
there exist $A_{5}\in(1,\infty)$ determined solely by $d,\theta,A_{1}$ and
$A_{6}\in(1,\infty)$ determined solely by $d,\theta,A_{3}$ such that
\begin{align}\label{e:ann1-GSC}
B_{\theta}(x,\theta(x,y)/A_{5})\subset B_{d}(x,d(x,y))&\subset B_{d}(x,A_{1}d(x,y))\subset B_{\theta}(x,A_{5}\theta(x,y)),\\
B_{d}(x,d(x,y)/A_{6})\subset B_{\theta}(x,\theta(x,y))&\subset B_{\theta}(x,A_{3}\theta(x,y))\subset B_{d}(x,A_{6}d(x,y)).
\label{e:ann2-GSC}
\end{align}
Then by \ref{EHI} for $(K,d,\measure,\mathcal{E},\mathcal{F})$ and
$(K,\theta,\measure,\mathcal{E},\mathcal{F})$, Remark \ref{rmk:harnack} and \cite[Lemma 5.22]{BCM}
(note also \cite[Theorem 5.4 and Lemma 5.2-(e)]{BCM}), there exist $A_{7}\in(A_{2},\infty)$
and $A_{8}\in(A_{4},\infty)$ such that for any $(z,s)\in K\times(0,\infty)$,
\begin{align}\label{e:capDd-GSC}
\begin{split}
	\Capa\bigl(B_{d}(z,s),K\setminus B_{d}(z,A_{1}s)\bigr)\asymp\Capa\bigl(B_{d}(z,s),K\setminus B_{d}(z,A_{6}^{2}s)\bigr)&\\
	\textrm{if $s<\diam_{d}(K)/A_{7}$,}&
\end{split}\\
\begin{split}
	\Capa\bigl(B_{\theta}(z,s),K\setminus B_{\theta}(z,A_{3}s)\bigr)\asymp\Capa\bigl(B_{\theta}(z,s),K\setminus B_{\theta}(z,A_{5}^{2}s)\bigr)&\\
	\textrm{if $s<\diam_{\theta}(K)/A_{8}$.}&
\end{split}
\label{e:capDtheta-GSC}
\end{align}
(To be precise, the definition of capacity between sets in \cite[Section 5]{BCM} is slightly different
from ours, but they are easily seen to be equivalent to each other by virtue of \cite[Lemma 2.2.7-(ii)]{FOT}.)
Moreover, the quasisymmetry of $\theta$ to $d$ again and \eqref{e:diamQS} show that
by taking $A_{8}$ large enough we may further assume that
$\theta(x,y)<\diam_{\theta}(K)/A_{8}$ implies $d(x,y)<\diam_{d}(K)/A_{7}$.
Now, if $\theta(x,y)\geq\diam_{\theta}(K)/A_{8}$, then \eqref{e:DM2-d-GSC} clearly holds
for some sufficiently large $C\in(0,\infty)$ independent of $x,y$ since
$\mu\bigl(B_{\theta}(x,\theta(x,y))\bigr)\asymp\mu(K)$ by \ref{VD} of $(K,\theta,\mu)$
and $d(x,y)\asymp\diam_{d}(K)$ by the quasisymmetry of $\theta$ to $d$ and \eqref{e:diamQS}.
Otherwise $\theta(x,y)<\diam_{\theta}(K)/A_{8}$, which implies $d(x,y)<\diam_{d}(K)/A_{7}$, hence
\begin{align*}
C_{2}\frac{\mu\bigl(B_{\theta}(x,\theta(x,y))\bigr)}{\theta(x,y)^{\beta}}
	\geq\Capa\bigl(B_{\theta}(x,\theta(x,y)),K\setminus B_{\theta}(x,A_{3}\theta(x,y&))\bigr)\\
\geq\Capa\bigl(B_{d}(x,d(x,y)/A_{6}),K\setminus B_{d}(x,A_{6}d(x,y))\bigr)
	&\asymp d(x,y)^{d_{\mathrm{f}}-d_{\mathrm{w}}}
\end{align*}
by \eqref{e:cap-TC-GSC}, \eqref{e:ann2-GSC}, \eqref{e:capDd-GSC} and \eqref{e:cap-GSC}, and similarly
\begin{align*}
C_{1}d(x,y)^{d_{\mathrm{f}}-d_{\mathrm{w}}}
	\geq\Capa\bigl(B_{d}(x,d(x,y)),K\setminus B_{d}(x,A_{1}d(x,y&))\bigr)\\
\geq\Capa\bigl(B_{\theta}(x,\theta(x,y)/A_{5}),K\setminus B_{\theta}(x,A_{5}\theta(x,y))\bigr)
	&\asymp\frac{\mu\bigl(B_{\theta}(x,\theta(x,y))\bigr)}{\theta(x,y)^{\beta}}
\end{align*}
by \eqref{e:cap-GSC}, \eqref{e:ann1-GSC}, \eqref{e:capDtheta-GSC}, \eqref{e:cap-TC-GSC}
and \ref{VD} of $(K,\theta,\mu)$, proving \eqref{e:DM2-d-GSC} and thereby \ref{it:condition-PHI-GSC-c}.

Conversely, assume \ref{it:condition-PHI-GSC-c}. Since $K$ is connected, \ref{VD} of $(K,\theta,\mu)$
implies \ref{RVD} of $(K,\theta,\mu)$ by Remark \ref{r:doubling}-\ref{it:RVD-uniformly-perfect}.
Note that \eqref{e:ann1-GSC} remains valid and that for each $A_{3},A_{4}\in(1,\infty)$
we still have \eqref{e:ann2-GSC}, \eqref{e:capDd-GSC} and \eqref{e:capDtheta-GSC}.
Note also that by the connectedness of $K$ and \cite[Theorem 11.3]{Hei}
there exist $\lambda,\alpha\in[1,\infty)$ such that $\theta$ is
$\eta_{\alpha,\lambda}$-quasisymmetric to $d$ with $\eta_{\alpha,\lambda}$ as in Definition \ref{d:qs}.
Let $x\in K$ and let $s_{1},s_{2}\in(0,\diam_{d}(K)/A_{2})$ satisfy $s_{1}\leq s_{2}$.
Then $d(x,y)=s_{1}$ and $d(x,z)=s_{2}$ for some $y,z\in K$ by the connectedness of $K$,
$\mu(B_{d}(x,s_{1}))\asymp\mu\bigl(B_{\theta}(x,\theta(x,y))\bigr)$ and
$\mu(B_{d}(x,s_{2}))\asymp\mu\bigl(B_{\theta}(x,\theta(x,z))\bigr)$ by
\eqref{e:ann1-GSC} and \ref{VD} of $(K,\theta,\mu)$, and therefore by
\eqref{e:DM2-d-GSC}, \eqref{e:cap-GSC}, the $\eta_{\alpha,\lambda}$-quasisymmetry
of $\theta$ to $d$ and \eqref{e:diamQS},
\begin{equation}\label{e:capacity-good}
\begin{split}
\frac{\mu(B_{d}(x,s_{2}))\Capa\bigl(B_{d}(x,s_{1}),K\setminus B_{d}(x,A_{1}s_{1})\bigr)}{\mu(B_{d}(x,s_{1}))\Capa\bigl(B_{d}(x,s_{2}),K\setminus B_{d}(x,A_{1}s_{2})\bigr)}
	&\asymp\frac{d(x,z)^{d_{\mathrm{w}}-d_{\mathrm{f}}}\mu\bigl(B_{\theta}(x,\theta(x,z))\bigr)}{d(x,y)^{d_{\mathrm{w}}-d_{\mathrm{f}}}\mu\bigl(B_{\theta}(x,\theta(x,y))\bigr)}\\
\asymp\biggl(\frac{\theta(x,z)}{\theta(x,y)}&\biggr)^{\beta}
	\in\biggl[\frac{1}{2\lambda}\Bigl(\frac{s_{2}}{s_{1}}\Bigr)^{1/\alpha},2^{\alpha}\lambda\Bigl(\frac{s_{2}}{s_{1}}\Bigr)^{\alpha}\biggr].
\end{split}
\end{equation}
It follows from \ref{EHI} for $(K,d,\measure,\mathcal{E},\mathcal{F})$,
Remark \ref{rmk:harnack} and \eqref{e:capacity-good} that \cite[Proposition 6.16]{BCM}
is applicable to $\mu$ and implies that $\mu\in\mathcal{A}(K,d,\measure,\mathcal{E},\mathcal{F})$.
In particular, $(K,\theta,\mu,\mathcal{E}^{\mu},\mathcal{F}^{\mu})$ satisfies \ref{EHI} by Lemma \ref{l:EHI-TC-QS}.
Finally, to show \ref{eq:capbeta} for $(K,\theta,\mu,\mathcal{E}^{\mu},\mathcal{F}^{\mu})$,
let $A_{3},A_{4}\in(1,\infty)$ satisfy $A_{4}\geq 2$, and choose
$A_{6},A_{7},A_{8}\in(1,\infty)$ with $A_{7}>A_{2}$ and $A_{8}>A_{4}$
so that \eqref{e:ann2-GSC}, \eqref{e:capDd-GSC} and \eqref{e:capDtheta-GSC} hold.
Again thanks to the quasisymmetry of $\theta$ to $d$ and \eqref{e:diamQS},
by taking $A_{8}$ large enough we may further assume that $d(x,y)<\diam_{d}(K)/A_{7}$
for any $x,y\in K$ with $\theta(x,y)<\diam_{\theta}(K)/A_{8}$.
Let $(x,s)\in K\times(0,\diam_{\theta}(K)/(A_{5}A_{8}))$, so that
$\theta(x,y)=s=\theta(x,z)/A_{5}$ for some $y,z\in K$ by the connectedness of $K$.
Then by \eqref{e:ann2-GSC}, \eqref{e:capDd-GSC}, \eqref{e:cap-GSC} and \eqref{e:DM2-d-GSC},
\begin{align*}
\Capa\bigl(B_{\theta}(x,s),K\setminus B_{\theta}(x,A_{3}s)\bigr)
	\geq\Capa\bigl(B_{d}&(x,d(x,y)/A_{6}),K\setminus B_{d}(x,A_{6}d(x,y))\bigr)\\
\asymp d(x,y)^{d_{\mathrm{f}}-d_{\mathrm{w}}}
	&\asymp\frac{\mu\bigl(B_{\theta}(x,\theta(x,y))\bigr)}{\theta(x,y)^{\beta}}
	=\frac{\mu\bigl(B_{\theta}(x,s)\bigr)}{s^{\beta}},
\end{align*}
and similarly by \eqref{e:capDtheta-GSC}, \eqref{e:ann1-GSC}, \eqref{e:cap-GSC},
\eqref{e:DM2-d-GSC} and \ref{VD} of $(K,\theta,\mu)$,
\begin{align*}
\Capa\bigl(B_{\theta}(x,s),K\setminus B_{\theta}(x,A_{3}&s)\bigr)
	\asymp\Capa\bigl(B_{\theta}(x,\theta(x,z)/A_{5}),K\setminus B_{\theta}(x,A_{5}\theta(x,z))\bigr)\\
	&\leq\Capa\bigl(B_{d}(x,d(x,z)),K\setminus B_{d}(x,A_{1}d(x,z))\bigr)\\
&\asymp d(x,z)^{d_{\mathrm{f}}-d_{\mathrm{w}}}
	\asymp\frac{\mu\bigl(B_{\theta}(x,\theta(x,z))\bigr)}{\theta(x,z)^{\beta}}
	\asymp\frac{\mu\bigl(B_{\theta}(x,s)\bigr)}{s^{\beta}},
\end{align*}
proving \eqref{e:cap-TC-GSC} for $(x,s)\in K\times(0,\diam_{\theta}(K)/(A_{5}A_{8}))$,
namely \ref{eq:capbeta} for $(K,\theta,\mu,\mathcal{E}^{\mu},\mathcal{F}^{\mu})$.
Now $(K,\theta,\mu,\mathcal{E}^{\mu},\mathcal{F}^{\mu})$ satisfies \ref{PHI}
by Theorem \ref{t:phichar}, showing Theorem \ref{t:condition-PHI-GSC}-\ref{it:condition-PHI-GSC-a}.
\end{proof}

\begin{proof}[Proof of Theorem \textup{\ref{t:condition-PHI-GSC}}]
By Proposition \ref{p:condition-PHI-GSC}, it suffices to prove the equivalence of
Proposition \ref{p:condition-PHI-GSC}-\ref{it:condition-PHI-GSC-c} and \ref{it:condition-PHI-GSC-b}.
We can verify it in exactly the same way as the proof of Proposition \ref{p:condition-PHI-pcf},
by considering the scale $\mathscr{S}=\{\Lambda_{s}\}_{s\in(0,1]}$ on $\Sigma$ defined by $\Lambda_{1}:=\{\emptyset\}$ and
\begin{equation}\label{e:scale-GSC}
\Lambda_{s}:=\{w\mid\textrm{$w=w_{1}\dots w_{n}\in W_{*}\setminus\{\emptyset\}$, $l^{1-\lvert w\rvert}>s\geq l^{-\lvert w\rvert}$}\}
\end{equation}
for each $s\in(0,1)$, which clearly satisfies \eqref{e:LF-pcf}, and by using instead of $R_{\mathcal{E}}$
the Euclidean metric $d$ on $K$, which is easily seen to satisfy \eqref{e:adapted-RE-pcf}
with $d$ in place of $R_{\mathcal{E}}$; note that since \eqref{e:DM2-RE-pcf} needs to be
replaced by \eqref{e:DM2-d-GSC} we also need to replace 
$R_{\mathcal{E}}(x,y)\mu\bigl(B_{R_{\mathcal{E}}}(x,R_{\mathcal{E}}(x,y))\bigr)$
in \eqref{e:DM2-volume-RE-pcf} by
$d(x,y)^{d_{\mathrm{w}}-d_{\mathrm{f}}}\mu\bigl(B_{d}(x,d(x,y))\bigr)$.
\end{proof}

By virtue of Theorem \ref{t:condition-PHI-GSC}, the whole of Subsection \ref{ssec:attain-harmonic-func-pcf}
can be easily adapted for the present case, and below we explicitly give
the details of the adaptation for concreteness. We begin with stating
the main result of this subsection, which requires the following definition.
Recall \eqref{eq:HomeoPlus} for $\homeo^{+}$ and
Definition \ref{d:harmonic-GSC} for $\mathcal{H}_{0}$.

\begin{definition}\label{d:GetaC-GSC}
We define $\mathcal{P}(K)$ by \eqref{eq:PK-pcf}, equip $\mathcal{P}(K)$ with the topology of weak convergence,
and for each $(\eta,C)\in\homeo^{+}\times(1,\infty)$ we define $\attainsss(\eta,C)=\attainsss_{N,l,S}(\eta,C)$ by
\begin{equation}\label{e:GetaC-GSC}
\attainsss(\eta,C)
:=\biggl\{(\theta,\mu)\biggm\vert
	\begin{minipage}{242pt}
		$\theta$ is a metric on $K$ and $\eta$-quasisymmetric to $d$,
		$\mu\in\mathcal{P}(K)$,
		$C^{-1}\leq r_{w}\mu(K_{w})/(\diam_{\theta}(K_{w}))^{2}\leq C$
		for any $w\in W_{*}$
	\end{minipage}
	\biggr\},
\end{equation}
which is considered as a subset of $\contfunc(K\times K)\times\mathcal{P}(K)$. We also set
\begin{equation}\label{e:GNlS}
\attainsss:=\attainsss_{N,l,S}:=\bigcup_{(\eta,C)\in\homeo^{+}\times(1,\infty)}\attainsss(\eta,C)
\end{equation}
and for each subset $\mathcal{Z}$ of $\attainsss$ define
$\mathcal{H}_{0}(\mathcal{Z})\subset\mathcal{H}_{0}$ by \eqref{e:H0Z-pcf} and
$\widetilde{\mathcal{H}}_{0}(\mathcal{Z})\subset\mathcal{H}_{0}/\mathbb{R}\one_{K}$
by \eqref{e:H0tildeZ-pcf}.
\end{definition}

Since $\mu$ is a Radon measure on $K$ with full support for any $(\theta,\mu)\in\attainsss$,
it follows from Theorem \ref{t:condition-PHI-GSC} with $\beta=2$ and \eqref{e:cgauge} that
\begin{equation}\label{e:condition-PHI-GSC-G}
\mathcal{G}(K,d,\measure,\mathcal{E},\mathcal{F})
	=\{a\mu\mid\textrm{$(\theta,\mu)\in\attainsss$, $a\in(0,\infty)$}\}.
\end{equation}
In particular,
\begin{equation}\label{e:condition-PHI-GSC-G-attained}
\begin{minipage}{330pt}
\emph{$\mathcal{G}(K,d,\measure,\mathcal{E},\mathcal{F})\not=\emptyset$,
i.e., the infimum in \eqref{e:dcw} is attained for $(K,d,\measure,\mathcal{E},\mathcal{F})$,
if and only if $\attainsss\not=\emptyset$, namely
$\attainsss(\eta,C)\not=\emptyset$ for some $(\eta,C)\in\homeo^{+}\times(1,\infty)$.}
\end{minipage}
\end{equation}
It turns out that in this case $\mathcal{H}_{0}(\attainsss)\not=\emptyset$, i.e.,
$(\theta_{h},\Gamma(h,h))\in\attainsss$ for some $h\in\mathcal{H}_{0}$ and some
$\theta_{h}\in\mathcal{J}(K,d)$, which is the main result of this subsection and stated as follows.
\emph{We take arbitrary $(\eta,C)\in\homeo^{+}\times(1,\infty)$, define
$\tilde{\eta}\in\homeo^{+}$ by $\tilde{\eta}(t):=1/\eta^{-1}(t^{-1})$
($\tilde{\eta}(0):=0$) and fix them throughout the rest of this subsection.}

\begin{theorem}\label{t:attain-harmonic-func-GSC}
If $\attainsss(\eta,C)\not=\emptyset$, then $\mathcal{H}_{0}(\attainsss(\eta,C))\not=\emptyset$,
i.e., there exist $h\in\mathcal{H}_{0}$ and a metric $\theta_{h}$ on $K$
such that $(\theta_{h},\Gamma(h,h))\in\attainsss(\eta,C)$.
\end{theorem}

Moreover, as in Subsection \ref{ssec:attain-harmonic-func-pcf},
a slight addition to our proof of Theorem \ref{t:attain-harmonic-func-GSC}
also shows the following theorem, which will be useful in studying further the problem
of whether the infimum in \eqref{e:dcw} is attained for generalized Sierpi\'{n}ski carpets.
Recall that $(\mathcal{F}/\mathbb{R}\one_{K},\mathcal{E})$ is a Hilbert space
as observed in Lemma \ref{l:global-PI-GSC}.

\begin{theorem}\label{thm:H0GetaC-compact-GSC}
\begin{enumerate}[label=\textup{(\arabic*)},align=left,leftmargin=*,topsep=5pt,parsep=0pt,itemsep=2pt]
\item\label{it:H0GetaC-compact-GSC} $\widetilde{\mathcal{H}}_{0}(\attainsss(\eta,C))$ is compact in norm in $(\mathcal{F}/\mathbb{R}\one_{K},\mathcal{E})$.
\item\label{it:H0GetaC-compact-scaling-GSC} If $h\in\mathcal{H}_{0}(\attainsss(\eta,C))$, then
	$\mathcal{E}(h\circ F_{w},h\circ F_{w})^{-1/2}h\circ F_{w}\in\mathcal{H}_{0}(\attainsss(\eta,C))$
	for any $w\in W_{*}$.
\end{enumerate}
\end{theorem}

We remark that in Theorem \ref{thm:H0GetaC-compact-GSC}-\ref{it:H0GetaC-compact-scaling-GSC}
we have $\mathcal{E}(h\circ F_{w},h\circ F_{w})=r_{w}\Gamma(h,h)(K_{w})>0$ for any $w\in W_{*}$ by
Lemma \ref{l:scaling-energy-meas-GSC} below and the lower inequality in \eqref{e:GetaC-GSC} for $\mu=\Gamma(h,h)$.

The rest of this subsection is devoted to the proof of Theorems \ref{t:attain-harmonic-func-GSC}
and \ref{thm:H0GetaC-compact-GSC}, which goes in exactly the same way as that of Theorem \ref{t:attain-harmonic-func}
and Proposition \ref{p:H0GetaC-compact-pcf} except for some modifications
required due to $\#V_{0}=\infty=\dim\mathcal{H}_{0}$ and explained in detail below.

\begin{proposition}\label{p:GetaC-compact-GSC}
$\attainsss(\eta,C)$ is a compact subset of $\contfunc(K\times K)\times\mathcal{P}(K)$.
\end{proposition}

\begin{proof}
The proof of Proposition \ref{p:GetaC-compact} remains valid also in this case,
except that $R_{\mathcal{E}}$ needs to be replaced by $d$ and that $s$ in the last
paragraph needs to be defined as $s:=l^{-\lvert w\rvert}$.
\end{proof}

\begin{corollary}\label{c:GetaC-compact-GSC}
Let $\{(\theta_{n},\mu_{n})\}_{n\in\mathbb{N}}\subset\attainsss(\eta,C)$,
$\mu\in\mathcal{P}(K)$ and suppose that $\{\mu_{n}\}_{n\in\mathbb{N}}$ converges
to $\mu$ in $\mathcal{P}(K)$. Then there exists a metric $\theta$ on $K$ such that
$(\theta,\mu)\in\attainsss(\eta,C)$.
\end{corollary}

\begin{proof}
The proof of Corollary \ref{c:GetaC-compact} remains valid with
Proposition \ref{p:GetaC-compact-GSC} applied in place of Proposition \ref{p:GetaC-compact}.
\end{proof}

\begin{lemma}\label{l:GetaC-scaling-GSC}
Let $(\theta,\mu)\in\attainsss(\eta,C)$, $w\in W_{*}$ and define
$(\theta_{w},\mu_{w})\in \contfunc(K\times K)\times\mathcal{P}(K)$ by \eqref{e:GetaC-scaling}.
Then $(\theta_{w},\mu_{w})\in\attainsss(\eta,C)$.
\end{lemma}

\begin{proof}
The proof of Lemma \ref{l:GetaC-scaling} remains valid also in this case,
except that $R_{\mathcal{E}}$ needs to be replaced by $d$.
\end{proof}

\begin{lemma}\label{l:Fw-star}
Let $w\in W_{*}$.
\begin{enumerate}[label=\textup{(\arabic*)},align=left,leftmargin=*,topsep=5pt,parsep=0pt,itemsep=2pt]
\item\label{it:Fw-star-measure}$\int_{K}\lvert u\circ F_{w}\rvert\,d\measure=(\#S)^{\lvert w\rvert}\int_{K_{w}}\lvert u\rvert\,d\measure$
	for any Borel measurable function $u\colon K\to[-\infty,\infty]$, and hence a bounded linear operator
	from $L^{2}(K,\measure)$ to itself is defined by $u\mapsto u\circ F_{w}$.
\item\label{it:Fw-star-SSDF}$u\circ F_{w}\in\mathcal{F}$ and \eqref{e:SSDF-form} holds for any $u,v\in\mathcal{F}$.
\item\label{it:Fw-star-H0}$h\circ F_{w}\in\mathcal{H}_{0}$ for any $h\in\mathcal{H}_{0}$.
\end{enumerate}
\end{lemma}
\begin{proof}
\ref{it:Fw-star-measure} is immediate from $\measure=(\#S)^{\lvert w\rvert}\measure\circ F_{w}$,
and \ref{it:Fw-star-SSDF} follows from \eqref{e:SSDF-domain}, \eqref{e:SSDF-form}, the
denseness of $\mathcal{F}\cap \contfunc(K)$ in $(\mathcal{F},\mathcal{E}_{1})$ and
the completeness of $(\mathcal{F},\mathcal{E}_{1})$; see \cite[Proof of Lemma 3.3]{Kaj22}.
To see \ref{it:Fw-star-H0}, let $h\in\mathcal{H}_{0}$, and let
$v\in\mathcal{F}\cap \contfunc(K)$ satisfy $\supp_{\measure}[v]\subset K\setminus V_{0}$.
Then since $K_{\tau}\cap K_{w}=F_{\tau}(V_{0})\cap F_{w}(V_{0})$ for any $\tau\in W_{\lvert w\rvert}\setminus\{w\}$
by \cite[Proposition 1.3.5-(2)]{Kig01}, we can define $v^{w}\in \contfunc(K)$
by $v^{w}\vert_{K_{w}}:=v\circ F_{w}^{-1}$ and $v^{w}\vert_{K\setminus K_{w}}:=0$ and have
$v^{w}\in\mathcal{F}\cap \contfunc(K)$ by \eqref{e:SSDF-domain} and
$\supp_{\measure}[v^{w}]\subset K_{w}\setminus F_{w}(V_{0})=K_{w}\setminus V_{\lvert w\rvert}\subset K\setminus V_{0}$,
and it therefore follows from \eqref{e:SSDF-form} for $h,v^{w}$ and $h\in\mathcal{H}_{0}$ that
$r_{w}^{-1}\mathcal{E}(h\circ F_{w},v)
	=\sum_{\tau\in W_{\lvert w\rvert}}r_{\tau}^{-1}\mathcal{E}(h\circ F_{\tau},v^{w}\circ F_{\tau})
	=\mathcal{E}(h,v^{w})=0$,
proving $h\circ F_{w}\in\mathcal{H}_{0}$.
\end{proof}

\begin{lemma}\label{l:scaling-energy-meas-GSC}
Suppose that $\attainsss\not=\emptyset$. Let $u\in\mathcal{F}$ and $w\in W_{*}$.
Then $\Gamma(u,u)(F_{w}(A))=r_{w}^{-1}\Gamma(u\circ F_{w},u\circ F_{w})(A)$
for any Borel subset $A$ of $K$, and in particular
$\Gamma(u,u)(K_{w})=r_{w}^{-1}\mathcal{E}(u\circ F_{w},u\circ F_{w})$. Moreover, if
$\Gamma(u,u)(K_{w})>0$, then \eqref{e:scaling-energy-meas} holds for any Borel subset $A$ of $K$.
\end{lemma}

\begin{proof}
First, by Lemma \ref{l:Fw-star}-\ref{it:Fw-star-SSDF} and \cite[Lemma 3.11-(ii)]{Hin05}
it holds, \emph{independently of the assumption $\attainsss\not=\emptyset$}, that
for any $v\in\mathcal{F}$ and any $n\in\mathbb{N}\cup\{0\}$,
\begin{equation}\label{e:pre-scaling-energy-meas-GSC}
\Gamma(v,v)(A)=\sum_{\tau\in W_{n}}\frac{1}{r_{\tau}}\Gamma(v\circ F_{\tau},v\circ F_{\tau})(F_{\tau}^{-1}(A))
	\quad\textrm{for any Borel subset $A$ of $K$.}
\end{equation}
By $\attainsss\not=\emptyset$ we can take $(\theta,\mu)\in\attainsss$,
then $\mu$ is a minimal energy-dominant measure of $(\mathcal{E},\mathcal{F})$
by Theorem \ref{t:condition-PHI-GSC} and Proposition \ref{p:metmeas}-\ref{it:conseq-PHI2-meas}, thus
$\Gamma(v,v)\ll\mu$ for any $v\in\mathcal{F}$, and hence
\begin{equation}\label{e:energy-V0-zero}
\Gamma(v,v)(V_{0})=0\qquad\textrm{for any $v\in\mathcal{F}$}
\end{equation}
since $\mu(V_{0})=0$ by Theorem \ref{t:condition-PHI-GSC}.
Recalling that for any $\tau\in W_{\lvert w\rvert}\setminus\{w\}$ we have
$K_{\tau}\cap K_{w}=F_{\tau}(V_{0})\cap F_{w}(V_{0})$ by \cite[Proposition 1.3.5-(2)]{Kig01}
and hence $F_{\tau}^{-1}(K_{w})=F_{\tau}^{-1}(K_{\tau}\cap K_{w})\subset V_{0}$,
we see from \eqref{e:pre-scaling-energy-meas-GSC} and \eqref{e:energy-V0-zero} that
for any Borel subset $A$ of $K$,
\begin{align*}
\Gamma(u,u)(F_{w}(A))&=\sum_{\tau\in W_{\lvert w\rvert}}\frac{1}{r_{\tau}}\Gamma(u\circ F_{\tau},u\circ F_{\tau})(F_{\tau}^{-1}(F_{w}(A)))\\
&=\frac{1}{r_{w}}\Gamma(u\circ F_{w},u\circ F_{w})(A).
\end{align*}
The rest of the proof goes in exactly the same way as that of Lemma \ref{l:scaling-energy-meas}.
\end{proof}

\begin{remark}\label{rmk:energy-V0-zero}
In fact, \eqref{e:energy-V0-zero} holds without supposing $\attainsss\not=\emptyset$
by \cite[Proposition 4.15]{Hin13} and hence so does Lemma \ref{l:scaling-energy-meas-GSC}.
The proof of \eqref{e:energy-V0-zero} presented in \cite[Section 5]{Hin13}, however, is long and difficult,
and since we later use Lemma \ref{l:scaling-energy-meas-GSC} only under the supposition that
$\attainsss\not=\emptyset$, we have decided to suppose it explicitly to keep our
present arguments independent of the demanding result \cite[Proposition 4.15]{Hin13}.
\end{remark}

\begin{lemma}\label{l:Lebesgue-points-GSC}
Suppose that $\attainsss\not=\emptyset$, let $(\theta,\mu)\in\attainsss$,
$u\in\mathcal{F}$ and set $f:=d\Gamma(u,u)/d\mu$. Then $\mu$-a.e.\ $x\in K$ is a
\emph{$(d,\mu)$-Lebesgue point} for $f$, i.e., satisfies
\begin{equation}\label{e:Lebesgue-points-GSC}
\lim_{s\downarrow 0}\frac{1}{\mu(B_{d}(x,s))}\int_{B_{d}(x,s)}\lvert f(y)-f(x)\rvert\,d\mu(y)=0.
\end{equation}
\end{lemma}

\begin{proof}
The proof of Lemma \ref{l:Lebesgue-points} remains valid also in this case,
except that $R_{\mathcal{E}}$ needs to be replaced by $d$.
\end{proof}

\begin{lemma}\label{l:Lebesgue-pt-scaling-GSC}
Suppose that $\attainsss\not=\emptyset$, let $(\theta,\mu)\in\attainsss$, $u\in\mathcal{F}$,
let $f\colon K\to[0,\infty)$ be a Borel measurable $\mu$-version of $d\Gamma(u,u)/d\mu$ and let $x\in K$
satisfy \eqref{e:Lebesgue-points-GSC}. Then \eqref{e:Lebesgue-pt-scaling} holds
for any $\omega\in\pi^{-1}(x)$ and any $w\in W_{*}$.
\end{lemma}

\begin{proof}
The proof of Lemma \ref{l:Lebesgue-pt-scaling} remains valid also in this case,
except that $R_{\mathcal{E}}$ needs to be replaced by $d$.
\end{proof}

\begin{proposition}\label{p:Lebesgue-pt-scaling-GSC}
Suppose that $\attainsss\not=\emptyset$, let 
$(\theta,\mu)\in\attainsss$, $u\in\mathcal{F}$, let $f\colon K\to[0,\infty)$
be a Borel measurable $\mu$-version of $d\Gamma(u,u)/d\mu$, let $x\in K$
satisfy \eqref{e:Lebesgue-points-GSC} and $f(x)>0$, and let $\omega\in\pi^{-1}(x)$.
For each $n\in\mathbb{N}\cup\{0\}$,
define $\mu_{n}:=\mu_{[\omega]_{n}}\in\mathcal{P}(K)$ by \eqref{e:GetaC-scaling}
with $w=[\omega]_{n}$ and, noting that $\Gamma(u,u)(K_{[\omega]_{n}})>0$ by
Lemma \textup{\ref{l:Lebesgue-pt-scaling-GSC}}, define $u_{n}:=u_{[\omega]_{n}}\in\mathcal{F}$
by \eqref{e:scaling-energy-meas} with $w=[\omega]_{n}$. If $v\in\mathcal{F}$ and
$\{n_{k}\}_{k\in\mathbb{N}}\subset\mathbb{N}$ is strictly increasing and satisfies
$\lim_{k\to\infty}\mathcal{E}(v-u_{n_{k}},v-u_{n_{k}})=0$, then $\Gamma(v,v)\in\mathcal{P}(K)$
and $\{\mu_{n_{k}}\}_{k\in\mathbb{N}}$ converges to $\Gamma(v,v)$ in $\mathcal{P}(K)$.
\end{proposition}

\begin{proof}
The proof of Proposition \ref{p:Lebesgue-pt-scaling} remains valid also in this case,
except that it is because of $\attainsss\not=\emptyset$, \eqref{e:energy-V0-zero} and
Lemma \ref{l:scaling-energy-meas-GSC} that $\Gamma(v,v)(F_{w}(V_{0}))=0$ for any $w\in W_{*}$.
\end{proof}

So far, except for the issue of the validity of \eqref{e:energy-V0-zero} treated
in the proof of Lemma \ref{l:scaling-energy-meas-GSC} and Remark \ref{rmk:energy-V0-zero},
our current discussion has been almost the same as the corresponding part of
Subsection \ref{ssec:attain-harmonic-func-pcf}. On the other hand, the concluding parts
of the proofs of Proposition \ref{p:H0GetaC-compact-pcf}-\ref{it:H0GetaC-compact-pcf}
and Theorem \ref{t:attain-harmonic-func} rely on the compactness of
$\{h\in\mathcal{H}_{0}/\mathbb{R}\one_{K}\mid\mathcal{E}(h,h)=1\}$
implied by $\dim\mathcal{H}_{0}/\mathbb{R}\one_{K}<\infty$ and thereby cannot be
extended directly to the present case of a generalized Sierpi\'{n}ski carpet,
where $\dim\mathcal{H}_{0}/\mathbb{R}\one_{K}=\infty$ by $\#V_{0}=\infty$.
We overcome this difficulty by establishing Proposition \ref{p:harmonic-func-compact-GSC}
below on the basis of Lemma \ref{l:global-PI-GSC} and the following lemma and applying it with
the help of the compactness of $\attainsss(\eta,C)$ from Proposition \ref{p:GetaC-compact-GSC}.

\begin{lemma}[Reverse Poincar\'e inequality]\label{l:reverse-PI-GSC}
There exists $C_{\mathrm{RP}}\in(0,\infty)$ such that for any $(x,s)\in K\times (0,\infty)$,
any $a\in\mathbb{R}$ and any function $h\in\mathcal{F}$ that is
$\mathcal{E}$-harmonic on $B_{d}(x,2s)$,
\begin{equation}\label{e:reverse-PI-GSC}
\int_{B_{d}(x,s)}d\Gamma(h,h)
	\leq\frac{C_{\mathrm{RP}}}{s^{d_{\mathrm{w}}}}\int_{B_{d}(x,2s)\setminus B_{d}(x,s)}\lvert h-a\rvert^{2}\,d\measure.
\end{equation}
\end{lemma}

\begin{proof}
This is a special case of \cite[Lemma 3.3]{KM} with $\Psi(s)=s^{d_{\mathrm{w}}}$,
whose assumption $\on{CS}(\Psi)$ formulated in \cite[Definition 2.6-(b)]{KM} is implied
in the current situation by Lemma \ref{l:volume-GSC}, Theorem \ref{thm:BBKT-HK}
and \cite[Theorem 5.5]{AB}; see also Remark \ref{rmk:reverse-PI-GSC} below.
\end{proof}

\begin{remark}\label{rmk:reverse-PI-GSC}
To be precise, \cite[Lemma 3.3]{KM} assumes additionally that $h\in L^{\infty}(X,\measure)$,
but this assumption can be dropped by replacing \cite[the first four lines of (3.9)]{KM}
with the following, where $h_{n}:=(-n)\vee(h\wedge n)$ for $n\in\mathbb{N}$:
\begin{align*}
0&=\lim_{n\to\infty}\mathcal{E}(h,h_{n}\varphi^{2})
	=\lim_{n\to\infty}\Gamma(h,h_{n}\varphi^{2})(X)
	\qquad\textrm{(by \cite[(3.7) and (2.4)]{KM})}\\
&=\lim_{n\to\infty}\biggl(\int_{X}\varphi^{2}\,d\Gamma(h,h_{n})+2\int_{X}\varphi h_{n}\,d\Gamma(h,\varphi)\biggr)
	\quad\textrm{(by \cite[Lemma 3.2.5]{FOT})}\\
&\geq\limsup_{n\to\infty}\Biggl(\int_{X}\varphi^{2}\,d\Gamma(h,h_{n})-2\sqrt{\int_{X}\varphi^{2}\,d\Gamma(h,h)\int_{X}h_{n}^{2}\,d\Gamma(\varphi,\varphi)}\Biggr)\\
&\qquad\textrm{(by \cite[Proof of Lemma 5.6.1]{FOT})}\\
&=\int_{X}\varphi^{2}\,d\Gamma(h,h)-2\sqrt{\int_{X}\varphi^{2}\,d\Gamma(h,h)\int_{X}h^{2}\,d\Gamma(\varphi,\varphi)}\\
&\qquad\begin{minipage}{245pt}(by the Cauchy--Schwarz inequality for $\textstyle\int_{X}\varphi^{2}\,d\Gamma(\cdot,\cdot)$ and\\\hspace*{3pt}\cite[Theorem 1.4.2-(iii)]{FOT}).\end{minipage}
\end{align*}
\end{remark}

\begin{proposition}\label{p:harmonic-func-compact-GSC}
Let $\{h_{k}\}_{k\in\mathbb{N}}\subset\mathcal{H}_{0}$ converge weakly in
$(\mathcal{F},\mathcal{E}_{1})$ to $h\in\mathcal{F}$\textup{(, so that $h\in\mathcal{H}_{0}$
since $\mathcal{H}_{0}$ is weakly closed in $(\mathcal{F},\mathcal{E}_{1})$)},
and assume that there exist a Radon measure $\nu$ on $K$ and $n_{0}\in\mathbb{N}$ such that
$\nu(F_{w}(V_{0}))=0$ and $\limsup_{k\to\infty}r_{w}^{-1}\mathcal{E}(h_{k}\circ F_{w},h_{k}\circ F_{w})\leq\nu(K_{w})$
for any $w\in W_{*}$ satisfying $\lvert w\rvert\geq n_{0}$ and $K_{w}\cap V_{0}\not=\emptyset$.
Then $\lim_{k\to\infty}\mathcal{E}_{1}(h-h_{k},h-h_{k})=0$.
\end{proposition}

\begin{proof}
Thanks to the weak convergence in $(\mathcal{F},\mathcal{E}_{1})$
of $\{h_{k}\}_{k\in\mathbb{N}}$ to $h$, we have
\begin{equation}\label{e:harmonic-func-weak-L2-GSC}
\lim_{k\to\infty}\int_{K}\lvert h-h_{k}\rvert^{2}\,d\measure=0
\end{equation}
by the compactness of the inclusion map from $(\mathcal{F},\mathcal{E}_{1})$ to
$L^{2}(K,\measure)$ stated in Lemma \ref{l:global-PI-GSC} and
\cite[Chapter 21, Theorem 9]{Lax}, and for any $w\in W_{*}$,
$\{h_{k}\circ F_{w}+\mathbb{R}\one_{K}\}_{k\in\mathbb{N}}\subset\mathcal{H}_{0}/\mathbb{R}\one_{K}$
converges weakly in $(\mathcal{F}/\mathbb{R}\one_{K},\mathcal{E})$
to $h\circ F_{w}+\mathbb{R}\one_{K}$ since
$u\mapsto u\circ F_{w}+\mathbb{R}\one_{K}$ is a bounded linear operator from
$(\mathcal{F},\mathcal{E}_{1})$ to $(\mathcal{F}/\mathbb{R}\one_{K},\mathcal{E})$
by Lemma \ref{l:Fw-star}-\ref{it:Fw-star-SSDF}. In particular,
\begin{equation}\label{e:harmonic-func-weak-Fw-star-GSC}
\mathcal{E}(h\circ F_{w},h\circ F_{w})
	\leq\liminf_{k\to\infty}\mathcal{E}(h_{k}\circ F_{w},h_{k}\circ F_{w})
	\qquad\textrm{for any $w\in W_{*}$.}
\end{equation}
Moreover, recalling that $K\setminus V_{0}$ is open in $K$ and non-empty and
letting $\tau\in W_{*}$, $z_{\tau}\in K_{\tau}$ and $s_{\tau}\in(0,\infty)$ satisfy
$K_{\tau}\subset B_{d}(z_{\tau},s_{\tau})\subset B_{d}(z_{\tau},2s_{\tau})\subset K\setminus V_{0}$,
we see from \eqref{e:pre-scaling-energy-meas-GSC} with $A=K_{\tau}$ and
Lemma \ref{l:reverse-PI-GSC} with $a=0$ that for any $v\in\mathcal{H}_{0}$,
\begin{equation}\label{e:reverse-PI-GSC-apply}
\frac{1}{r_{\tau}}\mathcal{E}(v\circ F_{\tau},v\circ F_{\tau})
	\leq\Gamma(v,v)(K_{\tau})\leq\Gamma(v,v)(B_{d}(z_{\tau},s_{\tau}))
	\leq\frac{C_{\mathrm{RP}}}{s_{\tau}^{d_{\mathrm{w}}}}\int_{K}v^{2}\,d\measure,
\end{equation}
which with $v=h-h_{k}$ for $k\in\mathbb{N}$, together with
\eqref{e:harmonic-func-weak-L2-GSC}, shows that
\begin{equation}\label{e:harmonic-func-norm-interior-GSC}
\lim_{k\to\infty}\mathcal{E}((h-h_{k})\circ F_{\tau},(h-h_{k})\circ F_{\tau})=0.
\end{equation}

Now let $n\in\mathbb{N}$ satisfy $n\geq n_{0}$, set
$V_{0,n}:=\bigcup_{w\in W_{n},\,K_{w}\cap V_{0}\not=\emptyset}K_{w}$ and
$W^{\circ}_{n}:=\{\tau\in W_{n+N}\mid K_{\tau}\not\subset V_{0,n}\}$, so that
$\{w\in W_{n}\mid K_{w}\cap V_{0}\not=\emptyset\}\cup W^{\circ}_{n}$
is a partition of $\Sigma$ (recall Definition \ref{d:partition}-\ref{it:partition})
and each $\tau\in W^{\circ}_{n}$ satisfies \eqref{e:harmonic-func-norm-interior-GSC}
since $K_{\tau}\subset B_{d}(z_{\tau},l^{-n-1})\subset B_{d}(z_{\tau},2l^{-n-1})\subset K\setminus V_{0}$
for any $z_{\tau}\in K_{\tau}$. Then since $\nu(F_{w}(V_{0}))=0$ and
$\limsup_{k\to\infty}r_{w}^{-1}\mathcal{E}(h_{k}\circ F_{w},h_{k}\circ F_{w})\leq\nu(K_{w})$
for any $w\in W_{n}$ with $K_{w}\cap V_{0}\not=\emptyset$ and hence
$\sum_{w\in W_{n},\,K_{w}\cap V_{0}\not=\emptyset}\nu(K_{w})=\nu(V_{0,n})$ and
$\nu(V_{0})=\nu\bigl(V_{0}\cap\bigcup_{w\in W_{n},\,K_{w}\cap V_{0}\not=\emptyset}F_{w}(V_{0})\bigr)=0$
by \cite[Proposition 1.3.5-(2)]{Kig01} and $V_{0}\subset V_{n}$,
it follows from Lemma \ref{l:Fw-star}-\ref{it:Fw-star-SSDF},
\eqref{e:harmonic-func-norm-interior-GSC} for $\tau\in W^{\circ}_{n}$
and \eqref{e:harmonic-func-weak-Fw-star-GSC} that
\begin{align*}
&\limsup_{k\to\infty}\mathcal{E}(h-h_{k},h-h_{k})\\
&=\limsup_{k\to\infty}\sum_{\tau\in\{w\in W_{n}\mid K_{w}\cap V_{0}\not=\emptyset\}\cup W^{\circ}_{n}}\frac{1}{r_{\tau}}\mathcal{E}((h-h_{k})\circ F_{\tau},(h-h_{k})\circ F_{\tau})\\
&=\limsup_{k\to\infty}\sum_{w\in W_{n},\,K_{w}\cap V_{0}\not=\emptyset}\frac{1}{r_{w}}\mathcal{E}(h\circ F_{w}-h_{k}\circ F_{w},h\circ F_{w}-h_{k}\circ F_{w})\\
&\leq\limsup_{k\to\infty}\sum_{w\in W_{n},\,K_{w}\cap V_{0}\not=\emptyset}\frac{2}{r_{w}}\bigl(\mathcal{E}(h\circ F_{w},h\circ F_{w})+\mathcal{E}(h_{k}\circ F_{w},h_{k}\circ F_{w})\bigr)\\
&\leq\sum_{w\in W_{n},\,K_{w}\cap V_{0}\not=\emptyset}\frac{2}{r_{w}}\Bigl(\mathcal{E}(h\circ F_{w},h\circ F_{w})+\limsup_{k\to\infty}\mathcal{E}(h_{k}\circ F_{w},h_{k}\circ F_{w})\Bigr)\\
&\leq\sum_{w\in W_{n},\,K_{w}\cap V_{0}\not=\emptyset}\frac{4}{r_{w}}\limsup_{k\to\infty}\mathcal{E}(h_{k}\circ F_{w},h_{k}\circ F_{w})\\
&\leq\sum_{w\in W_{n},\,K_{w}\cap V_{0}\not=\emptyset}4\nu(K_{w})
	=4\nu(V_{0,n})\xrightarrow{n\to\infty}4\nu(V_{0})=0,
\end{align*}
which along with \eqref{e:harmonic-func-weak-L2-GSC} proves $\lim_{k\to\infty}\mathcal{E}_{1}(h-h_{k},h-h_{k})=0$.
\end{proof}

\begin{proof}[Proof of Theorem \textup{\ref{thm:H0GetaC-compact-GSC}}]
Recall \eqref{e:GetaC-GSC} for $\attainsss(\eta,C)$,
\eqref{e:H0Z-pcf} for $\mathcal{H}_{0}(\mathcal{Z})$ and
\eqref{e:H0tildeZ-pcf} for $\widetilde{\mathcal{H}}_{0}(\mathcal{Z})$.
\begin{enumerate}[label=\textup{(\arabic*)},align=left,leftmargin=*,topsep=5pt,parsep=0pt,itemsep=2pt]
\item[\ref{it:H0GetaC-compact-scaling-GSC}]
	The proof of Proposition \ref{p:H0GetaC-compact-pcf}-\ref{it:H0GetaC-compact-scaling-pcf} remains
	valid also in this case, except that Proposition \ref{p:harmonic-pcf}-\ref{it:harmonic-pcf},
	Lemmas \ref{l:scaling-energy-meas} and \ref{l:GetaC-scaling} need to be replaced by
	Lemma \ref{l:Fw-star}-\ref{it:Fw-star-H0}, Lemmas \ref{l:scaling-energy-meas-GSC} and \ref{l:GetaC-scaling-GSC}, respectively.
\item[\ref{it:H0GetaC-compact-GSC}]Let $\{h_{n}\}_{n\in\mathbb{N}}\subset\mathcal{H}_{0}(\attainsss(\eta,C))$,
	so that $\{\Gamma(h_{n},h_{n})\}_{n\in\mathbb{N}}\subset\mathcal{P}(K)$ and hence
	$\{h_{n}\}_{n\in\mathbb{N}}\subset\{h\in\mathcal{H}_{0}\mid(\Gamma(h,h)(K)=)\,\mathcal{E}(h,h)=1\}$.
	Setting $v_{n}:=h_{n}-\int_{K}h_{n}\,d\measure\in\mathcal{H}_{0}$ for each $n\in\mathbb{N}$,
	by $\mathcal{E}(\one_{K},\one_{K})=0$ and Lemma \ref{l:global-PI-GSC} we have
	$\mathcal{E}_{1}(v_{n},v_{n})\leq C_{\mathrm{P}}+1$ for any $n\in\mathbb{N}$,
	therefore by \cite[Chapter 10, Theorem 7]{Lax} there exist $h\in\mathcal{F}$
	and a strictly increasing sequence $\{n'_{j}\}_{j\in\mathbb{N}}\subset\mathbb{N}$ such that
	$\{v_{n'_{j}}\}_{j\in\mathbb{N}}$ converges weakly in $(\mathcal{F},\mathcal{E}_{1})$ to $h$,
	and $h\in\mathcal{H}_{0}$ since $\mathcal{H}_{0}$ is weakly closed in $(\mathcal{F},\mathcal{E}_{1})$.
	Further, recalling that $\mathcal{P}(K)$ is a compact metrizable topological space
	by \cite[Theorems 9.1.5 and 9.1.9]{Str}, we can choose a strictly increasing sequence
	$\{j_{k}\}_{k\in\mathbb{N}}\subset\mathbb{N}$ such that
	$\{\Gamma(h_{n_{k}},h_{n_{k}})\}_{k\in\mathbb{N}}$ converges to some $\nu\in\mathcal{P}(K)$
	in $\mathcal{P}(K)$, where $n_{k}:=n'_{j_{k}}$, and then
	$(\vartheta,\nu)\in\attainsss(\eta,C)$ for some metric $\vartheta$ on $K$ by
	$\{h_{n_{k}}\}_{k\in\mathbb{N}}\subset\mathcal{H}_{0}(\attainsss(\eta,C))$ and
	Corollary \ref{c:GetaC-compact-GSC}. In particular, for any $w\in W_{*}$ we have
	$\nu(F_{w}(V_{0}))=0$ by \eqref{e:GetaC-GSC} and Theorem \ref{t:condition-PHI-GSC} and
	\begin{equation}\label{eq:H0GetaC-compact-GSC-boundary-energy}
	\limsup_{k\to\infty}\frac{1}{r_{w}}\mathcal{E}(v_{n_{k}}\circ F_{w},v_{n_{k}}\circ F_{w})
		\leq\limsup_{k\to\infty}\Gamma(h_{n_{k}},h_{n_{k}})(K_{w})
		\leq\nu(K_{w})
	\end{equation}
	by $\mathcal{E}(\one_{K},\one_{K})=0$ and \eqref{e:pre-scaling-energy-meas-GSC}
	with $A=K_{w}$, thus $\lim_{k\to\infty}\mathcal{E}_{1}(h-v_{n_{k}},h-v_{n_{k}})=0$
	by Proposition \ref{p:harmonic-func-compact-GSC} and hence
	$\lim_{k\to\infty}\mathcal{E}(h-h_{n_{k}},h-h_{n_{k}})=0$ by $\mathcal{E}(\one_{K},\one_{K})=0$.
	It follows that \eqref{e:energy-meas-converge} holds for any Borel subset $A$ of $K$,
	so that $\Gamma(h,h)\in\mathcal{P}(K)$ and $\{\Gamma(h_{n_{k}},h_{n_{k}})\}_{k\in\mathbb{N}}$
	converges to $\Gamma(h,h)$ in $\mathcal{P}(K)$, whence
	$(\vartheta,\Gamma(h,h))=(\vartheta,\nu)\in\attainsss(\eta,C)$.
	Therefore $h\in\mathcal{H}_{0}(\attainsss(\eta,C))$, which together with
	$\lim_{k\to\infty}\mathcal{E}(h-h_{n_{k}},h-h_{n_{k}})=0$ proves that
	$\widetilde{\mathcal{H}}_{0}(\attainsss(\eta,C))$ is (sequentially) compact
	in norm in $(\mathcal{F}/\mathbb{R}\one_{K},\mathcal{E})$.
	\qedhere
\end{enumerate}
\end{proof}

\begin{proof}[Proof of Theorem \textup{\ref{t:attain-harmonic-func-GSC}}]
By the assumption $\attainsss(\eta,C)\not=\emptyset$ we can take
$(\theta,\mu)\in\attainsss(\eta,C)$. Choose $u\in\mathcal{H}_{0}$ so that $\mathcal{E}(u,u)>0$; 
such $u$ exists, e.g., by \cite[Propositions 3.7 and 3.10]{Kaj22}.
Let $f\colon K\to[0,\infty)$ be a Borel measurable $\mu$-version of $d\Gamma(u,u)/d\mu$.
Then $\mu\bigl(f^{-1}((0,\infty))\bigr)>0$ by $\int_{K}f\,d\mu=\Gamma(u,u)(K)=\mathcal{E}(u,u)>0$,
and therefore by Lemma \ref{l:Lebesgue-points-GSC} there exists $x\in K$
with the properties \eqref{e:Lebesgue-points-GSC} and $f(x)>0$.
Let $\omega\in\pi^{-1}(x)$, and for each $n\in\mathbb{N}\cup\{0\}$,
as in Proposition \ref{p:Lebesgue-pt-scaling-GSC} define
$(\theta_{n},\mu_{n}):=(\theta_{[\omega]_{n}},\mu_{[\omega]_{n}})\in \contfunc(K\times K)\times\mathcal{P}(K)$
by \eqref{e:GetaC-scaling} and $u_{n}:=u_{[\omega]_{n}}\in\mathcal{F}$
by \eqref{e:scaling-energy-meas} with $w=[\omega]_{n}$, so that
$\{(\theta_{n},\mu_{n})\}_{n\in\mathbb{N}\cup\{0\}}\subset\attainsss(\eta,C)$
by Lemma \ref{l:GetaC-scaling-GSC} and
$\{u_{n}\}_{n\in\mathbb{N}\cup\{0\}}\subset\{h\in\mathcal{H}_{0}\mid\mathcal{E}(h,h)=1\}$
by Lemma \ref{l:Fw-star}-\ref{it:Fw-star-H0}. Then setting
$v_{n}:=u_{n}-\int_{K}u_{n}\,d\measure\in\mathcal{H}_{0}$ for each $n\in\mathbb{N}\cup\{0\}$,
by $\mathcal{E}(\one_{K},\one_{K})=0$ and Lemma \ref{l:global-PI-GSC} we have
$\mathcal{E}_{1}(v_{n},v_{n})\leq C_{\mathrm{P}}+1$ for any $n\in\mathbb{N}\cup\{0\}$,
hence by \cite[Chapter 10, Theorem 7]{Lax} there exist $h\in\mathcal{F}$
and a strictly increasing sequence $\{n'_{j}\}_{j\in\mathbb{N}}\subset\mathbb{N}$ such that
$\{v_{n'_{j}}\}_{j\in\mathbb{N}}$ converges weakly in $(\mathcal{F},\mathcal{E}_{1})$ to $h$,
and $h\in\mathcal{H}_{0}$ since $\mathcal{H}_{0}$ is weakly closed in $(\mathcal{F},\mathcal{E}_{1})$.
Further, by $\{(\theta_{n'_{j}},\mu_{n'_{j}})\}_{j\in\mathbb{N}}\subset\attainsss(\eta,C)$,
Proposition \ref{p:GetaC-compact-GSC} and the metrizability of
$\contfunc(K\times K)\times\mathcal{P}(K)$ we can choose a strictly increasing sequence
$\{j_{k}\}_{k\in\mathbb{N}}\subset\mathbb{N}$ such that
$\{(\theta_{n_{k}},\mu_{n_{k}})\}_{k\in\mathbb{N}}$ converges to some
$(\vartheta,\nu)\in\attainsss(\eta,C)$ in $\contfunc(K\times K)\times\mathcal{P}(K)$,
where $n_{k}:=n'_{j_{k}}$. Then for any $w\in W_{*}$, we have $\nu(F_{w}(V_{0}))=0$
by \eqref{e:GetaC-GSC} and Theorem \ref{t:condition-PHI-GSC}, and by using
$\mathcal{E}(\one_{K},\one_{K})=0$, \eqref{e:pre-scaling-energy-meas-GSC} with $A=K_{w}$,
\eqref{e:scaling-energy-meas} from Lemma \ref{l:scaling-energy-meas-GSC},
the definition of $\mu_{n_{k}}=\mu_{[\omega]_{n_{k}}}$ from \eqref{e:GetaC-scaling},
\eqref{e:Lebesgue-pt-scaling} from Lemma \ref{l:Lebesgue-pt-scaling-GSC}, $f(x)\in(0,\infty)$
and the convergence of $\{\mu_{n_{k}}\}_{k\in\mathbb{N}}$ to $\nu$ in $\mathcal{P}(K)$, we obtain
\begin{align}
&\limsup_{k\to\infty}\frac{1}{r_{w}}\mathcal{E}(v_{n_{k}}\circ F_{w},v_{n_{k}}\circ F_{w})\nonumber\\
	&\leq\limsup_{k\to\infty}\Gamma(u_{n_{k}},u_{n_{k}})(K_{w})
	=\limsup_{k\to\infty}\frac{\Gamma(u,u)(K_{[\omega]_{n_{k}}w})}{\Gamma(u,u)(K_{[\omega]_{n_{k}}})}\nonumber\\
&=\limsup_{k\to\infty}\frac{\Gamma(u,u)(K_{[\omega]_{n_{k}}w})/\mu(K_{[\omega]_{n_{k}}w})}{\Gamma(u,u)(K_{[\omega]_{n_{k}}})/\mu(K_{[\omega]_{n_{k}}})}\mu_{n_{k}}(K_{w})\nonumber\\
&=\frac{f(x)}{f(x)}\limsup_{k\to\infty}\mu_{n_{k}}(K_{w})
	\leq\nu(K_{w}).
\label{eq:attain-harmonic-func-GSC-boundary-energy}
\end{align}
Thus $\lim_{k\to\infty}\mathcal{E}_{1}(h-v_{n_{k}},h-v_{n_{k}})=0$
by Proposition \ref{p:harmonic-func-compact-GSC}, hence
$\lim_{k\to\infty}\mathcal{E}(h-u_{n_{k}},h-u_{n_{k}})=0$ by $\mathcal{E}(\one_{K},\one_{K})=0$,
and it follows from Proposition \ref{p:Lebesgue-pt-scaling-GSC} that
$\{\mu_{n_{k}}\}_{k\in\mathbb{N}}$ converges to $\Gamma(h,h)\in\mathcal{P}(K)$ in
$\mathcal{P}(K)$, whence $(\vartheta,\Gamma(h,h))=(\vartheta,\nu)\in\attainsss(\eta,C)$.
\end{proof}

\begin{remark}\label{rmk:Lebesgue-pt-scaling-subseq}
Note that the relatively short proof of Proposition \ref{p:harmonic-func-compact-GSC}
above is enabled by the assumed properties of the Radon measure $\nu$ on $K$.
The existence of such $\nu$ seems difficult to verify for a general sequence
$\{h_{k}\}_{k\in\mathbb{N}}\subset\mathcal{H}_{0}$ converging weakly in
$(\mathcal{F},\mathcal{E}_{1})$, but can be obtained via the compactness of
$\attainsss(\eta,C)$ from Proposition \ref{p:GetaC-compact-GSC} in the situations of
the proofs of Theorems \ref{thm:H0GetaC-compact-GSC} and \ref{t:attain-harmonic-func-GSC}
above as observed in \eqref{eq:H0GetaC-compact-GSC-boundary-energy} and
\eqref{eq:attain-harmonic-func-GSC-boundary-energy}.
In fact, Hino \cite[Proposition 4.18]{Hin13} has proved a similar sequential
compactness \emph{without assuming the existence of such $\nu$},
at the price of its long difficult proof given in \cite[Section 5]{Hin13}.
\end{remark}

\section{Further remarks and open problems} \label{sec:open-problem}
We conclude the present paper with mentioning some related open problems.
\begin{problem}
	Does $\dcw <\infty$ characterize the elliptic Harnack inequality for symmetric jump process?
\end{problem}
\noindent It is not clear if the equivalence between \ref{it:ehichar-a} and \ref{it:ehichar-b}
in Theorem \ref{t:ehichar} extends to jump processes. Despite the progress made in the diffusion case,
the characterization and stability of the elliptic Harnack inequality is still open for jump processes.

Our study of the Gaussian uniformization problem in Section \ref{sec:GUP} gives only partial
answers, both for a general MMD space and for the MMD space of Brownian motion on $\mathbb{R}^{n}$,
except for an explicit answer for $\mathbb{R}$ in Theorem \ref{t:1dgaussian} and
an implicit (unsatisfactory) one for $\mathbb{R}^{2}$ in Proposition \ref{p:strongainf}.
In particular, the following problem is left open.

\begin{problem} \label{p:gaussianadmissible}
	Characterize explicitly all the Gaussian admissible measures for
	the MMD space of Brownian motion on $\mathbb{R}^{n}$, $n \ge 2$
	(see \eqref{e:gaussianadmissible} for the definition).
\end{problem}

For Problem \ref{p:pcf-non-attainment} below, let $(K,R_{\mathcal{E}},\measure,\mathcal{E},\mathcal{F})$
denote the MMD space resulting as in Subsection \ref{ssec:pfcsss-preliminaries}
from a post-critically finite self-similar structure
$\mathcal{L}=(K,S,\{F_{i}\}_{i\in S})$ with $\#S\geq 2$ and $K$ connected and a regular harmonic structure
$(D,\mathbf{r})$ on $\mathcal{L}$. Recall that $\mathcal{H}_{0}$ has been defined in Definition \ref{d:harmonic-pcf}
as $\mathcal{H}_{0}:=\{h\in \contfunc (K)\mid\textrm{$h$ is $0$-harmonic}\}$ under this setting.
\begin{problem} \label{p:pcf-non-attainment}
Provide simple sufficient conditions for the non-attainment of the conformal walk dimension for
the MMD space $(K,R_{\mathcal{E}},\measure,\mathcal{E},\mathcal{F})$.
\end{problem}
\noindent In view of the non-attainment results in Subsection \ref{ssec:examples} for the Vicsek set
(Corollary \ref{cor:Vicsek}) and the higher-dimensional Sierpi\'{n}ski gaskets
(Theorem \ref{thm:SGN-NOT-attained}), it seems natural to expect that
the conformal walk dimension would typically fail to be attained for the MMD space
$(K,R_{\mathcal{E}},\measure,\mathcal{E},\mathcal{F})$ as in Problem \ref{p:pcf-non-attainment}.
This expectation, however, does not seem very easy to verify, since the behavior of
the linear maps $\mathcal{H}_{0}\ni h\mapsto h\circ F_{w}\in\mathcal{H}_{0}$ could be
difficult to analyze for a given pair of $\mathcal{L}$ and $(D,\mathbf{r})$ and might
not allow a proof by contradiction based
on Theorem \ref{t:attain-harmonic-func} and Proposition \ref{p:H0GetaC-compact-pcf}
as achieved in the proof of Theorem \ref{thm:SGN-NOT-attained}.

For Problems \ref{p:SCattainment}, \ref{p:SCdoublingharmonic} and \ref{p:SCenergyfullsupport} below,
let $N,K,d,\measure,V_{0},(\mathcal{E},\mathcal{F}),\mathcal{H}_{0}$ be as introduced
in Subsection \ref{ssec:attain-harmonic-func-GSCs} as pieces of the framework of the canonical
self-similar Dirichlet form on an arbitrary generalized Sierpi\'{n}ski carpet $K$;
in particular, recall that $\mathcal{H}_{0}$ has been defined in Definition \ref{d:harmonic-GSC} as
$\mathcal{H}_{0}:=\{h\in\mathcal{F}\mid\textrm{$h$ is $\mathcal{E}$-harmonic on $K\setminus V_{0}=K\cap(0,1)^{N}$}\}$.
\begin{problem} \label{p:SCattainment}
	Is the conformal walk dimension attained for the canonical MMD space
	$(K,d,\measure,\mathcal{E},\mathcal{F})$ over the generalized Sierpi\'{n}ski carpet $K$?
\end{problem}
\noindent As a matter of fact, the authors have proved in a recent ongoing work that
\emph{the conformal walk dimension is NOT attained for the two-dimensional standard Sierpi\'{n}ski carpet}.
It is therefore likely that the answer to Problem \ref{p:SCattainment} would be
negative also for a given generalized Sierpi\'{n}ski carpet. Note, however, that in view of
Theorem \ref{t:attain-harmonic-func-GSC} the negative answer to Problem \ref{p:SCattainment} would mean
only that the energy measure $\Gamma(h,h)$ of every $h\in\mathcal{H}_{0}\setminus\mathbb{R}\one_{K}$
would fail to satisfy $\Gamma(h,h)\in\mathcal{G}(K,d,\measure,\mathcal{E},\mathcal{F})$,
which is a much stronger requirement than just \ref{VD} of $(K,d,\Gamma(h,h))$.
In particular, the following problem remains non-trivial regardless of the actual
answer to Problem \ref{p:SCattainment}.
\begin{problem} \label{p:SCdoublingharmonic}
	Does there exist $h\in\mathcal{H}_{0}\setminus\mathbb{R}\one_{K}$
	such that its energy measure $\Gamma(h,h)$ satisfies the volume doubling property
	with respect to the Euclidean metric $d$?
\end{problem}
\noindent Problem \ref{p:SCdoublingharmonic} appears very challenging since
we do not know even the answer to the following much simpler question.
\begin{problem} \label{p:SCenergyfullsupport}
	Does there exist $h\in\mathcal{H}_{0}\setminus\mathbb{R}\one_{K}$
	such that its energy measure $\Gamma(h,h)$ has full support?
\end{problem}
\noindent It is tempting to conjecture that the energy measure $\Gamma(h,h)$
of every $h\in\mathcal{H}_{0}\setminus\mathbb{R}\one_{K}$
has full support. This can be viewed as a unique continuation principle
for harmonic functions on a given generalized Sierpi\'{n}ski carpet.

We expect that there is a version of Theorems \ref{t:attain-harmonic-func} and \ref{t:attain-harmonic-func-GSC} for Ahlfors regular conformal dimension on self-similar spaces.
In particular, we expect that if the Ahlfors regular conformal dimension $p>1$ is attained on a self-similar space then there exists a ``$p$-harmonic function'' such that its
``energy measure'' is a $p$-Ahlfors regular measure with respect to a metric in the conformal gauge. This motivates the following problem.
\begin{problem}\label{p:nonlinear-energy}
	Define non-linear analogs of Dirichlet space, energy measure and harmonic functions on self-similar spaces (for example, the Sierpi\'{n}ski carpet).
\end{problem}

Using the theory of Dirichlet forms and its relationship to diffusion process we have a notion of Sobolev space $W^{1,2}$ on the Sierpi\'{n}ski carpet  whose seminorm can be formally thought of as the integral  $(\int \abs{\nabla f}^2 )^{1/2}$ on a dense subspace of $L^2$. For example in \cite{KZ92}, the $W^{1,2}$-seminorm  and the Dirichlet energy $\int \abs{\nabla f}^2$ is constructed as suitably renormailized version of discrete Dirichlet energy  $\sum_{x \sim y} (\hat{f}(x)-\hat{f}(y))^2$, where $\hat{f}$ is a discretization of the function $f$ on a sequence of  graph approximations $V_n$.
A key ingredient in the constructions in \cite{BB89} and \cite{KZ92} is the following submultiplicative and supermultiplicative inequalities for resistance between opposite faces. Let $R_n$ denote the resistance between the opposite faces for the $n$-th level approximation of the Sierpi\'{n}ski carpet. Then the  submultiplicative and supermultiplicative inequalities  are given by $R_n R_m \gtrsim R_{m+n}$ and $R_n R_m \lesssim R_{m+n}$. These inequalities have been generalized in the non-linear context for the Sierpi\'{n}ski carpet by Bourdon and Kleiner for any $p>1$  \cite[Lemma 4.4]{BK13}.  This suggests that one could construct a non-linear version of Dirichlet form with a $W^{1,p}$ seminorm (formally denoted by $\norm{\nabla f}_p$). This seminorm is defined on a dense subspace $\mathcal F_p$ of $L^p$ and is  conjectured to  satisfy the following properties: 
\begin{enumerate}[label=\textup{\arabic*.},align=left,leftmargin=*,topsep=5pt,parsep=0pt,itemsep=2pt]
\item (Closability) If $f_n$ is a Cauchy sequence in the $\norm{\nabla f}_p$-seminorm and if $f_n \to 0$ in $L^p$ then $f_n$ converges to $0$ in the $\norm{\nabla f}_p$-seminorm.
\item (Regularity) 
	Let $K$ denote the Sierpi\'{n}ski carpet. Then
	$\mathcal{F}_p \cap \contfunc(K)$ is dense with respect to the uniform norm in $\contfunc(K)$
	and is dense with respect to the $f \mapsto \norm{\nabla f}_p + \norm{f}_p$ norm. 
\end{enumerate}
This yields the notion of $p$-harmonic functions which are defined as minimizers of the $p$-Dirichlet energy $\norm{\nabla f}_p^p$.
On the basis of Theorem \ref{t:attain-harmonic-func-GSC}, we have the following conjecture:
if the Sierpi\'{n}ski carpet attains the Ahlfors regular conformal dimension, then there exists
a $p$-harmonic function whose energy measure (formally written as the measure $A \mapsto \int_A \abs{\nabla f}^p$)
is an optimal Ahlfors regular measure, where $p$ is the Ahlfors regular conformal dimension of the Sierpi\'{n}ski carpet.
The previous discussion on unique continuation question in the linear case ($p=2$) also applies
to $p$-harmonic functions due to the above mentioned relationship to the attainment problem for the Ahlfors regular conformal dimension.
These conjectures serve as a motivation to develop a theory of non-linear Dirichlet forms on fractals and develop methods to obtain the elliptic Harnack inequality and quantitative unique continuation principle for $p$-harmonic functions.  Added in revision: there has been recent progress on Problem \ref{p:nonlinear-energy}  in \cite{Kig21,Shi}. These results are quite satisfactory when $p$ is strictly larger than the Ahlfors regular conformal dimension.

\begin{acknowledgements}
The authors would like to thank Ryosuke Shimizu for his valuable comments on an
earlier version of this paper, and the anonymous referees for
their careful reading of this manuscript and helpful suggestions.
The authors would also like to express gratitude to Shouhei Honda for having communicated
to the first-named author the idea of taking scaling limits at Lebesgue points to
reduce the analysis to the case of harmonic functions, which the authors have adapted
in Subsections \ref{ssec:attain-harmonic-func-pcf} and \ref{ssec:attain-harmonic-func-GSCs}.
This paper was revised while the second-named author was at the Mathematical Sciences
Research Institute in Berkeley, California, during the Spring 2022 semester on
Analysis and Geometry of Random Spaces, which was supported by
the National Science Foundation under Grant No. DMS-1928930.
\end{acknowledgements}

\begin{appendices}

\section*{Appendix: Index of words, phrases, symbols and abbreviations}
\addcontentsline{toc}{section}{Appendix: Index of words, phrases, symbols and abbreviations}

\subsection*{Words and phrases}

\begin{itemize}[align=left,leftmargin=*,topsep=5pt,parsep=0pt,itemsep=2pt]
\item $(1,p)$-Poincar\'e inequality --- Definition \ref{d:p-poincare}
\item attainment problem --- Problem \ref{prb:attainment-Gaussian-unif}, second paragraph of Section \ref{sec:GUP}
\item $A_{\infty}$-related --- Definition \ref{d:ainf}
\item bi-Lipschitz, bi-Lipschitz equivalent --- Definition \ref{d:qs}
\item boundary of a hyperbolic space --- fourth paragraph of Subsection \ref{s:gromov}
\item caloric --- Definition \ref{d:harmonic}
\item canonical Dirichlet form on $\GSC(N,l,S)$ --- Definition \ref{dfn:GSCDF}
\item conformal gauge --- Definition \ref{d:cgauge}
\item critical set --- Definition \ref{d:V0Vstar}
\item distortion function --- Definition \ref{d:qs}
\item doubling (measure) --- Definition \ref{d:vd-rvd}
\item doubling (metric space) --- first paragraph of Subsection \ref{s:hf}
\item $\mathcal{E}$-harmonic --- Definition \ref{d:harmonic}
\item energy measure --- Definition \ref{d:EnergyMeas}
\item full quasi-support --- Definition \ref{d:admissible}
\item Gaussian admissible measures --- \eqref{e:gauss}, \eqref{e:gaussianadmissible}
\item Gaussian uniformization problem --- Problem \ref{prb:attainment-Gaussian-unif}, second paragraph of Section \ref{sec:GUP}
\item generalized Sierpi\'{n}ski carpet --- Definition \ref{dfn:GSC}
\item gentle --- Definition \ref{d:gentle}
\item Gromov hyperbolic space --- second paragraph of Subsection \ref{s:gromov}
\item harmonic structure --- Definition \ref{d:harmonic-str}
\item $K_D$-doubling (metric space) --- first paragraph of Subsection \ref{s:hf}
\item $K$-gentle, $(K_h,K_v)$-gentle --- Definition \ref{d:gentle}
\item $K_P$-uniformly perfect --- first paragraph of Subsection \ref{s:hf}
\item maximal semi-metric induced by $h$ --- Definition \ref{d:maxmetric}
\item Menger sponge --- paragraph following Definition \ref{dfn:GSC}
\item metric doubling property --- first paragraph of Subsection \ref{s:hf}
\item metric measure space --- first paragraph of Subsection \ref{ssec:MMD-EnergyMeas}
\item MMD (metric measure Dirichlet) space --- second paragraph of Subsection \ref{ssec:MMD-EnergyMeas}
\item minimal energy-dominant measure --- Definition \ref{d:minimal-energy-dominant}
\item nest --- paragraph before Definition \ref{d:admissible}
\item $n$-harmonic ($\mathcal{E}$-harmonic on $K\setminus V_{n}$) --- Definition \ref{d:harmonic-pcf}
\item $N$-dimensional (standard) Sierpi\'{n}ski gasket --- Example \ref{exmp:SGs}
\item partition (of the shift space $\Sigma$) --- Definition \ref{d:partition}
\item post-critical set --- Definition \ref{d:V0Vstar}
\item post-critically finite (p.-c.f.)\ --- Definition \ref{d:V0Vstar}
\item power quasisymmetry --- Definition \ref{d:qs}
\item quasi-closed --- paragraph before Definition \ref{d:admissible}
\item quasi-open --- paragraph before Definition \ref{d:admissible}
\item quasisymmetry --- Definition \ref{d:qs}
\item regular (harmonic structure) --- Definition \ref{d:harmonic-str}
\item resistance form --- paragraph following Definition \ref{d:harmonic-str}
\item resistance metric --- paragraph following Definition \ref{d:harmonic-str}, especially \ref{it:RF4}
\item reverse volume doubling property --- Definition \ref{d:vd-rvd}
\item $(R_{\mathcal{E}},\mu)$-Lebesgue point --- Lemma \ref{l:Lebesgue-points}
\item scale --- Definition \ref{d:scale-pcf}
\item self-similar measure (with weight $(r_{i}^{d_{\mathrm{H}}})_{i\in S}$) --- paragraph before Lemma \ref{l:volume-pcf}
\item self-similar structure --- Definition \ref{d:sss}
\item Sierpi\'{n}ski carpet --- paragraph following Definition \ref{dfn:GSC}
\item smooth measure --- Definition \ref{d:smooth}
\item strong $A_{\infty}$-related --- Definition \ref{d:strongainf}
\item two-dimensional standard Sierpi\'{n}ski carpet --- paragraph following Definition \ref{dfn:GSC}
\item uniformly perfect --- first paragraph of Subsection \ref{s:hf}
\item upper gradient --- Definition \ref{d:uppergrad}
\item Vicsek set --- Example \ref{exmp:Vicsek}
\item volume doubling property --- Definition \ref{d:vd-rvd}
\item $\varepsilon$-net --- Definition \ref{d:net}
\item $\mu$-intrinsic metric --- Definition \ref{d:dint-pcf}
\end{itemize}

\subsection*{Symbols}

\begin{itemize}[align=left,leftmargin=*,topsep=5pt,parsep=0pt,itemsep=2pt]
\item $\one_{A}=\one_{A}^{X}$: indicator function of $A\subset X$ on a set $X$ --- Notation \ref{ntn:intro}-\ref{it:ntn-intro-indicator}
\item $\sA(X,d,m,\sE,\sF)$: admissible measures --- Definition \ref{d:admissible}
\item $B(x,r)=B_{d}(x,r)$: open ball in metric $d$ --- first paragraph of Subsection \ref{ssec:MMD-EnergyMeas}
\item $\overline{B}(x,r)=\overline{B}_{d}(x,r)$: closed ball in metric $d$ --- first paragraph of Subsection \ref{ssec:MMD-EnergyMeas}
\item $c_{R_{\mathcal{E}}}$ --- Lemma \ref{l:RE-scaling-pcf}
\item $\contfunc(X)$: space of continuous functions on $X$ --- Notation \ref{ntn:intro}-\ref{it:ntn-intro-contfunc}
\item $\contfunc_{\mathrm{c}}(X)$: space of continuous functions on $X$ with compact supports --- Notation \ref{ntn:intro}-\ref{it:ntn-intro-contfunc}
\item $\Capa_1(A)$: $1$-capacity of a set $A$ --- \eqref{e:defCap1}
\item $\Capa(A,B)$: capacity between sets $A$ and $B$ --- \eqref{e:defCap}
\item $\mathcal{C}_{\mathcal{L}}$: critical set of $\mathcal{L}$ --- Definition \ref{d:V0Vstar}
\item $\dcw$: conformal walk dimension --- Definition \ref{dfn:dcw}
\item $d_{\mathrm{f}}$: Euclidean Hausdorff dimension of $\GSC(N,l,S)$ --- Framework \ref{frmwrk:GSC}
\item $d_{h}$: maximal semi-metric induced by $h$ --- Definition \ref{d:maxmetric}
\item $d_{\mathrm{H}}$: Hausdorff dimension of $(K,R_{\mathcal{E}})$ --- Lemma \ref{l:volume-pcf} and preceding paragraph
\item $\diam_{d}(A)=\diam(A,d)$: diameter of $A$ in metric $d$ --- first paragraph of Subsection \ref{ssec:MMD-EnergyMeas}
\item $d_{\on{int}}$: intrinsic metric --- Definition \ref{d:dint}
\item $d_{\on{int}}^{\mu}$: $\mu$-intrinsic metric --- Definition \ref{d:dint-pcf}
\item $d_{\mathrm{w}}$: walk dimension of the Dirichlet form on $\GSC(N,l,S)$ --- Definition \ref{dfn:GSCDF}
\item $(D,\mathbf{r})$: harmonic structure --- Definition \ref{d:harmonic-str}
\item $\sD_l(B)$: descendants of generation $l$ --- \eqref{e:defD}
\item $D_{\mathcal{S}}$: combinatorial metric on the hyperbolic filling --- Definition \ref{d:hypfilling}.
\item $\mathcal{E}_{1}$: inner product on $\mathcal{F}$ for a Dirichlet form $(\mathcal{E},\mathcal{F})$ --- second paragraph of Subsection \ref{ssec:MMD-EnergyMeas}
\item $\mathcal{E}^{(n)}$: discrete Dirichlet form on $V_{n}$ --- \eqref{e:E0-pcf}, \eqref{e:Em-pcf}
\item $F_{w}$ --- Definition \ref{d:shift}
\item $\sF_e$: extended Dirichlet space --- Definition \ref{d:harmonic}
\item $\sF_{\on{loc}}$: space of functions locally in $\sF$ --- Definition \ref{d:dint}
\item $\GSC(N,l,S)$ --- Framework \ref{frmwrk:GSC}
\item $\sG(X,d,m,\sE,\sF)$: Gaussian admissible measures --- \eqref{e:gauss}, \eqref{e:gaussianadmissible}
\item $\sG_\beta(X,d,m,\sE,\sF)$: sub-Gaussian admissible measures --- \eqref{e:subgaussadmissible}
\item $\attainsss=\attainsss_{\mathcal{L},(D,\mathbf{r})}$ (for post-critically finite $\mathcal{L}$) --- Definition \ref{d:GetaC-pcf}
\item $\attainsss=\attainsss_{N,l,S}$ (for $\GSC(N,l,S)$) --- Definition \ref{d:GetaC-GSC}
\item $\attainsss(\eta,C)=\attainsss_{\mathcal{L},(D,\mathbf{r})}(\eta,C)$ (for post-critically finite $\mathcal{L}$) --- Definition \ref{d:GetaC-pcf}
\item $\attainsss(\eta,C)=\attainsss_{N,l,S}(\eta,C)$ (for $\GSC(N,l,S)$) --- Definition \ref{d:GetaC-GSC}
\item $\homeo^{+}$: group of homeomorphisms of $[0,\infty)$ --- Definition \ref{d:GetaC-pcf}
\item $\mathcal{H}_{0}$ (for $\GSC(N,l,S)$) --- Definition \ref{d:harmonic-GSC}
\item $\mathcal{H}_{0}(\mathcal{Z})$: set of $0$-harmonic functions attaining $\dcw$ --- \eqref{e:H0Z-pcf} in Definition \ref{d:GetaC-pcf}
\item $\widetilde{\mathcal{H}}_{0}(\mathcal{Z})$ --- \eqref{e:H0tildeZ-pcf} in Definition \ref{d:GetaC-pcf}
\item $\mathcal{H}_{n}$ (for post-critically finite $\mathcal{L}$) --- Definition \ref{d:harmonic-pcf}
\item $\interior_{\mathbb{R}^{N}}(A)$: interior of $A\subset\mathbb{R}^{N}$ in $\mathbb{R}^{N}$ --- paragraph before Definition \ref{dfn:GSC}
\item $\mathcal{I}_{0}$: group of cubic symmetries of $\GSC(N,l,S)$ --- Definition \ref{dfn:GSC-isometry}
\item $\sJ(X,d)$: conformal gauge of $(X,d)$ --- Definition \ref{d:cgauge}
\item $K_{s}(x)$ --- Definition \ref{d:scale-pcf}
\item $K_{w}$ --- Definition \ref{d:shift}
\item $\mathcal{L}=(K,S,\{F_{i}\}_{i\in S})$: self-similar structure --- Definition \ref{d:sss}
\item $\on{Lip}u(x)$: pointwise Lipschitz constant --- Definition \ref{d:lip}
\item $\on{Lip}(X)$: space of Lipschitz functions on $X$ --- Definition \ref{d:lip}
\item $N^{1,p}(X)=N^{1,p}(X,d,m)$ --- Definition \ref{d:uppergrad}
\item $N_{\on{loc}}^{1,p}(X)=N_{\on{loc}}^{1,p}(X,d,m)$ --- Definition \ref{d:uppergrad}
\item $\mathbb{N}$: set of \emph{positive} integers --- Notation \ref{ntn:intro}-\ref{it:ntn-intro-natural-numbers}
\item $\sN$: non-peripheral vertices --- Definition \ref{defE}
\item $\mathcal{P}(K)$: set of Borel probability measures on $K$ --- \eqref{eq:PK-pcf} in Definition \ref{d:GetaC-pcf}
\item $\mathcal{P}_{\mathcal{L}}$: post-critical set of $\mathcal{L}$ --- Definition \ref{d:V0Vstar}
\item $Q_{0}$: $N$-dimensional unit cube $[0,1]^{N}$ --- Framework \ref{frmwrk:GSC}
\item $Q_{1}$ --- Framework \ref{frmwrk:GSC}
\item $r_{w}$ (for post-critically finite $\mathcal{L}$) --- \eqref{e:Em-pcf}
\item $r_{w}$ (for $\GSC(N,l,S)$) --- Theorem \ref{t:condition-PHI-GSC}
\item $R_{\mathcal{E}}$: resistance metric --- paragraph following Definition \ref{d:harmonic-str}, especially \ref{it:RF4}
\item $\supp_{m}[f]$: support of $\lvert f\rvert\,dm$ --- second paragraph of Subsection \ref{ssec:MMD-EnergyMeas}
\item $\sS$: hyperbolic filling --- Definition \ref{d:hypfilling}, Subsection \ref{s:hf}
\item $\mathscr{S}=\{\Lambda_{s}\}_{s\in(0,1]}$: scale on $\Sigma$ associated with $\mathbf{r}=(r_{i})_{i\in S}$ --- Definition \ref{d:scale-pcf}
\item $u_{w}$ --- \eqref{e:scaling-energy-meas} in Lemma \ref{l:scaling-energy-meas}
\item $U_{s}(x)$ --- Definition \ref{d:scale-pcf}
\item $V_{n}$: set of level-$n$ boundary points of $\mathcal{L}$ --- Definition \ref{d:V0Vstar}
\item $V_{*}$: set of arbitrary level boundary points of $\mathcal{L}$ --- Definition \ref{d:V0Vstar}
\item $w^{\infty}$ --- Definition \ref{d:shift}
\item $W_{n}$: set of words of length $n$ --- Definition \ref{d:shift}
\item $W_{*}$: set of words of arbitrary length --- Definition \ref{d:shift}
\item $(X,d,m,\mathcal{E},\mathcal{F})$: MMD (metric measure Dirichlet) space --- Subsection \ref{ssec:MMD-EnergyMeas}
\item $\Gamma_{k+1}(B)$ --- Definition \ref{d:Gamk}
\item $\Gamma(f,f)$: energy measure --- Definition \ref{d:EnergyMeas}
\item $\tilde{\eta}$: multiplicative dual of $\eta$ --- paragraphs before Theorems \ref{t:attain-harmonic-func} and \ref{t:attain-harmonic-func-GSC}
\item $\theta_{w}$ (for $(\theta,\mu)\in\attainsss(\eta,C)$) --- \eqref{e:GetaC-scaling} in Lemma \ref{l:GetaC-scaling}
\item $\Lambda_{s}$ --- Definition \ref{d:scale-pcf}
\item $\Lambda_{s,x}$ --- Definition \ref{d:scale-pcf}
\item $\Lambda^{1}_{s,x}$ --- Definition \ref{d:scale-pcf}
\item $\mu_{w}$ (for $(\theta,\mu)\in\attainsss(\eta,C)$) --- \eqref{e:GetaC-scaling} in Lemma \ref{l:GetaC-scaling}
\item $\pi$ (for a self-similar structure $\mathcal{L}$): projection from $\Sigma$ to $K$ --- Definition \ref{d:sss}
\item $\sigma$ --- Definition \ref{d:shift}
\item $\sigma_{w}$ --- Definition \ref{d:shift}
\item $\Sigma$: (one-sided) shift space --- Definition \ref{d:shift}
\item $\Sigma_{w}$ --- Definition \ref{d:shift}
\item $[\omega]_{n}$ --- Definition \ref{d:shift}
\item $(\cdot)^{+}$: positive part of a $[-\infty,\infty]$-valued quantity --- Notation \ref{ntn:intro}-\ref{it:ntn-intro-max-min}
\item $(\cdot)^{-}$: negative part of a $[-\infty,\infty]$-valued quantity --- Notation \ref{ntn:intro}-\ref{it:ntn-intro-max-min}
\item $(\cdot \vert \cdot)_{\cdot}$: Gromov product --- \eqref{e:gromovproduct}
\item $\#A$: cardinality (number of elements) of a set $A$ --- Notation \ref{ntn:intro}-\ref{it:ntn-intro-cardinality}
\item $\leq$ (for measures) --- Notation \ref{ntn:intro}-\ref{it:ntn-intro-measure}
\item $\leq$ (for elements of $W_{*}$) --- Definition \ref{d:partition}
\item $\leq$ (for partitions of $\Sigma$): refinement --- Definition \ref{d:partition}
\item $\lesssim$: inequality up to constant multiples --- Notation \ref{ntn:intro}-\ref{it:ntn-intro-ineq-mod-const}
\item $\ll$: absolute continuity of a measure with respect to another --- Notation \ref{ntn:intro}-\ref{it:ntn-intro-measure}
\item $\vee$: maximum of two $[-\infty,\infty]$-valued quantities --- Notation \ref{ntn:intro}-\ref{it:ntn-intro-max-min}
\item $\wedge$: minimum of two $[-\infty,\infty]$-valued quantities --- Notation \ref{ntn:intro}-\ref{it:ntn-intro-max-min}
\end{itemize}

\subsection*{Abbreviations}

\begin{itemize}[align=left,leftmargin=*,topsep=5pt,parsep=0pt,itemsep=2pt]
\item $\operatorname{cap}(\beta)$: capacity estimate --- Definition \ref{d:cap}
\item $\operatorname{CS}(\beta)$: cutoff Sobolev inequality --- Definition \ref{d:pi-cs}
\item (E): enhanced subadditivity estimate --- Definition \ref{defE}
\item EHI: elliptic Harnack inequality --- Definition \ref{d:harnack}
\item \textup{(GSC1)},\textup{(GSC2)},\textup{(GSC3)},\textup{(GSC4)}: requirements for $\GSC(N,l,S)$ to be a generalized Sierpi\'{n}ski carpet --- Definition \ref{dfn:GSC}
\item \textup{(GSCDF)} --- Theorem \ref{thm:GSCDF}
\item (H1),(H2),(H3): conditions on weight functions in a hyperbolic filling --- Assumption \ref{a:h1-h3}
\item $\on{HKE}(\beta)$: heat kernel estimate --- Definition \ref{d:HKE}
\item \textup{(ND)$_{\mathbb{N}}$},\textup{(ND)$_{2}$},\textup{(NDF)}: equivalent formulations of \ref{it:GSC3} --- Proposition \ref{prop:ND}
\item $\on{PHI(\beta)}$: parabolic Harnack inequality --- Definition \ref{d:harnack}
\item $\operatorname{PI}(\beta)$: Poincar\'e inequality --- Definition \ref{d:pi-cs}
\item RVD: reverse volume doubling property --- Definition \ref{d:vd-rvd}
\item (S1),(S2): conditions on weight functions in a hyperbolic filling --- Theorem \ref{t:suff}
\item VD: volume doubling property --- Definition \ref{d:vd-rvd}
\end{itemize}

\end{appendices}

\end{document}